\documentclass[11pt,amscd,combelow,anysize,amsxtra]{amsart} %reqno places equation numbers on the right
\usepackage[margin=1in]{geometry}

\usepackage{amsmath}
\usepackage{hyperref}
\usepackage[capitalize]{cleveref}
\usepackage{graphicx}
\usepackage{refcount}

\usepackage{calligra}
\usepackage{amsthm,amssymb,mathrsfs,mathtools,bm,eucal,tensor} % math related
\usepackage{microtype} % latex technical issues
\usepackage[scaled]{beramono,berasans}
\usepackage{enumerate,comment,braket,xspace,tikz-cd} %utilities
\usetikzlibrary{arrows}
\usepackage[all,cmtip]{xy} % utilities
\usepackage[utf8]{inputenc} % input encoding
\usepackage[T1]{fontenc} % font encoding
\usepackage{lmodern}
\definecolor{linkcolor}{HTML}{005050}
\usepackage{xparse}
\usepackage{url}
\usepackage[toc,page]{appendix}
\usepackage{stmaryrd}
\usepackage{lscape}

\usepackage{subcaption}

\usepackage{pgfplots} %utilities
\pgfplotsset{compat=newest}

\theoremstyle{plain}
\newtheorem{thm-intro}{Theorem}
\newtheorem{thm}{Theorem}[section]

\newtheorem*{thm*}{Theorem}
\newtheorem{lem}[thm]{Lemma}
\newtheorem{lemma}[thm]{Lemma}
\newtheorem*{lem*}{Lemma}
\newtheorem{prop}[thm]{Proposition}

\newtheorem{Convention}[thm]{Convention}
\newtheorem{cor}[thm]{Corollary}

\newtheorem{cor-intro}[thm-intro]{Corollary}

\theoremstyle{definition}
\newtheorem{definition}[thm]{Definition}

\newtheorem{defin-intro}[thm-intro]{Definition}

\theoremstyle{remark}

\newtheorem{remark}[thm]{Remark}

\newtheorem{ex}[thm]{Example}

\newtheorem{eg-intro}[thm-intro]{Example}
\newtheorem{rem-intro}[thm-intro]{Remark}
\numberwithin{equation}{section}

\newcommand{\fa}{\mathfrak{a}}
\newcommand{\GL}{\mathrm{GL}}

\newcommand{\Sp}{\mathrm{Sp}}
\newcommand{\SO}{\mathrm{SO}}

\newcommand{\ol}{\overline}
\newcommand{\bG}{\mathbb{G}}
\newcommand{\Jac}{\mathrm{Jac}}
\newcommand{\bH}{\mathbb{H}}
\newcommand{\fc}{\mathfrak{c}}
\newcommand{\fh}{\mathfrak{h}}
\newcommand{\reg}{\mathrm{reg}}
\newcommand{\wt}{\widetilde}

\newcommand{\canT}{\mathbf{T}}
\newcommand{\canA}{\mathbf{A}}

\newcommand{\G}{\mathbb{G}}
\newcommand{\rG}{\mathrm{G}}

\newcommand{\ani}{\mathrm{ani}}
\renewcommand{\H}{\mathbf{H}}

\newcommand{\Gr}{\mathrm{Gr}}
\newcommand{\cl}{\mathrm{cl}}
\newcommand{\Orb}{\mathrm{Orb}}
\newcommand{\fgl}{\mathfrak{gl}}
\newcommand{\Herm}{\mathrm{Herm}}
\newcommand{\bfh}{\mathfrak{h}}
\newcommand{\bfg}{\mathfrak{g}}
\renewcommand{\t}{\, ^t}
\renewcommand{\bH}{\mathbf{H}}
\renewcommand{\bG}{\mathbf{G}}

\newcommand{\bF}{\mathbb{F}}
\newcommand{\e}{\epsilon}
\newcommand{\st}{\mathrm{st}}
\newcommand{\bR}{\mathrm{R}}
\newcommand{\fe}{\mathfrak{e}}
\newcommand{\cent}{\mathrm{cent}}
\newcommand{\br}{\mathrm{br}}
\newcommand{\unitary}{\mathbf{U}}
\newcommand{\Siegel}{\mathbf{S}}
\newcommand{\siegel}{\mathbf{S}}
\newcommand{\Prym}{\mathrm{Prym}}
\newcommand{\canC}{\mathbf{C}}
\newcommand{\Ob}{\mathrm{Ob}}

\newcommand{\fR}{\mathfrak{R}}

\newcommand{\Out}{\mathrm{Out}}
\newcommand{\vol}{\mathrm{vol}}
\newcommand{\Irr}{\mathrm{Irr}}
\newcommand{\good}{\mathrm{good}}
\newcommand{\bad}{\mathrm{bad}}

\newcommand{\cameral}{\underline}
\newcommand{\U}{\mathrm{U}}
\newcommand{\SOrb}{\mathrm{SO}}
\newcommand{\rs}{\mathrm{rs}}
\newcommand{\iso}{\overset{\sim}{\longrightarrow}}
\renewcommand{\top}{{2d}}

\usepackage{bbold}

\newcommand{\quash}[1]{}

\newcommand{\myfatslash}{\mathbin{\mkern-6mu\fatslash}}

\newcommand{\C}{\mathbb C}

\newcommand{\Q}{\mathbb Q}

\newcommand{\Z}{\mathbb Z}

\newcommand\git{{/\!\!/}}

\newcommand{\rH}{\mathrm H}

\newcommand{\fD}{\mathfrak D}

\newcommand{\fS}{\mathfrak S}

\newcommand{\fX}{\mathfrak X}

\newcommand{\fm}{\mathfrak m}

\newcommand{\fs}{\mathfrak s}
\newcommand{\ft}{\mathfrak t}

\newcommand{\fg}{\mathfrak{g}}

\newcommand{\cA}{\mathcal A}
\newcommand{\A}{\mathbb A}

\newcommand{\cB}{\mathcal B}

\newcommand{\cE}{\mathcal E}
\newcommand{\cF}{\mathcal F}

\newcommand{\cG}{\mathcal G}

\newcommand{\cL}{\mathcal L}
\newcommand{\cM}{\mathcal M}
\newcommand{\cN}{\mathcal N}
\newcommand{\cO}{\mathcal O}
\newcommand{\cP}{\mathcal P}

\newcommand{\cT}{\mathcal T}
\newcommand{\cU}{\mathcal U}

\DeclareFontFamily{U}{BOONDOX-calo}{\skewchar\font=45 }
\DeclareFontShape{U}{BOONDOX-calo}{m}{n}{<-> s*[1.05] BOONDOX-r-calo}{}
\DeclareFontShape{U}{BOONDOX-calo}{b}{n}{<-> s*[1.05] BOONDOX-b-calo}{}
\DeclareMathAlphabet{\mathcalboondox}{U}{BOONDOX-calo}{m}{n}
\newcommand{\bbD}{\mathbb D}

\newcommand{\bA}{\mathbf A}

\newcommand{\bP}{\mathbf P}

\newcommand{\bX}{\mathbf X}

\newcommand{\sX}{\mathscr X}

\newcommand{\wh}{\widehat}

\newcommand{\ul}{\underline}

\makeatletter
\let\save@mathaccent\mathaccent
\newcommand*\if@single[3]{%
	\setbox0\hbox{${\mathaccent"0362{#1}}^H$}%
	\setbox2\hbox{${\mathaccent"0362{\kern0pt#1}}^H$}%
	\ifdim\ht0=\ht2 #3\else #2\fi
}
\newcommand*\rel@kern[1]{\kern#1\dimexpr\macc@kerna}
\newcommand*\widebar[1]{\@ifnextchar^{{\wide@bar{#1}{0}}}{\wide@bar{#1}{1}}}
\newcommand*\wide@bar[2]{\if@single{#1}{\wide@bar@{#1}{#2}{1}}{\wide@bar@{#1}{#2}{2}}}
\newcommand*\wide@bar@[3]{%
	\begingroup
	\def\mathaccent##1##2{%
		\let\mathaccent\save@mathaccent
		\if#32 \let\macc@nucleus\first@char \fi
		\setbox\z@\hbox{$\macc@style{\macc@nucleus}_{}$}%
		\setbox\tw@\hbox{$\macc@style{\macc@nucleus}{}_{}$}%
		\dimen@\wd\tw@
		\advance\dimen@-\wd\z@
		\divide\dimen@ 3
		\@tempdima\wd\tw@
		\advance\@tempdima-\scriptspace
		\divide\@tempdima 10
		\advance\dimen@-\@tempdima
		\ifdim\dimen@>\z@ \dimen@0pt\fi
		\rel@kern{0.6}\kern-\dimen@
		\if#31
		\overline{\rel@kern{-0.6}\kern\dimen@\macc@nucleus\rel@kern{0.4}\kern\dimen@}%
		\advance\dimen@0.4\dimexpr\macc@kerna
		\let\final@kern#2%
		\ifdim\dimen@<\z@ \let\final@kern1\fi
		\if\final@kern1 \kern-\dimen@\fi
		\else
		\overline{\rel@kern{-0.6}\kern\dimen@#1}%
		\fi
	}%
	\macc@depth\@ne
	\let\math@bgroup\@empty \let\math@egroup\macc@set@skewchar
	\mathsurround\z@ \frozen@everymath{\mathgroup\macc@group\relax}%
	\macc@set@skewchar\relax
	\let\mathaccentV\macc@nested@a
	\if#31
	\macc@nested@a\relax111{#1}%
	\else
	\def\gobble@till@marker##1\endmarker{}%
	\futurelet\first@char\gobble@till@marker#1\endmarker
	\ifcat\noexpand\first@char A\else
	\def\first@char{}%
	\fi
	\macc@nested@a\relax111{\first@char}%
	\fi
	\endgroup
}
\makeatother

\newcommand{\ev}{\mathrm{ev}}
\usetikzlibrary{decorations.markings} %arrows for open immersions and closed immersions
\tikzset{
  closed/.style = {decoration = {markings, mark = at position 0.5 with { \node[transform shape, xscale = .8, yscale=.4] {/}; } }, postaction = {decorate} },
  open/.style = {decoration = {markings, mark = at position 0.5 with { \node[transform shape, scale = .7] {$\circ$}; } }, postaction = {decorate} }
}

\DeclareMathOperator{\Aut}{Aut}

\DeclareMathOperator{\codim}{codim}

\DeclareMathOperator{\End}{End}

\DeclareMathOperator{\Gal}{Gal}
\DeclareMathOperator{\Hom}{Hom}

\DeclareMathOperator{\Pic}{Pic}

\DeclareMathOperator{\Res}{Res}

\DeclareMathOperator{\Spec}{Spec}

\DeclareMathOperator{\Sym}{Sym}

\DeclareMathOperator{\tr}{tr}
\DeclareMathOperator{\val}{val}

\title{Relative Hitchin fibrations and endoscopy for Unitary Friedberg--Jacquet varieties}
\author{Thomas Hameister and Spencer Leslie}

\begin{document}

\begin{abstract}
    We prove a relative fundamental lemma for the unitary symmetric variety $\GL_n(k')\backslash \U_{2n}$, proving a conjecture of \cite{sl_endoscopy,LXZfund}. Our proof reduces to the Lie algebra and argues over function fields in positive characteristic, using methods from \cite{ngo-laumon,ngo}. In particular, we reinterpret orbital integrals in terms of the cohomology of relative affine Springer fibers  and globalize this using unitary versions of Hitchin moduli spaces for symmetric pairs via a product formula. As the main global input, we prove a matching of perverse sheaves between relative Hitchin moduli spaces and their corresponding endoscopic variants.
\end{abstract}

\maketitle

\setcounter{tocdepth}{1}
\tableofcontents

\section{Introduction}

Fix an unramified quadratic extension of non-archimedean local fields $F'/F$ of positive characteristic. In the body of this paper, we prove an endoscopic fundamental lemma for the Lie algebra of the symmetric space $\GL_{n,F'}\backslash\U_{2n}$ (Theorem \ref{thm: main_intro} below). Applying motivic integration techniques \cite{CHL} (see also Julia Gordon's appendix to \cite{YunJR}) and descent arguments similar to  \cite{Lesliedescent} sketched in the appendix, our statement implies the endoscopic fundamental lemma for $\GL_{n,F'}\backslash\U_{2n}$ in characteristic zero. This confirms a conjecture in \cite{LXZfund} and fits into a program towards relative endoscopy initiated by the second author \cite{sl_endoscopy}. Its resolution is an essential step towards the stabilization of the relative trace formula associated to twisted unitary Friedberg--Jacquet periods of automorphic forms on $\U_{2n}$ considered in \cite[Conjecture 1.1]{LXZfund}.

This work constitutes the first example of an endoscopic relative fundamental lemma to be proved utilizing geometric techniques pioneered by Ng\^o and Ng\^o--Laumon \cite{ngo-laumon, ngo_endoscopie, ngo}. The moduli spaces arising naturally in the proof are expected to be important objects of independent interest in the relative Langlands program \cite{bzsv}, see for example \cite{HML}. We note that the methods used in this paper can be used with minor modification to re-prove the fundamental lemma associated to the Lie algebra of $\U_n\times \U_n\backslash \U_{2n}$, which was proved in \cite{spencer_annals} using methods from $p$-adic harmonic analysis. 

\subsection{Background and motivation}

The study of relative fundamental lemmas arises from the method of comparison of relative trace formulas, an important tool used to study questions of distinction for automorphic representations. More precisely, let $H\subset G$ be reductive groups over a number field $F$ with $X=H\backslash G$ spherical, and let $\mathbb{A}_F$ denote the ring of adeles of $F$. A cuspidal automorphic representation $\pi$ of $G$ is \emph{$H$-distinguished} if the period integral 
\[
\int_{[H]}\varphi(h)\,dh
\]
is nonzero for some $\varphi\in\pi$, where $[H] = H(F)\backslash H(\mathbb{A}_F)$. A general philosophy arising from the relative Langlands program \cite{sv,bzsv} is that properties of period integrals can be understood through $L$-functions determined by $X$ (or more precisely, the $G$-Hamiltonian variety ${M}=T^\ast X$). For example, for linear periods of $\mathrm{GL}_n\times\mathrm{GL}_n\subset\mathrm{GL}_{2n}$, Friedberg and Jacquet show that $H$-distinction of $\pi$ is given by a condition that $\pi$ be of ``symplectic type'' and that $L(\frac{1}{2},\pi,\mathrm{std})\neq 0$ \cite{FJ}; from the point of view of relative Langlands duality, this result reflects that the dual Hamiltonian to $T^\ast X$ is the $\check{G}= \GL_{2n}(\mathbb{C})$-variety $\check{M} = T^\ast(\GL_{2n}(\mathbb{C})\times^{\Sp_{2n}(\mathbb{C})}\mathbb{C}^{2n})$. The results of this paper concern the unitary version of Friedberg and Jacquet's linear periods given by 
\begin{equation}\label{cases of interest}
  \Res_{F'/F}\GL_{n}\subset \U_{2n}\qquad \text{and}\qquad \U_n\times \U_n\subset \U_{2n}.  
\end{equation}
This was one of several cases considered in recent work of Leslie, Xiao, and Zhang, which in particular proposes conjectures detecting distinction in these case in terms of base change $L$-functions \cite{LXZfund,LXZrtf}. As explained in \cite{sl_endoscopy}, the precise base change $L$-functions may be explained via equipping $\check{M}$ with the correct action of the Galois group of $F$.

To establish theorems relating period integrals to $L$-values, an essential tool is the relative trace formula (RTF). The RTF admits a spectral expansion, written in terms of period integrals, and a geometric expansion, in terms of orbital integrals. A powerful technique is to analyze the geometric expansion, comparing it to the geometric expansion of a second RTF (possibly from another group) the spectral expansion of which encodes the $L$-values of interest. This approach has been very successful in several cases, and a family of such comparisons was conjectured and partially realized in \cite{LXZrtf} in the cases \eqref{cases of interest}. Importantly, however, the RTFs for  unitary Friedberg--Jacquet periods are not \emph{stable}: the stable or geometric orbits of the $H$-action on $X$ do not agree with the rational orbits of $H(F)$ acting on $X(F)$. The RTF encoding the $L$-values is stable, so that the comparison requires that the unitary RTF be \emph{stabilized}. The results of \cite{LXZrtf} effect the first steps of the procedure by expressing the geometric side of the RTF as a sum
  \begin{equation*}
\mathrm{RTF} = \underbrace{\mathrm{SRTF}}_{\text{Stable part of RTF}} + \quad{\sum_{a\in (X\sslash H)(F)}\sum_\kappa \mathrm{O}_a^\kappa},
  \end{equation*}
where the stable part is amenable for comparisons to the linear RTF, and $\mathrm{O}_a^\kappa$ is a so-called $\kappa$-orbital integral (here, $\kappa$ runs over characters of a certain Galois cohomology group). The essential problem of {endoscopy and stabilization} is giving a meaningful interpretation for these error terms. Such a mechanism was identified in \cite{sl_endoscopy} when $X$ is a symmetric variety such as the cases considered here, so that to a given character $\kappa$, we will consider an \emph{endoscopic variety}. This problem was largely avoided in \cite{LXZfund} by applying simple trace formula techniques to ensure that (most of) the $\kappa$-terms to do not appear, but at the cost of weaker global results. 
\begin{remark}
It turns out that the theory of endoscopy for the \emph{Siegel Friedberg--Jacquet variety } $X^{\Siegel}_n=\Res_{F'/F}\GL_{n}\backslash \U_{2n}$ is more complicated than that of the unitary case $X^U_n=\U_n\times \U_n\backslash \U_{2n}$ considered in \cite{Lesliedescent,spencer_annals}. In fact, we require an enhancement of the theory of \cite{sl_endoscopy} which to an elliptic endoscopic datum $\fe$ for $X_n^\siegel$ produces for us the elliptic endoscopic varieties 
\[
X_a^\siegel\times X_b^\Siegel \qquad\qquad (a+b=n)
\]
along with an additional datum of an automorphism of the dual Hamiltonian $\check{M}_\fe$. Appropriately taking this automorphism into account produces the correct endoscopic varieties
\[
X_a^\siegel\times X_b^\unitary \qquad\qquad (a+b=n);
\]
in particular, when $a=0$ and $b=n$, we obtain $X_n^\unitary=\U_n\times \U_n\backslash \U_{2n}$ as an endoscopic variety for $X_n^\Siegel$. In this case the fundamental lemma was established in characteristic zero in \cite{LXZfund}. A fuller account of the extension to \cite{sl_endoscopy} which explains this is forthcoming.
\end{remark}

The relative fundamental lemma established in this paper relates the $\kappa$-orbital integrals to stable orbital integrals for the corresponding endoscopic variety.  Combined with the main result in \cite{Lesliedescent,spencer_annals}, which handle the result for the $X^\unitary = \U_n\times \U_n\backslash\U_{2n}$ case, our main result is the key local harmonic-analytic statement required for the stabilization of the RTF associated to the general (biquadratic) twisted Friedberg--Jacquet periods discussed in \cite{LXZfund}.

It is a stunning fact that all known proofs of the classical fundamental lemma of Langlands and Shelstad for a quasi-split group come from a reduction to positive characteristic, where powerful geometric techniques in either perverse sheaves or $p$-adic integration can be applied \cite{ngo,GWZ}.  Subsequent work has established variants of these results for the Jacquet--Rallis fundamental lemma \cite{YunJR,griffin2} and the classical fundamental lemma for the spherical Hecke algebra \cite{griffin} using generalizations of affine Springer fibers and moduli spaces of Higgs bundles. The proof in \cite{spencer_annals} proceeds via local and global harmonic analysis in characteristic zero, but it is not clear if such an approach works more generally. To handle the Siegel case, we build on work of the first author with Morrissey \cite{HM}, introducing relative moduli spaces and studying the relative Hitchin fibration to prove the corresponding endoscopic fundamental lemma.

\subsection{Main Theorem}

Let $F$ be a non-archimedean local field of characteristic $p>n$ with ring of integers $\cO$ and with residue field $k = \bF_q$. Let $F'$ be the unique non-split, unramified extension of $F$ with residue field $k'\simeq \mathbb{F}_{q^2}$. Let $\sigma\in \Gal(k'/k)$ be the order 2 generator, and for $x\in \mathrm{Mat}_n(F')$ we denote $x^* = {}^t\sigma(x)$, where $\sigma(x)$ indicates entry-wise conjugation. Let
\[
\U_n = \{g\in \GL_n(F')\mid g\cdot g^* = 1\}
\]
denote the rank $n$ unitary group over $F$ and 
\[
\Herm_n = \{x\in \mathfrak{gl}_n\mid x = x^*\}
\]
the space of $n\times n$ Hermitian matrices. Consider the following two representations:  
\begin{itemize}
    \item The ``Siegel'' case: $H^\siegel = \mathrm{Res}_{F'/F}\; \GL_n$ acting on $V^\siegel = \Herm_n\times \Herm_n$ by
\[
h\cdot(x,y) = (h\cdot x\cdot h^*, (h^*)^{-1}x h^{-1}).
\]
    \item The ``unitary'' case:  $H^\unitary = \U_n\times \U_n$ acting on $V^\unitary = \mathfrak{gl}_n(F')$ by 
    \[
(h_1,h_2)\cdot x = h_1xh_2^*.
    \]
\end{itemize}
We have an identification $V^\siegel\git H^\siegel\simeq V^\unitary\git H^\unitary$ of GIT quotients, and we denote this common quotient by $\fc$. Let $\fc^\mathrm{rs}$ denote the image of regular, semisimple matrices. For any $a\in \fc(\cO)\cap \fc^{\mathrm{rs}}(F)$ we have a Kostant-Rallis section $[\e](a)\in V^\siegel(F)$, and we consider the orbital integral
\[
\Orb_{a}^{\siegel,\kappa} = \sum_{\gamma\sim [\e](a)}\langle\kappa,\mathrm{inv}(\gamma,\gamma_0)\rangle\int_{I_\gamma\backslash H^\siegel(F)} \mathbb{1}_{V^\siegel(\cO)}(h^{-1}\cdot \gamma)\, \frac{dh}{dt}
\]
where the sum is over all $H^\Siegel(F)$-orbits of $\gamma$ in the same orbit as $[\e](a)$ over $\breve F = F\otimes_k\ol k$; $I_\gamma$ is the centralizer of $\gamma$ in $H^\siegel(F)$; and $dh$ and $dt$ are Haar measures on $H(F)$ and $I_\gamma$, respectively, normalized so that $\vol(H^\siegel(\cO),dh)=1$. Here $\mathrm{inv}(\gamma,\gamma_0)\in H^1(F,I_{\gamma_0})$ is the cohomological invariant encoding the relative position of the $H^\Siegel(F)$-orbits of $\gamma$ and $\gamma_0$, and the parameter $\kappa\in H^1(F,I_{\gamma_0})^D$ is a character arising from an endoscopic datum $\fe$ for $X^\Siegel$ (see Section \ref{sec: endoscopic spaces}). We may similarly define the stable orbital integral $\SOrb_a^\unitary$ by summing over all lifts $\gamma\in V^\unitary(F)$ of $a$ (corresponding to $\kappa=1$). 

As explained in Section \ref{sec: endoscopic spaces}, an elliptic endoscopic datum $\fe$ for $X^\Siegel$ determines a character $\kappa$ as well as an endoscopic variety of the form $X^\fe = X^\siegel_r\times X^\unitary_s$ with $r+s= n$. For such a variety, the stable orbital integral is the product $\SOrb_{(a^\siegel,a^\unitary)}^\fe=\SOrb_{a^\siegel}^\siegel\times\SO^\unitary_{a^\unitary}$. Let $\fc_\fe$ denote the GIT quotient of $X^\fe$ by $H^\fe= \Res_{F'/F}\GL_r\times [\U_s\times \U_s]$.

\begin{thm}
\label{thm: main_intro}
    (= Theorem \ref{thm: main}) For every elliptic endoscopic datum $\fe = (\kappa,\rho)$ over $F$ (see \S\ref{sec: endoscopic spaces}), there is a canonical transfer map
    \[
    \nu_\fe\colon \fc_\fe\to \fc
    \]
    such that for any $a_\fe\in \fc_\fe(\cO)$ with image $a = \nu_\fe(a_\fe)\in \fc(\cO)\cap \fc(F)^{\mathrm{rs}}$, we have
    \[
\Orb_a^{\siegel,\kappa} = q^{r_\fe(a)}\SOrb_{a_\fe}^\fe
    \]
    where $r_\fe(a)$ is a nonnegative integer, defined in \S\ref{sec: cases and statements}. 
\end{thm}

\subsubsection{Transfer to characteristic zero}\label{Sec: on transfer to characteristic zero}
Below we discuss the proof of Theorem \ref{thm: main_intro} for $F$ of characteristic $p>n$. We first outline the argument to transfer this result to characteristic zero. The strategy is to apply the methods of motivic integration and the transfer principle as discussed in \cite{CluckersHalesLoeser} for the case of the classical endoscopic fundamental lemma and in the appendix by Julia Gordon to \cite{YunJR} in the case of the Jacquet--Rallis fundamental lemma.

Due to the standard nature of this argument, we omit the complete argument, noting that it is a routine check to verify that the terms defining the orbital integrals are definable in the Denef--Pas language. Indeed, many of the terms (the field extension $F'/F$, the Lie algebras, the orbital integrals, et cetera) are shown to be definable with appropriate changes to the statements in the Jacquet--Rallis setting \cite[Appendix]{YunJR}. The case-at-hand is more involved due to the presence of the $\kappa$-orbital integral and the transfer factor. Both may be handled by appropriately adapting the arguments of \cite[Section 8]{CluckersHalesLoeser} for the classical fundamental lemma to our setting. More precisely, Cluckers--Hales--Loeser show how to define the character $\kappa$ (Section 8.1 of \emph{loc. cit.}), the normalization of the transfer factor in terms of the Kostant section (Section 8.2), and the pairing of stable orbits (Section 8.3). An inspection of the definition of the transfer factor for the Siegel case in Appendix \ref{Appendix A} and the unitary case in \cite{Lesliedescent} shows that precisely these same ingredients are required to handle the transfer factors in our setting.  

With this in mind, we obtain the following corollary of the transfer principle and Theorem \ref{thm: main_intro}.

\begin{cor}\label{Cor: transfer}
    There exists an algorithmically computable constant $M>0$ such that for any characteristic-zero local field $F$ of residue characteristic larger than $M$, the result stated in Theorem \ref{thm: main_intro} holds for $F$ and its unique unramified extension $F'/F$.
\end{cor}

\subsection{Reformulation as Lattice Counts}

We reformulate Theorem \ref{thm: main_intro} as a point count of a moduli space of lattices which we call the \emph{relative affine Springer fiber (RASF)}. In particular, to an $n$-tuple $a = (a_2,\dots, a_{2n})\in\cO^n$, we may form the $F'$-algebra
\[
R_a(F')= F'[x]/(P_a(x)),\quad \text{for }P_a(x):=x^{2n}+a_2x^{2n-2}+\cdots+a_{2n-2}x^2+a_{2n}
\]
with the obvious sub-$\cO'$-algebra $R_a(\cO')$. Then $R_a(F')$ carries natural involutions $\iota(x) = -x$ and $\sigma$ from the Galois action of $F'/F$.

We say $\Lambda\subset R_a(F')$ is an $R_a(\cO')$-lattice if it is stable under the $R_a(\cO')$-action and it is free of rank $2n$ as an $\cO'$-module. There are natural non-degenerate Hermitian forms $h^\siegel$ and $h^\unitary$ on $R_a(F')$ constructed via the trace pairing by
\[
h^\siegel(f,g) = \mathrm{Tr}_{R_a(F')/F'}\left( \frac{f\cdot \sigma(g)}{P_a'} \right),\quad h^\unitary(f,g) = h^\siegel(xf,g).
\]
For $\star \in \{\siegel,\unitary\}$, we define $\cM^\star(a)$ as the moduli of $R_a(\cO')$-lattices $\Lambda$ such that $\iota(\Lambda) = \Lambda$ and $\Lambda$ is self dual with respect to $h^\star$. We call $\cM^\star(a)$ the relative affine Springer fiber (RASF) at $a$. It is a analogue of affine Springer fibers, which are subvarieties of the affine Grassmannian introduced by Kazhdan and Lusztig \cite{kl} of central importance to geometric representation theory. In fact, $\cM^\star(a)$ can be considered as the fixed points of the involution $\iota$ acting on the unitary affine Springer fiber studied in \cite{ngo-laumon}. We remark on generalizations to other spherical varieties in \S\ref{sec: general def of rasf}.

The RASF $\cM^\star(a)$ is representable by a locally-finite-type, projective scheme. It admits an action of a group of symmetries $\cP(J_a)$, which is independent of $\star\in \{\siegel,\unitary\}$. Together, we show that $\cM^\star(a)$ and $\cP(J_a)$ satisfy the point-counting properties of \cite[Hypotheses 8.1.9]{ngo}; namely, the stabilizers of $\cP(J_a)$ acting on $\cM^\star(a)$ are of finite type, and there is a canonical lattice $\Gamma_a\subset \cP(J_a)$ which acts freely on $\cM^\star(a)$ such that $\cM^\star(a)/\Gamma_a$ and $\cP(J_a)/\Gamma_a$ are finite type. In stark contrast to the case of affine Springer fibers, however, the quotient $\cM^\star(a)/\Gamma_a$ need not be connected, even for an optimally chosen lattice. Instead, $\Gamma_a$ only acts transitively on irreducible components within the connected components of $\cM^\star(a)$. This phenomena obstructs standard approaches to proving dimension formulas for RASFs, and so establishing our dimension formula (Proposition \ref{prop: dim of RASF}) and the equidimensionality of $\cM^\star(a)$ (Proposition \ref{prop: equidimensionality}) is one of the more technical parts of the paper and can be viewed as one of our main results. Indeed, inspired by work of Jingren Chi \cite{jingren}, we use global methods to approach these questions, exploiting properness of the relative Hitchin map over a carefully constructed family.

For an endoscopic symmetric space $X^\fe = X^\siegel_r\times X^\unitary_s$, the RASF at $a_\fe=(a^\siegel,a^\unitary)$ is a product $\cM^\fe(a^\siegel,a^\unitary) = \cM^\siegel(a^\siegel)\times \cM^\unitary(a^\unitary)$. Moreover, if $\nu_\fe(a_\fe)=a$, there is a canonical morphism of centralizers $J_a\to J_{\fe,a_\fe}$ (Proposition \ref{prop: centralizers under transfer}) which induces a surjection between symmetry groups $\cP(J_a)\to \cP(J_{\fe,a})$. In particular, $\cP(J_a)$ acts on $\cM^\fe(a_\fe)$, and there are finitely many $\cP(J_a)$-orbits on both $\cM^\siegel(a)$ and $\cM^\fe(a_\fe)$. Therefore, the sums
\[
\#[\cM^\siegel(a)/\cP(J_a)](k)_\kappa = \sum_{O}\frac{\langle \kappa,\mathrm{cl}(O)\rangle}{\#\Aut_{\cP(J_a)}(O)(k)},\quad \#[\cM^\fe(a_\fe)/\cP(J_a)](k) = \sum_{O}\frac{1}{\#\Aut_{\cP(J_a)}(O)(k)},
\]
 over points $O$ in the groupoid $[\cM^\siegel(a)/\cP(J_a)](k)$, respectively $[\cM^\fe(a_\fe)/\cP(J_a)](k)$, are well-defined. Here, $\langle \kappa,\mathrm{cl}(O)\rangle$ denotes a certain sign defined in \S\ref{sec: point counting for RASFs} and determined by the endoscopic datum. Proposition \ref{prop: point counting for P(J')} shows how the orbital integrals appearing in Theorem \ref{thm: main_intro} can be written in terms of these weighted point counts.

\begin{thm}
    \label{thm: local main_intro}
    (= Theorem \ref{thm: local main}) Let $\fe$ be an endoscopic datum, and suppose that $a\in \fc^\heartsuit(\cO_v)$ and $a_\fe\in \fc_\fe^
    \heartsuit(\cO)$ are such that $\nu_\fe(a_\fe) = a$. Then, for all finite extensions $K/k$, we have
    \[
    \#[\cM^\Siegel(a)/\cP(J_a)](K)_\kappa = q^{r_\fe(a)}\;\#[\cM^\fe(a_\fe)/\cP(J_a)](K)
    \]
\end{thm}

Note that Theorem \ref{thm: local main_intro} combined with Grothendieck--Lefschetz reduces Theorem \ref{thm: main_intro} to a computation of cohomology of RASFs.

\subsection{Global moduli spaces}

To study point counts on RASFs, we construct global moduli spaces called \emph{relative Hitchin spaces}, and establish a product formula (Proposition \ref{prop: product formula}) to relate these to the RASFs from the previous section. These global moduli spaces are unitary analogues of Hitchin systems for symmetric varieties, which have been studied extensively by Schaposnik \cite{schaposnik2013spectral,schaposnik2014introduction,schaposnik, schaposnik2018geometric}, Pe\'on-Nieto and Garc\'ia-Prada \cite{gppn}, and the first author with Morrissey \cite{HM}.

Fix a smooth, projective, geometrically-connected curve $\Sigma$ of genus at least 2 and a divisor $D$ of degree at least $2g$. Fix an \'etale double cover $\vartheta\colon \Sigma'\to \Sigma$, with $\Sigma'$ smooth, projective, and geometrically connected as well. From this, we may form  moduli spaces $\cM^\star = \mathrm{Sect}_\Sigma([V^\star_D/H^\star])$, where the subscript $D$ denotes the induced bundle on $\Sigma$, and $H^\star$ and $V^\star$ are taken over $k$ instead of $F$. Here, $\mathrm{Sect}_\Sigma$ denotes the space of maps $\Sigma\to [V^\star_D/H^\star]$ whose composition with the structure map for $[V^\star_D/H^\star]\to \Sigma$ is the identity.

We may also describe the moduli spaces $\cM^\star$ more explicitly in terms of fixed point stacks of involutions on moduli of unitary Higgs bundles. Namely, we let $\cN$ describe the moduli of unitary Higgs bundles of Laumon and Ng\^o, whose $S$-points are given by
\[
    \cN(S) = \left\{ (\cE',h,\phi')\left| \begin{array}{c} \cE'\text{ is a rank $2n$ vector bundle on }\Sigma'\times S, \\ h\colon \cE'\to \sigma^*(\cE')^\vee\text{ is a Hermitian form on $\cE'$, i.e. $h = \sigma^*h^\vee$,}  \\ \phi'\colon \cE'\to \cE'\otimes D,\; \text{satisfying $\sigma^*(\phi')^\vee\circ h +h\circ \phi' = 0$}\end{array}\right. \right\}.
 \]
We have involutions $\iota^\star$ on $\cN$ given by
\[
\iota^\star(\cE',h,\phi') = \begin{cases}
(\cE',-h,-\phi') & \text{if }\star = \siegel \\
(\cE',h,-\phi') & \text{if }\star = \unitary,
\end{cases}
\]
and show in Proposition \ref{prop: fixed points of involution on N} that $\cM^\star$ is a distinguished component of the fixed point stack $\cN^{\iota^\star}$, and is equal to $\cN^{\iota^\star}$ when we restrict to a suitable anisotropic locus. From this we obtain concrete realizations of $\cM^\siegel$ and $\cM^\unitary$; for example,
\[
\cM^\Siegel(S) = \left\{ (\cE',h,\phi_{12}',\phi_{21}')\left| \begin{array}{c}\cE'\text{ is a rank $n$ vector bundle on $\Sigma'\times S$, }\phi_{12}'\colon \cE'\to \sigma^*(\cE')^\vee\otimes D\\ \text{and }\phi_{21}'\colon \sigma^*(\cE')^\vee\to \cE'\otimes D\text{ such that }\sigma^*\phi_{ij}'^\vee = \phi_{ij}' \\
\text{and }h\colon \cE'\to \sigma^*(\cE')^\vee\text{ is a Hermitian form on $\cE'$}\end{array} \right.\right\}.
\]

For $\star\in \{\siegel,\unitary\}$, there is a Hitchin-type map 
\[
f^\star\colon \cM^\star\to \cA
\]
where $\cA = \bigoplus_{i=1}^n H^0(\Sigma,D^{\otimes 2i})$ is an affine space. These maps essentially compute the $H^\star$-invariants of the relative Higgs field $\phi'$. As in the case of the classical Hitchin moduli space, this morphism reveals a remarkable amount of structure in $\cM^\star$. Indeed, we consider a family of curves $\wt{\Sigma}_a'$ over $\cA$ equipped with involutions $\iota$ (induced by the involution of $G$ determining $X^\star$) and $\sigma$ (induced by the Galois action), so that $\wt{\Sigma}_a'$ gives a globalization of the local ring $R_a(F')$. Whenever $\wt\Sigma_a'$ is reduced, we show that the fiber $f^{\star,-1}(a)$ admits a spectral description as the moduli of $\iota$-fixed elements in a unitary version of the compactified Prym variety of $\wt\Sigma_a'$ (Proposition \ref{prop: spectral description for unitary higgs}). In particular, there is a commutative group stack $\cP$, the relative Prym variety of $\wt\Sigma_a'/\iota\to \wt\Sigma_a'/\langle\sigma,\iota\rangle$, which acts on the fibers of $f^\star$ with finitely many orbits whenever $\wt\Sigma_a'$ is reduced. When $\wt\Sigma_a'$ is smooth (a generic condition on $\cA$), the fiber $f^{\star,-1}(a)$ is a finite unions of $\cP_a$-torsors.

There is a Beauville--Laszlo-style gluing map from RASFs to relative Higgs bundles. This produces a product formula (Proposition \ref{prop: product formula})
\begin{equation}
\label{eqn: product_intro}
    \#[\cM^\star_a/\cP_a](k)_\kappa = \prod_{v\in |\Sigma|}\#[\cM^\star_v(a)/\cP(J_{a,v})](k_v)_\kappa
\end{equation}
for $a$ in an elliptic locus $\cA^\ani\subset \cA$ (\S\ref{sec: open subsets}), where only finitely many terms on the right-hand side are nontrivial.

Over $\cA^\ani$, the symmetry group $\cP_a$ acts on the complex $\bR\cameral{f}_*^{\siegel,\ani}\ol\Q_\ell$ through its sheaf of connected components $\pi_0(\cP_a)$. In fact, after base change by an \'etale map $\cameral\cA^\ani\to \cA^\ani$ (\S\ref{sec: component groups of P}), there is a surjection $\mathbb{X}_*(\wh\canC)\times \cameral\cA^\ani\to \pi_0(\cP_a)$ over $\cameral\cA^\ani$ which allows us to decompose the cohomology into $\kappa\in \wh\canC[2]$ constituent pieces
\[
\bR\cameral{f}_*^{\siegel,\ani}\ol\Q_\ell = \bigoplus_\kappa (\bR\cameral{f}_*^{\siegel,\ani}\ol\Q_\ell)_\kappa,\quad\quad \bR\cameral{f}_*^{\fe,\ani}\ol\Q_\ell = \bigoplus_\kappa (\bR\cameral{f}_*^{\fe,\ani}\ol\Q_\ell)_\kappa
\]
where $(\bR\cameral{f}_*^{\star,\ani}\ol\Q_\ell)_\kappa\subset \bR\cameral{f}_*^{\star,\ani}\ol\Q_\ell$ is the subsheaf on which $\mathbb{X}_*(\wh\canC)$ acts through evaluation at $\kappa$. We prove the following global result.

\begin{thm}
    \label{thm: global main_intro}
    (= Theorems \ref{thm: global main} and \ref{thm: global main over ani}) For any elliptic endoscopic datum $\fe$, there is a canonical transfer map on Hitchin bases
    \[
    \nu_\fe\colon \cameral\cA_\fe\to \cameral\cA,
    \]
    which is a closed embedding. For any $\kappa\in \wh{\canC}[2]$, we have an isomorphism
    \[
    \nu_\fe^*\big(\bR \cameral{f}^{\Siegel,\ani}_*\ol{\Q}_\ell)_{\kappa}[2r_\fe(D)](r_\fe(D)) =  (\bR\cameral{f}^{\fe,\ani}_*\ol{\Q}_\ell)_{\st}
    \]
    over $\ol{k}$ restricted to the locus $\cameral \cA_\fe^\ani$ which induces an isomorphism of semisimplifications over $k$.
\end{thm}

We note that we do not prove Theorem \ref{thm: global main_intro} directly, but rather first prove it over an open subset $\cameral\cA_\fe^\good\subset \cameral\cA_\fe^\ani$, defined so that the fibrations $f^{\siegel,\ani}$ and $f^{\fe,\ani}$ satisfy the conditions of Ng\^o's support theorem on $\cameral\cA_\fe^\good$ (see Definition \ref{def: good locus}). The codimension of the complement of $\cameral\cA_\fe^\good$ in $\cameral\cA_\fe^\ani$ tends towards infinity as $\deg D$ goes to infinity, so together with an appropriate interpolation result, this is enough to deduce Theorem \ref{thm: local main_intro} (and hence Theorem \ref{thm: main_intro}). To obtain Theorem \ref{thm: global main_intro} over the full anisotropic locus, we reverse the flow of information and use Theorem \ref{thm: local main_intro} and the product formula to conclude the global result.

We prove Theorem \ref{thm: global main_intro} over $\cA_\fe^\good$ roughly as follows:  For each $\kappa$, the decomposition theorem of Beilinson, Bernstein, Deligne, and Gabber \cite{BBD} implies that  $(\bR\cameral{f}_*^\ani\ol\Q_\ell)_\kappa$ is a semisimple perverse sheaf over $\ol k$. By Ng\^o's support theorem \cite{ngo}, to determine the supports of the simple summands of $(\bR\cameral{f}_*^\ani\ol\Q_\ell)_\kappa$, it is enough to study the top degree ordinary cohomology $(\bR^{\top} \cameral{f}_*\ol\Q_\ell)_\kappa$ along $\delta$-critical subsets.

The key concept that allows us to compute $(\bR^{\top} \cameral{f}_*\ol\Q_\ell)_\kappa$ is the theory of the \emph{regular quotient}, developed for symmetric varieties by the first author and Morrissey in \cite{HM}. In particular, in \emph{loc. cit.}, we study the substack of regular $X^\star$-Higgs bundles $\cM^{\star,\reg}\subset \cM^\star$ (\S\ref{sec: geometry over regular}), showing that the relative Hitchin map restricted to $\cM^{\star,\reg}$ factors as
\begin{equation}
\label{eqn: factorization over reg}
\cM^{\star,\reg}\to \cA^\reg\xrightarrow{\varphi} \cA,
\end{equation}
where the first map is a torsor over the smooth Picard stack $\varphi^*\cP$, and the second $\varphi$ is an explicit \'etale map.

The geometry of $\varphi$ reflects the invariant theory of regular, or maximal-dimensional, $H^\star$-orbits on $V^\star$. In \cite{HM}, it is shown that there is a scheme $V^{\star,\reg}\myfatslash H^\star$ which naturally parametrizes all regular $H^\star$-orbits. There is also a natural comparison map 
\[
V^{\star,\reg}\myfatslash H^\star\to \fc,
\]
which sends any regular $H^\star$-orbit to the unique closed orbit in its closure. Moreover, $V^{\star,\reg}\myfatslash H^\star$ carries a natural $\G_m$-action, and twisting with respect to this action yields the map
\[
\varphi\colon \cA^\reg:=\mathrm{Sect}_\Sigma(V^{\star,\reg}_D\myfatslash H^\star)\to \cA:= \mathrm{Sect}_\Sigma(\fc_D).
\]
We show in \emph{loc. cit.} that $V^{\star,\reg}_D\myfatslash H^\star$ is the non-separated space obtained by gluing copies of the GIT quotient $\fc$ indexed by components of the nilpotent cone of the symmetric variety (in the sense of \cite{KR}; see also \cite{sekiguchi,levy}) away from a \emph{branching divisor} $\fD^\br$. This divisor plays an important role in the classical limit of the global conjectures of \cite{bzsv}; in joint work of the first author with Morrissey and Luo \cite{HML}, it is conjectured that for tempered spherical varieties $X$ satisfying mild conditions, $\fD^\br$ encodes the weights of the dual symplectic representation, and hence the data of the full dual Hamiltonian, associated to $M=T^*X$.

In this paper, we show that the factorization \eqref{eqn: factorization over reg} carries over to the unitary settings of this paper, with $\star\in \{\unitary,\siegel\}$ distinguished by the Galois action on the nilpotent cone. We introduce an open set $\cA^\dagger\subset \cA^\ani$ (Definition \ref{def: dagger locus}) using the branching divisor $\fD^\br$, and show that the regular locus $\cM^{\star,\reg}\subset \cM^\star$ is dense in every fiber (Lemma \ref{lem: density of regular locus}) over $\cA^\dagger$. A significant part of the technical work in proving Theorem \ref{thm: global main_intro} is in showing that $\cA^\dagger$ is large enough; namely, for $\deg(D)$ large enough, we show that $\cA^\dagger$ has nonempty intersection with all $\delta$-critical loci (See Proposition \ref{prop: density of dagger locus in delta strata}).

Deducing Theorem \ref{thm: main_intro} from Theorem \ref{thm: global main_intro} still requires an interpolation argument. Namely, for a general pair $a_\fe, a$ as in Theorem \ref{thm: main_intro}, we construct global points in $\cA^\ani$ so that all terms in the right-hand side of \eqref{eqn: product_intro} are nonzero and simple enough to compute by hand (see \S\ref{sec: cases and statements}). We note that the nonzero condition is not vacuous; $\kappa$-orbital integrals for $X^\siegel$ are frequently zero. To characterize this, we study a cohomological obstruction for this non-vanishing in \S\ref{sec: obstruction} and use this description to construct suitable globalizations for arbitrary pairs in Theorem \ref{thm: main_intro} (see Proposition \ref{prop: existence of interpolation point}).

\subsection{Structure of the article}

In \S\ref{sec: invariant theory (whole section)}, we discuss invariant theory that underlies both the local and global formulations of our main results. Most results in this section are not new, the exception being the $k$-rational results in \S\ref{sec: rational structures}. In \S\ref{sec: unitary Higgs}, we review some basic results of Ng\^o and Laumon on unitary Higgs bundles, which we will use to deduce some basic properties of our global moduli spaces. In \S\ref{sec: relative Hitchin spaces}, we introduce and study the relative Hitchin fibration. In \S\ref{sec: RASF}, we introduce and study relative affine Springer fibers. The crucial technical result in this section is the dimension formula and equidimensionality (Propositions \ref{prop: dim of RASF} and \ref{prop: equidimensionality}), the proofs of which use global methods. In \S\ref{sec: point counting (whole)}, we recall the formalism of point counting for groupoids and show that our orbital integrals can be interpreted as point counts on RASFs. We also deduce a few cases of Theorem \ref{thm: main_intro}, including the case where both sides of the equality are zero. In \S\ref{sec: support theorem (whole)}, we outline how to apply Ng\^o's support theorem to our case, including the computation of top degree cohomology over $\cameral\cA^\dagger$. Finally, in \S\ref{sec: proof of main} we finish the proofs of the main theorems. The key new technical piece is the interpolation step:  For any local $a$, we must extend $a$ to a point in the Hitchin base with nonzero point counts at every place of $\Sigma$. Appendix \ref{Appendix A} develops the descent of Theorem \ref{Thm: fundamental lemma intro} and Corollary \ref{Cor: transfer} from the infinitesimal symmetric variety $V^\Siegel$ to the fundamental lemma for the variety $X^\Siegel$. In particular, we formulate the notion of transfer factors for these varieties.

\subsection{Acknowledgments}

T. Hameister thanks Ng\^o Bao Ch\^au for his role in introducing him to related moduli spaces for his PhD thesis. He also thanks Siddharth Mahendraker, Xinchun Ma, Zeyu Wang and Griffin Wang for useful discussions. S. Leslie also thanks Jayce Getz for several conversations regarding endoscopy and tools needed to consider relative variants of Ng\^o's arguments, and Tasho Kaletha for conversations about transfer factors in the relative setting. 

S. Leslie was partially supported by NSF grants DMS 2200852 and 2440741, as well as the Sloan Foundation. He wishes to thank the Institute for Advanced Study, where this document was completed.

Prior to the editing process, AI was not used in any significant way in this work, neither in obtaining results nor writing. During editing, ChatGPT 5.6 Astra was prompted to identify errors in a draft, which led to corrections in the proof of Proposition \ref{prop: obstructed case} and minor modifications to the interpolation argument in Section 9.2 along with numerous typos throughout the paper.

\section{Invariant theory for symmetric varieties}
\label{sec: invariant theory (whole section)} 

\subsection{The representation $\bH$ acting on $\fh^\perp$}
\label{sec: invariant theory}

 For this section (and this section only), we assume throughout that the base field $k$ is algebraically closed such that $\mathrm{char}(k)\neq2$. The results of this section are discussed in greater detail in \cite{HM}, with most of the technical steps in the proofs in characteristic $p$ contained in Levy's thesis \cite{levy}. While we will keep this section focused on our example of interest, all constructions can be made for general symmetric varieties with minimal modifications (cf. \cite{HM}), and much remains true for general spherical varieties (see \cite{HML} for a summary).

Let $\bG = \GL_{2n}$. While results in this section do not depend on it, we will assume that the characteristic of $k$ is greater than $n$ for Corollary \ref{cor: delta regularity in the limit}. We have an algebraic involution $\theta\colon \bG\to \bG$ given by
\[
\theta(g) = I_{n,n}gI_{n,n}, \quad\text{for }I_{n,n} = \begin{pmatrix}
	I_n & \\ & -I_n 
	\end{pmatrix}.
\]
The fixed points under this action are given by the Levi subgroup $\bH = \GL_n\times \GL_n$ inside $\bG$. Moreover, the action of $\theta$ induces a decomposition on the level of Lie algebras
\[
\fg = \fh\oplus \fh^\perp
\]
where $\fh = \mathrm{Lie}(\bH)$ is the fixed points of the action of $\theta$ and $\fh^\perp$, the orthogonal complement of $\fh$ under a nondegenerate $G$ invariant form on $\fg$, is the $(-1)$-eigenspace of the $\theta$ action. The group $\bH$ acts on $\fh^\perp$ by restricting the adjoint action. 

We also define the following analogue of Cartan subalgebras adapted for the symmetric setting.

\begin{definition}
	A $\theta$-Cartan subalgebra $\fa$ is a maximal (with respect to inclusion) abelian Lie algebra consisting of semi-simple elements contained in $\fh^\perp$.
\end{definition}

All $\theta$-Cartan subalgebras in $\fh^\perp$ are conjugate under $\bH$ (\cite[Thm 2.11]{levy}); we will usually fix the choice
\begin{equation}
	\label{eqn: theta Cartan}
\fa = \left\{ \left.\begin{pmatrix} & \delta \\ \delta &  \end{pmatrix}\right| \delta \text{ is a diagonal $n\times n$ matrix}  \right\}.
\end{equation}
The theory of spherical varieties endows the $\theta$-Cartan $\fa$ with a root system of type $B_n$, whose roots on the choice of Cartan \eqref{eqn: theta Cartan} are described by
\[
\Phi_{\bX,\fa} = \{\pm\delta_i^* \pm \delta_j^* \mid 1\leq i< j\leq n\}\cup\{\pm\delta_i^*\mid 1\leq i\leq n\},
\]
where $\delta_i^*$ is the coordinate function for the $i$-th entry of the $n\times n$ diagonal matrix $\delta$. We may see this as a renormalization of the restricted root system of $\bX$.

We have a Weyl group $W_\fa$ for the root system $\Phi_{\bX,\fa}$, which is presented as $W_\fa = N_\bG(\fa)/C_\bG(\fa) = N_{\bH}(\fa)/C_{\bH}(\fa)$. For general symmetric varieties $W_\fa$ is a subquotient of the usual Weyl group $W$, but in our case, $W_\fa$ is the subgroup of $W$ given by  $W_\fa = \mathfrak{S}_n\rtimes \{\pm 1\}^n$. The action of $W_\fa$ on \eqref{eqn: theta Cartan} is by letting $\mathfrak{S}_n$ permute the diagonal entries of $\delta$ and letting $\{\pm1\}^n$ change the signs of the entries $\delta_i$ of $\delta$.

There is a construction also of the canonical torus of the symmetric variety, which arises as follows. Let $\canT$ denote the canonical torus for $\bG$. Then, for any choice of Borel $B$ such that $B\cap \theta(B)$ is a torus, $\canT = B/[B,B]$ acts on $\bH\backslash\bG/[B,B]$ through the quotient $\canA$, which we call the \emph{canonical torus} for $\bX = \bH\backslash \bG$. The canonical torus $\canA$ comes equipped with a root system $\Phi_{\bX}$ and an action of a Weyl group $W_\bX$. The map on Lie algebras $\fa\to \mathrm{Lie}(\canA)$ induced by any choice of Borel $B$ as above is an isomorphism which identifies the two respective root systems, and so we do not distinguish between the two Lie algebras.

We also have a notion of regularity and semi-simplicity for elements of $\fh^\perp$.

\begin{definition}
	We say $x\in \fh^\perp$ is \emph{semisimple} if it is semisimple as an element of $\fg$.
	
	Let $I\subset \fh^\perp\times \bH$ denote the group scheme of centralizers, i.e.
	\[
I_x = \{h\in \bH\mid h\cdot x = x\}	.
	\]
	Then, we say that $x\in \fh^\perp$ is \emph{regular} if $\dim I_x$ is minimal. We denote the regular locus of $\fh^\perp$ by $(\fh^\perp)^{\reg}$. We denote the regular semi-simple locus by $(\fh^\perp)^{\mathrm{rs}}$.
\end{definition}

While the notion of regularity for $x\in \fh^\perp$ under the action of $\bH$ is generally distinct from regularity of $x\in\fg$ under the action of $\bG$, the two notions agree for $\bX = \GL_{n}\times \GL_n\backslash\GL_{2n}$.

The following result traces back to Kostant and Rallis in characteristic zero \cite{KR} and to Paul Levy in characteristic $p\neq 2$ \cite{levy}.

\begin{thm}
	\label{thm: levy's results}
	 Let $\fc:=\fh^\perp\git \bH:= \Spec(k[\fh^\perp]^\bH)$ denote the GIT quotient of $\fh^\perp$ by $\bH$, and let $\chi\colon \fh^\perp\to \fc$ denote the quotient map.
	\begin{enumerate}
		\item (\cite[Thm 4.9]{levy}) There is an isomorphism $\fc\simeq \fa\git W_X$. 
		\item (\cite[Lem 6.30]{levy}) The map $\chi$ admits a section $\epsilon$, which is valued in $(\fh^\perp)^\reg$. Moreover, this section can be made $\G_m$-equivariant after base change by $t\mapsto t^2$.
	\end{enumerate}
\end{thm}
We refer to such a section $\epsilon: \fc\to (\fh^\perp)^\reg$ as a Kostant--Rallis section. 

Returning to our centralizer group scheme $I\subset \fh^\perp\times \bH$, we call the restriction of $I$ to $(\fh^\perp)^\reg$ the regular centralizer group scheme and denote it by $I^\reg$. We have the following package describing $I^\reg$.

\begin{prop}
	\label{prop: regular centralizer package}
	(\cite[Prop 4.1]{HM}) Recall that $\chi\colon \fh^\perp\to \fc$ denotes the quotient map; we also denote $\chi^\reg = \chi|_{(\fh^\perp)^\reg}$.
	\begin{enumerate}
		\item The regular centralizer group scheme $I^\reg$ is a smooth commutative group scheme over $(\fh^\perp)^\reg$ which descends to a smooth commutative group scheme $J$ over $\fc$.
		\item There is a canonical map $\chi^*J\to I$ extending the isomorphism $(\chi^\reg)^*J\simeq I^\reg$.
	\end{enumerate}
\end{prop}

The group scheme $J$ admits a useful description in terms of spectral covers. We begin by recalling some facts and definitions from invariant theory (cf. \cite{HM}). Recall that $\fc_{2n} = \fgl_{2n}\git \GL_{2n}\simeq \ft_{\GL_{2n}}\git \mathfrak{S}_{2n}$ admits a spectral cover $\fs_{2n} = \ft_{\GL_{2n}}\git \mathfrak{S}_{2n-1}$. Explicitly, if
\[
A_{2n}:=k[\fc_{2n}] = k[a_1,\dots, a_{2n}]
\]
for $a_i(X)= \mathrm{tr}\bigwedge^iX$ the $i$-th elementary symmetric function of the eigenvalues of $X$, then
\[
k[\fs_{2n}] = A_{2n}[x]/(x^{2n}+a_1x^{2n-1}+\cdots +a_{2n})
\]
The natural morphism $\fs_{2n}\to \fc_{2n}$ is finite and flat of degree $2n$.

\begin{lem}\cite[Lemma 5.1]{LXZfund}\label{Lem: cat quot split}
    	The GIT quotient $\fc=\fh^\perp\git \bH$  is an affine space $\fc= \A^n$. More explicitly, for $x=\begin{pmatrix}
    	    &X\\Y
    	\end{pmatrix}\in\fh^\perp$,
        \[
        \chi(x) = (b_1(x), \cdots, b_n(x)),
        \]
        where $b_i(x) = (-1)^i\tr \bigwedge^i(XY)$. In particular, there is a commutative diagram
        \[
        \begin{tikzcd}
            \fh^\perp \ar[r]\ar[d,"\chi"]&\fgl_{2n}\ar[d,"\chi_{2n}"]\\
            \fc\ar[r] &\fc_{2n},
        \end{tikzcd}
        \] with the natural map $\fc\to \fc_{2n}$ given by setting $b_i=a_{2i}$ and $a_i = 0$ for all odd $i$. In particular $\fc\simeq \A^n$ is naturally equipped with a $\G_m$-action by the exponents $2,4,\cdots,2n$.
\end{lem}
Writing $A = k[\fc] = k[a_2,a_4,\dots, a_{2n}]$, the base change $\fs = \fs_{2n}\times_{\fc_{2n}}\fc$ is thus a finite, flat, degree-$2n$ cover of $\fc$ with ring of functions
\[
k[\fs] = A[x]/(x^{2n}+a_2x^{2n-2}+\cdots +a_{2n})
\]
The spectral cover $\fs$ therefore carries an involution $\iota$, defined over $\fc$, given by $\iota(x) = -x$, and the quotient $\fs/\iota$ remains a finite, flat cover of degree $n$. We have the following description of centralizers.

\begin{prop}
	\label{prop: universal spectral}
	(\cite[Prop. 6.2]{HM}) We have an isomorphism 
	\[
J\simeq \mathrm{Res}^{\fs/\iota}_\fc(\G_m)
	\]
	of group schemes over $\fc$.
\end{prop}

\subsection{Rational structures in invariant theory}
\label{sec: rational structures}
The group $\bH = \GL_n\times \GL_n$ and $\bH$ representation $\bfh^\perp$ admit involutions, which in turn will allow us to deduce specific rational forms. First, we have two involutions on $\bH$:  
\begin{itemize}
	\item The involution $\tau_1\colon \bH\to \bH$ given by the transpose-inverse $h = (h_1,h_2)\mapsto \t h^{-1} = (\t h_1^{-1},\t h_2^{-1})$;
	\item The involution $\tau_2\colon \bH\to \bH$ given by the swapping map on the factors $h = (h_1,h_2)\mapsto (\t h_2^{-1},\t h_1^{-1})$.
\end{itemize}
We have involutions on $\bfh^\perp$ corresponding to each of the $\tau_j$ above. Namely, if we identify $\bfh^\perp \simeq \{(A,B)\mid A,B\text{ $n\times n$ matrices}\}$, then we have:
\begin{itemize}
	\item The involution $\tau_1\colon \bfh^\perp\to \bfh^\perp$ given by the swapping-transpose map $(A,B)\mapsto (-\t B,-\t A)$.
	\item The involution $\tau_2\colon \bfh^\perp\to \bfh^\perp$ given by the transpose map $(A,B)\mapsto (-\t A,-\t B)$.
\end{itemize}
The involutions $\tau_i$ on $\bH$ and $\bfh^\perp$ are compatible in the sense that $\tau_i(h\cdot x) = \tau_i(h)\cdot \tau_i(x)$ for all $h\in \bH,\, x\in \bfh^\perp$.

We now define $k$ rational forms of the invariant theoretic spaces of the previous section. Let $k = \bF_q$ be a finite field, and let $k'$ be a quadratic extension of $k$.  We denote by $\sigma$ the Frobenius element of $k$, thought of as the generator in $\mathrm{Gal}(k'/k)$.
Consider the Weil restrictions
\[
\bG' = \Res_{k'/k}\bG_{k'}\quad \text{and}\quad \bH' = \Res_{k'/k}\bH_{k'}
\]
The involution $\sigma$ acts on $\bG'$ and $\bH'$, and hence also on $(\bfh')^\perp\subset \bfg' = \mathrm{Lie}(\bG')$. The involutions $\tau_i$ similarly induce involutions on $\bH'$ and $(\bfh')^\perp$ which we call by the same name, and which commute with the action of $\sigma$. We form the two rational twists:
\begin{itemize}
	\item The split-inert or unitary case:
\[
H^\unitary =  (\bH')^{\tau_1\sigma}\quad \text{and}\quad V^\unitary = (\bfh'^\perp)^{\tau_1\sigma}
\]
	\item The inert-inert or Siegel case:
\[
H^\siegel = (\bH')^{\tau_2\sigma}\quad \text{and}\quad V^\siegel = (\bfh'^\perp)^{\tau_2\sigma}
\]
\end{itemize}

More precisely, set $g^* = {}^t\sigma(g)$ for any matrix $g$. We have the unitary group $\U_n$ consisting of $g\in \GL_n(k')$ such that $g\cdot g^* = 1$. Recall that $\Herm_n = \{x\in \mathfrak{gl}_n(k')\mid x = x^*\}$ is the space of $n\times n$ Hermitian matrices.

The above construction produces the following spaces.
\begin{itemize}
	\item \underline{the split-inert/Siegel case} Here, $H^\siegel = \Res_{k'/k}\GL_n$ acts on $V^\siegel=\Herm_n\times \Herm_n$ by $g\cdot (x,y) = (gxg^*,(g^*)^{-1}yg^{-1})$. This is the infinitesimal model of the symmetric variety $X^\Siegel = \Res_{k'/k}\GL_n\backslash \U_{2n}$.
	\item \underline{the inert-inert/unitary case} Here $H^\unitary=\U_{n}\times \U_n$ acts on $V^\unitary=\mathfrak{gl}_n(k')$ by $(g_1,g_2)\cdot x = g_1xg_2^*$. This is the infinitesimal model of the symmetric variety $X^\unitary = (\U_{n}\times \U_n)\backslash \U_{2n}$.
\end{itemize}

\begin{remark}
    The terminology of \emph{split-inert} and \emph{inert-inert} comes from \cite{LXZfund,LXZrtf}, and is motivated by a more general bi-quadratic setup. In this paper, we opt for the self-contained terminology of the \emph{Siegel} and \emph{unitary} cases as we do not discuss this more general setting.
\end{remark}

We keep the same notation as Section \ref{sec: invariant theory} with $V^\star$ replacing $\fh^\perp$ (for $\star = \siegel$ or $\unitary$).  The following is an easy variation on Lemma \ref{Lem: cat quot split}.

\begin{lem}\cite[Lemma 5.3]{LXZfund}
	\label{lem: GIT quotient does not depend on form}
	The GIT quotients $\fc$ for each of $X^\Siegel$ and $X^\unitary$ above are isomorphic as $k$-varieties. Setting $\fc = V^\star\git H^\star$ for this quotient, we have $\fc\simeq \bA^n$ is an affine space over $k$ equipped with a $\G_m$-action by the exponents $2,4,\cdots,2n$.
\end{lem}

Moreover, the following result identifies the regular centralizer group schemes for $X^\unitary$ and $X^\Siegel$.

\begin{thm}
\label{thm: regular centralizers for rational forms}
    Consider the $k$ rational form $X^\star$ for $\star\in \{\unitary,\siegel\}$. There exists a canonical smooth commutative group $k$-scheme $J = J^\star$ over $\fc$ such that we have an isomorphism $(\chi^\reg)^*J = I^\reg$ defined over $k$ which extends uniquely to a morphism $\chi^*J\to I$. Moreover, $J$ satisfies the following:
    \begin{enumerate}
        \item The group schemes $J$ for $X^\unitary$ and $X^\Siegel$ are canonically isomorphic.
        \item $J$ can be explicitly described as the kernel of the norm map
        \[
\mathrm{Nm}\colon \Res^{(\fs/\iota)_{k'}}_{\fc}(\G_m)\to\Res^{\fs/\iota}_{\fc}(\G_m),
        \]
        which takes an $S$ point $\xi\in \Gamma\!\left(S\times_\fc(\fs/\iota)_{k'},\cO_{S\times_\fc(\fs/\iota)_{k'}}^\times\right)$ to $\xi\cdot \sigma(\xi)$.
    \end{enumerate}
\end{thm}
\begin{proof}
    \newcommand{\symb}{\dagger} 
    By Proposition \ref{prop: regular centralizer package}, if we write $\breve{I}^\reg = I^\reg\otimes_k\ol k$, then there is a unique group scheme $\breve{J}$ such that $(\chi^\reg)^*\breve{J}\simeq \breve{I}^\reg$, with a unique extension to a map $(\chi^\reg)^*\breve{J}\to \breve{I}$. Moreover, by Proposition \ref{prop: universal spectral}, we may write $\breve{J} = \mathrm{Res}^{\fs/\iota}_\fc(\G_m)$.

    Consider the $k$-model $J^{\symb}$ of $\breve{J}$ given by the following Frobenius action:  Let $\sigma$ act diagonally on $\G_m\times \fs/\iota$, with the action on $\G_m$ by inversion and the action on $\fs/\iota$ induced by the natural action of $\sigma$ on $\fs$ such that $\sigma(x) = -x$. We note that $J^{\symb}$ is simply the kernel of the norm map as in (2) above. It is clear that $J^{\symb}$ is smooth and commutative. We claim that over the regular semisimple locus $(\chi^{\rs})^*J^{\symb}\simeq I^\rs$. Indeed, for either case if we use our chosen model of $G=\U_{2n}$, one verifies that the centralizer $C\subset H^\star$ of $\fa\cap V^\star$ for $\fa$ as in \eqref{eqn: theta Cartan} is a group scheme over $k$ identified with $C(\ol k)\simeq \G_m(\ol k)^n$ with Frobenius acting by inversion composed with the Galois action on coefficients. Since all regular semisimple elements are conjugate, and since $\fa\cap V^\star\neq \emptyset$ for $\star\in \{\unitary,\siegel\}$, we get our stated identification on the regular semisimple locus.

    Now, since $\breve{I}^\reg$ is smooth and separated, it admits a unique $k$-model agreeing with $I^\rs$ over $\fc^\rs$. Since $\chi^*J^{\symb}$ and $I^\reg$ are two such $k$ models, we conclude $(\chi^\reg)^*J^{\symb}\simeq I^\reg$. Moreover, since $J^{\symb}$ is smooth, there exists a canonical extension $\chi^*J^{\symb}\to I$. Hence, the result is proven with the $k$ model $J := J^{\symb}$ of $\breve{J}$.
\end{proof}

While the GIT quotient is independent of form, the existence of Kostant-Rallis sections is sensitive to which $k$ form of $\bX$ we consider.
\begin{thm}
\label{thm: section for rational forms}
    The section in Theorem \ref{thm: levy's results}(2) descends to a section 
    \[
    \fc:= V^\siegel\git H^\siegel\to V^{\siegel,\reg}
    \]
    over $k$ for $X^\Siegel$. However, such a section does not exist for $X^\unitary$.
\end{thm}
\begin{proof}
    Suppose that $e$ is a regular, nilpotent element in $V^\star(\ol k)$, and let $\lambda$ be an associated character for $e$ valued in $H^\star(\ol k)$, i.e. $\mathrm{Ad}(\lambda(t))e = t^2e$ for all $t\in \ol k^\times$. Such an associated character exists by \cite[Corollary 5.4]{levy}. By \cite[Lemma 6.30]{levy}, a section to the map $\chi\colon V^\star\to \fc$ over the algebraic closure $\ol{k}$ is given by $e+\mathfrak{v}$ where $e$ is a regular nilpotent element in $\fh^\perp$ and $\mathfrak{v}$ is any choice of $\lambda$ graded subspace of $\fh^\perp$ such that a slightly enhanced version of the complementary property
    \[
    [e,\fh]\oplus \mathfrak{v} = \fh^\perp
    \]
    holds. (``Enhanced'' here means that it remains complementary in a smaller subspace of $\fh^\perp$. See \cite[\S6.5]{levy} for a precise statement.) 
    
    Suppose that we can find a $k$ rational point $e\in V^\star(k)$ which is regular and nilpotent. Then since $k$ is perfect of characteristic good for $G$, by \cite[Theorem 26]{mcninch}, we can further assume that the associated cocharacter $\lambda$ is defined over $k$, and so takes values in $H^\star$. As $[e,\fh]\subset \fh^\perp$ is stable under the Frobenius action, we can therefore choose $\mathfrak{v}$ to be stable under Frobenius as well. In this case, Levy's section over $\ol k$ is stable under Frobenius and hence descends to a section over $k$. Note that the existence of a section over $k$ implies the existence of a $k$ rational regular, nilpotent element $e$, so we have proved that the existence of a section over $k$ is equivalent to the existence of a regular, nilpotent element over $k$.

    It reduces, therefore, to study for which rational forms there exists a $k$-rational regular nilpotent element $e\in V^\star(k)$. Let $\mathcal{N}^\star\subset V^\star$ denote the nilpotent cone in $V^\star$. Geometrically, there are two $\rH$-orbits of regular nilpotent elements $\mathcal{N}^{\star,reg}$, represented by the elements
    \[
    \begin{pmatrix}
        &e_n\\I_n&
    \end{pmatrix}  \quad\text{and}\quad \begin{pmatrix}
        &I_n\\e_n&
    \end{pmatrix},
    \] 
    where $e_n\in\fgl_n$ is regular nilpotent. It is not hard to see that the Frobenius element permutes these two orbits in the unitary case and stabilizes each them in the Siegel case. Hence in the unitary case $\mathcal{N}^{\unitary,\reg}(k)=\emptyset$ and in the Siegel case $\mathcal{N}^{\siegel,\reg}(k)\neq\emptyset$. The result follows.
\end{proof}

We will fix a choice of $k$-rational section for $X^\siegel$ for the rest of this paper, denoting it by $\e$ and referring to it as a \emph{Kostant--Rallis section}.

\subsection{Endoscopic symmetric varieties}
\label{sec: endoscopic spaces}

The \emph{endoscopic symmetric varieties} for a rational form $X^\star$ were studied by the second author for $X^\unitary$ in \cite{spencer_annals}, where a fundamental lemma was proven for $X^\unitary$ over a $p$-adic field, and for general symmetric varieties satisfying mild assumptions in \cite{sl_endoscopy}. We summarize the latter theory here for the special case of $X^\Siegel$.

Let $\canT$ denote the canonical torus for $G$. We have a canonical involution $\theta$ on the torus $\canT$ whose centralizer $\canC = \canT^\theta$ is identified with the centralizer subgroup of a generic regular, semisimple element of $\fh^\perp$ (see \cite[Proposition 1.12]{leslie} and \cite[Lemma 4.13]{HM}). Let $\wh{\bG}$ denote the Langlands dual group of $G$, considered over $\C$, with dual torus $\wh{\canT}$. Recall that we have a canonical torus $\canA$ associated to the symmetric variety $\mathbf{X}$, which fits into an exact sequence
\[
1\to \canC\to \canT\to \canA\to 1.
\]
This exact sequence dualizes to an exact sequence
\[
1\to \wh{\canA}\to \wh{\canT}\to \wh{\canC}\to 1.
\]

Recall the dual group $\wh{\bG}_X$ associated to $\mathbf{X}$, which in our case is the subgroup $\wh{\bG}_X = \Sp_{2n}$ of $\wh{\bG} = \GL_{2n}$. The homogeneous variety $\wh{\mathbf{X}} = \wh{\bG}_X\backslash\wh{\bG}$ is itself a symmetric variety for an involution $\wh{\theta}$. Moreover, the torus $\wh{\canT}$ acts on $\wh{\bX}$ freely through the quotient $\wh{\canT}\to \wh{\canC}$, so that the above sequence extends to a commutative diagram of
\[
\begin{tikzcd}
    \wh{\canA}\ar[r]\ar[d]&\wh{\canT}\ar[r]\ar[d]&\wh{\canC}\ar[d]\\
  \wh{\bG}_X\ar[r]&\wh{\bG} \ar[r]&\wh{\bX}.&
\end{tikzcd}
\]
 The variety $\wh{\mathbf{X}}$ is a quasi-Hamiltonian variety, admitting a group-valued moment map
\[
s\colon \wh{\mathbf{X}}\to \wh{\bG}, \qquad \wh{\bG}_X\cdot g\longmapsto g^{-1}\cdot \hat{\theta}(g).
\]
We will refer to the map $s$ as the \emph{symmetrization map} for $\wh{\mathbf{X}}$. Note that $s$ is almost a section of the canonical map $\wh{\canT}\to \wh{\canC}$ in the sense that $s$ is injective and the composition 
\[
\wh{\canC}\xrightarrow{s} \wh{\canT}\to \wh{\canC}
\]
is the map $x\mapsto x^2$.

We now recall the classification of elliptic endoscopic data for the unitary group $\U_{2n}$. Let $\bG = \GL_{2n}$ and $\rho_G$ be the $\Out(\bG)$-torsor which defines the unitary group $\U_{2n}$. For $\kappa\in \wh{\canT}$, we may take the neutral component of the centralizer $\wh{\bG}_\kappa$ of $\kappa$ in $\wh{\bG}$, and similarly for $(\wh{\bG}\rtimes \Out(\bG))_\kappa$ in $\wh{\bG}\rtimes \Out(\bG)$. Let $\pi_0(\kappa)$ be the group of connected components of $(\wh{\bG}\rtimes \Out(\bG))_\kappa$, which is a finite group fitting into a short exact sequence
\[
1\to \wh{\bG}_\kappa\to (\wh{\bG}\rtimes \Out(\bG))_\kappa\to \pi_0(\kappa)\to 1.
\]
The group $\pi_0(\kappa)$ admits morphisms
\begin{equation}
\label{eqn: pi 0 kappa to Out G}
\Out(\wh{\bG}_\kappa)\leftarrow \pi_0(\kappa)\to \Out(\wh{\bG})
\end{equation}
where the rightward map is induced by the projection onto the second factor and the leftward map is induced by the action of $(\wh{\bG}\rtimes \Out(\bG))_\kappa$ on $\wh{\bG}_\kappa$.

Let $F$ be either a non-archimedean local field with residue field $k$ or a function field $F = k(\Sigma)$ for $\Sigma$ a smooth, projective, geometrically-connected curve over $k$ of genus at least 2. Consider a quadratic extension $F'/F$. For example, in the case of a function field, such an extension is determined by an \'etale, degree 2 cover
\[
\vartheta\colon \Sigma'\to \Sigma,
\]
with $\Sigma'$ a smooth, projective, geometrically-connected curve over $k$. Let $\sigma\in \Gal(F'/F)$ be the canonical degree 2 involution, and let $G = \U_{2n}$ be the associated unitary group. That is, if $F$ is a non-archimedean local field,
\[
G = \U_{2n} = \{g\in \GL(F')\mid g\cdot ^t\!\!\sigma(g) = 1 \},
\]
while if $F = k(\Sigma)$ is a function field,
\[
G = \vartheta_*(\GL_{2n}\times \Sigma')^{\tau\sigma}
\]
where $\tau$ is the transpose map on $\GL_{2n}$ and $\sigma$ acts through the automorphism group of the cover. This group $G$ is determined by an $\Out(\bG)$-torsor  $\rho_G$ over $F$.

\begin{definition}
    Fix $x\in\Sigma(k)$ for the function field case and $x\in \Sigma^\bullet = \Spec(F)$ in the case of a non-archimedean local field. An \emph{endoscopic datum} for the unitary group $\U_{2n}$ is a pair $\fe =(\kappa,\rho)$, where $\kappa\in \wh{\canT}$ and $\rho$ is a homomorphism
    \[
    \rho\colon \pi_1(\Sigma,x)\to \pi_0(\kappa)
    \]
    which induces $\rho_G$ via the rightward map in \eqref{eqn: pi 0 kappa to Out G}.

    We call an endoscopic datum $(\kappa,\rho)$ \emph{elliptic} if $(Z(\wh{G}_\kappa)^{\pi_1(\Sigma,x)})^\circ = \{1\}$, where $(\bullet)^\circ$ denotes the neutral component and $\pi_1(\Sigma,x)$ acts through $\rho$.
\end{definition}

If the unitary group $\U_{2n}$ is defined over the local field $F = k(\!(\varpi)\!)$ for $k$ a finite field, then it is easy to verify that $(\kappa,\rho)$ is elliptic if and only if $\kappa\in \wh{\canC}$ is 2-torsion. Moreover, in this case, the map $\pi_0(\kappa)\to \Out(\bG)$ is a bijection of order-2 groups. Hence, the data $(\kappa,\rho)$ is completely determined by the data of $\kappa$ in such cases.

An endoscopic datum determines a quasisplit group $G^\fe$ over $\Sigma$ by considering the induced quasisplit group from the split group $\bG_\kappa$ dual to $\wh{\bG}_\kappa$ together with the $\Out(\bG_\kappa)$-torsor induced by $\rho_\kappa$ via the leftward arrow in \eqref{eqn: pi 0 kappa to Out G}. The associated endoscopic groups $G^\fe$ for $F$ a non-archimedean field are the unitary groups $\U_r\times \U_s$ for $r+s = 2n$.

We now add the data of a relative endoscopic datum.

\begin{definition}
\label{def: endoscopic datum}
    An \emph{endoscopic datum} for a symmetric variety $X^\siegel$, or \emph{relative endoscopic datum}, is a pair $(\kappa,\rho)$, where $\kappa\in \wh{\canC}$ and $(s(\kappa),\rho)$ is an endoscopic datum for the unitary group $G = \U_{2n}$.

    We say that $(\kappa,\rho)$ is \emph{elliptic} if $(s(\kappa),\rho)$ is elliptic as a an endoscopic datum for $\U_{2n}$.
\end{definition}

In \cite{sl_endoscopy}, the second author produced a symmetric space $Y^{\fe}$ for $G^\fe$, which satisfies the following properties.

\begin{thm}
\label{thm: existence of endoscopic symmetric variety}
    (\cite[Theorem 1.4]{sl_endoscopy}) For any elliptic relative endoscopic datum $(\kappa, \rho)$ for $X$ a ``well adapted'' symmetric space, there exists a symmetric space
    \[
    Y^\fe = H_{Y^\fe}\backslash G^\fe,
    \]
    unique up to $G^\fe$-inner form, where $G^\fe$ is the endoscopic group of $G$ determined by the endoscopic datum $(s(\kappa),\rho)$, such that
    \begin{enumerate}
        \item The canonical torus for $Y^\fe$ is identified with the canonical torus for $X$;
        \item $\wh{\bG}_{Y^\fe}$ is the centralizer of $\kappa$ in $\wh\bG_X$;
        \item There is a canonical map $\Aut_d(X)\to \Aut_d(Y^\fe)$ between the groups of doubling automorphisms (see \S3.4 of \emph{loc cit.}) such that the geometric 1-cocycle (see \S3.4.3 of \emph{loc cit.}) of $Y^\fe$ is induced from the geometric 1-cocycle of $X$ via this map.
        \item (\cite[Proposition 8.5]{sl_endoscopy}) There is canonical transfer morphism 
        \[
    \nu_\fe\colon \fc_\fe\to \fc = T^*X\git G,
        \]
        where $\fc_\fe = (\fg^\fe)^{\theta^\fe = -1}\git H_{Y^\fe}$, which is finite and flat.
    \end{enumerate}
\end{thm}

Conditions (1) through (3) above are needed to describe $Y^\fe$ up to $G^\fe$-inner form, but only conditions (1) and (4) will be needed in this paper.

We will now specialize to the case $X = X^\siegel$ and explicitly describe the transfer morphism $\nu_\fe$ of Theorem \ref{thm: existence of endoscopic symmetric variety}. The reader unfamiliar with \cite{sl_endoscopy} can, for the purposes of this paper, take the following as a definition.

Consider an elliptic relative endoscopic datum $(\kappa,\rho)$ for $X^\siegel$. By our earlier remarks, $\kappa\in \wh\canC$ must be 2 torsion, and $\rho$ is completely determined by $\kappa$. View $\kappa = (\kappa_1,\dots, \kappa_n)$ as an $n$-tuple with $\kappa_i\in \{\pm1\}$, and suppose that $\kappa$ has $r$ indices $+1$ and $s$ indices $-1$ (with $r+s = n$). Then,
\[
Y^\fe = X^{\Siegel}_{r}\times X^{\siegel}_s : = \big(\Res_{F'/F}\GL_r\backslash \U_{2r}\big)\times \big(\Res_{F'/F}\GL_s\backslash \U_{2s}\big).
\]
     For $\kappa$ having $\kappa_i = 1$ for $r$ indices and $\kappa_i = -1$ for $s$ indices, the transfer map $\nu_\fe$ is the morphism
    \[
    \fc_r\times \fc_s \simeq \fa\git \big((\mathfrak{S}_r\times \mathfrak{S}_s)\rtimes \{\pm1\}^n\big)\to \fa\git W_X
    \]
    with subscripts $r$ and $s$ denoting the rank of the symmetric variety factor, and the final map induced by the embedding $\mathfrak{S}_r\times \mathfrak{S}_s\subset \mathfrak{S}_n$ given by the distribution of $+1$'s and $-1$'s in $\kappa$. It is evident that this morphism is finite and flat. 

    Let $\fs_\fe\to \fc_\fe$ denote the spectral cover for $\fe$, with involution $\iota_\fe$, which is finite, flat of degree $2n$ over $\fc_\fe$. The following characterizes $\fs_\fe$ and its relationship to $\fs$.

    \begin{prop}
\label{prop: centralizers under transfer}
    There is a canonical partial normalization map $\fs_\fe\to \fs\times_\fc\fc_\fe$, which induces a canonical morphism $\nu_\fe^*J\to J_\fe$. The latter map of regular centralizers is an isomorphism over $\nu_\fe^{-1}(\fc \setminus \fD^\cent)$.
\end{prop}
\begin{proof}
     If $X^\fe = X^\siegel_r\times X^\unitary_s$ for $r+s = n$, then identify $\fc_\fe \simeq \fc_r\times \fc_s$, with $\fc_r$ the GIT quotient for $X^\siegel_r$ and $\fc_s$ for $X^\unitary_s$. Let $\fs_r\to \fc_r$ and $\fs_s\to \fc_s$ denote the respective spectral covers. Then, we have
    \[
k[\fs_\fe] = k[\fs_r]\times k[\fs_s],
    \]
    so that $\fs_\fe = \fs_r\sqcup \fs_s$. Let $P_r(x), P_s(x)$ be the polynomials in $x$ such that $k[\fs_r] = A[x]/P_r(x)$ and $k[\fs_s] = A[x]/P_s(x)$. The coefficients of $P_r,P_s$ are the coordinate functions on $\fc_r,\fc_s$. Therefore, the morphism $\fc_r\times \fc_s\to \fc_n$, being induced by the inclusion $\fS_r\times \fS_s\subset \fS_n$, induces on any point $a\in \nu_{\fe}(\fc_\fe)$ a canonical factorization
    \[
  x^{2n}+a_2x^{2n-2}+\cdots +a_{2n} = P_r(x)\cdot P_s(x)  
    \]
    so that $k[\fs\times_\fc\fc_\fe] = A[x]/\big(P_r(x)\cdot P_s(x)\big)$. Hence, we obtain a $\iota$ equivariant morphism 
    \[
\fs_\fe\to \fs\times_{\fc}\fc_\fe
    \]
    over $\fc_\fe$ which is the partial normalization $\fs_r\sqcup \fs_s\to \fs_r\cup\fs_s$ at the intersection of $\fs_r$ and $\fs_s$ in $\fs\times_\fc\fc_\fe$. 
    
    In particular, we also deduce a partial normalization map
    \[
\fs_\fe/\iota_\fe\to \fs/\iota\times_{\fc}\fc_\fe
    \]
    Since $\mathrm{Res}^{\fs_\fe/\iota_\fe}_{\fc_\fe}(\G_m) = \mathrm{Res}^{\fs_\fe/\iota_\fe}_{\fs/\iota\times_\fc\fc_\fe}\left(\mathrm{Res}^{\fs/\iota\times_\fc\fc_\fe}_{\fc_\fe}(\G_m)\right)$, we obtain the morphism $\nu_\fe^*J\to J_\fe$ by the adjuction map for the Weil restriction $\mathrm{Res}^{\fs_\fe/\iota_\fe}_{\fs/\iota\times_\fc\fc_\fe}$. In particular, this morphism is an isomorphism over the open set in $\fc_\fe$ over which $\fs_r$ and $\fs_s$ are disjoint in $\fs\times_\fc\fc_\fe$, which includes the preimage of $\fc\setminus\fD^\cent$.
\end{proof}

In \cite{LXZfund}, it was noted that $Y^\fe$ is not the endoscopic space needed to formulate a fundamental lemma for $X^\siegel$. Indeed, the endoscopic variety $Y^\fe$ does not distinguish between the endoscopic datum for $\kappa = \mathbf{1}_r\times (-\mathbf{1}_s)$ and for $\kappa = (-\mathbf{1}_r)\times \mathbf{1}_s$ (and the corresponding orbital integrals are different and the full endoscopic datum does distinguish between the two). 

In \cite{LXZfund}, the authors propose the following endoscopic space, which we give as a definition.

\begin{definition}\label{def: end space}
    For an elliptic relative endoscopic datum $\fe = (\kappa,\rho)$, with $\kappa_i = 1$ for $r$ indices and $\kappa_i = -1$ for $s$ indices, we call the variety
    \[
    X^\fe = H^\fe\backslash G^\fe = \big(\Res_{F'/F}\GL_r\backslash \U_{2r}\big)\times \big((\U_s\times \U_s)\backslash \U_{2s}\big).
    \]
    the \emph{endoscopic symmetric space} associated to $\fe$. 
\end{definition}

Since $\fc_\fe$ and $J_\fe$ are the same for $X^\fe$ and $Y^\fe$, there is still a unique transfer morphism $\nu_\fe$ with canonical isomorphism $\nu_\fe^*J\simeq J_\fe$.

\begin{remark}
    The definition of $X^\fe$ above is more or less \emph{ad hoc} as it is not clear how to extract the implied comparison of orbital integrals from the formalism of \cite{sl_endoscopy}. A generalization of \emph{loc. cit.} that accommodates this case more naturally is ongoing work of the second author.
\end{remark}

\subsection{Divisors in $\fc$}
\label{sec: divisors}

In this section, we introduce several important divisors in the GIT quotient $\fc$ which will control the geometry of our local and global moduli spaces.

We consider first the discriminant divisor of the root system $\Phi_\bX$. Namely, we consider the function
\[
\prod_{\alpha\in \Phi_\bX} \mathrm{d}\alpha\in k[\fa]
\]
where we are viewing each root as a function $\mathrm{d}\alpha$ on $\fa$. As this function is $W_X$-invariant, it induces a divisor in $\fc$ which we denote by $\fD_\bX$ and call the \emph{discriminant divisor}. For the $\theta$-Cartan $\fa$ as in \eqref{eqn: theta Cartan} we have
\[
\fD_\bX = \left((-1)^{\frac{n^2+n}{2}}\left( \prod_i \delta_i^2 \right)\left(\prod_{i<j}(\delta_i-\delta_j)^2\right)\right) = \left(a_{2n}\cdot \mathrm{Disc}(x^n+a_2x^{n-1}+a_4x^{n-2}+\cdots +a_{2n}) \right)
\]
where $\delta_i$ is the coordinate function for the $i$-th diagonal entry in $\delta$ in \eqref{eqn: theta Cartan}, and $\mathrm{Disc}$ denotes the discriminant of the polynomial in $x$, which is a polynomial function of the coefficients $a_{2i}$. We record here the following basic observation about $\fD_\bX$.

\begin{lem}
	(\cite[Lem 1.10.1]{ngo}) The divisor $\fD_\bX$ is a reduced, effective divisor on $\fc$, and the map $\pi\colon \fa\to \fc$ is \'etale on the complement of $\fD_\bX$. Moreover, this divisor is stable under the action of $\Out(\bH)$.
\end{lem}

We have a subdivisor of $\fD_\bX$ corresponding to the type $A_n$ root system on $\fa$, given by 
\[
\fD^\cent = \left((-1)^{\frac{n^2-n}{2}}\prod_{i<j}(\delta_i-\delta_j)^2\right) = \left(\mathrm{Disc}(x^n+a_2x^{n-1}+a_4x^{n-2}+\cdots +a_{2n}) \right).
\]
The centralizer group scheme $J$ is closely related to the type $\GL_n$ centralizers (see the spectral description given in Proposition \ref{prop: universal spectral} and part (2) of Theorem \ref{thm: regular centralizers for rational forms}), and so we call this divisor the \emph{centralizer divisor} associated to $\bX$. The following hints at its importance in describing regular centralizers using the spectral cover $\fs$.

\begin{lem}
	Let $\fs$ denote the spectral cover over $\fc$ described in \S\ref{sec: invariant theory}. The map $\fs/\iota\to \fc$ is \'etale away from the divisor $\fD^\cent$.
\end{lem}
\begin{proof}
    Ramification of the spectral cover $\fs\to \fc$ can occur over a point $a = (a_i)$ in $\fc$ if and only if the corresponding spectral equation
    \[
    Q_a(x) = x^{n}+a_{2}x^{n-1}+\cdots+a_{2n}
    \]
    has repeated roots. Since $\fD^\cent$ is cut out by the discriminant of the polynomial $Q_a(x)$, this is precisely the condition that $a\in \fD^\cent$.
\end{proof}

We finally come to the most subtle divisor, the branching divisor for $\bX$. In \cite{HM}, it was shown that there exists a divisor $\fD^\br\subset \fc$ such that the fibers $\chi^{-1}(x)$ contain exactly 2 regular $\bH$-orbits when $x\in \fD^\br$ and exactly 1 regular $\bH$-orbit when $x\in\fc - \fD^\br$. We call this divisor the \emph{branching divisor} for $\bX$. The following structure result will serve as a model from which the regular parts of our moduli spaces will be understood.

\begin{thm}
	\label{thm: regular quotient}
	(\cite[Theorems 3.7 and 3.37]{HM}) Let $(\fh^\perp)^\reg\myfatslash \bH$ denote the nonseparated scheme obtained by gluing two copies of $\fc$ away from the branching divisor $\fD^\br$, i.e.
	\[
	(\fh^\perp)^\reg\myfatslash \bH = \fc\coprod_{\fc\setminus\fD^\br}\fc.
	\]
	Then, the natural map $[\chi^\reg]\colon [(\fh^\perp)^\reg/\bH]\to \fc$ factors through
	\[
[(\fh^\perp)^\reg/\bH]\to (\fh^\perp)^\reg\myfatslash \bH\to \fc
	\]
	with the first map being a gerbe banded by the regular centralizer group scheme, trivialized by a choice of Kostant-Rallis section $[\epsilon]$.
\end{thm}

The branching divisor $\fD^\br$ can be described for more general spherical varieties, and it has a role in the relative Langlands program matching the determinant of the associated symplectic representation for tempered spherical varieties (see \cite{HML}). It is a surprising feature of this case that it is complementary to $\fD^\cent$ in the full discriminant divisor.

\begin{lem}
	(\cite[Prop 3.44, 3.45]{HM}) For the spherical variety $\bX = (\GL_n\times \GL_n)\backslash \GL_{2n}$, we have $\fD^\br = (a_{2n})$ is the divisor given by the determinant. In particular, we have a decomposition of reduced divisors
	\[
\fD_\bX = \fD^\cent+\fD^\br.
	\]
\end{lem}

Finally, we related the divisors above to their endoscopic counterparts via the transfer map. Let $X^\fe$ be an endoscopic symmetric variety for $X$, and let $\nu_\fe\colon \fc_\fe\to \fc$ be the corresponding transfer map (defined in \S\ref{sec: endoscopic spaces}). Choose $\Psi\subset \Phi_{X,\fa}-\Phi_{X^\fe,\fa}$ such that for any $\{\pm \alpha\}\subset \Phi_{X,\fa}-\Phi_{X^\fe,\fa}$, exactly one of $\alpha$ and $-\alpha$ is in $\Psi$. Then, by \cite[\S1.10.3]{ngo} the divisor $\nu_\fe^*\fD_\bX$ decomposes as
\[
\nu_\fe^*\fD_\bX = \fD_{\bX^\fe}+2\fR_{\fe}
\]
where $\fR_\fe = \left(\prod_{\alpha\in \Psi}\alpha\right)$ is the (reduced) ramification divisor for the map $\nu_\fe$. This decomposition admits a further refinement for $\fD^\cent$ and $\fD^\br$ as follows.

\begin{lem}
	\label{lem: pullback of divisors along transfer map}
	For $\fR_\fe$ as defined above, we have the following decompositions of pullbacks
	\[
	\nu_{\fe}^*\fD^\cent = \fD_\fe^\cent+2\fR_\fe,\quad \nu_{\fe}^*\fD^\br = \fD^\br_\fe.
	\]
\end{lem}
\begin{proof}
	Write $X^\fe = X^\unitary_{m}\times X^\Siegel_{n-m}$. We may explicitly express the map $\nu_\fe\colon \fc_\fe\to \fc$ as 
	\[
	\fa\git (W_{X^\unitary_{m}}\times W_{X^\Siegel_{n-m}}) = \fa\git \big((\mathfrak{S}_m\times \mathfrak{S}_{n-m})\rtimes \{\pm1\}^n\big)\to \fa\git W_\fa
	\]
	whose ramification is evidently described by a subdivisor of $\fD^\cent$.
\end{proof}

\section{Unitary Higgs Bundles}
\label{sec: unitary Higgs}

We review here the notation and definitions of the moduli spaces of unitary Higgs bundles of Ng\^o and Laumon \cite{ngo-laumon}. We will use some of their results to deduce properties of the relative Hitchin moduli spaces and their unitary relatives in \S\ref{sec: relative Hitchin spaces}.

First we fix the following data:  fix smooth, projective, geometrically-connected curves $\Sigma',\Sigma$ over $k$ with an \'etale degree 2 cover $\vartheta\colon \Sigma'\to \Sigma$. Let $\sigma\in \Aut(\Sigma'/\Sigma)$ denote the order-two generator of the automorphism group of the cover. We also fix a divisor $D$ on $\Sigma$ and let $S$ denote a test $k$-scheme.

\begin{definition}
    The moduli space of $\U_n$ Higgs bundles is the functor $\cN_n$ which takes values
    \[
    \cN_n(S) = \left\{ (\cE',h,\Phi')\left| \begin{array}{c} \cE'\text{ is a rank $n$ vector bundle on }\Sigma'\times S, \\ h\colon \cE'\to \sigma^*(\cE')^\vee\text{ is a Hermitian form on $\cE'$, i.e. $h = \sigma^*h^\vee$,}  \\ \Phi'\colon \cE'\to \cE'\otimes D,\; \text{satisfying $\sigma^*(\Phi')^\vee\circ h +h\circ \Phi' = 0$}\end{array}\right. \right\}.
    \]
\end{definition}

Note that, in the above, $(-)^\vee = \Hom(-,\cO_{\Sigma'})$ denotes the usual dual of vector bundles. In \cite{ngo-laumon}, it is shown that $\cN_n$ is representable by a smooth, finite-type algebraic stack. 

We now recall the construction of the Hitchin morphism. The involution $\sigma$ induces a decomposition
\[
\vartheta_*\cO_{\Sigma'} = \cO_\Sigma\oplus \cL
\]
where $\cL$ is the $(-1)$-eigenspace of $\sigma$ on $\vartheta_*\cO_{\Sigma'}$. Note that $\cL^{\otimes 2}\simeq \cO_\Sigma$. 

\begin{definition}
    Let $\cB_n$ denote the affine space
    \[
\cB_n = \bigoplus_{i=1}^n H^0(\Sigma,\cL(D)^{\otimes i}),
    \]
    where $\cL(D):=\cL\otimes D$. The Hitchin morphism $f\colon \cN_n\to \cB_n$ is the morphism taking $(\cE',h,\Phi')$ to the coefficients of the characteristic polynomial of $\Phi'$, i.e. $\tr\bigwedge^i\Phi'\in H^0(\Sigma,\cL(D)^{\otimes i})$.
\end{definition}

Fibers of the morphism $f$ have an interpretation via spectral covers similar to those associated to classical groups by Hitchin \cite{hitchin}. To state it, let $\wt{\Sigma}'$ and $\wt{\Sigma}$ be the base change of the spectral cover $\fs_n\to \fc$ along the respective morphisms fitting in the diagram
\[
\xymatrix{
\wt{\Sigma}' \ar[r]\ar[d]^-{\pi_a'} & \wt{\Sigma}\ar[d]^-{\pi_a}\ar[r] & (\fs_n)_{\cL(D)}\ar[d] \\
\Sigma'\times \cB_n\ar[r]^-{\vartheta} & \Sigma\times \cB_n\ar[r]^-{\mathrm{ev}} & (\fc_n)_{\cL(D)}.
}
\]
For every $a\in \cB_n(S)$, we denote by $\pi_a'\colon \wt{\Sigma}_a'\to \Sigma'$ the corresponding finite, flat cover of $\Sigma'$, and similarly for $\pi_a\colon \wt{\Sigma}_a\to \Sigma$. The automorphism $\sigma\in \Aut(\Sigma'/\Sigma)$  acts naturally on $\wt{\Sigma}'$, and  it is easy to verify that $\wt\Sigma_a = \wt\Sigma_a'/\sigma$ under this action.

\begin{Convention}\label{Convention for sheaves}For $\cG$ a sheaf on $\Sigma'$, we denote 
\[
\cG^\vee = \Hom_{\cO_{\Sigma'}}(\cG,\cO_{\Sigma'}),
\]
while for $\cF$ a sheaf on a spectral cover $\wt{\Sigma}'_a$ for $a\in \cA(S)$, we denote 
\[
\cF^\vee = \Hom_{\cO_{\wt{\Sigma}'_a}}\left(\cF,\omega_{\wt{\Sigma}_a'/\Sigma'\times S}\right).
\]
where $\omega_{\wt{\Sigma}_a'/\Sigma'\times S}$ is the relative dualizing sheaf for the finite, flat cover $\wt{\Sigma}_a'\to \Sigma'\times S$.
\end{Convention}
\begin{prop}
\label{prop: spectral description for unitary higgs}
    Let $\cN_a = f^{-1}(a)$. For all $a\in \cB_n(S)$ such that $\wt{\Sigma}_a'$ is reduced, we have
    \[
    \cN_a\simeq \left\{(\cF,h)\mid \cF\in \ol{\Pic}(\wt{\Sigma}'_a/S),\; h\colon \cF\to \sigma^*(\cF)^\vee\text{ Hermitian, i.e. }h = \sigma^*(h)^\vee \right\}
    \]
    where the $\ol{\Pic}$ denotes the compactified relative Picard scheme classifying torsion-free, generic rank-1 sheaves on $\wt{\Sigma}_a$, and the superscript indicates that we are considering the sub-functor classifying $\cF'$ such that $\cF'\simeq \sigma^*(\cF')^\vee$. 
\end{prop}

Given a torsion-free, generic rank-1 sheaf $\cF'$, we construct the corresponding unitary Higgs bundle by taking $\pi_{a,*}'(\cF')$, with the isomorphism $\cF'\simeq \sigma^*(\cF')^\vee$ providing the Hermitian structure $h$ and multiplication by the tautological section providing the section $\Phi'$.

The entire moduli space is acted on by its group of symmetries 
\[
\cP_\cN = \Pic(\wt{\Sigma}'/\cB_n)^{\sigma^* = (-)^{\vee}}
\]
in the natural way.

\begin{definition}
    The \emph{anisotropic locus} $\cB_n^\ani\subset \cB_n$ is the set of $a\in \cB_n$ such that $\pi_0(\cP_{\cN,a})$ is finite; that is, such that the morphism $\wt{\Sigma}_a'\to \wt{\Sigma}_a$ induces a bijection on irreducible components. (See \cite[Lem 2.8.2]{ngo-laumon}.)
\end{definition}

It was shown in \cite{ngo-laumon} that $\cB_n^\ani$ is a nonempty, open subset of $\cB_n$; the restriction $\cN_n^\ani = f^{-1}(\cB_n^\ani)$ is a Deligne--Mumford stack; and the Hitchin morphism $f^\ani = f|_{\cN_n^\ani}$ is proper.

\section{The relative Hitchin moduli space}
\label{sec: relative Hitchin spaces}

\subsection{Global moduli spaces}
\label{sec: global construction}

We keep the same notation as in \S\ref{sec: unitary Higgs} and formulate the global forms of the rational symmetric spaces in  Section \ref{sec: rational structures}. Namely, let $\Sigma$ be a smooth, geometrically-connected projective curve of genus $g\geq 2$ over the finite field $k = \bF_q$, and let $\vartheta\colon \Sigma'\to \Sigma$ be an \'etale 2-cover over $k$, with $\Sigma'$ also smooth, projective, and geometrically connected. Let $\Aut(\Sigma'/\Sigma) = \langle \sigma \rangle$ for $\sigma$ an involution.

The involution $\sigma$ acts on the Weil restrictions
\[
\bG' = \vartheta_*\bG\quad \text{and}\quad \bH' = \vartheta_*\bH,
\]
 and hence also on the bundle $(\bfh')^\perp\subset \bfg' = \mathrm{Lie}(\bG')$. The involutions $\tau_i$ induce involutions on $\bH'$ and $(\bfh')^\perp$ which we call by the same name, and which commute with the action of $\sigma$. We form the two twists $(H^\unitary,V^\unitary)$ and $(H^\siegel, V^\siegel)$ as before.

In each case, $V^\star$ admits a $\G_m$-action by scaling, which commutes with the action of $H^\star$, and hence, we may form the stack $[V^\star_D/H^\star]$ over $\Sigma$, where $V^\star_D = V^\star\times^{\G_m}D^\times$ is the $D$-twisted form of $V'$. Here and below, the superscript $\star$ denotes either $\star = \unitary$ for the split-inert case (using $\tau_1$) or $\star = \Siegel$ for the inert-inert case (using $\tau_2$).  Both are rational twists of the case
\[
H = (\bH')^\sigma = \bH\times \Sigma \quad\text{and} \quad V=(\bfh'^\perp)^\sigma = \bfh^\perp\times \Sigma.
\]

Define the moduli stack $\cM^\star=\cM_{D}^\star$ as the mapping stack
\[
\cM_{D}^{\star} = \mathrm{Sect}_\Sigma([V^\star_D/H^\star]).
\]
Explicitly, we may describe $\cM^\star$ as functors of points classifying certain unitary Higgs bundles with additional data. 

\begin{lemma}
    The moduli  $\cM^\star$ represent the functors sending a scheme $S$ to
\begin{equation}
\label{eqn: unitary functor of points}
\cM^\unitary(S) = \left\{ (\cE_1',\cE_2',h_1,h_2,\phi')\left| \begin{array}{c}\cE_i'\text{ are rank $n$ vector bundles on $\Sigma'\times S$,} \\
\text{$h_i\colon \cE_i'\to \sigma^*(\cE_i')^\vee$ is a Hermitian form, i.e. $\sigma^*h_i^\vee = h_i$,}\\\text{and }\phi'\colon \cE_1'\to \cE_{2}'\otimes D\text{ such that }\sigma^*(\phi')^\vee\circ h_1+h_2\circ \phi'= 0 \end{array} \right.\right\}
\end{equation}
\begin{equation}
\label{eqn: siegel functor of points}
\cM^\Siegel(S) = \left\{ (\cE',h,\phi_{12}',\phi_{21}')\left| \begin{array}{c}\cE'\text{ is a rank $n$ vector bundle on $\Sigma'\times S$, }\phi_{12}'\colon \cE'\to \sigma^*(\cE')^\vee\otimes D\\ \text{and }\phi_{21}'\colon \sigma^*(\cE')^\vee\to \cE'\otimes D\text{ such that }\sigma^*\phi_{ij}'^\vee = \phi_{ij}' \\
\text{and }h\colon \cE'\to \sigma^*(\cE')^\vee\text{ is a Hermitian form on $\cE'$}\end{array} \right.\right\}.
\end{equation}
\end{lemma}
We will describe these moduli spaces in terms of spectral covers in \S\ref{sec: invariant theory spectral}. First we show how they arise from taking fixed points under the appropriate involutions on the stack $\cN_{2n}$.

For $\star\in\{\unitary,\Siegel\}$, we define the maps $\mu^\star: \cM^\star\to \cN_{2n}$ as follows. In the unitary case, the map $\mu^\unitary$ is given on $S$-points by
\begin{equation}
    \mu^\unitary(\cE_1',\cE_2',h_1,h_2,\phi') = (\cE_1'\oplus \cE_2', h_1\oplus h_2,\Phi')
\end{equation}
where the Higgs field $\Phi'$ is given by the sum $\phi'\oplus \sigma^*(\phi')^\vee$, viewing the latter as a map $\cE_2'\to \cE_1'\otimes D$ by
    \[
    \cE_2'\xrightarrow{h_2} \sigma^*(\cE_2')^\vee \xrightarrow{\sigma^*(\phi')^\vee} \sigma^*(\cE_1')^\vee\otimes D\xrightarrow{h_1} \cE_1'\otimes D.
    \]
    It is now simple to check that $\mu^\unitary(\cE_1',\cE_2',h_1,h_2,\phi')\in \cN_{2n}(S)$.

    In the Siegel case, the map $\mu^\Siegel$ is given on $S$-points by
\begin{equation}
\mu^\Siegel(\cE',h,\phi_{12}',\phi_{21}')= (\cE'\oplus \sigma^\ast(\cE')^\vee, h\oplus \sigma^\ast(h)^\vee, \phi_{12}'\oplus\phi_{21}')\in \cN_{2n}(S).
\end{equation}

\subsection{Fixed Points on $\cN_{2n}$}

Denote by $\cN = \cN_{2n}$ the moduli space of unitary Higgs bundles of rank $2n$ (see \S\ref{sec: unitary Higgs}). To relate the moduli spaces defined in the previous section to fixed point stacks, we first consider involutions at the level of invariant theory.  Recall that $G = \U_{2n}$ is the unitary group over $k$, defined relative to the quadratic extension $k'$. Let $\theta^\star:G\to G$ denote the involution for which $H^\star = G^{\theta^\star}$. More precisely, write $G=\U(V)$ where $(V,h)$ is our fixed Hermitian space over $k$ with $h:V\overset{\sim}{\longrightarrow} \sigma^\ast(V)^\vee$; we sometimes write $\langle\cdot,\cdot\rangle:V\otimes V\to k'$ for the induced $\sigma$-sesquilinear form 
\[
\langle v, w\rangle := h(v)(\sigma^\ast({w})).
\] For any $g\in \GL(V)$ set $g^\dagger\in\GL( \sigma^\ast(V)^\vee)$ such that
\[
g^\dagger(h(v))(\sigma^\ast(w)) = h(v)(\sigma^\ast(g^{-1}w)) = \langle v,g^{-1}w\rangle;
\]
then $g\in \U(V)$ if and only if the diagram
\[
         \begin{tikzcd}
             V\ar[r,"h"]\ar[d,"{g}"]&\sigma^\ast(V)^\vee \ar[d,"{g^\dagger}"]\\
             V\ar[r,"h"]& \sigma^\ast(V)^\vee
         \end{tikzcd}
         \]
         commutes.
         
Let $J^\star: V\to V$ be the involution inducing an eigenspace decomposition $V=V_+\oplus V_-$ such that
\begin{enumerate}
    \item $(\star= \unitary)$ $\langle J^\unitary v,J^\unitary w\rangle= \langle v,w\rangle$. This implies that
    \[
    J^{\unitary,\dagger}\circ h = h\circ J^\unitary,
    \]which in particular forces the restriction of $h$ to $V_\pm$ is a non-degenerate Hermitian form.
    \item $(\star = \siegel)$  $\langle J^\siegel v,w\rangle+\langle v,J^\siegel w\rangle= 0.$ Note that this implies $\langle J^\Siegel v,J^\Siegel w\rangle= -\langle v,w\rangle$ since $(J^\siegel)^2=I$. This implies that
    \[
    J^{\siegel,\dagger}\circ h = -h\circ J^\siegel,
    \] and that the restriction of $h$ to $V_\pm$ vanishes so that $V_- \simeq V_+^\ast$.
\end{enumerate}
     There is an evident involution $\theta^\star:[\fg/G]\to [\fg/G]$ induced by $\theta^\star(g) = J^\star gJ^{\star,-1}$and $\theta^{\star}(x) = -J^\star xJ^{\star,-1}$ on the Lie algebra.

    Recall that an $S$-point of $[\fg/G]$ consists of the data of a triple $(\cE',h,\phi')$, with $\cE'$ a vector bundle on $S\otimes_kk'$, $\phi\colon \cE'\to \cE'$ a linear map, and $h\colon \cE'\to \sigma^*(\cE')^\vee$ a Hermitian form on $\cE'$, i.e. $h = \sigma^*(h)^\vee$.
\begin{lem}
\label{lem: standardization of involution}
    There is an automorphism $[\fg/G]\to [\fg/G]$ which brings $\theta^\star$ to the involution $\iota^\star$ which sends $(\cE',h,\phi')\in [\fg/G](S)$ to
    \[
\iota^\star(\cE',h,\phi') = \begin{cases}
(\cE',-h,-\phi') & \text{if }\star = \siegel \\
(\cE',h,-\phi') & \text{if }\star = \unitary
\end{cases}
    \]
\end{lem}
\begin{proof}
    To a triple $(\cE',h,\phi')\in [\fg/G](S)$, $\theta^\star$ induces a triple $(\theta^\star_*(\cE'),\theta^\star_*(h),\theta^\star_*(\phi'))$ as follows:  Fix an \'etale covering $\{U_\alpha\to S\otimes_kk'\}$ which is stable under the involution $\sigma$, with local trivializations $\nu_\alpha: \cE'|_{U_\alpha}\iso V\otimes \mathcal{O}_\alpha$ and descent data $\{g_{\alpha,\beta}\}$. By definition the diagrams
       \[
         \begin{tikzcd}
             V\otimes \mathcal{O}_{\alpha,\beta}\ar[r,"h_{\alpha,\beta}"]\ar[d,"{g_{\alpha,\beta}}"]&\sigma^\ast(V)^\vee \otimes \mathcal{O}_{\alpha,\beta}\ar[d,"{g_{\alpha,\beta}^\dagger}"]\\
             V\otimes \mathcal{O}_{\alpha,\beta}\ar[r,"h_{\alpha,\beta}"]& \sigma^\ast(V)^\vee \otimes \mathcal{O}_{\alpha,\beta}
         \end{tikzcd}
         \]
         commute for all $\alpha,\beta$. Then,
     \begin{enumerate}
         \item  the bundle $\theta_\ast^\star(\cE')$ is obtained from the descent data $\{J^\star g_{\alpha,\beta}J^{\star,-1}\}$. Intuitively, the local sections $\nu_\alpha$ induce local sections $J^\star v$ of $\theta_\ast^\star(\cE')$.
         \item the map $\theta_\ast^\star(\phi'): \theta_\ast^\star(\cE')\to \theta_\ast^\star(\cE')$ is obtained by locally conjugating by $J^\star$ and gluing together via $\{  J^\star g_{\alpha,\beta}J^{\star,-1}\}$
         \item We first claim that there is a natural isomorphism
         \[
         \theta_\ast^\star(\sigma^\ast(\cE')^\vee) \simeq \sigma^\ast(\theta_\ast^\star(\cE'))^\vee.
         \]
         To see this, identify $\{g_{\alpha,\beta}^\dagger\}$ as the descent data for $\sigma^\ast(\cE')^\vee$. Then the isomorphism follows from the matrix identity
         \[
         J^{\star,\dagger} g_{\alpha,\beta}^\dagger  (J^{\star,-1})^{\dagger} =   (J^{\star} g_{\alpha,\beta}J^{\star,-1})^\dagger.
         \]
         With this identity, the map $\theta^\star_\ast(h): \theta^\star_\ast(\cE') \longrightarrow\sigma^\ast( \theta^\star_\ast(\cE'))^\vee$ is obtained by gluing the local maps $\theta^\star_\ast(h)|_{U_\alpha}\equiv h$. We note that is well defined since
         \[
         (J^\star g_{\alpha,\beta}J^{\star,-1})^\dagger\circ h = h\circ(J^\star g_{\alpha,\beta}J^{\star,-1}),
         \]
         which can be checked using the compatibility between $J^\star$ and $h$.
     \end{enumerate}
    From the above, we see that the adjoint action by $J^{\star,-1}$ realizes an isomorphism $\theta^\star_*(\cE)\simeq \cE$ which identifies $\theta^\star_*(\phi)$ with $-\phi$. Finally, the Hermitian form $h$ is carried to $J^{\star,\dagger}hJ^{\star,-1} = \pm h$, with negative sign for $\star = \siegel$ and positive for $\star= \unitary$ due to the relations between $J^\star$ and $h$ above the lemma. 
\end{proof}
As a consequence of Lemma \ref{lem: standardization of involution}, we will work with the more explicit involutions $\iota^\star$. The involutions $\iota^\star$ induce involutions on the moduli of unitary Higgs bundles $\cN$ by sending $(\cE',h,\Phi')\in \cN(S)$ to
\begin{equation}
    \label{eqn: involution on cN}
\iota^\star(\cE',h,\Phi') = \begin{cases}
    (\cE',-h,-\Phi') & \text{if }\star = \siegel \\
    (\cE',h,-\Phi') & \text{if }\star = \unitary
\end{cases}
\end{equation}
Form the homotopy fixed point stack $\cM^{\star,\flat} = \cN^{\iota^\star}$ as the fixed points in $\cN$ of $\Z/2\Z$, with nontrivial element acting by $\iota^\star$. Explicitly, for a test scheme $S$,
\begin{equation}
\label{eqn: homotopy fixed points}
\cM^{\star,\flat}(S) =  \left\{ (\cE',h,\Phi',\alpha)\left| \begin{array}{c}(\cE',h,\Phi')\in \cN(S)\\
\alpha: (\cE',h,\Phi')\simeq \iota^\star(\cE',h,\Phi')\\
\iota^\star\alpha\circ\alpha = Id_{(\cE',h,\Phi')}\end{array} \right.\right\} 
\end{equation}
where in the above, $\alpha$ is an isomorphism in the groupoid $\cN(S)$. The functor $\cM^{\star,\flat}$ is representable by an algebraic stack by \cite[Theorem 3.3]{romagny1} and \cite[Prop 2.2.1]{ngo-laumon}. Moreover, since $\cN$ is locally of finite type, $\cM^{\star,\flat}$ is also locally of finite type by \cite[Proposition 3.7]{romagny1}.

For $r+s = 2n$, let $\cM^\unitary_{r,s}$ denote the moduli space whose $S$-points are
\[
\cM^\unitary_{r,s}(S) = \left\{ (\cE_1',\cE_2',h_1,h_2,\phi')\left| \begin{array}{c}\cE_1',\cE_2'\text{ are rank $r$,$s$ vector bundles on $\Sigma'\times S$, respectively,} \\
\text{$h_i\colon 
\cE_i'\to \sigma^*(\cE_i')^\vee$ is a Hermitian form, i.e. $\sigma^*h_i^\vee = h_i$,}\\\text{and }\phi'\colon \cE_1'\to \cE_{2}'\otimes D\text{ such that }\sigma^*(\phi')^\vee\circ h_1+h_2\circ \phi'= 0 \end{array} \right.\right\}
\]

\begin{prop}
    \label{prop: fixed points of involution on N}
    There are isomorphisms $\cM^{\siegel,\flat} \simeq \cM^\siegel$ and 
    \[
    \cM^{\unitary,\flat}\simeq \bigsqcup_{r+s = 2n} \cM^\unitary_{r,s}.
    \]
    In particular, $\cM^\unitary = \cM^\unitary_{n,n}\subset \cM^{\unitary,\flat}$ is a distinguished (disjoint) component.
\end{prop}
\begin{proof}
    Consider the data of a point $(\cE',h,\Phi',\alpha)\in \cM^{\star,\flat}(S)$ as in the right-hand side of \eqref{eqn: homotopy fixed points}. We can consider $\cE' = \cE_1'\oplus \cE_2'$, with $\cE_1'$ and $\cE_2'$ being the $+1$ and $-1$ eigenspaces of $\alpha$ acting on $\cE'$, respectively. Then, the condition on $\alpha$ implies that the following diagram commutes
    \[
    \xymatrix{
    \cE'\ar[r]^-{\alpha}\ar[d]_-{\Phi'} & \cE'\ar[d]^-{-\Phi'} \\
    \cE'\otimes D\ar[r]^-{\alpha(D)} & \cE'\otimes D
    }
    \]
    In particular, restricting the top left corner to $\cE_i'\subset \cE'$, we see that $\Phi'$ must be in the form $\Phi' = \Phi_{12}'\oplus \Phi'_{21}$ where $\Phi'_{ij}\colon \cE_i'\to \cE_j'\otimes D$. 

    Now, we specialize to the case $\star =\siegel$. In this case, the Hermitian form $h$ satisfies the following condition:  let $\langle\cdot,\cdot\rangle\colon \cE'\otimes \sigma^*(\cE')\to \cO_{X'}$ denote the pairing associated to $h$. Then the following diagram commutes
    \[
    \xymatrix{
    \cE'\otimes \sigma^*(\cE')\ar[r]^-{\langle\cdot,\cdot\rangle}\ar[d]_-{\alpha\otimes \sigma^*\alpha} & \cO_{X'}\\
    \cE'\otimes \sigma^*(\cE')\ar[ru]_-{-\langle\cdot,\cdot\rangle} & 
    }
    \]
    In particular, for all open sets $U\subset \Sigma'$ and all sections $s\in \Gamma(\cE_i',U')$ and $t\in \Gamma(\cE_j',U')$ ($i,j\in \{1,2\}$), we have
	\[
\langle s,t\rangle = -\big\langle(-1)^{i-1}s,(-1)^{j-1}t\big\rangle = (-1)^{i+j-1}\langle s,t\rangle.
	\]
	That is, the condition on $h$ is equivalent to requiring 
	\[
\langle\cdot,\cdot\rangle\in (\cE_1'\otimes \cE_2')^*\oplus (\cE_2'\otimes \cE_1')^{*}
	\]
    which is to say that $h$ determines an isomorphism $h\colon \cE_2'\xrightarrow{\sim}\sigma^*(\cE_1')^\vee$. This allows us to identify the tuple $(\cE',h,\Phi',\alpha)$ with the equivalent data $(\cE_1',h,\Phi'_{12},\Phi'_{21})\in \cM^\siegel(S)$.

    Consider now the case $\star = \unitary$. For this case, the condition on $h$ is that the pairing $\langle\cdot,\cdot\rangle$ fit into the commutative diagram
    \[
    \xymatrix{
    \cE'\otimes \sigma^*(\cE')\ar[r]^-{\langle\cdot,\cdot\rangle}\ar[d]_-{\alpha\otimes \sigma^*\alpha} & \cO_{X'} \\
    \cE'\otimes \sigma^*(\cE')\ar[ru]_-{\langle\cdot,\cdot\rangle} & 
    }
    \]
    This condition implies that, for all open sets $U\subset \Sigma'$ and all sections $s\in \Gamma(\cE_i',U')$ and $t\in \Gamma(\cE_j',U')$ ($i,j\in \{1,2\}$), we have
	\[
\langle s,t\rangle = -\big\langle(-1)^{i-1}s,(-1)^{j-1}t\big\rangle = (-1)^{i+j}\langle s,t\rangle.
    \]
    so that 
	\[
\langle\cdot,\cdot\rangle\in (\cE_1'\otimes \cE_1')^*\oplus (\cE_2'\otimes \cE_2')^{*}.
	\]
    In other words, $h$ is a direct sum of Hermitian forms $h_i$ on $\cE_i'$. From the relation
    \[
    \sigma^*(\Phi')^\vee\circ h+h\circ \Phi' = 0
    \]
    in the definition of unitary Higgs bundles, it follows that $\Phi'$ is completely determined by $h_1,h_2$, and $\Phi'_{12}$, and so the data of $(\cE',h,\Phi',\alpha)$ is equivalent to the data of $(\cE_1',\cE_2',h_1,h_2,\Phi'_{12})$. depending on the ranks of $\cE_1'$ and $\cE_2'$, this determines a point in $\cM_{r,s}^\unitary$.

    It is not hard to show the maps on $S$-points above are functorial, fully faithful, and essentially surjective.
\end{proof}

\subsection{The relative Hitchin fibration}
\label{sec: hitchin fibration}
We now introduce the relative Hitchin fibration, whose geometry is governed by the invariant theory of \S\ref{sec: invariant theory}.

\begin{definition}
   The relative Hitchin base $\cA=\cA_D$ is the space of sections 
    \[
    \cA = \Gamma(\Sigma,\fc_{\cL(D)}) = \bigoplus_{i=1}^{n} H^0(\Sigma,\cL(D)^{\otimes 2i}) = \bigoplus_{i=1}^{n} H^0(\Sigma,D^{\otimes 2i})
    \]
    where $\fc_{\cL(D)} = \fc\times^{\G_m}\cL(D)$, and we note in the last equality that $\cL^{\otimes 2} = \cO$.
\end{definition}

As $\fc_{\cL(D)} = \fc_D$, we will use these two interchangeably. Note that the second equality relies on Lemma \ref{Lem: cat quot split}.

The Hitchin base $\cA$ is an affine space, and composition with the map $[\chi]\colon [\fh^\perp/\bH]\to \fc$ induces morphisms
\[
f^\star\colon \cM^\star\to \cA,
\]
which we call the \emph{relative Hitchin map for $X^\star$}.  On $S$-points, the maps are explicitly given by (using notation as in \S\ref{sec: relative Hitchin spaces})
\begin{align*}
    f^\unitary(\cE_1',\cE_2',h_1,h_2,\phi') &= \mathrm{char}(\sigma^*(\phi')^\vee\circ\phi')\\
f^\Siegel(\cE',h,\phi_{12}',\phi_{21}') &= \mathrm{char}(\phi_{12}'\circ \phi_{21}')
\end{align*}
Note that although the constructions differ, Lemma \ref{lem: GIT quotient does not depend on form} ensures that the Hitchin bases for $\cM^\unitary$ and $\cM^\Siegel$ are indeed identified.

Given a choice of $D'$ such that $(D')^2 = D$, for the form $X^\siegel$, Theorem \ref{thm: section for rational forms} along with an argument mirroring \cite[Proposition 2.5]{ngo_endoscopie} produces a Kostant-Rallis section which globalizes to a section of $h^\siegel$, which we denote by 
\begin{equation}\label{eqn: KR section of siegel}
 [\e] = [\e_{D'}]\colon \cA\to \cM^\siegel.   
\end{equation}
 Note that no such section exists in the unitary case $X^\unitary$.

\subsection{Spectral curves}
\label{sec: spectral curves}

Recall that $\fs\to \fc$ denotes the universal spectral cover restricted to $\fc$, see \S\ref{sec: invariant theory}. Note that $\fs$ carries a canonical $\G_m$-action (given by scaling $x$) and the map $\fs\to \fc$ is $\G_m$-equivariant with respect to this action. Hence we may form the twist $\fs_{D} = \fs\times^{\G_m}D^\times$.

Recall that in \S\ref{sec: unitary Higgs}, we defined the spectral covers $\wt{\Sigma}_a'$ and $\wt{\Sigma}_a$ associated to a point $a\in \cA(S)$ by taking the base change of the following diagram
\[
\xymatrix{
	\wt{\Sigma}'\ar[r]\ar[d] & \wt{\Sigma}\ar[r]\ar[d] & \fs_{D}\ar[d] \\
	\Sigma'\times \cA\ar[r] & \Sigma\times \cA\ar[r]^-{\mathrm{ev}} & \fc_{D}
}
\]
The involution $\iota(x) = -x$ on $\fs$ induces an involution, which we also denote by $\iota$, on $\wt\Sigma'$ and $\wt\Sigma$ by base change. We denote by $\ol\Sigma' = \wt\Sigma'/\iota$ and $\ol{\Sigma} = \wt\Sigma/\iota$. The action of $\sigma\in \Aut(\Sigma'/\Sigma)$ lifts to an involution on $\wt\Sigma_a'$ by letting $\sigma(x) = -x$ while $\sigma$ acts via the Galois action on coefficients. As previously noted, $\wt\Sigma_a = \wt\Sigma_a'/\sigma$. In particular, for any $a\in \cA(S)$, we have a commutative diagram of spectral curves
\begin{equation}\label{eqn: diagram of covers}
   \begin{tikzcd}  
	\wt{\Sigma}'_a\ar[r,"/\sigma"]\ar[d,swap,"/\iota"]& \wt{\Sigma}_a\ar[d] \\
	\ol\Sigma'_a \ar[r] &\ol{\Sigma}_a,
\end{tikzcd} 
\end{equation}
with each arrow being an \'etale 2-cover over $k$. There is clearly a universal such diagram over $\Sigma\times \mathcal{A}$ pulling back to this one for any $a\in \cA(S)$

For a general $a\in \cA(S)$, the spectral cover $\wt\Sigma_a$ need not be smooth or even reduced. To control the behavior of $\wt\Sigma_a$, we will now introduce the smooth locus $\cA^\diamondsuit$ and the heart locus $\cA^\heartsuit$ in the Hitchin base. The notation and definitions below are chosen to mirror those of \cite[\S4.5,\S4.7]{ngo}.

\begin{definition}
	Let $\mathrm{ev}\colon \Sigma\times \cA\to \fc_D$ denote the evaluation map. We denote by $\cA^\heartsuit$ the projection of $\mathrm{ev}^{-1}\big((\fc\setminus \fD^\cent)_D\big)$ to $\cA$. That is, $\cA^\heartsuit\subset\cA$ is the open subvariety whose $\ol{k}$-points are given by
	\[
\cA^\heartsuit(\ol{k}) = \{ a\in \cA(\ol{k})\mid a(\Sigma_{\ol{k}})\not\subset \fD_D^\cent \}.
	\]
		
	Likewise, the \emph{smooth locus} $\cA^\diamondsuit$ is the open subvariety of $\cA$ consisting of $a\in \cA(S)$ whose associated spectral curve $\wt{\Sigma}_a'$ is smooth and irreducible. It contains all points $a$ of $\cA$ whose image $a(\Sigma)$ meets the divisor $\fD^\cent_D$ transversely.

    We denote $\cM^{\star,\heartsuit} = (f^\star)^{-1}(\cA^\heartsuit)$ and $\cM^{\star,\diamondsuit} = (f^\star)^{-1}(\cA^\diamondsuit)$.
\end{definition}

\begin{lem}
	Assume that $\deg(D) > 2g$. The smooth locus $\cA^\diamondsuit$ is a \emph{nonempty} open subset of $\cA$, and $\cA^\diamondsuit\subset \cA^\heartsuit$.
\end{lem}
\begin{proof}
	The only nontrivial claim is that $\cA^\diamondsuit$ is nonempty. Since $a\in \cA^\diamondsuit(S)$ holds in particular whenever $a$ intersects the full divisor $\fD_X = \fD^\cent+\fD^\br$ transversely, an identical argument to \cite[Prop 4.7.1]{ngo} proves this.
\end{proof}

Similar to \cite[\S4.5 and \S4.7]{ngo}, we have the following interpretations of $\cA^\heartsuit$ in terms of the spectral cover $\wt{\Sigma}_a'$, with identical proof.

\begin{prop}
    If $a\in \cA^\heartsuit(k)$, the spectral curve $\wt{\Sigma}_a'$ (and hence also $\wt{\Sigma}_a$, $\ol{\Sigma}_a'$, $\ol{\Sigma}_a$) is reduced over $k$.
\end{prop}

\subsection{Spectral description of relative Hitchin fibers}
\label{sec: invariant theory spectral}

In this section we interpret the moduli spaces $\cM^\star$ in terms of appropriate spectral covers. Recall our Convention \ref{Convention for sheaves}.

Proposition \ref{prop: spectral description for unitary higgs} gives the description of $\cN^\heartsuit$ in terms of spectral curve
\begin{equation}
\label{eqn: spectral for unitary higgs}
\cN^\heartsuit(S) = \{(a,\cF',h)\mid a\in \cB^\heartsuit(S),\; \cF'\in \ol{\Pic}(\wt{\Sigma}_a'),\; h\colon \cF'\xrightarrow{\sim}\sigma^*(\cF')^\vee\text{ a Hermitian form}\}.
\end{equation}
Recall that a triple $(a,\cF',h)$ gives a $\U_{2n}$ Higgs bundle by considering the vector bundle $E' = \pi_{a,*}'(\cF')$ on $\Sigma'\times S$ (with $\pi_a'\colon \wt\Sigma'_a\to \Sigma'\times S$ the covering map) together with the Hermitian form $E\to \sigma^*(E)^\vee$ obtained by descent from $h$. The section $\Phi'$ is obtained by multiplication by the tautological section.

We have an involution on $\cB$ given by sending $a = (a_1,\dots, a_{2n})$ to $\iota_{\cB}(a) = \left((-1)^ia_i\right)$. The fixed points $\cB^{\iota_{\cB}}$ gives the relative Hitchin base $\cA$. Recall further that we are denoting by $\iota$ the involution on the spectral curve $\wt{\Sigma}_a'$ defined over $\Sigma'$ (given by negating the tautological section).

\begin{lem}
\label{lem: involution on spectral data}
    The involution $\iota^\star$ defined in \eqref{eqn: involution on cN} restricts to the involution on $\cN^\heartsuit$ sending the data $(a,\cF',h)\in \cN^{\heartsuit}(S)$ in the right hand side of \eqref{eqn: spectral for unitary higgs} to
    \[
    \iota^\star(a,\cF',h) = \begin{cases}
        (\iota_\cB(a),\iota^*\cF',-\iota^*h) & \text{if }\star = \siegel \\
        (\iota_\cB(a),\iota^*\cF',\iota^*h) & \text{if }\star = \unitary.
    \end{cases}
    \]
\end{lem}
\begin{proof}
For spectral data $(a,\cF',h)\in \cN^\heartsuit(S)$ in the right hand side of \eqref{eqn: spectral for unitary higgs}, let $(\cE',h,\Phi')$ be the associated $\U_{2n}$ Higgs bundle. We have $\pi_{a,*}'(\cF') = \pi_{a,*}'(\iota^*\cF')$ since $\pi_{a}'\circ\iota = \pi_a'$. Moreover, $\iota$ has the effect of negating the tautological section, so that multiplication by the tautological section on $\cE'=\pi_{a,*}'(\iota^*\cF')$ is $-\Phi'$. The Hermitian form $\pm \iota^*h$ descends to $\pm h$ on $\cE'$, and so we deduce that $\iota^\star(a,\cF',h)$ corresponds exactly to $\iota^\star(\cE',h,\Phi')$. 
\end{proof}

The following result can be seen as a unitary version of \cite[Proposition 2.2]{schaposnik}.

\begin{cor}
    \label{cor: spectral description using fixed points}
    The functor $\cM^{\star,\heartsuit}$ is the homotopy fixed point stack of $\cN^\heartsuit$ under the action of $\iota^\star$. In particular, $\cM^{\star,\heartsuit}$ is isomorphic to the functor
    \[
\cM^{\star,\heartsuit}(S) = \left\{ (a,\cF',h,\alpha)\;\left| \; \begin{array}{c} a\in \cA^\heartsuit(S),\; \cF'\in \ol{\Pic}(\wt{\Sigma}_a'/S),\; \\ h\colon \cF'\xrightarrow{\sim}\sigma^*(\cF')^\vee\text{ is Hermitian,}\; \\\alpha\colon (a,\cF',h)\xrightarrow{\sim}\iota^\star(a,\cF',h)\text{ such that }\iota^\star\alpha\circ\alpha = \mathrm{Id}\end{array} \right.\right\}
    \]
    where in the above $\alpha$ is the data of an isomorphism $\alpha\colon \cF'\to \iota^*\cF'$ which transports $h$ to $-\iota^*h$ in the Siegel case and $\iota^\star h$ in the unitary case. That is, $\alpha$ must fit into a commutative diagram
    \[
\xymatrix{
	\cF\otimes \cF\ar[r]^-{h}\ar[d]_-{\alpha\otimes \alpha} & \omega_{\wt{\Sigma}_a'/\Sigma'\times S} \\
	\iota^*\cF\otimes \iota^*\cF\ar[ur]_-{\pm \iota^*h}
}		
	\] 
    with the horizontal maps being the pairings induced by $h$ and $\pm \iota^*h$ and the sign on the latter determined by the case $\star\in \{\siegel,\unitary\}$ ($+1$ in the case $\star = \unitary$ and $-1$ in the case $\star = \siegel$).
\end{cor}
\begin{proof}
    The statement mostly follows from Proposition \ref{prop: fixed points of involution on N} and Lemma \ref{lem: involution on spectral data}. It suffices to note that, in the case $\star = \unitary$, the image of $\cM^\unitary_{r,s}$ under the relative Hitchin map $h^\unitary$ is disjoint from $\cA^\heartsuit$, so that $\cM^{\unitary,\heartsuit} = h^{\unitary,-1}(\cA^\heartsuit)$.
\end{proof}

\section{Symmetries of relative Hitchin fibers}

It was observed in \cite{ngo_endoscopie} that globalization of endoscopic groups is intimately tied to the groups of symmetries acting on the Hitchin fibration. For example, in type $A$, this manifests in an action of Picard stacks of spectral covers on the moduli of Higgs bundles. Here, we define a smooth, commutative group stack $\cP$ which acts on the relative Hitchin moduli $\cM^\star$ over $\cA$, and give a spectral description of $\cP$ and its action. We will study its relation to endoscopy further in \S\ref{sec: component groups of P}.

Recall that the Hitchin base $\cA$ is identified with the space of sections $\Gamma(\Sigma,\fc_{D})$. For any $a\in \cA(S)$, we consider the group scheme $J_a = a^*J$ defined as the pullback of the group scheme $J$ from Theorem \ref{thm: regular centralizers for rational forms} to $\Sigma \times S$. In particular, $J_a$ is a smooth, commutative group scheme over $\Sigma\times S$.

\begin{definition}
	Let $\cP$ denote the functor sending a $k$-scheme $S$ to
	\[
	\cP(S) = \{(a,\cT)\mid a\in \cA(S),\; \cT\text{ is a $J_a$-torsor on }\Sigma\times S\}.
	\]
    We again denote $\cP^\heartsuit = \cP|_{\cA^\heartsuit}$.
\end{definition}

Let $\wt\Sigma'$ be the spectral cover of $\Sigma'\times \cA$ defined in \S\ref{sec: spectral curves}. Recall from \eqref{eqn: diagram of covers} the quotients $\bar{\Sigma}' = \wt{\Sigma}'/\iota$ and $\ol\Sigma = \ol\Sigma'/\sigma$. The spectral description of $J$ (see Theorem \ref{thm: regular centralizers for rational forms}) gives the following spectral description for $\cP$.

\begin{prop}
	\label{prop: spectral description of P star}
	Over $\cA^\heartsuit$, we have an identification $\cP^\heartsuit\simeq \Prym_{/\cA^\heartsuit}(\ol{\Sigma}',\ol{\Sigma})$, where $\Prym_{/\cA^\heartsuit}(\ol{\Sigma}',\ol{\Sigma})$ denotes the group of pairs $(\cG',\xi)$ where $\cG'$ is in the kernel of the norm map
	\[
\Pic(\ol{\Sigma}'/\cA^\heartsuit)\to \Pic(\ol{\Sigma}'/\cA^\heartsuit),\quad \cG'\mapsto \cG'\otimes \sigma^*\cG'
	\]
    and $\xi\colon \cG'\simeq \sigma^*\cG'^{\otimes -1}$ is an isomorphism realizing this.

	Moreover, the action of $\cP^\heartsuit$ on $\cM^{\star,\heartsuit}$ is identified with the map
	\[
	\cP^\heartsuit\times \cM^{\star,\heartsuit}\to \cM^{\star,\heartsuit}
	\]
	\[
	(\cG',\cF')\mapsto q'^*\cG'\otimes \cF'
	\]
	where in the above, $q'\colon \wt{\Sigma}'\to \ol{\Sigma}'$ is the quotient map by $\iota$. In this statement, we view $\cM^{\star,\heartsuit}$ in terms of the spectral description (cf. Corollary \ref{cor: spectral description using fixed points}). The Hermitian form on $q'^*\cG'\otimes \cF'$ is induced by the corresponding forms on $\cG'$ and $\cF'$ by 
    \[
    q'^*\cG'\otimes \cF'\xrightarrow{q'^*\xi\otimes h}q'^*\sigma^*\cG'^{\otimes -1}\otimes (\sigma^*\cF')^\vee = \sigma^*(q'^*\cG'\otimes \cF')^\vee,
    \]
    and $\alpha$ is induced by
    \[
    q'^*\cG\otimes \cF\xrightarrow{1\otimes \alpha}q'^*\cG\otimes \iota^*\cF = \iota^*(q'^*\cG\otimes \cF ).
    \]
\end{prop}

\begin{remark}
In the above result, $\Prym$ refers to the full kernel of the norm map, not just the identity component. For example, since $\vartheta$ is unramified, the $\Prym$ above has two connected components when $\ol{\Sigma}_a'$ is irreducible. In fact, it is easy to compute the component group in general; see Lemma \ref{lem: component group of Pa}.
\end{remark}

\begin{proof}
    By Theorem \ref{thm: regular centralizers for rational forms}, the regular centralizer group $J$ can be described as the kernel of the morphism
\[
\mathrm{Nm}\colon \Res^{(\fs/\iota)_{k'}}_{\fc}(\G_m)\to\Res^{\fs/\iota}_{\fc}(\G_m),
\]
    where $\mathrm{Nm}$ is norm map. Since torsors over $J_a$ for $a\in \cA^\heartsuit(S)$ are identified with $H^1(\Sigma\times S,J_a)$, the above description of $J_a$ describes $\cP_a$ as the kernel of the norm map
    \[
\Pic(\ol{\Sigma}'_a)\to \Pic(\ol{\Sigma}_a),\quad \cG\mapsto \cG\otimes \sigma^*\cG	.
    \]
    The group $\cP_a$ is identified with the sections of $\pi_{a,*}(\cO_{\wt\Sigma_a}^\times)$ that preserve the standard Hermitian form $h\colon \cO_{\wt\Sigma_a}\simeq \sigma^*\cO_{\wt\Sigma_a}^\vee$ via the map $\cG\mapsto q'^*\cG$, and the action is given by the natural action of this unitary group on $\ol{\Pic}(\wt\Sigma_a')$. 
\end{proof}

\subsection{Open subvarieties of the relative Hitchin base}
\label{sec: open subsets}

In this section, we introduce some more (open) subvarieties of the Hitchin base that will be relevant later, namely the anisotropic locus $\cA^\ani$ (Definition \ref{def: ani locus}) and the dagger locus $\cA^\dagger$ (Definition \ref{def: dagger locus}), both of which will be crucial in geometric stabilization.

\begin{definition}
\label{def: ani locus}
    The \emph{anisotropic locus} in $\cA$, denoted by $\cA^\ani$, is the set of all $a\in \cA^\heartsuit(S)$ such that the double cover $\wt{\Sigma}_a'\to \wt{\Sigma}_a=\wt{\Sigma}_a'/\sigma$ induces a bijection on irreducible components, i.e. the involution $\sigma$ acts trivially on the set of irreducible components of the spectral cover $\wt{\Sigma}_a'$.
    
    We denote $\cM^{\star,\ani} =\cM^\star\times_\cA\cA^\ani$ and $f^{\star,\ani}\colon \cM^{\star,\ani}\to \cA^\ani$ the restriction of $f^\star$ to $\cM^{\star,\ani}$, and likewise for $\cP^{\ani}$ and $g^\ani\colon \cP^\ani\to \cA^\ani$.
\end{definition}

\begin{remark}
    Note that the above definition uses the spectral curve $\wt{\Sigma}_a'$, not $\ol{\Sigma}_a'$. The above condition implies the corresponding condition for $\ol{\Sigma}_a'$. As such, the above can be thought of as an intersection between 2 elliptic loci, one for the relative Hitchin system itself and one for unitary Hitchin system $\cN_{2n}$ from \cite{ngo-laumon}. While analysis of the full elliptic locus for the relative Hitchin system is important for the relative trace formula \cite{LXZrtf}, this smaller locus suffices for our local aims.
\end{remark}

\begin{lem}
For $a\in \cA(S)$, we have $a\in \cA^\ani(S)$ if and only if $\pi_0(\Prym(\wt{\Sigma}_a',\wt{\Sigma}_a))$ is finite. In this case, $\pi_0(\cP_a)$ is also finite.
\end{lem}
\begin{proof}
    It is easy to see that the Prym variety $\Prym(\wt{\Sigma}_a',\wt{\Sigma}_a)$ has finite component group if and only if for each component of $\wt{\Sigma}_a$, its preimage under the cover $\wt{\Sigma}_a'\to \wt{\Sigma}_a$ is a geometrically connected cover, which is exactly the condition of $\cA^\ani$. 
    
    Now, suppose that $\wt{\Sigma}_a'\to \wt{\Sigma}_a$ induces a bijection on irreducible components. This map is $\iota$ equivariant, so after passing to the quotient by $\iota$ in diagram \eqref{eqn: diagram of covers}, we deduce that $\ol{\Sigma}_a'\to \ol{\Sigma}_a$ also induces a bijection on irreducible components. Therefore, finiteness of $\pi_0(\cP_a)$ follows from the spectral description Proposition \ref{prop: spectral description of P star} using an identical argument to the first paragraph.
\end{proof}

\begin{lem}
\label{lem: inclusions of opens in relative bases}
    The locus $\cA^\ani$ is open in $\cA$, and we have inclusions $\cA^\diamondsuit\subset \cA^\ani\subset \cA^\heartsuit$. In particular, $\cA^\ani$ is nonempty.
\end{lem}
\begin{proof}
    That this locus is open follows from \cite[Lem 2.8.1]{ngo-laumon}. For $a\in \cA^\diamondsuit(S)$, the spectral cover $\wt{\Sigma}_a'\to \Sigma'\times S$ restricts to a smooth and connected, hence irreducible, curve over any geometric point $s$ of $S$. Hence, $\wt{\Sigma}_a'\to \wt{\Sigma}_a$ induces an injection (and so also a bijection) on irreducible components, and the inclusion $\cA^\diamondsuit\subset \cA^\ani$ follows.
\end{proof}

\begin{prop}
    \label{prop: smoothness and properness of relative hitchin fibration}
    For $\star\in \{\unitary,\siegel\}$, the relative Hitchin morphism $f^{\star,\ani}\colon \cM^{\star,\ani}\to \cA^\ani$ is proper and $\cM^{\star,\ani}$ is a smooth, finite type Deligne--Mumford stack.
\end{prop}
\begin{proof}
	Since the anisotropic condition is the intersection of the anisotropic condition for $\U_{2n}$ Higgs bundles, $\cM^{\star,\ani}$ is the fixed point stack of $\iota^\star$ acting on $\cN^\ani$ by Proposition \ref{prop: fixed points of involution on N} and Corollary \ref{cor: spectral description using fixed points}. Hence, smoothness follows from \cite[Theorem 1.2.1]{romagny2}. By \cite[Theorem 3.3]{romagny1}, the morphism $\cM^{\star,\ani}\to \cN^\ani$ is representable, and by \cite[Prop 2.8.4]{ngo-laumon}, $\cN^{\ani}$ is Deligne--Mumford. Hence, $\cM^{\star,\ani}$ is Deligne--Mumford. 
	
	We now note that since $\cB$ is affine, it is separated. Therefore, as $f^\ani\colon \cN^{\ani}\to \cB^\ani$ is proper (and so, in particular, separated), $\cN^\ani$ is also separated. Properness of $f^{\star,\ani}$ now follows by \cite[Proposition 3.7]{romagny1} since $f^\ani$ is proper (see \S\ref{sec: unitary Higgs} and \cite[Prop 2.8.4]{ngo-laumon}). Moreover, \cite[Proposition 3.7]{romagny1} implies the map $\cM^{\star,\ani}\to \cN^{\ani}$ is proper and so in particular is finite type. Hence, $\cM^{\star,\ani}$ is locally of finite type since $\cN^\ani$ is. (See \cite[Prop 6.1.3]{ngo}.)
\end{proof}

In fact, the open subset $\cA^\ani\subset \cA^\heartsuit$ can be arbitrarily large relative to the degree of $D$, as the following lemma makes plain.

\begin{lem}
\label{lem: codim of anisotropic}
    The codimension of the complement $\cA^\heartsuit - \cA^\ani$ in $\cA^\heartsuit$ is at least $\deg(D)$.
\end{lem}
\begin{proof}
    Identical to \cite[Cor 6.3.6]{ngo}.
\end{proof}

\begin{remark}
    Note that, while we use the above lemma in \S\ref{sec: dimension formula}, we will need a slightly stronger version (see Lemma \ref{lem: complement of ani locus}) for the proof of the main result. 
\end{remark}

We finally introduce the following subset of the anisotropic base.

\begin{definition}
	\label{def: dagger locus}
	Let $\cA^\dagger\subset\cA^\ani$ denote the subvariety of $\cA^{ani}$ whose $S$-points $a\colon \Sigma\times S\to \fc_D$ satisfy that image intersects transversely with the branching divisor $\fD_D^\br$.
	
	We likewise denote $\cM^\dagger = \cM\times_\cA\cA^\dagger$ and $h^\dagger = h|_{\cA^\dagger}$ for the restrictions of $\cM$ and $h$, respectively, to $\cA^\dagger$.
\end{definition}

The subset $\cA^\dagger$ will be used in \ref{sec: support thm for relative case} in tandem with the support theorem. We note here the following density statement, which follows from an argument mirroring that of \cite[Prop 4.7.1]{ngo}.

\begin{lem}
	\label{lem: density of dagger locus}
		The set $\cA^\dagger$ is open and dense in $\cA$ as long as $\deg(D)>2g$. 
\end{lem}

\subsection{Geometry over the regular locus}
\label{sec: geometry over regular}
We consider here the geometry of the regular locus in the moduli spaces $\cM^\star$, which is the global analogue of the regular locus of the representation $V^\star$ of $H^\star$. It turns out that this geometry is considerably richer than that of usual Higgs bundles, which will be an important technical tool in our proof of geometric stabilization.

\begin{definition}
    The \emph{regular locus} in $\cM^\star$ is the mapping space 
    \[
    \cM^{\star,\reg} := \mathrm{Sect}_\Sigma([V_D^{\star,\reg}/H^\star])
    \]
    where $V^{\star,\reg} = [(\fh'^\perp)^\reg]^{\tau_\star\sigma}$ is the open locus in $V_D^\star$ with the notion of regularity defined in \S\ref{sec: invariant theory}.
\end{definition}

We recall from Theorem \ref{thm: regular quotient} that
\[
[\chi^\reg]\colon [(\fh^\perp)^\reg/\bH]\to \fc
\]
factors through the nonseparated scheme $(\fh^\perp)^\reg\myfatslash \bH$ and the map
\[
[(\fh^\perp)^\reg/\bH]\to (\fh^\perp)^\reg\myfatslash \bH
\]
is a gerbe banded by the regular centralizer group scheme $J$. We introduce the following global analogue of this intermediate space, which is studied in \cite[\S5]{HM}. 

\begin{definition}
    The \emph{regular base} is the mapping space $\cA^{\star,\reg} = \mathrm{Sect}_\Sigma(V^{\star,\reg}_D\myfatslash \bH^\star)$.
\end{definition}

Note that the action of $\Aut(\Sigma'/\Sigma)$ on the regular nilpotent orbits in $V^{\star,\reg}_D$ is different in the unitary and Siegel cases; in particular, $\cA^{\unitary,\reg}$ and $\cA^{\Siegel,\reg}$ are not isomorphic. We have the following important properties of $\cA^{\star,\reg}$.

\begin{prop}
\label{prop: factorization thru A reg}
    The relative Hitchin morphism 
    \[
    f^\star|_{\cM^{\star,\reg}}\colon \cM^{\star,\reg}\to \cA
    \]
    factors canonically through morphisms 
    \[
\cM^{\star,\reg}\xrightarrow{f^{\star,\reg}}\cA^{\star,\reg}\xrightarrow{\varphi^\star}\cA
    \]
    Moreover, $\varphi^\star$ is \'etale and the smooth, commutative group stack $(\varphi^\star)^*\cP$ acts on $\cM^{\star,\reg}$ over $\cA^\reg$, realizing $\cM^{\star,\reg}$ as a $(\varphi^\star)^*\cP$-torsor.
\end{prop}
\begin{proof}
    This is a restatement of \cite[Theorem 5.7 and Proposition 5.6]{HM}. In \emph{loc. cit.} is assumed that $k$ is algebraically closed, but all properties in the proposition statement can be checked after base changing to $\ol{k}$.
\end{proof}

In fact, it is relatively straightforward to describe $\cA^{\star,\reg}$ over the open set $\cA^\dagger\subset\cA$ (see Definition \ref{def: dagger locus}). To prepare for this, we first prove the following lemma.

\begin{lem}
\label{lem: nilpotent covers in each case}
    Consider the map
    \begin{equation}
    \label{eqn: twisted chevalley}
    V_D^{\star,\reg}\myfatslash H^\star\to \fc_D.
    \end{equation}
    In the Siegel case ($\star = \Siegel$), the restriction to the zero fiber $\Sigma = 0_D\subset \fc_D$ is the trivial 2-cover
    \[
    \Sigma\sqcup\Sigma\to \Sigma.
    \]
    In the unitary case ($\star = \unitary$), the restriction to the zero fiber is our fixed \'etale double cover
    \[
    \vartheta\colon \Sigma'\to \Sigma.
    \]
\end{lem}
\begin{proof}
    Choose $\infty\in \Sigma(k)$. The action of $\pi_1(\Sigma,\infty)$ on the fiber of \eqref{eqn: twisted chevalley} over $\infty\in \Sigma = 0_D$ acts through the quotient map
    \[
    \pi_1(\Sigma,\infty) \to \Aut(\Sigma'/\Sigma),
    \]
    where the latter acts through via the Frobenius on the set of regular nilpotent $H^\star$-orbits in $V^\star$. This action is nontrivial on the set of regular nilpotent orbits in the unitary case and trivial in the Siegel case. Since the cover $\vartheta^{-1}(0_D)\to \Sigma$ is \'etale, it is determined by this action.
\end{proof}

Consider the evaluation map $\cA^\dagger\to \Sym^{2n\deg(D)}(\Sigma)$ sending $a\colon \Sigma\times S\to \fc_D$ to $a^{-1}(\fD^\br_D)$. By Lemma \ref{lem: nilpotent covers in each case}, this evaluation map lifts to morphisms
\[
\cA^{\Siegel,\reg}\times_\cA\cA^\dagger\to \Sym^{2n\deg(D)}(\Sigma\sqcup\Sigma) = \bigsqcup_{i=0}^{2n\deg(D)}\Sym^i(\Sigma)\times \Sym^{2n\deg(D)-i}(\Sigma)
\]
\[
\cA^{\unitary,\reg}\times_\cA\cA^\dagger\to \Sym^{2n\deg(D)}(\Sigma')
\]
These evaluation maps capture the entire behavior of the morphism $\varphi^\star$ over $\cA^\dagger$.

\begin{prop}
    \label{prop: A reg over dagger locus}
   The following diagrams, with horizontal arrows given by the evaluation morphisms described above and with the right vertical arrows induced by the double covers $\Sigma\sqcup \Sigma\to \Sigma$ and $\Sigma'\to \Sigma$ respectively, are Cartesian.
   \[
   \xymatrix{
\cA^{\Siegel,\reg}\times_\cA\cA^\dagger\ar[r]\ar[d] & \Sym^{2n\deg(D)}(\Sigma\sqcup\Sigma)\ar[d] \\
\cA^\dagger\ar[r] & \Sym^{2n\deg(D)}(\Sigma)
   } \hspace{2cm}
   \xymatrix{
\cA^{\unitary,\reg}\times_\cA\cA^\dagger\ar[r]\ar[d] & \Sym^{2n\deg(D)}(\Sigma')\ar[d] \\
\cA^\dagger\ar[r] & \Sym^{2n\deg(D)}(\Sigma)
   }
   \]
\end{prop}
\begin{proof}
    Since the diagrams above are commutative, and since the left vertical arrows are \'etale of the same degree as the right vertical arrows, the proposition follows.
\end{proof}

\subsection{The component groups $\pi_0(\cP_a)$}
\label{sec: component groups of P}

We will now consider the component groups $\pi_0(\cP_a)$ for $a\in \cA^\ani(S)$ and relate them to the character lattice of the dual torus $\wh{\canC}$ (see Lemma \ref{lem: component group of Pa} and Proposition \ref{prop: C surjects on components of P}). In addition to utility in calculations, this will be necessary for decomposing $\bR f^{\star,\ani}_*\ol{\Q}_\ell$ into $\kappa$ isotypic components.

The first description is a simple calculation using the spectral description in Proposition \ref{prop: spectral description of P star}.

\begin{lem}
    \label{lem: component group of Pa}
    For $a\in \cA^\ani(S)$, let $\mathrm{Irr}(\ol{\Sigma}_a)$ denote the set of irreducible components of $\ol{\Sigma}_a$ (which is in bijection with $\mathrm{Irr}(\ol{\Sigma}_a')$). Then, we have 
    \[
    \pi_0(\cP_a) = (\Z/2\Z)^{\#\mathrm{Irr}(\ol{\Sigma}_a)}.
    \]
\end{lem}
\begin{proof}
    The proof is identical to \cite[Lem 2.8.2]{ngo-laumon}. Namely, the unramified morphism $\ol\Sigma_a'\to \ol\Sigma_a$ induces an exact sequence
    \[
    1\to \G_{m,\ol \Sigma_a}\to \Res_{\ol\Sigma_a'/\ol\Sigma_a}(\G_m)\xrightarrow{\beta} \Res_{\ol\Sigma_a'/\ol\Sigma_a}(\G_m)^{\sigma^*=(\cdot)^{-1}}\to 1
    \]
    where $\beta$ is the anti-norm map $\beta(\psi) = \psi\cdot \sigma(\psi)^{-1}$. This sequence therefore induces an exact sequence on cohomology
    \[
    \Pic(\ol\Sigma_a)\to \Pic(\ol\Sigma_a')\to \cP_a\to 1.
    \]
    The above sequence remains exact on passing to $\pi_0$ since taking component groups is right exact, and by the condition that $a\in \cA^\ani(S)$, we get $\Pic(\ol\Sigma_a)\simeq \Pic(\ol\Sigma_a')\simeq \Z^{\#\Irr(\ol \Sigma_a)}$. The natural map
    \[\pi_0\big(\Pic(\ol\Sigma_a)\big)\to \pi_0\big(\Pic(\ol\Sigma_a')\big)\]
    is identified with the multiplication-by-2 map $\Z^{\#\Irr(\ol \Sigma_a)}\xrightarrow{2} \Z^{\#\Irr(\ol \Sigma_a)}$, and so the lemma follows.
\end{proof}

We now connect the component groups of fibers with the character group of the dual torus $\wh{\canC}$. To do so, we must first briefly introduce a cameral-type cover. For wider applications of the geometry of Hitchin systems to relative endoscopic fundamental lemmas, we expect a greater use of similar cameral covers will be necessary.

\begin{definition}
\label{def: relative cameral cover}
	The cameral cover $\cameral{\Sigma}$ over $\Sigma$ and $\cameral{\Sigma}'$ over $\Sigma'$ are determined by the Cartesian diagrams
	\[
\xymatrix{
\cameral{\Sigma}'\ar[r] \ar[d] & \cameral{\Sigma} \ar[r] \ar[d] & (\fa\git \{\pm1\}^n)_D \ar[d] \\
\cA^\ani\times \Sigma'\ar[r] & \cA^\ani\times\Sigma\ar[r] & \fc_D.
}	
	\]
\end{definition}

Since $\fs = \fa\git (\mathfrak{S}_{n-1}\ltimes\{\pm1\}^n)$, any of the $n$ natural embeddings $\mathfrak{S}_{n-1}\to \mathfrak{S}_n$ induce factorizations 
\begin{equation}
	\label{eqn: factorizations of cameral cover}
\xymatrix{
\cameral{\Sigma}'\ar[r]\ar[d] & \ol{\Sigma}'\ar[r]\ar[d] & \cA^\ani\times \Sigma'\ar[d] \\
\cameral{\Sigma}\ar[r] & \ol{\Sigma}\ar[r] & \cA^\ani\times \Sigma.
}
\end{equation}

Now, fix $\infty\in \Sigma(k)$ which splits in $\Sigma'$ and let $\infty'_1,\infty'_2\in \Sigma'(k)$ denotes the two points lying over $\infty$. We define the following modifications of the anisotropic Hitchin base.

\begin{definition}
\label{def: cameral cover of base}
		Denote by $\cA_{\infty}^\ani\subset \cA^\ani$ the subset of points $a$ of $\cA^\ani$ such that the finite, flat map 
		\[
            \cameral{\Sigma}_a'\to \Sigma'\to \Sigma
		\]
	is unramified at $\infty$. Note that since $\vartheta$ is unramified, the diagram \eqref{eqn: factorizations of cameral cover} induces a commutative diagram
   		\[
        \begin{tikzcd}
            \cameral{\Sigma}_a'\ar[r]\ar[d]&\overline{\Sigma}_a'\ar[d]\ar[r]& \Sigma'\ar[d,"\vartheta"]\\\cameral{\Sigma}_a\ar[r]&\overline{\Sigma}_a \ar[r]&\Sigma
        \end{tikzcd}
		\] with \'etale maps.
		
	Denote $\cameral{\cA}^\ani=\cameral{\cA}^\ani_{\infty}\subset \cA^\ani_\infty\times \cameral{\Sigma}'$ for the space of pairs $(a,\cameral{\infty}')$ such that $a\in \cA^\ani_\infty$ and $\cameral{\infty}'\in \cameral{\Sigma}_a'$ has image $\infty$ under $\cameral{\Sigma}_a'\to \Sigma$. We will drop the subscript $\infty$ in what follows for notational simplicity. By definition of $\cA^\ani_{\infty}$, the map $\cameral{\cA}^\ani \to \cA^\ani_{\infty}$ is \'etale of degree $2(n!)$. {Note that $\sigma\in \Aut(\Sigma'/\Sigma)$ lifts to an action on $\cameral{\cA}^\ani$ over $\cA^\ani_{\infty}$ via its action on $\cameral{\Sigma}'$.}
\end{definition}

Now for $i=1,2$, we let $Z_{i}\subset \cA_\infty^\ani\times\ol\Sigma'$ denote the space of pairs $(a,\ol\infty')$ such that $\ol\infty'\in \ol{\Sigma}'_a$ has image $\infty'_i\in \Sigma'(k)$.  {Note that $Z_{i}\to \cA^\ani_\infty$ is \'etale of degree $n$ and $\sigma$ induces an isomorphism $Z_1\to Z_2$.}

 We have a concrete realization of $\cameral{\cA}^\ani$ as follows.
\begin{lem}
	\label{lem: cameral description via Z infinity} 
	There is an isomorphism of spaces over $\cA^\ani_\infty$
	\[
\cameral{\cA}^\ani\simeq (Z_{1}^n\setminus \Delta_1) \bigsqcup (Z_{2}^n\setminus \Delta_2),
	\]
	where $\Delta_i$ denotes the union of all pairwise diagonals in the $n$-fold product $Z_{i}^n$. {This isomorphism intertwines the action of $\sigma$ on $\cameral{\cA}^\ani$ with the induced isomorphism $(Z_{1}^n\setminus \Delta_1) \to(Z_{2}^n\setminus \Delta_2)$.}
\end{lem}
\begin{proof} 
	The $n$ distinct factorizations of \eqref{eqn: factorizations of cameral cover} induce a map  
	\begin{equation}
		\label{eqn: product of embedding maps}
\fa\git \{\pm1\}^n	\to \prod_{i=1}^n\fa\git (\mathfrak{S}_{n-1}\rtimes \{\pm1\}^n)
	\end{equation}
	sending
	\[
	(x_1,\dots, x_n) \mapsto \prod_i (x_i,[x_1,\dots,x_{i-1},x_{i+1},\dots, x_n]),
	\]
	where square brackets denote unordered tuples. The resulting morphism  $\cameral{\cA}^\ani\to Z_{1}^n \bigsqcup Z_2^n$ obtained by composing the evaluation map is defined over $\cA_\infty^\ani$ has image in the complement of $\Delta$ since we are assuming the underlying point $a$ is in $\cA_\infty^\ani$. Moreover, both sides are finite and \'etale over $\cA_\infty^\ani$ of degree $2(n!)$.  Since \eqref{eqn: product of embedding maps} is an injection, the map $\cameral{\cA}^\ani\to Z_{1}^n \bigsqcup Z_2^n$ is injective and hence an isomorphism.
\end{proof}

Recall that the canonical torus $\canT$ of $G = \U_{2n}$ fits into an exact sequence
\[
1\to \canC\to \canT\to \canA\to 1,
\]
where $\canC$ is the fixed points of a canonical involution on $\canT$. Let $\Omega$ be the set of weights of the defining representation of $\GL_{2n}$. That is, $\Omega = \{\varepsilon_1,\dots, \varepsilon_{2n}\}$ where $\varepsilon_i$ is the $i$-th coordinate function on the lattice $\mathbb{X}^*(\canT)$. The Frobenius $\sigma$ acts on $\Omega$ and we can choose an ordering of the $\varepsilon_i$ so that $\varepsilon_{i+n} = \sigma(\varepsilon_i)$ for all $i = 1,\dots, n$.

For $\omega\in \Omega$, we have an induced function $\omega\colon \cameral\cA^\ani\to Z$ which, under the identification $\cameral\cA^\ani\simeq (Z_1^n\setminus \Delta_1)\sqcup (Z_2^n\setminus \Delta_2)$ of Lemma \ref{lem: cameral description via Z infinity}, is expressed by
\[
\varepsilon_i(\ol\infty'_1,\dots, \ol\infty'_n) = \ol\infty'_i
\]
\[
\varepsilon_{i+n}(\ol\infty'_1,\dots, \ol\infty'_n) = \sigma(\ol\infty'_i)
\]
for $i = 1,\dots, n$.

Denote $Z = Z_1\sqcup Z_2$ for the space of pairs $(a,\ol{\infty}')\in \cA_\infty^\ani\times_{\cA} \ol{\Sigma}'$ such that $\ol{\infty}'\mapsto \infty$. There is a natural action of $\sigma$ on $Z$ from the isomorphism $\sigma\colon Z_1\to Z_2$.

\begin{cor}
	\label{cor: cameral cover indexes components of P}
	The base change 
	\[
Z\times_{\cA^\ani_\infty }\cameral{\cA}^\ani\to \cameral{\cA}^\ani
\]
 is canonically isomorphic to the trivial cover
	\[
\coprod_{\omega\in \Omega}\cameral{\cA}^\ani\to \cameral{\cA}^\ani,
	\]
	where $\Omega$ is the set of weights of the defining representation of $\U_{2n}$, with $\sigma$ acting freely on the components via the conjugation action on the weight lattice of $\U_{2n}$. In particular, for every weight $\omega\in \Omega$, we have a section
	\[
	y_\omega\colon \cameral{\cA}^\ani\to \ol{\Sigma}'\times_{\cA^\ani_\infty }\cameral{\cA}^\ani.
	\]
\end{cor}
\begin{proof}
    By Lemma \ref{lem: cameral description via Z infinity}, a point in the left hand side is a tuple $\cameral a = (a,\ol{\infty}',\cameral\infty')$ where $\ol\infty'$ has image $\infty$; $\cameral\infty'=(\ol\infty_1',\dots, \ol\infty_n')$ for $\ol\infty'_i\in \ol\Sigma'_a$; there exists a choice of $i\in \{1,2\}$ so that each $\ol\infty'_i$ has image $\ol\infty'_i$; and $\ol\infty_j'\neq\ol\infty_\ell'$ for all $j\neq \ell$. In particular, since the map $\fa\git \{\pm1\}^n\to \fa\git(\fS_{n-1}\rtimes \{\pm 1\}^n)$ is of degree $(n-1)!$, there must exist some index $j$ such that $\ol\infty'$ and $\ol\infty'_j$ have the same image in $\fs_D$. In other words, such that either $\ol\infty' = \ol\infty'_j$ or $\ol\infty' = \sigma(\ol\infty'_j)$; and so there exists a unique $\omega_{\cameral a}\in \Omega$ such that $\omega_{\cameral a}(\cameral\infty') = \ol\infty'$. Therefore, we have an explicit isomorphism
    \[
Z\times_{\cA^\ani_\infty }\cameral\cA^\ani\to \bigsqcup_{\omega\in \Omega}\cameral\cA^\ani
    \]
    \[
\cameral a=(a,\ol{\infty}',\cameral\infty')\mapsto \big(\cameral \infty'\big)_{\omega_{\cameral a}}
    \]
    where the subscript $\omega_{\cameral a}$ denotes the $\omega_{\cameral a}$-th copy of $\cameral \cA^\ani$. The inverse of the above map is given by $(\cameral\infty')_\omega\mapsto (a,\omega(\cameral\infty'),\cameral\infty')$
\end{proof}

We now apply this to study the components of $\cP$. Denote by $\cameral{p}\colon \cameral{\cA}^\ani\to \cA^\ani_\infty$, and $\cameral{\cP} = \cameral{p}^*(\cP|_{\cA^\ani_\infty})$. By \cite[Proposition 6.2]{ngo_endoscopie}, there is a constructible sheaf $\pi_0(\cameral{\cP}/\cameral{\cA}^\ani)$ on $\cameral{\cA}^\ani$ whose fiber over any $a$ in $\cameral{\cA}^\ani$ is the component group of the fiber $\cameral{\cP}_a$.

Corollary \ref{cor: cameral cover indexes components of P} gives a canonical map
\[
\mathbb{X}_*(\canT)\times \cameral\cA^\ani\to \Pic\big(\ol{\Sigma}'\times_{\cA_\infty^\ani}\cameral\cA^\ani\big/\cameral\cA^\ani\big)
\]
\[
(\lambda,a)\mapsto \cO_{\ol \Sigma_a'}\left( \sum_{\omega\in \Omega}\langle \omega,\lambda\rangle y_\omega \right)
\]
When $\lambda\in \mathbb{X}_*(\canC)$, we have $\langle \omega,\lambda\rangle = -\langle \sigma(\omega),\lambda\rangle$ for $\omega\in\Omega$, where $\lambda$ is viewed in $\mathbb{X}_*(\canT)$. The line bundle above then becomes
\[
\cO_{\ol \Sigma_a'}\left( \sum_{\omega\in \Omega}\langle \omega,\lambda\rangle y_\omega \right) = \cO_{\ol \Sigma_a'}\left( \sum_{i=1}^n\langle \varepsilon_i,\lambda\rangle (y_i - \sigma(y_i)) \right).
\]
In particular, this carries a unique Hermitian structure, and so the image lies in the Prym variety $\cameral\cP$, and we deduce a morphism
\[
\mathbb{X}_*(\canC)\times \cameral\cA^\ani\to \cP
\]
over $\cameral\cA^\ani$. This induces a surjection on $\pi_0$ by Lemma \ref{lem: component group of Pa}, and hence we obtain the following.

\begin{prop}
	\label{prop: C surjects on components of P}
    There is a canonical surjective morphism of sheaves on $\cameral{\cA}^\ani$
    \begin{equation}
    \label{eqn: C surjects on components of P}
    \mathbb{X}^*(\wh{\canC})\times \cameral{\cA}^\ani\to \pi_0(\cameral{\cP}/\cameral{\cA}^\ani)
    \end{equation}
    which is given explicitly by identifying $\mathbb{X}^*(\wh{\canC})\simeq \Z^n$ and sending $((d_1,\dots,d_n),a)$ to the component of the bundle
    \[
\cO\left( \sum_{i=1}^n d_i(y_i - \sigma(y_i)) \right) \in \cameral{\cP}_a
    \]
    where $y_1,\dots, y_n$ are the sections in Corollary \ref{cor: cameral cover indexes components of P} for the elements $\varepsilon_1,\dots,\varepsilon_n\in \Omega$.
\end{prop}
\begin{proof}
    Immediate from Corollary \ref{cor: cameral cover indexes components of P} and the discussion above.
\end{proof}

We denote $\cameral{g}^\ani\colon \cameral{\cP}^\ani\to \cameral{\cA}^\ani$ and $\cameral{f}^{\star,\ani}\colon \cameral{\cM}^{\star,\ani}\to \cameral{\cA}^\ani$ for the base change of the respective structure maps $g$ and $f$ to $\cameral{\cA}^\ani$. Note that $\cameral\cP$ acts on $\bR \cameral{g}_!^\ani\ol{\Q}_\ell$ and $\bR \cameral{f}^\ani_*\ol{\Q}_\ell$ through $\pi_0(\cameral\cP/\cameral{\cA}^\ani)$. Hence, the map \eqref{eqn: C surjects on components of P} defines $\kappa$-decompositions
\begin{equation}\label{eqn: kappa decomp}
    \bR \cameral{g}_!^\ani\ol{\Q}_\ell = \bigoplus_\kappa (\bR \cameral{g}_!^\ani\ol{\Q}_\ell)_\kappa,\quad \bR \cameral{f}^\ani_*\ol{\Q}_\ell = \bigoplus_\kappa(\bR \cameral{f}^{\star,\ani}_*\ol{\Q}_\ell)_\kappa,
\end{equation}
where the each direct sum is finite--indexed over the same set--of $\kappa \in \wh{\canC}$, and the subscript $\kappa$ indicates the subsheaf on which the character group $\mathbb{X}^*(\wh{\canC})$ acts through the evaluation map at $\kappa$. 

\begin{cor}
\label{cor: kappa must be 2 torsion}
    For $\kappa\in \wh\canC$, we have $(\bR \cameral{f}^{\star,\ani}_*\ol{\Q}_\ell)_\kappa\neq 0$ if and only if $\kappa$ is 2-torsion in $\wh\canC$. In particular, $(\bR \cameral{f}^{\star,\ani}_*\ol{\Q}_\ell)_\kappa\neq 0$ if and only if $\kappa$ extends to an elliptic endoscopic datum $\fe$.
\end{cor}
\begin{proof}
    By Proposition \ref{prop: C surjects on components of P}, for each $a\in \cA^\ani(S)$ the map $\mathbb{X}^*(\wh\canC)\to \pi_0(\cP_a)$ factors through the quotient $\mathbb{X}^*(\wh\canC)/2\mathbb{X}^*(\wh\canC)$. In particular, if $(\bR \cameral{f}^{\star,\ani}_*\ol{\Q}_\ell)_\kappa$ has nonzero stalk at $a$, $\kappa$ must be order 2 in $\wh\canC$.

    Conversely, we note that for any distinct nonzero sections $s_i\in H^0(\Sigma,L)$, we can consider the point $a\in \cA^\ani(k)$ associated to the characteristic polynomial
    \[
\prod_{i=1}^n(\lambda^2-s_i).
    \]
    That is, $a_{2i}$ is the coefficient of $\lambda^{2(n-i)}$ in the above polynomial, which is a section in $H^0(\Sigma,L^{2i})$. The corresponding spectral cover $\ol\Sigma_a$ has $n$ components, each of which is isomorphic to $\Sigma$, and since $\ol\Sigma_a'=\ol\Sigma_a\times_\Sigma\Sigma'$, the same is true for $\ol\Sigma_a'$. We conclude that $a\in \cA^\ani$, and that $\pi_0(\cP_a) = \Z^{n}$ is maximal. Therefore, the stalk of $(\bR \cameral{f}^{\star,\ani}_*\ol{\Q}_\ell)_\kappa$ is nontrivial at $a$ for every 2 torsion element $\kappa$ and the converse holds.
\end{proof}

\subsection{The Endoscopic Case}

Let $\fe = (\kappa,\rho)$ be a relative endoscopic datum for $X^\siegel$ over $\Sigma$. From \S\ref{sec: endoscopic spaces}, we have the associated endoscopic space $X^\fe$, defined over $\Sigma$, which is a product of Siegel and unitary cases of complementary ranks. We may define all the same structures using the bundle of $H^\fe$-representations $V^\fe_D$. In particular, we have a moduli of $X^\fe$-Higgs bundles, $\cM^\fe$, which comes equipped with a Hitchin fibration
\[
f^\fe\colon \cM^\fe\to \cA_\fe,
\]
which itself is a product of Siegel and unitary cases. The transfer morphism $\nu_\fe\colon \fc_\fe\to \fc$ induces a morphism of Hitchin bases, which we also denote
\[
\nu_\fe\colon \cA_\fe\to \cA.
\]
Recall from Proposition \ref{prop: centralizers under transfer} that $\fs_\fe/\iota_\fe \to \fs/\iota\times_\fc\fc_\fe$ is a partial normalization which separates two components of $\fs/\iota\times_\fc\fc_\fe$ canonically. In particular,  for $\nu_\fe(a_\fe) = a$ there is a map $i_\fe\colon \ol\Sigma_{a_\fe}'\to \ol\Sigma_{a}'$ which realizes $\ol\Sigma_{a_\fe}'$ as a union of components of $\ol\Sigma_a'$. We describe this explicitly as follows:  Suppose that $\fe = (\kappa,\rho)$ is an endoscopic datum associated to a $2$-torsion element $\kappa\in \wh{\canC}$. Let $S_\kappa\subset \{1,\dots, n\}$ be the subset of indices for which $\kappa_i = 1$ and set $\#S_\kappa = r$. For $a\in \cA^\heartsuit(k)$, we have an associated characteristic polynomial $Q_a(\lambda^2)$, where $Q(\lambda)$ is a polynomial of degree $n$ with $i$-th coefficient in $H^0(\Sigma,L^{2i})$. A lift $a_\fe\in \cA^\ani_\fe(k)$ is the data of a factorization $Q_a = Q^\siegel Q^\unitary$, where $\deg(Q^\siegel) = r$ and $\deg(Q^\unitary) = n-r$. Furthermore, if $\ol\Sigma_{a_\fe}^\star$ denotes the vanishing locus of $Q^\star(\lambda)$ in the total space of $L$, then the image of the sections $\{y_j\mid j\in S_\kappa\}$ of Corollary \ref{cor: cameral cover indexes components of P} are in $\ol\Sigma_{a_\fe}^\siegel$ and $\{y_j\mid j\not\in S_\kappa\}$ are in $\ol\Sigma_{a_\fe}^\unitary$.

Therefore, for any such $a_\fe\in \cA_\fe^\heartsuit$, the spectral curve $\ol\Sigma_{a_\fe}$ is a disjoint union $\ol\Sigma_{a_\fe}=\ol\Sigma_{a_\fe}^\siegel\sqcup \ol\Sigma_{a_\fe}^\unitary$ while $\ol\Sigma_a = \ol\Sigma_{a_\fe}^\siegel\cup \ol\Sigma_{a_\fe}^\unitary$ is a (not necessarily disjoint) union of $\ol\Sigma_{a_\fe}^\siegel$ and $\ol\Sigma_{a_\fe}^\unitary$. Each $\ol\Sigma_{a_\fe}^\star$ is a union of components of $\ol\Sigma_{a}$ and the map $\ol\Sigma_{a_\fe}\to \ol\Sigma_{a}$ is 
induced by the inclusion maps of $\ol\Sigma_{a_\fe}^\star$ into $\ol\Sigma_{a}$. If $a\in \cA^\ani$ and $a_\fe\in\cA_\fe^\ani$, then this discussion can be lifted to the double covers 
$i_\fe\colon\ol\Sigma_{a_\fe}'=(\ol\Sigma_{a_\fe}^{\siegel})'\sqcup(\ol\Sigma_{a_\fe}^{\unitary})'\rightarrow \ol{\Sigma}'_a$.

By Lemma \ref{lem: cameral description via Z infinity}, the map $\nu_\fe$ lifts to the cover
\begin{equation}
\label{eqn: transfer on cameral}
\cameral\nu_\fe\colon \cameral \cA_\fe\to \cameral \cA.
\end{equation}
More precisely, if $\kappa = (1_r,-1_{n-r})$ the corresponding map is 
\[
\cameral\nu_\fe\left[\left(a_\fe,(\ol\infty_1',\dots, \ol\infty_r'),(\ol\infty_{r+1}',\dots, \ol\infty_n')\right) \right] = \left(\nu_\fe(a_\fe),(i_\fe(\ol\infty_1'),\dots,i_\fe( \ol\infty_n'))\right).
\] 
(When $\kappa$ is a permutation of $(1_r,-1_s)$, we interlace the orders as specified by the order of the signs on $\kappa$.) 
\begin{lem}
\label{lem: transfer on bases is closed embedding}
    The map \eqref{eqn: transfer on cameral} is a closed embedding over the heart locus $\cameral{\cA}^\heartsuit$.
\end{lem}
\begin{proof}
    Since the map $\fc_\fe\to \fc$ is induced by the inclusion of Weyl groups $W_{X^\fe}\subset W_X$, it follows from \cite[Proposition 7.2]{ngo_endoscopie} that \eqref{eqn: transfer on cameral} is finite and unramified. Therefore, it is enough to check the map is injective on geometric points. Suppose that $\left(a,(\ol\infty_1',\dots,\ol\infty_n')\right)\in \cameral\cA^\heartsuit(\ol k)$ is in the image of $\nu_\fe$. Then, the data of $\kappa$ and the ordering $(\ol\infty_1',\dots,\ol\infty_n')$ determine a decomposition $\ol\Sigma'_a = (\ol\Sigma_{a}^{\siegel})'\cup(\ol\Sigma_{a}^{\unitary})'$, which forces a unique preimage for $\left(a,(\ol\infty_1',\dots,\ol\infty_n')\right)$ in $\cameral\cA_\fe(\ol k)$.
\end{proof}

For any $\kappa\in \wh{\canC}$, we let $\cameral\cA_\kappa$ denote the support of $(\bR \cameral{g}_!^\ani\ol{\Q}_\ell)_\kappa$, which is a closed subvariety of $\cameral\cA^\ani$. 

\begin{prop}
    \label{prop: endoscopic description of components of P}
    For any $\kappa\in \wh{\canC}$, we have a decomposition of top degree cohomology
    \[
    \bR^\top\cameral{g}^\ani_!\ol{\Q}_{\ell} = \bigoplus_{\fe}\cameral\nu_{\fe,*}(\ol{\Q}_\ell)
    \]
    where the sum is over all endoscopic data $\fe = (\kappa,\rho)$. Moreover, the sheaf $\pi_0(\cameral\cP/\cameral{\cA}^\ani)$ is locally constant on an open subset of $\cameral\cA_\kappa$, $(\bR^\top\cameral g_!^\ani\ol \Q_\ell)_\kappa = \cameral\nu_{\fe,*}(\ol\Q_\ell)$ for the unique endoscopic datum $\fe = (\kappa,\rho)$ extending $\kappa$, and $\cameral \cA_\kappa=\cameral\nu_\fe(\cameral\cA_\fe)$, where $\fe = (\kappa,\rho)$ is the (unique) endoscopic datum containing $\kappa$.
\end{prop}
\begin{proof}
    The support of $(\bR\cameral{g}_!^\ani\ol{\Q}_{\ell})_\kappa$ can be computed by Lemma \ref{lem: component group of Pa} and Proposition \ref{prop: C surjects on components of P}. Namely, the support is the locus in $\cameral\cA^\ani$ for which the kernel of the evaluation map at $\kappa$, 
    \begin{equation}
    \label{eqn: evaluation at kappa}
    \mathbb{X}_*(\canC)\to \ol \Q_\ell^\times
    \end{equation}
    contains the kernel of \eqref{eqn: C surjects on components of P}. Corollary \ref{cor: kappa must be 2 torsion}, the only $\kappa$ with nontrivial $(\cameral{g}_!^\ani\ol{\Q}_{\ell,\cP})_\kappa$ are those for which $\kappa = (\kappa_1,\dots, \kappa_n)\in \wh\canC$ is 2-torsion; thus, $\kappa_i \in \{\pm1\}$ for all $i$. Let $\kappa\in \wh\canC[2]$ and let $S_\kappa\subset \{1,\dots,n\}$ denote the set of all indices $i$ for which $\kappa_i = -1$. The kernel of the map \eqref{eqn: evaluation at kappa} is given explicitly by
    \[
    \left\{ (x_1,\dots, x_n)\in \mathbb{X}_*(\wh\canC) \;\;\left|\;\; \sum_{i\in S_\kappa}x_i \equiv 0\mod 2 \right. \right\}
    \]
    The map \eqref{eqn: C surjects on components of P} is also quite explicit:  Recall $\Omega$ is the set of weights of the defining representation of $\U_{2n}$. Fix $\varepsilon_1,\dots, \varepsilon_n\in \Omega$ so that $\Omega = \{\varepsilon_1,\dots,\varepsilon_n,\sigma(\varepsilon_1),\dots,\sigma(\varepsilon_n\}$. At $a\in \cameral\cA^{\ani}(k)$, the spectral curve $\ol\Sigma_a'$ decomposes into irreducible components $C_i$. We put $S_i\subset \{1,\dots, n\}$ for the set of indices $j$ such that the image $y_j(a)$ of the section $y_j$ of Corollary \ref{cor: cameral cover indexes components of P} lies in $C_i$. We have a partition
    \[
\{1,\dots, n\} = S_1\sqcup \cdots \sqcup S_\ell.
    \]
Then the map sends 
    \[
    (x_1,\dots, x_n)\mapsto \left( \sum_{j\in S_i} x_j \right)_i\in (\Z/2\Z)^{\#\Irr(\ol\Sigma_a)}
    \]
    over $a$. In particular, we see that the containment of kernels happens exactly when $S_\kappa$ is one of the $S_i$; that is, the sections $y_j(a)$ for $j\in S_\kappa$ determine a union of components of $\ol\Sigma_a$. It follows immediately that $\cameral\cA_\kappa$ is the preimage of $\nu_\fe(\cA_\fe)$. Moreover, for fixed $a\in \cA^\ani$, the above a bijection between the set of $\kappa$ for which $a\in \cA_\kappa$ and subsets of $\mathrm{Irr}(\ol\Sigma_a)$, so that the natural map 
    \[
    \bigoplus_{\fe}\cameral\nu_{\fe,*}(\ol{\Q}_\ell) \to \bR^\top\cameral g^\ani\ol\Q_{\ell,\cameral\cP}
    \]
    is an isomorphism.
\end{proof}

\subsection{The $\delta$-stratification on $\cA^\ani$}
\label{sec: delta stratification}

The group stack $\cP$ is a gerbe over a commutative group scheme $P$. By Chevalley's structure theorem, for any $a\in \cA^\ani$, we have an exact sequence
\[
0\to R_a\to P_a\to A_a\to 0,
\]
where $R_a$ is the unique maximal connected affine subgroup of $P_a$ and $A_a$ is an abelian variety.

\begin{definition}
	The (global) $\delta$-invariant at $a\in \cA^\ani$ is the dimension of the affine part of $P_a$, i.e. $\delta(a) = \dim R_a$.

    We denote by $\cA_\delta$ the locally closed subset of $\cA^\ani$ consisting of $a$ in $\cA^\ani$ with $\delta(a) = \delta$.
\end{definition}

\begin{lem}
	\label{lem: delta invariant is same as GLn}
	Let $a\in \cA^\ani(\ol{k})$, and let $\delta_{n,a}$ denote the dimension of the affine part of the Jacobian $\Jac(\ol{\Sigma}_a)$. Then, we have $\delta(a) = \delta_{n,a}$.
\end{lem}
\begin{proof}
	Let $\delta_{n,a}'$ denote the dimension of the affine part of $\Jac(\ol{\Sigma}_a')$. Since $\ol{\Sigma}_a'\to \ol{\Sigma}_a$ is an unramified double cover, we have $\delta_{n,a}'= 2\delta_{n,a}$. Hence, the result follows from the short exact sequence
	\[
	1\to \Prym(\ol{\Sigma}_a',\ol{\Sigma}_a)\to \Jac(\ol{\Sigma}_a')\to \Jac(\ol{\Sigma}_a)\to 1
	\]
 by Proposition \ref{prop: spectral description of P star}.
\end{proof}
    
    We may form the relative Jacobian
\[
\mathrm{Jac}(\ol{\Sigma}/\cA)\to \cA,
\]
    which over the locus $\cA^\heartsuit$ is a relative Jacobian of a family of curves with reduced fibers. As the dimension of the affine part of $\mathrm{Jac}(\ol{\Sigma}_a)$, $\delta_{n,a}$ is an upper semi-continuous function, which agrees with $\delta_a = \dim\cP_a$ by Lemma \ref{lem: delta invariant is same as GLn}. By \cite[\S4.4.1]{ngo} and \cite[Prop 5.7.2]{ngo}, we obtain the following as a corollary with the necessary constraint on characteristic of $k$.
    
\begin{cor}
\label{cor: delta regularity in the limit}
    Assume that $\mathrm{char}(k)>n$. For every integer $\delta_0\geq 0$, there exists some integer $N_0$  such that if $\deg(D)\geq N_0$, then for every $\delta\leq \delta_0$, we have 
	\[
	\mathrm{codim}_{\cA^\ani}(\cA_\delta)\geq \delta.
	\]
\end{cor}

\section{Relative affine Springer fibers}
\label{sec: RASF}

In this section, we prove basic results about certain relative affine Springer fibers (RASFs) which arise in the local study of endoscopic orbital integrals for Friedberg--Jacquet varieties. In particular, we give moduli and lattice descriptions of RASFs; we study the component groups of RASFs for split elements (Proposition \ref{prop: dim formula split case}); and we relate the RASFs to relative Hitchin fibers via a product formula (Proposition \ref{prop: product formula}). We use the product formula together with a global argument to obtain a dimension formula for RASFs (Theorem \ref{prop: dim of RASF}).

\subsection{Lattice definition of RASFs}
\label{sec: lattice description of RASF}

Similar to the $\GL_n$ case of usual affine Springer fibers (cf. \cite[\S3.6]{ngo} for example), relative affine Springer fibers admit a definition in terms of moduli spaces of lattices. We choose to define these local spaces first in terms of lattices because this definition is at once elementary and makes clear questions of representability. We then give a definition that is likely to generalize to other spherical varieties in \S\ref{sec: general def of rasf}, relying on an important lifting property (Proposition \ref{prop: lifts of a over rs}) which we prove at the end of the present section.

We begin with a definition of lattices.

\begin{definition}
    Let $F$ be a non-archimedean local field and let $\cO\subset F$ be the corresponding ring of integers. Let $R$ denote an $\cO$-algebra which is free as an $\cO$-module. An $R$-lattice in $R\otimes_\cO F$ is an $R$-submodule $\Lambda\subset R\otimes_\cO F$ which is free as an $\cO$-module and for which $\Lambda\otimes_\cO F = R\otimes_\cO F$. 
\end{definition}

Now we specialize to our settings. Throughout this section, we let $F = k(\!(\varpi_v)\!)$ and we let $F'=F'_v$ be a fixed unramified quadratic extension of $F$ with $\sigma\in \Gal(F'/F)$ the nontrivial element in the Galois group. In particular, $F'=k'(\!(\varpi_v)\!)$, where $k'$ is the unique quadratic extension $k'$ of $k$. (For simplicity, we will drop the subscripts $v$ throughout this section.) Let $\fc^\heartsuit(\cO):=\fc(F)^{\mathrm{rs}}\cap \fc(\cO_F)$ denote all elements $a\in \fc(\cO_F)$ which are regular, semisimple over $F$, i.e. for which the centralizer $J_a$ is a torus. For $a = (a_{2i})_{i=1}^n\in \fc^\heartsuit(\cO_F)$, let $P_a(t) = x^{2n}+a_2x^{2n-2}+\cdots +a_{2n}\in \cO_F[t]$ denote the corresponding characteristic polynomial. Let 
\begin{equation}\label{eqn: def of R_a}
   \text{$R_a(F') = F'[x]/(P_a(x))$ and $R_a(\cO') = \cO'[x]/(P_a(x))$.} 
\end{equation}
There are two natural involutions on these algebras:  We let $\iota(f(x)) = f(-x)$ and 
\[
\sigma\left(\sum_{i=0}^{2n-1}\alpha_i x^i\right) = \sum_{i=0}^{2n-1}\sigma(\alpha_i)\cdot (-x)^i,
\]
where $\sigma(\alpha_i)$ denotes the Frobenius action on coefficients. We set $R_a(F) = R_a(F')^{\sigma}$ and $R_a(\cO) = R_a(\cO')^{\sigma}$. The involution $\iota$ descends to an involution on $R_a$ which we also denote by $\iota$. We note that we have Cartesian diagrams of algebras analogous to \eqref{eqn: diagram of covers}:
\[
\begin{tikzcd}  
	R_a(F')\ar[r,"/\sigma"]\ar[d,swap,"/\iota"]& R_a(F)\ar[d] \\
	R_a(F')^\iota \ar[r] & R_a(F)^\iota
\end{tikzcd} \qquad \qquad 
\begin{tikzcd}  
	R_a(\cO')\ar[r,"/\sigma"]\ar[d,swap,"/\iota"]& R_a(\cO)\ar[d] \\
	R_a(\cO')^\iota \ar[r] & R_a(\cO)^\iota
\end{tikzcd}
\]

A non-unitary version of a relative affine Springer fiber could be defined as the functor classifying $R_a(\cO)$-lattices $\Lambda$ in $R_a(F)$ satisfying $\iota(\Lambda) = \Lambda$. To describe our unitary versions $\cM^\star(a)$, we need to consider the additional structure of a Hermitian form on $R_a(F')$ whose compatibility with $\iota$ is dependent on whether we are in the Siegel or unitary case.

We begin with $\cM^\Siegel(a)$. Consider the natural pairing $h^\Siegel\colon R_a(F')\times R_a(F')\to F'$ by
\[
h^\Siegel(f,g) = \mathrm{tr}_{R_a(F')/F'}\left( \frac{f\cdot \sigma(g)}{P'_a(x)} \right)
\]
The pairing $h^\Siegel$ is clearly Hermitian and is nondegenerate by Euler's formula. We call an $\cO'$-lattice $\Lambda$ in $R_a(F')$ \emph{$h^\Siegel$-self-dual} if $\Lambda = \Lambda^{\Siegel-\vee}$, where $\Lambda^{\Siegel-\vee}$ is the dual lattice with respect to $h^\Siegel$, i.e.
\begin{equation}
\label{eqn: Siegel dual lattice}
\Lambda^{\Siegel-\vee} = \{f\in R_a(F')\mid h^\Siegel(f,g)\in \cO' \text{ for all }g\in \Lambda\}.
\end{equation}

\begin{definition}
\label{def: siegel RASF}
    The Siegel relative affine Springer fiber $\cM^\Siegel(a)$ is the functor classifying all $R_a(\cO')$-lattices $\Lambda$ in $R_a(F')$ such that $\iota(\Lambda) = \Lambda$ and $\Lambda$ is $h^\Siegel$-self-dual, i.e. for any finitely generated $k$-algebra $R$,
    \[
\cM_v^\Siegel(a)(R) = \{ \Lambda\subset R_a(F')\hat{\otimes}_k R\mid \Lambda \text{ is an $R_a(\cO')$-lattice satisfying }\iota(\Lambda) = \Lambda\text{ and }\Lambda^{\siegel-\vee} = \Lambda \},
    \]
    where $\hat{\otimes}_k$ is the completed tensor product.
\end{definition}

Note that the Hermitian form $h^\Siegel$ restricts to zero on the vector subspace $R_a(F')^{\iota}$. In particular, it is ill-suited for $\cM_v^\unitary(a)$, which is meant to model the local case where the Hermitian form restricts to a non-degenerate Hermitian form on $R_a(F')^{\iota}$. We instead define the form $h^\unitary\colon R_a(F')\times R_a(F')\to F'$ as follows:  Denote by $R_a(F')^{\pm\iota}$ the subset of invariant/anti-invariant functions $f\in R_a(F')$ under the involution $\iota$. For any $f,g\in R_a(F')^{\iota}$, respectively $f,g\in R_a(F')^{-\iota}$, we have $f\cdot \sigma(g)\in R_a(F')^{\iota}$ and so we may define
\[
h^\unitary(f,g) = \mathrm{tr}_{R_a(F')^{\iota}/F'}\left( \frac{f\cdot \sigma(g)}{Q'_a(x^2)} \right)
\]
where $Q_a(x) = x^n+a_2x^{n-1}+\cdots+a_{2n}$ is the polynomial in $t$ such that $Q_a(x^2) = P_a(x)$. When $f\in R_a(F')^{\iota}$ and $g\in R_a(F')^{-\iota}$, we put $h^\unitary(f,g) = 0$. One may now easily verify that $h^\unitary$ induces a non-degenerate Hermitian form on $R_a(F')$.

For any lattice $\Lambda$, we may define the dual lattice $\Lambda^{\unitary-\vee}$ in the same way as \eqref{eqn: Siegel dual lattice}, and we again say that $\lambda$ is $h^\unitary$-self-dual if $\Lambda = \Lambda^{\unitary-\vee}$.

\begin{definition}
\label{def: unitary RASF}
    The unitary relative affine Springer fiber $\cM^\unitary(a)$ is the functor classifying all $R_a(\cO')$-lattices $\Lambda$ in $R_a(F')$ such that $\iota(\Lambda) = \Lambda$ and $\Lambda$ is $h^\unitary$-self-dual, i.e. For any $k$-algebra $R$,
    \[
\cM^\unitary(a)(R) = \{ \Lambda\subset R_a(F')\hat{\otimes}_k R\mid \Lambda \text{ is an $R_a(\cO')\hat\otimes_kR$-lattice satisfying }\iota(\Lambda) = \Lambda\text{ and }\Lambda^{\unitary-\vee} = \Lambda \}.
    \]
\end{definition}

\begin{ex}
	\label{ex: rank 1 lattices}
	We compute the $k$-points of the spaces $\cM_v^\star(a)$ above in the case $n = 1$. Let $a = \varpi_v^m$. We have 
\[
R_a(F') = F'[x]/(x^2 +\varpi_v^m).
\]	
	It is easy to see that a lattice $\Lambda\subset R$ which is fixed by the involution $\iota$ is generated by $\varpi^i$, $\varpi^jx$, where $j\leq i\leq j+m$.
	
	The Hermitian form $h^\Siegel$ can be computed by means of Euler's formula:  in terms of the $F'$-basis $\{1,t\}$ for $R_a(F')$, it has matrix
	\[
	\begin{pmatrix}
		0 & 1 \\ 1 & 0
	\end{pmatrix}.
	\]
	In particular, we can compute $\Lambda^{\Siegel-\vee} = \langle \varpi_v^{-j},\varpi_v^{-i}x \rangle$, and so $\Lambda$ is $h^\Siegel$-self-dual if and only if $\Lambda = \langle \varpi_v^i,\varpi_v^{-i}x\rangle$ for $0\leq 2i\leq m$. 
	
	The Hermitian form $h^\unitary$ can be computed in terms of the same $F'$-basis as
	\[
\begin{pmatrix}
	1 & 0 \\ 0 & \varpi_v^m
\end{pmatrix}	.
	\]
	We have $\Lambda^{\unitary-\vee} = \langle \varpi_v^{-i},\varpi_v^{m-j}x \rangle$, and so $\Lambda$ is $h^\unitary$-self-dual if and only if $m$ is even and $\Lambda = \langle 1,\varpi_v^{m/2}x\rangle$. 

    We note that we have $\cM_v^\siegel(a)(k') = \cM_v^{\unitary}(a)(k')$ where $k'$ is the quadratic extension $k'$ of $k$ for which $F'\otimes_kk'$ splits, which in all cases is an infinite disjoint union of points.
\end{ex}

\begin{cor}
	\label{cor:  representability and basic props}
	The functors $\cM^\star(a)$ (for $\star = \siegel$ or $\unitary$) are representable by projective ind-schemes which are locally of finite type.
\end{cor}
\begin{proof}
    The space of $R_a(F')$-lattices is representable by a projective ind-scheme which is locally of finite type by \cite[\S3, Proposition 1]{kl}. As the above moduli spaces of lattices are closed subvarieties of the moduli space of all $R_a(F')$-lattices, the result is now immediate.
\end{proof}

Before concluding this section, we remark on another application of the local rings $R_a(F')$ to lifting elements $a\in \fc^\heartsuit(\cO_v)$. The following result will be used to give a more conceptual definition of RASFs in the following section.

\begin{prop}
	\label{prop: lifts of a over rs}
	For $\star\in \{\siegel,\unitary\}$ and for any $a\in \fc^\heartsuit(\cO_F)$, there is a canonical lift $[\gamma]\in [V^\star/H^\star](F)$ of $a$.
\end{prop}
\begin{proof}
	Let $[x]$ denote the multiplication-by-$x$ endomorphism on the ring $R_a(F')$. We claim that the tuple $(R_a(F'),h^\star,[x])$ gives an $F$-point of the stack $[V^\star/H^\star]$, with the characteristic polynomial of $[x]$ being given by $a\in \fc^\heartsuit(\cO_F)$. For this, we note that an $F$-point of the stack $[V^\star/H^\star]$ is the following data:
\begin{itemize}
    \item Two rank $n$ vector bundles $\cE_1',\cE_2'$ on $\Spec(F')$ with Hermitian structure $h$ on $\cE_1'\oplus \cE_2'$ which either $(\star = \unitary)$ is a direct sum of Hermitian forms on $\cE_1'$ and $\cE_2'$ or $(\star = \siegel)$ is a direct sum of isomorphisms $\cE_1'\simeq  \sigma^*\cE_2
    '^\vee$ and $\cE_2'\simeq \sigma^*\cE_1'^\vee$;
    \item Endomorphisms $\phi_{ij}'\colon \cE_i'\to \cE_j'$ for $i\neq j$ in $\{1,2\}$ such that $\phi' = \phi_{12}'\oplus \phi_{21}'$ satisfies $\sigma^*(\phi')^\vee\circ h+h\circ \phi' = 0$. 
\end{itemize}
    Then, we identify $\cE_1' = R_a(F')^\iota$ with elements represented by polynomials concentrated in even degrees and $\cE_2'= x\cdot R_a(F')^\iota$ with those in odd degree. We take $\phi = [x]$, which interchanges the two. We can check that the condition on the Hermitian form is satisfied for the forms $h^\star$ on $R_a(F')$.
\end{proof}

\subsection{General RASFs}
\label{sec: general def of rasf}
For the rest of this paper, we denote $\cO_v = k[\![\varpi_v]\!]$ and $F_v = k(\!(\varpi_v)\!)$, together with their unramified extensions $\cO_v'$ and $F_v'$, which we assume to be non-split. We also denote $\Sigma_v = \Spec(\cO_v)$ and $\Sigma_v^\bullet = \Spec(F_v)$. For the purposes of this section, $H^\star$ denotes the quasisplit forms of the relevant symmetric subgroups defined over $\cO_v$ constructed in Section \ref{sec: rational structures}. In particular, $H^\star$ splits over the $\cO_v'$. Let $\Gr_{H^\star} = H^\star(F_v)/H^\star(\cO_v)$ denote the affine Grassmannian of $H^\star$.

To an element $\gamma\in V^{\star,\rs}(F_v)\cap V^\star(\cO_v)$, we may define $\fX_\gamma^{\star,\mathrm{pre}}$ to be the functor sending a $k$-algebra $R$ to
\begin{equation}
	\label{eqn: RASF functor of points with lift}
\fX_\gamma^{\star,\mathrm{pre}}(R) = \{ h\in \Gr_{H^\star}(R)\mid h^{-1}\cdot \gamma\in V^\star(R\hat{\otimes}_k\cO_v) \}.
\end{equation}
We define $\fX^\star_{\gamma}$ to be the reduced subfunctor of $\fX_\gamma^{\star,\mathrm{pre}}$.

We immediately note that $\fX_\gamma^\star$ only depends on the image of $\gamma$ in the stack $[V^\star/H^\star](F_v)$. Indeed, the image $[\gamma]\in [V^\star/H^\star](F_v)$ is the data of a tuple:
\begin{itemize}
    \item when $\star = \siegel$, $[\gamma] = (\cE',h,\phi_{12}',\phi_{21}')$, where $\cE'$ is a rank-$n$ vector bundle over $\Sigma_v^\bullet\wh{\otimes}_kk'$; $h\colon \cE'\to \sigma^*(\cE')^\vee$ such that $h = \sigma h^\vee$; $\phi_{12}'\colon E'\to \sigma^*(\cE')^\vee$; and $\phi_{21}'\colon \sigma^*(\cE')^\vee\to \cE'$; satisfying $\sigma^*\phi_{12}'^\vee = \phi_{21}'$.
    \item when $\star = \unitary$, $[\gamma] = (\cE_1',\cE_2',h_1,h_2,\phi')$, where $\cE_i'$ are rank $n$ vector bundles on $\Sigma_v^\bullet\wh{\otimes}_k k'$; $h_i\colon \cE_i\to \sigma^*(\cE_i)^\vee$ such that $h=\sigma^*h_i^\vee$; and $\phi'\colon \cE_1'\to \cE_2'$ satisfying $\sigma^*\phi'^\vee\circ h_1+h_2\circ \phi' = 0$.
\end{itemize}
By base change, $[\gamma]$ induces similar data $[\gamma]_R$ for any $k$ scheme $R$, replacing $\Sigma_v^\bullet$ with $\Sigma_v^\bullet\wh{\otimes}_k R$. We may reformulate \eqref{eqn: RASF functor of points with lift} as
\[
\fX_{[\gamma]}^{\star,\mathrm{pre}}(R) = \left\{ (\cE,\xi)\left| \begin{array}{c}
	\cE \in [V^\star/H^\star](R\hat{\otimes}_k\cO_v)\text{ and } \\
	\xi\colon \cE|_{\Sigma_v^\bullet\wh{\otimes}_kR}\xrightarrow{\sim} [\gamma]_R\text{ is a choice of trivialization}\end{array} \right.\right\}.
\]

\begin{prop}
	\label{prop: RASF lattice description}
	Let $a\in \fc^\heartsuit(\cO_v)$. For $\star \in \{\siegel,\unitary\}$, the RASF $\cM_v^\star(a)$ (see Definitions \ref{def: siegel RASF} and \ref{def: unitary RASF}) is isomorphic to $\fX_{[\gamma]}^\star$ where $[\gamma]\in [V^\star/H^\star](F_v)$ is the canonical lift of $a$ of Proposition \ref{prop: lifts of a over rs}.
\end{prop}
\begin{proof}
	Suppose that $(\cE,\xi)\in \fX_{[\gamma]}^\star(R)$. In particular, $\cE$ determines a point $(\Lambda,h_\Lambda)\in [\mathfrak{u}_{2n}/\U_{2n}](\cO_v\wh{\otimes}R)$, which is the data of a $R_a(\cO'_v)$-module $\Lambda$ equipped with a Hermitian form $h_\Lambda$ so that $h_\Lambda\circ [x]+[x]\circ h_{\Lambda} = 0$, where $[x]$ denotes the multiplication-by-$x$ operation on $\Lambda$. The isomorphism $\xi$ determines, in particular, a choice of isomorphism $\Lambda\otimes_{\cO'_v}F_v'\simeq R_a(F_v')$, which allows us to view $\Lambda$ as a submodule of $R_a(F_v')$.

    The Hermitian form $h_\Lambda$ extends to a Hermitian form $h_\Lambda$ on $R_a(F_v')$ under which $\Lambda$ is $h_\Lambda$ self dual. The remaining data in the pair $(\cE,\xi)$ involve the compatibility of this form with the involution $\iota$. In particular, we see that the induced Hermitian form $h_\Lambda$ must restrict to zero on $R_a(F_v')^{\iota}$ in the Siegel case and must be a direct sum of Hermitian forms on $R_a(F_v')^{\pm\iota}$ in the unitary case. Up to automorphism of $R_a(F_v)$ commuting with $\iota$ we can therefore normalize $h_\Lambda$ to $h^\star$, which gives a bijection between $\fX_{[\gamma]}^\star(R)$ and $\cM^\star_v(a)(R)$.
\end{proof}

\begin{remark}
	While this lattice definition of $\cM_v^\star(a)$ is special to the unitary Friedberg--Jacquet varieties and other symmetric varieties for classical groups, we expect an identical construction to $\fX_{[\gamma]}^\star$ to work for any spherical variety $X$. The obstruction is to find lifts $[\gamma]\in [T^*X/G](F_v)$ for any $a\in (T^*X\git G)^\heartsuit(\cO_v)$. This is of course nontrivial, and we leave this point to future work.
\end{remark}

For the Siegel case, there is a Kostant-Rallis section, which permits us to use the more naive definition in terms of $\fX_\gamma^\siegel$. The following ensures this is well defined.

\begin{lem}
\label{lem: kostant section and canonical lift agree}
	For $a\in \fc^\heartsuit(\cO_v)$ and a fixed Kostant-Rallis section $[\e]$, the canonical lift $[\gamma]\in [V^\siegel/H^\siegel](F_v)$ of Proposition \ref{prop: lifts of a over rs} is isomorphic to the image of $[\e](a)\in V^\siegel(F)$. In particular, $\cM_v^\siegel(a) \simeq \fX_{[\gamma]}^\siegel \simeq \fX_{[\e](a)}^\siegel$.
\end{lem}
\begin{proof}
    The image of $[\e](a)\in V^\siegel(F_v)$ in $[V^\siegel/H^\siegel](F_v)$ consists of the trivial $H^\siegel$-bundle together equipped with an endomorphism--which can be used to identify it with $R_a(F_v')$--and a compatible Hermitian form $h^\e$. The Hermitian form $h^\e$ is characterized up to scaling by the following properties:
    \begin{itemize}
        \item $h^\e$ is compatible with the endomorphism $[x]$ in the sense that $h^\e(x\cdot f,g) +h^\e(f,x\cdot g)$ for all $f,g\in R_a(F_v')$;
        \item $h^\e$ restricts to zero on $R_a(F_v')^\iota$.
    \end{itemize}
    We may write $h^\e(f,g) = h^\siegel(\eta f,g)$ for some automorphism $\eta$, which by the first condition is in $R_a(F_v')^\times$ and by the second is in $(R_a(F_v')^\times)^\iota$. In particular, $\eta$ determines an isomorphism between the image of $[\e](a)$ and $[\gamma]$.
\end{proof}

\begin{remark}
\label{rem: RASF computed with lift}
    If $k'/k$ is quadratic extension over which our unitary groups split, then $\cM_v^\siegel(a)\otimes_k k'\simeq \cM_v^\unitary(a)\otimes_k k'$. Hence, Lemma \ref{lem: kostant section and canonical lift agree} together with the conjugation invariance of $\fX_\gamma^\star$ imply that the spaces $\cM_v^\siegel(a)\otimes_k k'$ and $\cM_v^\unitary(a)\otimes_k k'$ can be computed using any lift $\gamma'\in V^\star(k')$. We will use this fact for algebraically closed fields in \S\ref{sec: split case}. 
\end{remark}

\subsection{The Obstruction Class}
\label{sec: obstruction}

In this section, we study a class $\mathrm{Ob}_v^\star(a)$, which serves as the obstruction to the existence of $k$-points in the RASFs $\cM_v^\star(a)$. This obstruction only exists in the unitary case ($\star = \unitary$) and that it is controlled by the parity of $\val(a_{2n})$.

Given $a\in \fc^\heartsuit(\cO_v)$, recall that we may form the ring $R_a(F_v')$, which comes equipped with a Hermitian form $h^\star$. The tuple $(R_a(F_v'),h^\star,[x])$ gives an $F_v$ point of the stack $[\mathfrak{u}_{2n}/\U_{2n}](F_v)$, which is the image of the canonical lift $[\gamma]$ of $a$ (see Proposition \ref{prop: lifts of a over rs}) under the map $[V^\star/H^\star]\to [\mathfrak{u}_{2n}/\U_{2n}]$. In particular, $(R_a(F_v'),h^\star)$ gives a $\U_{2n}$ torsor on $\Sigma_v^\bullet = \Spec(F_v)$.

\begin{definition}
	The obstruction class $\Ob_v^\star(a)$ is the class
	\[
\Ob_v^\star(a) = [(R_a(F_v'),h^\star)]\in H^1(F_v,\U_{2n}) \simeq \Z/2\Z
	\]
    If no superscript is written, we mean the unitary case, i.e. $\Ob_v(a) = \Ob_v^\unitary(a)$.
\end{definition}

Our convention to drop the superscripts in the unitary case is justified by the following lemma.

\begin{lem}
	\label{lem: obstruction is always trivial for siegel}
	If $\star = \siegel$, then $\Ob^\siegel_v(a)$ is trivial for all $a\in \fc^\heartsuit(\cO_v)$.
\end{lem}
\begin{proof}
	We can calculate $H^1(F_v,\U_{2n}) = \Z/2\Z$ with a $\U_{2n}$-bundle $(E,h)$ being taken to $\val_v(\det(h))\mod 2$. But by our definition of $h^\siegel$ and Euler's formula for trace, the matrix of $h^\siegel$ with respect to the basis $\{1,x,\dots, x^{2n-1}\}$ of $R_a(F_v')$ is lower triangular, with 1's on the anti-diagonal and nonzero entries only possible in the lower right triangular entries. In particular, $\det(h^\siegel) = 1$, so the class of $(R_a(F_v'),h^\siegel)$ in $H^1(F_v,\U_{2n})$ is trivial.
\end{proof}

The unitary case is more interesting. The following computes $\Ob_v(a) = \Ob_v^\unitary(a)$.

\begin{lem}
	\label{lem: obstruction for unitary}
	If $\star = \unitary$, then 
	\[
	\Ob_v(a) = \val (a^*\fD^\br) = \val(a_{2n})\in \Z/2\Z
	\]
	 is the parity of the valuation of the determinant component of $a = (a_2,\dots, a_{2n})\in \fc^\heartsuit(\cO_v)$.
\end{lem}
\begin{proof}
	Similar to Lemma \ref{lem: obstruction is always trivial for siegel}, it suffices to calculate the valuation of $\det(h^\unitary)$. By Euler's formulas (using an identical argument to Lemma \ref{lem: obstruction is always trivial for siegel} using the basis $1,x^2,\dots, x^{2n-2}$), we know the restriction of $h^\unitary$ to $R_a(F_v')^{\iota}$ has determinant $1$. On the other hand, the restriction of $h^{\unitary}$ to $R_a(F_v')^{-\iota}$  can be computed with respect to the basis $x,\dots, x^{2n-1}$ as follows. We can identify $R_a(F_v')^{\iota}\simeq R_a(F_v')^{-\iota}$ as $F_v'$-vector spaces by multiplying by $x$ and under this identification, the form $h^{\unitary}|_{R_a(F_v')^{-\iota}}$ is identified with the composition of $h^{\unitary}|_{R_a(F_v')^{\iota}}$ and the multiplication-by-$x^2$ endomorphism. As the former has determinant 1 and the latter has determinant $a_{2n}$, we conclude that $h^\unitary$ restricted to $R_a(F_v')^{-\iota}$ has determinant $a_{2n}$. Since $h^\unitary$ is defined so that $R_a(F_v')^{\iota}$ and $R_a(F_v')^{-\iota}$ are orthogonal with respect to $h^\unitary$, it follows that $\det(h^\unitary) = a_{2n}$, and the lemma follows.
\end{proof}

Finally, we prove that the obstruction class indeed captures the existence of rational points. This will end up being a non-vanishing statement for stable orbital integrals.

\begin{prop}
	\label{prop: obstruction measures existence of k points}
	The obstruction class $\Ob_v^\star(a)$ is trivial if and only if the RASF $\cM_v^\star(a)$ has a $k$ point. In particular, $\cM_v^\unitary(a)$ has a $k$ point if and only if $ \val(a^*\fD^\br) = \val(a_{2n})$ is even.
\end{prop}
\begin{proof}
	We always have $\Ob^\siegel_v(a) = 0$ by Lemma \ref{lem: obstruction is always trivial for siegel}. Moreover, by the existence of the Kostant-Rallis section, we have $\cM_v^\siegel(a)(k)\neq \emptyset$ for all $a\in \fc^\heartsuit(\cO_v)$. Therefore, the only nontrivial statement is for $\star = \unitary$.
		
	Suppose first that there exists a $h^\unitary$-self-dual lattice $\Lambda$ such that $\iota(\Lambda) = \Lambda$. Then, the Hermitian form $h^\unitary$ restricts to a perfect pairing on $\Lambda$, and so must have determinant a norm. Hence, the $\U_{2n}$ bundle $(R_a(F_v'),h^\unitary)$ is trivial, and $\Ob_v(a) = 0$. We note that this argument proves the stronger claim that, if $\Ob_v(a)\neq 0$, then there are no $h^\unitary$-self-dual lattices (regardless of stability under $\iota$).
	
	Now, suppose that $\Ob_v(a) = 0$. Then, $(R_a(F_v'),h^\unitary)$ is the trivial $\U_{2n}$ bundle, and moreover, by our formula for $h^\unitary$ and the proof of Lemma \ref{lem: obstruction for unitary}, the restriction of $h^\unitary$ to $R_a(F_v')^{\iota}$ and to $R_a(F_v')^{-\iota}$ are each isometric to the trivial Hermitian form. That is, there exist bases $\cB_{\pm}$ for $R_a(F_v')^{\pm\iota}$ with respect to which $h^\unitary$ has matrix the identity $I_n$ when restricted to $R_a(F_v')^{\pm\iota}$. In particular, we can consider the lattices $\Lambda_\pm\subset R_a(F_v')^{\pm\iota}$, defined by
	\[
\Lambda_\pm = \bigoplus_{e\in \cB_\pm}\cO_v'\cdot e	
	\]
	Then $\Lambda = \Lambda_+\oplus \Lambda_-$ is self dual and satisfies $\iota(\Lambda) = \Lambda$.
\end{proof}

\subsection{Symmetries of RASFs}
\label{sec: symmetries of RASFs}

Let $a\in \fc^\heartsuit(\cO_v):=\fc^\rs(F_v)\cap \fc(\cO_v)$. For any group $J_a'$ over $\cO_v$ with generic fiber $J_a'(F_v)\simeq J_a(F_v)$ agreeing with the group scheme of regular centralizers $J_a$, we consider the affine Grassmannian
\[
\cP_v(J_a') := \Gr_{J_a'} = J_a(F_v)/J_a'(\cO_v)
\]
For each $\gamma\in V^\star(F_v)$ lifting $a$, we obtain an identification of the fiber $J_a(F_v)$ with the centralizer $I_{\gamma}\subset H^\star(F_v)$. The latter group acts on the scheme $\fX_\gamma^\star$ by left multiplication. This action factors through $\cP_v(J'_a)$, and we deduce the following.

\begin{prop}
	\label{lem: symmetries of RASF}
	For any $\gamma\in V^\star(F_v)$ lifting $a$, the group scheme $\cP_v(J'_a)$ acts on $\fX_\gamma^\star$. In particular, there is an action of $\cP_v(J'_a)$ on  $\cM_v^\star(a)$.
\end{prop}

Analogously to the classical setting, this action consists of a free action by a lattice in $\cP_v(J'_a)$ together with a an action of a group scheme of finite type. More precisely, for $a\in \fc^\heartsuit(\cO_v)$, we let $\Gamma_a$ denote the lattice $\Gamma_a = \mathbb{X}_*(J_a)_{\breve F_v}$ of $\breve F_v$-rational cocharacters. Then $\Gamma_a$ inherits an action of Frobenius, and hence descends to an ind-scheme over $k$. If $\gamma\in V^\star(F_v)$ is a lift of $a$, and $I_\gamma\subset H^\star(F_v)$ its centralizer group, then we have a canonical identification $\Gamma_a \simeq \mathbb{X}_*(I_\gamma)_{\breve F_v}$. More explicitly, suppose that we have a factorization
\[
R_a(F_v)^{\iota} \simeq \prod_{j=1}^r K_j
\]
where $K_j$ is a finite extension of $F_v$ for all $j$. Let $\varpi_j$ be a uniformizer of $K_j$. Then, $(m_1,\dots, m_r)\in \Gamma_a\otimes_k k'\simeq \Z^r$ acts on $\cM_v^\star(a)\otimes_k k'$ by multiplication by $(\varpi_1^{m_1},\dots, \varpi_r^{m_r})\in R_a(F_v)^{\iota}$.

\begin{prop}
    \label{prop: lattice action on RASF}
    Fix $a\in \fc^\heartsuit(\cO_v)$.
    \begin{enumerate}
        \item The action of $\Gamma_a$ on $\cM_v^\star(a)$ is free;
        \item The quotient $\cM_v^\star(a)/\Gamma_a$ is a finite-type projective scheme, and $\cP_v(J_a)/\Gamma_a$ is a commutative group scheme of finite type;
        \item The stabilizers of the action of $\cP_v(J_a)$ on $\cM_v^\star(a)$ are finite type and affine.
    \end{enumerate}
\end{prop}
\begin{proof}
    As all claims are geometric, we will assume throughout that $k=\ol k$. Put $c = \val\bigg( \mathrm{Disc}(P_a)\bigg)$. The proof will be via explicit computation. For simplicity throughout this proof, we denote $R= R_a(\cO_v')\otimes_k\ol k$ and $\wt{R}$ the normalization of $R$. Note that our assumption that $k=\ol k$ does not mean that we work with $R_a(\breve\cO_v)$, but rather $R_a(\cO_v')\otimes_k\ol k\simeq R_a(\breve\cO_v)\oplus R_a(\breve\cO_v)$. If we decompose 
    \[
R\otimes_{\cO_v'}F_v' = R_a(F_v')\simeq \prod_j E_j
    \]
    as a product of finite extensions $E_i$ of $F_v'$, then $\wt{R}\simeq \prod_j\cO_{E_j}$. We denote by
    \[
v\colon R_a(F_v')^\times/\wt{R}^\times\xrightarrow{\sim} \prod_j \Z
    \]
    the valuation map on each $E_j^\times$. As $\wt{R}$ is stable under the involution $\iota$, we can transfer $\iota$ to the right-hand side, and so can define $\Gamma = \left( \prod_j \Z \right)^\iota$. The valuation map $v$ determines an embedding $v\colon \Gamma_a\to \Gamma$, which because $\Gamma_a$ and $\Gamma$ are free abelian groups of the same rank, has finite cokernel. 

    Since $\wt{R}$ is integral, we must have $\wt{R}\subset R^\vee$, where $R^\vee = (P_a')^{-1}\cdot R$ is the dual $\cO_v'$-lattice to $R$ in $R_a(F_v')$, where $P_a$ is as in \eqref{eqn: def of R_a} and $(P_a')^{-1}\cdot R$ denotes the fractional ideal generated by $(P_a')^{-1}$. As $\mathrm{Disc}(P_a)\in P_a'\cdot R$, it follows that $\varpi_v^c\wt{R}\subset R$. On the other hand, since every $\wt{R}$-lattice in $R_a(F_v')$ is principal for every $R$-lattice $\Lambda$, there exists some $f_\Lambda\in R_a(F_v')$ such that $\wt{R}\cdot \Lambda = f_\Lambda\cdot \wt{R}$. In particular, we can define $v(\Lambda) := v(f_\lambda)$. Note that $v(\Lambda)\in \Gamma$ when $\Lambda$ is an $\iota$-fixed lattice. From the inclusion $\varpi_v^c\wt{R}\subset R$ of the previous paragraph, we get
    \begin{equation}
        \label{eqn: inclusion of lattices}
        \varpi_v^cf_\Lambda\cdot \wt{R}\subset \Lambda\subset f_\Lambda\cdot \wt{R}.
    \end{equation}

    Now, let us prove (1). If $\check\lambda\cdot \Lambda = \Lambda$ for $\check\lambda\in \Gamma_a$ and $\Lambda\in \cM_v^\star(a)(\ol k)$, then taking valuations gives 
    \[
    v(\Lambda) = v(\check\lambda\cdot \Lambda) = v(\check\lambda(\varpi_v))+v(\Lambda),
    \]
    so that $v(\check\lambda(\varpi_v)) = 0$. Since $v$ is injective on $\Gamma_a$, this implies that $\check\lambda=0$, and so $\Gamma_a$ acts freely on $\cM_v^\star(a)$.

    We prove (2) by constructing a finite type, projective subvariety of $\cM_v^\star(a)$ which contains a fundamental domain for the $\Gamma_a$-action. For each of the finitely many cosets $s\in \Gamma/\Gamma_a$, we take a representative $f_s\in R_a(F_v')^\times$. Define $Y_s$ to be the moduli of all $\iota$-fixed $R$-lattices $\Lambda$ such that
    \[
    \varpi_v^cf_s\cdot \wt{R}\subset \Lambda\subset f_s\cdot \wt{R}
    \]
    We can view $Y_s$ as a closed subvariety of the Grassmannian of the finite dimensional vector space $f_s\cdot \wt{R}/\varpi_v^cf_s\cdot \wt{R}$, so $Y_s$ is a projective algebraic variety of finite type. Consider the space
    \[
    Y_\cM := \cM_v^\star(a)\cap \left( \bigcup_{s\in \Gamma/\Gamma_a} Y_s \right).
    \]
    Then, for any $\Lambda\in \cM_v^\star(a)$, pick $\check\lambda\in \Gamma_a$ and $s\in \Gamma/\Gamma_a$ such that $v(\Lambda)-v(\check\lambda(\varpi_v)) = v(f_s)$. Then, $\wt{R}\cdot (\check\lambda^{-1} \Lambda) = f_s\cdot\wt{R}$, and so by \eqref{eqn: inclusion of lattices}, we see that $\check\lambda^{-1}\cdot \Lambda\in Y_s\subset Y_\cM$. We conclude that $Y_\cM$ contains a fundamental domain for the $\Gamma_a$-action, and part (2) follows.

    Consider now the symmetry group $\cP_v(J_a)$. We compute $\cP_v(J_a)\simeq (R^\iota\otimes_{\cO_v'}F_v')^\times/(R^\iota)^\times$. Therefore, $\cP_v(J_a)$ fits into a short exact sequence
    \begin{equation}
    \label{eqn: decomposition of P(Ja) into lattice and finite parts}
    0\to (\wt{R^\iota})^\times/(R^\iota)^\times\to \cP_v(J_a)\to (R^\iota\otimes_{\cO_v'}F_v')^\times/(\wt{R^\iota})^\times\to 0
    \end{equation}
    where $\wt{R^\iota}$ is the normalization of $R^\iota$. In particular, $(R^\iota\otimes_{\cO_v'}F_v')^\times/(\wt{R^\iota})^\times\simeq \prod_i\Z$ is a product of copies of $\Z$ since $(\wt{R^\iota})^\times$ splits as a product of DVR's, and the the composition $\Gamma_a\to (R^\iota\otimes_{\cO_v'}F_v')^\times/(\wt{R^\iota})^\times$ is injective between abelian groups of the same rank and hence is of finite index. Therefore, in the resulting exact sequence
    \[
0\to (\wt{R^\iota})^\times/(R^\iota)^\times\to \cP_v(J_a)/\Gamma_a\to \left[(R^\iota\otimes_{\cO_v'}F_v')^\times/(\wt{R^\iota})^\times\right]/\Gamma_a\to 0,
    \]
    the cokernel is finite and the kernel is affine and of finite type. Hence, $\cP_v(J_a)/\Gamma_a$ is finite type. By (1), the stabilizer subgroup of any $m\in \cM_v^\star(a)(\ol k)$ under the action of $\cP_v(J_a)$ can be viewed as a subgroup of $\cP_v(J_a)/\Gamma_a$, which necessarily must be affine and of finite type since it is an extension of a finite type, affine group by a a finite group.    
\end{proof}

\begin{remark}
    In contrast to the case of usual affine Springer fibers, the lattice $\Gamma_a$ does not act transitively on the set of irreducible components of $\cM_v^\star(a)$. However, for each connected component of $\cM_v^\star(a)$, we will see that the lattice action is free and transitive on the set of irreducible components \emph{in that connected component} of $\cM_v^\star(a)$.
\end{remark}

\subsection{Dimension in the split and transverse cases}
\label{sec: split case}

The relative affine springer fibers $\cM_v^\star(a)$ are particularly simple in the case of \emph{split} $a\in \fc^\heartsuit(\cO_v)$. We will give an explicit description of the dimension and component groups in this case, which we will use in Section \ref{sec: dimension formula} to describe a general dimension formula. 

\begin{definition}
\label{def: split}
	We say $a\in \fc^\heartsuit(\cO_v)$ is \emph{split} if the centralizer $J_a$ is a split torus over $F_v$.
\end{definition} 

Recall that we use a superscript breve to denote base change to $\ol k$; for example $\breve{F}_v=F_v\hat{\otimes}_k\ol k$, $\breve\cO_v = \cO_v\hat{\otimes}_k\ol k$, etc. 

\begin{lem}
	\label{lem:  split representatives}
	Let $a\in \fc^\heartsuit(\cO_v)$ be split. Then, there is a lift $\gamma\in \fh^\perp(\breve F_v)$ of $a$ of the form
	\[
	\gamma = \begin{pmatrix}
		0 & A \\
		I_n & 0
	\end{pmatrix}	\quad \text{where }A = \begin{pmatrix}
		u_1\varpi_v^{\delta_1} & & & \\
		& u_2\varpi_v^{\delta_2} & & \\
		& & \ddots & \\
		& & & u_n\varpi_v^{\delta_n}
	\end{pmatrix}
	\]
	where $0\leq \delta_1\leq \cdots \leq \delta_n$ and $u_i\in \cO_v^\times$ such that $u_i\neq u_j$ whenever $\delta_i=\delta_j$. This lift is unique up to reordering of pairs $u_i,u_j$ such that $\delta_i=\delta_j$
\end{lem}

\begin{proof}
	Choose any lift $\gamma\in \fh^\perp(\breve F_v)$ of $a$. We have a canonical identification $I_\gamma\simeq J_a(\breve F_v)$, which is a split torus over $\breve F_v$ by assumption. Hence, the image of $I_\gamma$ under the each of the two projections $H(\breve F_v) = \GL_n(\breve F_v)\times \GL_n(\breve F_v)\to \GL_n(\breve F_v)$ is a split maximal torus of $\GL_n(\breve F_v)$, and up to $H(\breve F_v)$ conjugacy, we may assume that the image of $I_\gamma$ under each projection is the standard maximal torus consisting of diagonal matrices. Now, let
	\[
	\gamma = \begin{pmatrix}	
		0 & A \\ B & 0
	\end{pmatrix}
	\]
	and denote $T_{\GL_n}\subset \GL_n$ and $T\subset H$ the standard diagonal maximal tori. Then, it is easy to see that $(g,h)\in I_\gamma\subset T(\breve F_v)$ implies that $g\in C_{\GL_n(\breve F_v)}(BA)$ and $h\in C_{\GL_n(\breve F_v)}(AB)$. In particular, we see that $T_{\GL_n}(\breve F_v)\subset C_{\GL_n(\breve F_v)}(BA)$ and $T_{\GL_n}(\breve F_v)\subset C_{\GL_n(\breve F_v)}(AB)$. Since we have an equality of characteristic polynomials (with $\lambda$ a variable)
    \[
    \mathrm{char}(\gamma,\lambda) = \mathrm{char}(AB,\lambda^2) = \mathrm{char}(BA,\lambda^2),
    \]
    it follows that $AB$ and $BA$ are regular, semisimple over $\breve F_v$ with no zero eigenvalues. Hence, we conclude that $AB$ and $BA$ are diagonal, and since $A$ and $B$ have no zero eigenvalues, $A$ and $B$ are invertible. Therefore, we have
\[
\begin{pmatrix}	
		I_n & 0 \\ 0 & B^{-1}
	\end{pmatrix}\begin{pmatrix}	
		0 & A \\ B & 0
	\end{pmatrix}\begin{pmatrix}	
		I_n & 0 \\ 0 & B
	\end{pmatrix} = \begin{pmatrix}	
		0 & AB \\ I_n & 0
	\end{pmatrix}
\]
where $AB$ is diagonal, regular semisimple. This representation is unique up to reordering of the diagonal entries of $AB$, and so the lemma follows. 
\end{proof}

We now prove a dimension formula for affine Springer fibers for split $a$. Note that, by Remark \ref{rem: RASF computed with lift}, we can compute $\cM_v^\star(a)\otimes_k\ol k\simeq \fX^\star_\gamma$ using the lift $\gamma\in V^\star(\breve F_v) = \fh^\perp(\breve F_v)$ of Lemma \ref{lem:  split representatives}.

We may identify $\fc\simeq \fc_{\GL_n}$, with the corresponding $\G_m$ actions doubled in $\fc$. More explicitly, for $a = (a_1,\dots, a_{n})\in \fc_{\GL_n}$ thought of as a characteristic polynomial
\[
P_a(\lambda) = \lambda^{n} +a_1\lambda^{n-1}+\cdots +a_{n},
\]
we send $a$ to the same $n$-tuple in $\fc$, thought of as the polynomial $P_a(\lambda^2)$. If $a\in \fc$, we call $a_{\GL}\in \fc_{\GL_n}$ as above the \emph{associated element to $a$ in $\fc_{\GL_n}$.} 

The following (without component groups) appears in \cite{polak}. We give a simpler argument here.

\begin{prop}
	\label{prop: dim formula split case}
	Let $a\in \fc^\heartsuit(\cO_v)$ be split with $a_{\GL}$ the associated element in $\fc_{\GL_n}(\breve \cO_v)$. Then, the RASF $\cM_v^\star(a)$ is equidimensional with
	\begin{equation}
    \label{eqn: dim for split case}
	\dim\cM_v^\star(a) = \dim \cP_v(J_a) =\dim \cM_v(a_{\GL})
	\end{equation}
	where $\cM_v(a_{\GL})$ denotes the affine Springer fiber for the group $G=\GL_n$ over $a_{\GL}$. Moreover, if 
	\[
	\gamma = \begin{pmatrix}
		0 & A \\
		I_n & 0
	\end{pmatrix}	\quad \text{where }A = \begin{pmatrix}
		u_1\varpi_v^{\delta_1} & & & \\
		& u_2\varpi_v^{\delta_2} & & \\
		& & \ddots & \\
		& & & u_n\varpi_v^{\delta_n}
	\end{pmatrix}
	\]
	is a lift of $a$ as in Lemma \ref{lem:  split representatives}, then 
	\[
\mathrm{Irr}(\cM_v^\star(a)) \simeq \{(m^{(1)},m^{(2)})\in \Z^n\times \Z^n\colon 0\leq m_{i}^{(1)}-m_{i}^{(2)}\leq \delta_i\}
	\]
	with the action of the Frobenius $\sigma$ for $k$ coming from
	\[
	\sigma(m^{(1)},m^{(2)}) = \begin{cases}
		(-m^{(1)},-m^{(2)}) & \text{in the unitary case, $\star = \unitary$} \\
		(-m^{(2)},-m^{(1)}) & \text{in the Siegel case, $\star = \Siegel$} \\
	\end{cases}
	\]
\end{prop}
\begin{proof}
	To compute dimension and components, it is harmless to assume that $k = \ol{k}$ is algebraically closed and identify $H$ with $\GL_n\times \GL_n$. We begin by taking $\gamma$ to be slightly more general than the above form:  Assume that
	\begin{equation}
		\label{eqn: gamma after t action}
		\gamma = \begin{pmatrix}
			0 & A' \\
			A'' & 0
		\end{pmatrix}
	\end{equation}
	where $A',A''$ are diagonal and $A'A'' = A$ is of the form of Prop \ref{lem:  split representatives}. Consider the unipotent subgroup of $H$, which we denote by $N=N_n\times N_n$, with $N_n$ the standard upper triangular unipotent subgroup of $\GL_n$. Denote by $\fX^{N}_\gamma$ the unipotent RASF whose $\ol{k}$ points are given by
	\[
	\fX_\gamma^N(\ol{k}) = \{n\in \Gr_N\mid \mathrm{Ad}(n)^{-1}\cdot \gamma\in \fh^\perp(\cO_v) \}.
	\]
	We may similarly define $\fX_{\gamma'}^{N_n}$ for $\gamma'\in \mathfrak{gl}_n(\cO_v)$ a regular semisimple nil-element. Let $\gamma' = A = A'A''$ denote the diagonal matrix in the upper right block of $\gamma$. We have a map
	\begin{equation}
		\label{eqn: U diagonal comparison}
		\fX_\gamma^N\to \fX_{\gamma'}^{N_n}\times \fX_{\gamma'}^{N_n}
	\end{equation}
	taking $n$ to each of its components $(n_1,n_2)$. A simple calculation shows that $n = (n_1,n_2)\in N(F_v)$ represents an element of $\fX_\gamma^N$ if and only if $n_2^{-1}n_1, n_1^{-1}An_2\in \mathfrak{n}_n(\cO_v)$. Since $n_2$ and $n_1$ are invertible in a unipotent subgroup, the former condition holds if and only if $n_1 = n_2\cdot N_n(\cO_v)$, and so the above two conditions are equivalent to the conditions $n_1,n_2\in \fX_{\gamma'}^{N_n}$ and $n_1 = n_2\cdot N_n(\cO_v)$. That is, the map \eqref{eqn: U diagonal comparison} identifies $\fX_{\gamma}^N$ with the diagonal in the product $\fX_{\gamma'}^{N_n}\times \fX_{\gamma'}^{N_n}$.
	
	Now consider the Iwasawa decomposition $h = tnk$ for $t\in T(F_v)$ and $n\in N(F_v)$ uniquely determined up to right multiplication by $T(\cO_v)$ and $N(\cO_v)$, respectively. Note that when $t^{-1}\cdot \gamma$ is in $\fh^\perp(\cO_v)$, it is of the form \eqref{eqn: gamma after t action}, and hence, $\fX_{t^{-1}\cdot \gamma}^{N}\simeq \fX_{\gamma}^N$. Let $\delta = (\delta_1,\dots,\delta_n)$ denote the exponents of the diagonal entries of $A$. Then $t^{-1}\cdot \gamma$ is integral valued when $t^{-1}\cdot \gamma=\gamma_{\eta,\eta'}$ where $\eta,\eta'$ are $n$-tuples of nonnegative integers satisfying $\eta+\eta' = \delta$. We have a morphism
	\[
	\fX_\gamma\to \bigsqcup_{\substack{\eta+\eta' = \delta \\ \eta,\eta'\geq 0}}\fX_{\gamma_{\eta,\eta'}}^N
	\]
	whose fibers are discrete ind-schemes of dimension zero, which are torsors for the action of the lattice $\Gamma_a$. In particular, we at once deduce the claimed dimension formula and the description of component groups.
	
	It is a simple exercise to see the $\sigma$-action on $\Gr_H\simeq \Gr_{\GL_n}\times\Gr_{\GL_n}$ restricts to the actions described in the proposition statement.
\end{proof}

\begin{remark}
    An explicit formula for the dimension \eqref{eqn: dim for split case} will be given in \S\ref{sec: dimension formula} using basic invariants associated to $a$.
\end{remark}

We conclude this section with the \emph{transverse} case. Namely, we prove:

\begin{prop}
    \label{prop: transverse is dimension zero}
    Suppose that $a\in \fc^\heartsuit(\cO_v)$ intersects with the full discriminant divisor $\fD_X$ transversely. Then, the ring $R_a(\cO_v')$ is regular. Moreover, the affine Springer fiber $\cM^\star_v(a)$ has dimension zero.
\end{prop}

\begin{proof}
    It suffices to show that every $\ol k$-point of $\Spec R_a(\cO_v')$ is regular. An $\ol k$-point of $R_a(\cO_v')$ is an element $x_0\in \ol k$ such that $P_a(x_0)\equiv 0\mod \varpi_v$. Such an $x_0$ is non-regular if and only if $P_a(x)\in \fm_{x_0}^2$ where $\fm_{x_0} = (\varpi_v,x-x_0)\subset R_a(\cO_v')$ is the maximal ideal corresponding to $x_0$. More explicitly, let us write
    \[
    P_a(x) = \sum_{i\geq 0}P_i(x)\varpi_v^i,\quad P_i(x)\in k[x].
    \]
    Then, $x_0$ is non-regular if $(x-x_0)^2$ divides $P_0(x)$ and $P_1(x_0)=0$.

    Suppose that $x_0$ is a non-regular $\ol k$-point of $\Spec R_a(\cO_v')$. If $x_0 = 0$, then $P_a(x_0)\equiv 0\mod \varpi_c$ implies $a_{2n}\equiv 0 \mod \varpi_v$, and $P_1(x_0) = 0$ then implies that $a_{2n}\equiv 0\mod \varpi_v^2$. Therefore, $\val(a_{2n})\geq 2$, contradicting transversality with $\fD^\br=(a_{2n})\subset \fD_X$.

    If instead $x_0\neq 0$, then we claim that $x_0^2$ is a nonzero, non-regular point of $\Spec R_a(\cO_v')^{\iota}$. Indeed, 
    \[
    R_a(\cO_v')^{\iota} = \cO_v'[y]/(Q_a(x)),\quad\quad Q_a(x) = x^n+a_{2}x^{n-1}+a_{2n-2}x^{n-2}+\cdots +a_{2n}
    \]
    with $Q_a(x^2) = P_a(x)$. In particular, $x_0^2$ is a closed point of $\Spec R_a(\cO_v')^{\iota}$. If we express
    \[
    Q_a(x) = \sum_{i\geq 0}Q_i(x)\varpi_v^i,\quad P_i(x)\in k[x],
    \]
    then one can further check that $Q_1(x_0^2) = P_1(x_0)$ so that $x_0^2$ is non-regular as a point of $\Spec R_a(\cO_v')^{\iota}$ if and only if $x_0$ is a non-regular point of $\Spec R_a(\cO_v')$. Since $\fD^\cent = (\mathrm{Disc}\; Q_a(x^2))$, transversality implies $\val \mathrm{Disc}\; Q_a(x_0^2)\leq 1$. Hence, the regularity comes from the following lemma.

\begin{lemma}
    Let $Q\in\mathcal{O}_v'[x]$ be a monic polynomial with $\mathrm{val}\,\mathrm{Disc}(Q)\le 1$. Then $R_Q=\mathcal{O}_v'[x]/(Q)$ is a regular local ring.
\end{lemma}

\begin{proof}
Since $\mathrm{Disc}(Q)\neq 0$, $R_Q$ is reduced. Let $\widetilde{R_Q}$ be its normalization. Since the trace of an integral element is integral, we have $R_Q\subset\widetilde{R_Q}\subset (R_Q)^\vee$, where $(R_Q)^\vee$ is the
dual lattice to $R_Q$ in $R_Q\otimes_{\cO_v'}F_v'$. Hence $\widetilde{R_Q}$ is also a $\cO_v'$-lattice in $R_Q\otimes_{\cO_v'}F_v'$. Let $M$ be the matrix expressing a basis of $R_Q$ in terms of a basis of $\widetilde{R_Q}$. Then, we can compute the discriminant via the discriminant of lattices
\[
  \mathrm{Disc}(Q)=\mathrm{Disc}(R_Q)=\det(M)^2\,\mathrm{Disc}(\widetilde{R_Q}),
  \qquad \mathrm{Disc}(\widetilde{R_Q})\in\mathcal{O}_v'.
\]
Since $\mathrm{val}\,\mathrm{Disc}(Q)\le 1$, we are forced to have $\val \det M = 0$, so that $\wt{R_Q} = R_Q$ is normal. Therefore, $R_Q$ is regular since it is a reduced, normal, one-dimensional ring.
\end{proof}

    We now complete the proof of Proposition \ref{prop: transverse is dimension zero}. Since $R_a(\cO_v')$ is regular, it is a product of DVR's $\prod_{i=1}^{2n} R_i$, each of which is finite over $\cO_v'$. Let $\varpi_i$ denote a uniformizer in $R_i$ for each $i$. As a consequence, any $R_a(\cO_v')$-lattice $\Lambda$ in $R_a(F_v')$ must be of the form $\Lambda = \langle \ell_1\varpi_1,\dots, \ell_n\varpi_n\rangle$. Hence, $\cM_v^\star(a)\otimes_k\ol k$ embeds in a discrete space and $\dim \cM_v^\star(a) = 0$.
\end{proof}

\subsection{The product formula}

We will now build a bridge between the local models (RASFs) and global moduli spaces (relative Hitchin fibration) by means of a product formula (Proposition \ref{prop: product formula}). The proofs of relevant statements are similar to the case of \cite[\S4.15]{ngo}.

Let $a\in \cA^\heartsuit(k)$, and for any place $v\in |\Sigma|$, denote by $a_v$ the restriction of the map $a\colon \Sigma\to \fc_D$ to the infinitesimal disk $\Sigma_v\subset \Sigma$. It is easy to see  that $a_v\in \fc^\heartsuit(\cO_v)$. 

Let $U_a = a^{-1}(\fc_D^\rs)$ denote the preimage of the regular semisimple locus in $\Sigma$. We denote by $\breve\Sigma = \Sigma\otimes_k\ol{k}$ the base change to the algebraic closure $\ol{k}$. Recall that we denote by $\cM_a^\star = f^{\star,-1}(a)$ the fiber of the relative Hitchin map at $a$. We will first describe a gluing map 
\begin{equation}
	\label{eqn: naive gluing}
\prod_{v\in |\breve\Sigma-U_a|}\cM_v^\star(a_v)\to \cM_a^\star.
\end{equation}
Let $\wt{U}_a'\subset \wt{\Sigma}'_a$ be the preimage of $U_a$. First, note that the cover $\wt{U}'_a\to U_a$ is smooth and \'etale, so that $\omega_{\wt{U}'_a/U_a} = \cO_{\wt{U}_a'}$ and the canonical global section $x\in \Gamma\big(\wt{U}'_a,\cO_{\wt{U}'_a}\big)$ induced by the element $x\in k[\fs]$ is invertible. Note also that $\iota^*x=-x$. Using the description from Corollary \eqref{cor: spectral description using fixed points}, we obtain a canonical $X^\star$-Higgs bundle on $U_a$ with characteristic data $a|_{U_a}$ by considering the structure sheaf $\cO_{\wt{U}_a'}$ together with the identity map
\[
\alpha = \mathrm{id}\colon \cO_{\wt{U}_a'}\to \iota^*\cO_{\wt{U}_a'}\simeq \cO_{\wt{U}_a'}
\]
and the Hermitian form
\[
h^\star\colon \cO_{\wt{U}_a'}\to \sigma^*\cO_{\wt{U}_a'}\simeq \cO_{\wt{U}_a'}
\]
where $h^\unitary = \mathrm{id}$ and $h^\siegel = x$ is multiplication by the section $x$. We call the triple $(\cO_{\wt{U}_a'},h^\star,\alpha)$ the trivial $X^\star$-bundle.

The data of an $R_a(\cO_v')$-lattice in $R_a(F_v')$ gives gluing data along the punctured disk for a trivial rank $2n$ vector bundle on $\Sigma_v'$. The assumptions on self duality and $\iota$ fixed imply that we can glue the Hermitian forms $h^\star$ on $\Sigma_v'$ with $h^\star$ on $U_a'$ compatibly with the isomorphisms $\alpha$. By the formal gluing of Beauville and Laszlo \cite{beauville-laszlo}, we may glue such lattices with the trivial $X^\star$-bundle and therefore obtain the map \eqref{eqn: naive gluing}.

Likewise, we have a map $\prod_{v\in |\breve{\Sigma}-U_a|}\cP_v(J_a)\to \cP_a$, and the map \eqref{eqn: naive gluing} is compatible with the actions of these symmetry groups. We deduce a map
\begin{equation}
	\label{eqn: gluing map}
\prod_{v\in |\breve\Sigma - U_a|}\cM_v^\star(a_v)\wedge^{\prod_{v\in |\breve{\Sigma}-U_a|}\cP_v(J_a)} \cP_a\to \cM_a^\star
\end{equation}

We prove the following close analogue of \cite[Prop 4.15.1]{ngo},\cite[Thm 4.6]{ngo_endoscopie}. 

\begin{prop}
	\label{prop: product formula}
		Let $a\in \cA^\ani(\ol{k})$. Then the quotient of $\prod_{v\in |\breve\Sigma - U_a|}\cM_v^\star(a)\times \cP_a$ by $\prod_{v\in |\breve{\Sigma}-U_a|}\cP_v(J_a)$ is a proper Deligne--Mumford stack. Moreover, the map \eqref{eqn: gluing map} is a universal homeomorphism; namely, it induces a bijection of groupoids of $\ol{k}$-points. 
        
        If $a\in \cA^\ani(K)$ for an extension $K/k$, the bijection is compatible with the respective Galois actions by $\Gal(\ol k/K)$. In particular, we have an equivalence of groupoids
        \[
        \prod_{v\in |\Sigma_{K} - U_a|} [\cM_v^\star(a)/\cP_v(J_a)](K_v)\xrightarrow{\sim} [\cM_a^\star/\cP_a](K),
        \]
        for every $K/k$, with $\Sigma_{K} :=\Sigma\otimes_kK$ and $K_v$ the residue field of $\Sigma_K$ at the place $v$.
\end{prop}
\begin{proof}
	To show that the space $\prod_{v\in |\breve\Sigma - U_a|}\cM_v^\star(a)\wedge^{\prod_{v\in |\breve{\Sigma}-U_a|}\cP_v(J_a)} \cP_a$ is a proper Deligne--Mumford stack, we can argue identically to \cite[Lemme 4.15.1]{ngo} using the fact that the lattice $\Gamma_a$ acts freely on $\cM_v(a)$ with quotient $\cM_v(a)/\Gamma_a$ a projective, finite type scheme; see Proposition \ref{prop: lattice action on RASF}. Alternatively, one can deduce this from the morphism $\cM_v^\star(a)\to \cN_v(a)$ using \cite[Theorem 3.3]{romagny1}.
    
    By considering the quotient by $\cP_a$ on each side of \eqref{eqn: gluing map}, it suffices to show the following map is an equivalence of groupoids
	\begin{equation}
	\label{eqn: product formula on groupoids}
			\beta\colon \prod_{v\in |\breve\Sigma - U_a|} [\cM_v^\star(a)/\cP_v(J_a)](\ol{k})\to [\cM_a^\star/\cP_a](\ol{k}),
	\end{equation}
    compatible with the Galois action. First, we prove $\beta$ is fully faithful. If $\Lambda = \{\Lambda_v\}_v$ and $\Lambda' = \{\Lambda_v'\}_v$ are two $\ol k$-points in the left hand side such that $\beta(\Lambda)$ and $\beta(\Lambda')$ are isomorphic, then there exists $p\in \cP_a(\ol k)$ so that $p\cdot \beta(\lambda) = \beta(\Lambda')$. Since $\beta(\Lambda)$ and $\beta(\Lambda')$ are glued from the trivial $X^\star$-Higgs bundle over $U_a$, the element $p$ restricts to the identity over $U_a$, and therefore, can be obtained by gluing of $p_v\in \cP_v(J_a)$ for $v\in |\breve\Sigma - U_a|$. Then, $p_v\cdot \Lambda_v = \Lambda_v'$ for all places $v$, and we conclude that $\beta$ is fully faithful.

    We now prove that $\beta$ is essentially surjective. Let $m_0$ denote the trivial $X^\star$-Higgs bundle over $U_a$, and let $m\in \cM_a^\star(\ol{k})$. Since $m|_{U_a}$ is regular semisimple, there exists a $J_a$-torsor $p'$ on $U_a$. such that $m|_{U_0} = p'\cdot m_0$. Since $H^0(\ol k,H^\star) = 0$, $p'$ extends to a $J_a$-torsor on $\Sigma$, i.e. $p\in \cP_a(\bar k)$. Then, $p\cdot m$ carries a trivialization identifying $p\cdot m|_{U_a}\simeq m_0$. For $v\in |\breve\Sigma-U_a|$, the restriction of $p\cdot m$ to the infinitesimal disk $\Sigma_v$ comes equipped with gluing data over the punctured disk $\Sigma_v^\bullet$, which by the moduli description of $\cM_v^\star(a)$ is a point $m_v\in \cM_v^\star(a)(\ol k)$. The element $p$ gives an explicit isomorphism between $\beta\big((m_v)\big)$ and $m$, proving surjectivity.

    Finally, we note that the gluing map \eqref{eqn: naive gluing} intertwines the Galois actions. Indeed, the Frobenius $\sigma$ acts on $\cM^\star_a$ through $\Aut(\Sigma'/\Sigma)$, which restricts to actions of the Galois group $\mathrm{Gal}(F_v'/F_v)$ for every $v\in |\breve\Sigma|$.
\end{proof}

\begin{remark}
\label{rmk: product formula for Ja'}
    Let $a\in \cA^\ani(\ol k)$ be as above. Suppose that $J_a'$ is a smooth group scheme equipped with a map $J_a'\to J_a$, which is an isomorphism on $U_a\subset \breve\Sigma$. Then $\cP_v(J_a')$ acts on $\cM_v^\star(a)$ through $\cP_v(J_a)$, and moreover, we may form a Picard stack $\cP_a'$ of $J_a'$ torsors on $\Sigma$ which acts on $\cM^\star_a$. We claim that Proposition \ref{prop: product formula} holds replacing $\cP_a$ with $\cP_a'$ and $\cP_v(J_a)$ with $\cP_v(J_a')$. Indeed, by a Beauville-Laszlo gluing argument, the morphism 
    \[
    \cP_a'\to \cP_a
    \]
    is a surjective map with kernel a torsor under the group $\prod_{v\in |\breve\Sigma-U_a|} J_{a,v}(\cO_v)/J_{a,v}'(\cO_v)$. In particular, 
    \[
\cP_a\simeq \cP_a'\wedge^{\prod_{v\in |\breve\Sigma-U_a|}\cP_v(J_a')}\prod_{v\in |\breve\Sigma-U_a|}\cP_v(J_a).
    \]
    Therefore, the map $J_a'\to J_a$ induces an isomorphism
    \begin{align*}
\prod_{v\in |\breve\Sigma - U_a|}\cM_v^\star(a)\wedge^{\prod_{v\in |\breve{\Sigma}-U_a|}\cP_v(J_a')} \cP_a' &\simeq 
\prod_{v\in |\breve\Sigma - U_a|}\cM_v^\star(a)\wedge^{\prod_{v\in |\breve{\Sigma}-U_a|}\cP_v(J_a)} \left( \prod_{v\in |\breve{\Sigma}-U_a|}\cP_v(J_a) \wedge^{\prod_{v\in |\breve{\Sigma}-U_a|}\cP_v(J_a')} \cP_a'\right)\\
&\simeq  \prod_{v\in |\breve\Sigma - U_a|}\cM_v^\star(a)\wedge^{\prod_{v\in |\breve{\Sigma}-U_a|}\cP_v(J_a)} \cP_a.
    \end{align*}
    Therefore, Proposition \ref{prop: product formula} gives a morphism
    \[
\prod_{v\in |\breve\Sigma - U_a|}\cM_v^\star(a_v)\wedge^{\prod_{v\in |\breve{\Sigma}-U_a|}\cP_v(J_a')} \cP_a'\xrightarrow{\sim} \cM_a^\star ,
    \]
    which is a universal homeomorphism over $k$.
\end{remark}

\subsection{Local Constancy of RASFs}

In this section, we will prove the following proposition, mirroring \cite[Proposition 3.5.1]{ngo} in our relative context. The proof here is simpler due to our use of spectral covers.

\begin{prop}
	\label{prop: local constancy}
	Suppose $a\in \fc^\heartsuit(\cO_v)$ satisfies $\Ob_v^\star(a) = 0$, and let $\cM_{v}^\star(a)$ (for $\star = \siegel$ or $\unitary$) denote the associated affine Springer fibers and  $\cP_v(J_a)$ the group of symmetries. There exists an integer $N>0$ such that for all finite extensions $K$ of $k$ and for all $a'\in \fc^\heartsuit(\cO_v\otimes_k K)$ satisfying
	\[
	a\equiv a'\mod \varpi_v^N,
	\]
	the affine Springer fiber $\cM_v^\star(a')$ with the action of $\cP_v(J_{a'})$ is isomorphic to $\cM_v^\star(a)\otimes_k K$ with the action of $\cP_v(J_a)\otimes_k K$.
\end{prop}

We introduce some preliminaries before the proof below. Let $\wt{\Sigma}_{a,v}' = \Spec R_a(\cO'_v)$, where $R_a(\cO'_v)$ is as in \S\ref{sec: lattice description of RASF}. This cover $\wt{\Sigma}_{a,v}'$ fits into a Cartesian diagram
\[
\xymatrix{
\wt{\Sigma}_{a,v}'\ar[r]\ar[d] & \fs\ar[d] \\
\Sigma'_v\ar[r] & \fc
}
\]
Recall from \S\ref{sec: lattice description of RASF} that $\wt{\Sigma}_{a,v}'$ carries two involutions:  $\sigma$ acting by the nontrivial involution in $\Aut(\Sigma'/\Sigma)$ and $\tau$ the involution of the spectral cover induced by $x\mapsto -x$ (see \S\ref{sec: spectral curves} for a discussion).

\begin{lem}
    \label{lem: spectral cover mod N}
	For any $a\in \fc^\heartsuit(\cO_v)$ and integer $N_1$, there exists an integer $N>N_1$ such that for all finite extensions $K$ of $k$ and all $a'\in \fc(\cO_v\otimes_kK)$ satisfying
	\[
	a\equiv a'\mod \varpi_v^N,
	\]
	there is a $\sigma$- and $\tau$-equivariant isomorphism $\wt{\Sigma}_{a',v}'\simeq \wt{\Sigma}_{a,v}'\otimes_kK$ over $\Sigma_v'\otimes_kK$ which lifts the obvious one mod $\varpi_v^{N_1}$.
\end{lem}
\begin{proof}
    This is a consequence of \cite[Lemma 3.12]{artin} and is very close to the statement recorded in \cite[Lemma 3.5.2]{ngo}.
\end{proof}

\begin{proof}[Proof of Proposition \ref{prop: local constancy}]
    Fix $a\in \fc^\heartsuit(\cO_v)$. Let $N$ be as in Lemma \ref{lem: spectral cover mod N} for $N_1 = 1$. Without loss of generality, we assume that $k=K$, and suppose that $a'\in \fc^\heartsuit(\cO_v)$ is such that $a'\equiv a\mod \varpi_v^N$. Then, by Lemma \ref{lem: spectral cover mod N}, $R_a(\cO'_v)\simeq R_{a'}(\cO'_v)$ compatibly with the involutions $\sigma$ and $\iota$. Compatibility with $\sigma$ implies in particular that the isomorphism respects the Hermitian forms $h^\star$ on each side, and so from the lattice definition (see \S\ref{sec: lattice description of RASF}), we obtain an isomorphism $\cM_v(a)\simeq \cM_v(a')$
\end{proof}

\subsection{Dimension formula}
\label{sec: dimension formula}

In \cite{bez}, it was shown that regular locus in (ordinary) affine Springer fibers is dense, and hence one can deduce a formula for the dimension of an affine Springer fiber from the dimension of its group of symmetries. This method fails for RASF's as regularity is no longer a dense condition. It also fails in the multiplicative setting of \cite{jingren,griffin} for a similar reason:  multiplicative affine Springer fibers have rich component groups, related to Mirkovic--Vilonen cycles, and in particular, have no dense regular locus. In work of Jingren Chi \cite{jingren}, a dimension formula and equidimensionality property in the multiplicative case is proved by passing to the global setting. We follow an analogous argument, with some simplifications. In particular, for any $a_v\in \fc^\heartsuit(\cO_v)$, we construct a curve in the Hitchin base which interpolates between $a_v$ and a split element, and use upper semi-continuity of fiber dimensions for proper maps to deduce the dimension bound. We also collect as a corollary that the relative Hitchin morphism is flat over the anisotropic locus, from which equidimensionality follows.

Let $\Sigma$ be a smooth projective curve of genus $g\geq 2$ and $v\in |\Sigma|$. Recall from Section \ref{sec: lattice description of RASF} $\fc^\heartsuit(\cO_v) = \fc(\cO_v)\cap \fc^{\mathrm{rs}}(F_v)$. As we will be exchanging between the local and global settings in this section, we will denote by $a_v$ a fixed element in the local base $\fc^\heartsuit(\cO_v)$ and by $a$ as element in the global Hitchin base $\cA = \cA_{D}$.

Associated to any $a_v\in \fc^\heartsuit(\cO_v)$, we have two local invariants mimicking the invariants that control the dimension of affine Springer fibers in \cite{bez}. First, we put
\begin{equation}\label{eqn: dim cent}
    d_v^\cent(a_v):= \deg(a_v^*\fD^\cent)
\end{equation}
to be the multiplicity of intersection of $a_v$ at the centralizer divisor. We also have the invariant $c_v(a_v)$ controlling how split $a_v$ is (see Definition \ref{def: split}):  define
\[
c_v(a_v) = n - \mathrm{rank}\; \Gamma_a,
\]
where $\Gamma_a = \mathbb{X}_*(J_{a_v,\ol F_v})$ is the lattice of $\breve F_v$-rational cocharacters of the torus $J_{a_v,\breve F_v}$ over $\breve F_v$ acting on the RASF as defined in \S\ref{sec: symmetries of RASFs}. Note in particular that $c_v(a_v) = 0$ if and only if $a_v$ is split.

We begin by proving the following dimension formula.

\begin{thm}
	\label{prop: dim of RASF}
	Let $a_v\in \fc^\heartsuit(\cO_v)$. The RASF $\cM_v^\star(a_v)$ has dimension equal to that of $\cP_v(J_a)$. Explicitly, we have
	\[
	\dim \cM_v(a_v) = \dim \cP_v(J_{a_v}) = \frac{d_v^\cent(a)-c_v(a)}{2}
	\]
\end{thm}
\begin{remark}
    The final formula for $\dim \cP_v(J_{a_v})$ can be considered as a consequence of Bezrukavnikov's formula computing the dimension of symmetry groups for affine Springer fibers \cite{bez}. Namely, the spectral description of regular centralizers (see Theorem \ref{thm: regular centralizers for rational forms}) implies that the regular centralizer group scheme is controlled by a type $A$ root system in the sense of \cite[Theorem 7.7]{knop1996automorphisms} (also see \cite[Proposition 2.4.7]{ngo}), whose associated discriminant divisor is $\fD^\cent$. However, in the proof that follows, we give a simpler argument using the spectral description of $\cP_v(J_a)$.
\end{remark}
\begin{proof}
Our proof follows \cite[\S5.2]{jingren}, with some simplifications in our case. As dimension is a geometric property, we will base change to an algebraic closure $\ol k$ throughout the rest of this section.

First, we note that, there exists a point $m\in \cM_v^\reg(a_v)(\ol k)$ given by the image of a Kostant-Rallis section. The stabilizer of $m$ in $\cP_v(J_{a_v})$ is trivial since $m$ is regular, so we have $\dim \cM_v(a_v)\geq \dim \cP_v(J_{a_v})$. The rest of this proof will concern the opposite inequality.

Recall the algebra of functions $R = R_a(\cO_v)$ on $\Sigma_{a_v,v} = \Sigma_v\times_\fc \fs$ considered in \S\ref{sec: lattice description of RASF}, where we recall that $\fs$ is the spectral cover appearing in \S\ref{sec: invariant theory}. Then $R_a(F_v)=R\otimes_{\cO_v}F_v$ splits as a product of extensions $\prod_i E_i$ of $F_v$ of degrees $e_i$. Put $e=\mathrm{LCM}(e_i)$ to be the least common multiple of the degrees $e_i$; in particular, we have $e=1$ if and only if $a_v$ is split. Consider the family $a_t$ given by viewing $a$ as $n$-tuple of power series in $\varpi_v$ and substituting $\varpi_v\mapsto t\varpi_v+(1-t)\varpi_v^e$ in $a_v$, i.e.
\[
a_t = a_v\big(t\varpi_v+(1-t)\varpi_v^e\big)\in \fc^\heartsuit(\cO_v).
\]
We note that $\cM_{v}^\star(a_t)\simeq \cM_v^\star(a_1) = \cM_v^\star(a_v)$ for $t\neq 0$ and that $a_0 = a_v(\varpi_v^e)$ is split since $\varpi_v\mapsto \varpi_v^e$ splits $R_a(F_v)$. Let $N>0$ be as in Proposition \ref{prop: local constancy} for both $a_v$ and $a_0$, so that $\cM_v^\star(a_{v}')\simeq \cM_v^\star(a_v)$ whenever $a_v'\equiv a_v\mod \varpi_v^N$ and likewise for $a_0$. Fix a smooth, projective curve $\Sigma$ of genus $g\geq 2$ and an \'etale double cover $\vartheta\colon \Sigma'\to \Sigma$ so that at a distinguished place $v$, the extension of $\Sigma_v'$ by $\Sigma_v$ is the extension $\cO_v\subset \cO_v'$. Choose a line bundle $D$ on $\Sigma$ such that
\[
\deg(D)>\max\{N\cdot n, 2g\}.
\]
where $n =\dim\fa$ is the rank of the symmetric variety $X^\star$. Let $\cM^{\star,\ani} = \cM^{\star,\ani}_{\Sigma,D}$ be the relative Hitchin moduli space, which is smooth by Proposition \ref{prop: smoothness and properness of relative hitchin fibration}. For every $u\neq v$ in $|\Sigma|$, we have an affine linear map
\begin{equation}
    \label{eqn: local evaluation map for v and v'}
\cA = \cA_D\to \fc(\cO_v/\varpi_v^N)\times \fc(\cO_{u}/\varpi_{u}^2),
\end{equation}
which is surjective by Riemann-Roch since $\deg(D)\geq\max\{N\cdot n, 2g\}-1\geq \frac{1}{2}(2g+N+1)$. Let $C$ denote the curve in $H^0(\Sigma,\fc(\cO_v/\varpi_v^N))$ given by the family $a_t$. Define the closed subset $L_C\subset \cA$ to be the preimage of $C$, and let $Z$ denote the set of points $a$ such that $a|_{\Sigma - \{v\}}$ meets the discriminant divisor $\fD_X$ transversely. Put $Z_C = Z\cap L_C$. The set $Z$ is open in $\cA^\heartsuit$ and intersects each fiber of the map $L_C\to C$ non-trivially by surjectivity of \eqref{eqn: local evaluation map for v and v'}. Moreover, since by Lemma \ref{lem: codim of anisotropic}, the codimension of $\cA-\cA^\ani$ in $\cA$ is greater than $\deg(D)$ and $Z_C$ is open in $L_C$, which is of codimension $N\cdot n$ in $\cA$, it follows that $Z_C^\ani =Z_C\cap \cA^\ani$ is also nonempty and open in $L_C$. 

Composing \eqref{eqn: local evaluation map for v and v'} with projection onto the first factor gives
\begin{equation}
	\label{eqn: local evaluation map at v}
\cA = \cA_{D}\to \fc(\cO_v/\varpi_v^N),
\end{equation}
which is also a linear map between vector spaces. Choose a section $s$ of the map \eqref{eqn: local evaluation map at v} such that $s(a_0)\in Z_C^\ani$. As $Z_C^\ani\subset L_C$ is open, this implies that $s^{-1}(Z_C^\ani)$ is open in $C$. Therefore, we may choose a fixed $t\neq 0$ such that $s(a_t)\in Z_C^\ani$. 

As $f^{\star,\ani}$ is proper (Proposition \ref{prop: smoothness and properness of relative hitchin fibration}), by upper semi-continuity of fiber dimensions we have
\[
\dim \cM^\star_{s(a_t)} \leq \dim \cM^\star_{s(a_0)}
\]
By Proposition \ref{prop: dim formula split case}, we know that $\dim \cM^\star_{s(a_0)} = \dim  \cP_{s(a_0)}$, which by smoothness of the group scheme $\cP$ is equal to $\dim \cP_{s(a_t)}$. Using the product formula \ref{prop: product formula}, we deduce that
\begin{align*}
\dim\cM^\star_{s(a_t)} -\dim \cP_{s(a_t)} &= \sum_{v'\in |\breve\Sigma - (U_a\cup \{v\})|} \big(\dim \cM^\star_{v'}(a) - \dim \cP_{v'}(J_a)\big) + \big(\dim \cM^\star_{v}(a) - \dim \cP_{v}(J_a)\big) \\
& = \dim \cM^\star_{v}(a) - \dim \cP_{v}(J_a)
\end{align*}
where the second equality follows from Proposition \ref{prop: transverse is dimension zero}. Hence, we conclude that
\[
\dim \cM^\star_{v}(a) - \dim \cP_{v}(J_a) = \dim\cM^\star_{s(a_t)} -\dim \cP_{s(a_t)}\leq \dim\cM^\star_{s(a_0)} -\dim \cP_{s(a_0)} = 0
\]
so that $\dim \cM^\star_v(a) = \dim \cP_v(J_a)$, as desired. 

We now compute the dimension of $\cP_v(J_{a})$. Using notation from the proof of Proposition \ref{prop: lattice action on RASF}, we denote $R = R_a(\cO_v')\otimes_k\ol k$. From the short exact sequence \eqref{eqn: decomposition of P(Ja) into lattice and finite parts}, whose cokernel is a lattice and hence has dimension zero, we have
\[
\dim \cP_v(J_a) = \dim_{\ol k}\left( (\wt{R^\iota})^\times/ ({R^\iota})^\times\right) = \dim_{\ol k}\left( \wt{R^\iota}/ {R^\iota}\right) = \mathrm{length}_{\cO_v'}\left( \wt{R^\iota}/ {R^\iota}\right).
\]
We consider $R^\iota\subset \wt{R^\iota}$ as $\cO_v'$-lattices in $R^\iota\otimes_{\cO_v'}F_v'$. In particular, for an $\cO_v'$-lattice $\Lambda$ in $R^\iota\otimes_{\cO_v'}F_v'$ we may compute its discriminant by choosing an $\cO_v'$ basis $\ul e = \{e_1,\dots, e_n\}$ and putting
\[
\mathrm{disc}(\Lambda) = \det \mathrm{Tr}(\Lambda,\ul e), \quad\quad \mathrm{Tr}(\Lambda,\ul e) := \left( \mathrm{Tr}_{R^\iota\otimes_{\cO_v'}F_v'/F_v'}(e_ie_j) \right)_{1\leq i,j\leq n}.
\]
It is a standard exercise that $\mathrm{disc}(\Lambda)$ is well defined up to unit in $(\cO_v')^\times$. Choose $\cO_v'$ bases $\ul e = \{e_1,\dots,e_n\}$ for $R^\iota$ and $\tilde{\ul{e}} = \{\tilde e_1,\dots,\tilde e_n\}$ for $\wt R^\iota$ and let $M$ be the matrix with entries in $\cO_v'$ so that for all $i$,
\[
\tilde e_i = \sum_jM_{ij}e_j.
\]
Then, it is easy to check that $\val(\det M) = \mathrm{length}_{\cO_v'}(\wt R^\iota/R^\iota)$ and 
\begin{equation}
\label{eqn: relating bases of R iota and normalization}
    M^t\cdot \mathrm{Tr}(\wt R^\iota,\tilde{\ul{e}})\cdot M = \mathrm{Tr}(R^\iota,\ul e).
\end{equation} 
In particular, taking determinants and valuations of \eqref{eqn: relating bases of R iota and normalization}, we get
\[
\mathrm{length}_{\cO_v'}(\wt R^\iota/R^\iota) = \frac{1}{2}\left(\val \mathrm{disc}(R^\iota) - \val \mathrm{disc}(\wt R^\iota)\right).
\]
Finally, we compute these discriminants:  With respect to the basis $1,x^2,x^4,\dots,x^{2n-2}$, one can compute $\mathrm{disc}(R^\iota) = \mathrm{disc}(Q_a)$ is the discriminant of the polynomial $Q_a$, and so
\[
\val \mathrm{disc}(R^\iota) = \val \mathrm{disc}(Q_a) = d^\cent_v(a).
\]
where the last equality follows from the description of $\fD^\cent$ given in \S\ref{sec: divisors}.

On the other hand, $R^\iota\otimes_{\cO_v'}F_v'$ splits as a product $\prod_{i=1}^r K_i'$ where each $K_i'$ is a finite extension of $F_v'$. We likewise have a splitting of the $\sigma$ fixed points: Put $K_i = (K_i')^\sigma$ $k_i=[K_i':F_i'] = [K_i:F_i]$. The cocharacter lattice of the torus $J_{a,\breve F_v}$ is $\prod_{i=1}^r\mathrm{Res}_{K_i/F_v}(\bG_m)$
and so $c_v(a) = n-r$. But on the other hand, the discriminant of the DVR $\cO_{K_i'}$ in $K_i'$ is
\[
\val \mathrm{disc}(\wt R^\iota)=\sum_{i=1}^r (e_i-1) = n-r.
\]
Hence, $c_v(a) = \val \mathrm{disc}(\wt R^\iota)$, and the formula for dimension now follows.
\end{proof}

\begin{cor}
	\label{cor: ani is flat}
	The relative Hitchin fibration restricted to anisotropic locus $f^{\star,\ani}\colon \cM^{\star,\ani}\to \cA^\ani$ is flat.
\end{cor}
\begin{proof}
	Since $\cM^{\star,\ani}$ is smooth by Proposition \ref{prop: smoothness and properness of relative hitchin fibration}, by Miracle Flatness it suffices to show the fibers of $f^{\star,\ani}$ are equidimensional of the same dimension.  Since by Proposition \ref{prop: smoothness and properness of relative hitchin fibration} $f^{\star,\ani}$ is proper, the fiber dimension is upper semicontinuous on the source. Hence, it suffices to show the fiber dimensions are equal over any $a\in \cA^\ani$. For any $a\in \cA^\ani$, we have by product formula (Proposition \ref{prop: product formula}) and the dimension formula for RASFs (Theorem \ref{prop: dim of RASF}) that
	\[
\dim [\cM_a^\star/\cP_a] = \sum_{v\in |\breve\Sigma - U_a|} \dim [\cM_v^\star(a)/\cP_v(J_a)] = 0.
	\]
	Hence, $\dim \cM_a^\star = \dim\cP_a$ for all $a\in \cA^\ani$. Moreover, by uppersemicontiuity of local dimension of fibers, this constrains all components in $\cM_a^\star$ to also be of dimension $\dim \cP_a$. Finally, since $\cP$ is smooth over $\cA^\ani$, we conclude that $\dim \cM_a^\star$ is constant on $\cA^\ani$.
\end{proof}

\begin{prop}
\label{prop: equidimensionality}
	Let $a_v\in \fc^\heartsuit(\cO_v)$. The RASF $\cM_v^\star(a_v)$ is equidimensional.
\end{prop}
\begin{proof}
Let $N>0$ be as in Proposition \ref{prop: local constancy}. Choose a line bundle $D$ on $\Sigma$ such that 
\[
\deg(D)>\max\{N\cdot n, 2g\}.
\]
Let $\cM^{\star,\ani} = \cM^{\star,\ani}_{\Sigma,D}$ be the relative Hitchin moduli space, which is smooth by Proposition \ref{prop: smoothness and properness of relative hitchin fibration}.

Let $L\subset \cA$ be the subset of points $a$ such that
\[
a\equiv a_v\mod \varpi_v^N.
\]
In particular, $L$ is a closed subset of $\cA^\heartsuit$ since $a_v\in \fc^\heartsuit(\cO_v)\subset \fc^{\mathrm{rs}}(F_v)$. Alternatively, we note that $L$ is the preimage of $a_v$ under the local evaluation map
\[
\cA\to \fc(\cO_v/\varpi_v^N),
\]
which is surjective using Riemann--Roch by the assumption on $\deg(D)$. Therefore, the codimension of $L$ in $\cA$ is given by
\[
\mathrm{codim}_{\cA}(L) = N\cdot n.
\]
The complement of the anisotropic locus has codimension at least $\deg(D) >N\cdot n$ by Lemma \ref{lem: codim of anisotropic}. Hence, $L$ intersects non-trivially with the open set $\cA^\ani$, and the set $L^\ani = L\cap \cA^\ani$ is locally closed in $\cA$. We claim that $L^\ani$ contains a point $a$ such that for any $v'\in |\Sigma|\setminus \{v\}$, the image of $a$ under the local evaluation map
\begin{equation}
	\label{eqn: local eval map for x,y}
	\cA = \cA_{\Sigma,D}\to \fc(\cO_{v'}/\varpi_{v'}^2)
\end{equation}
is nonzero. That is, $a|_{\Sigma\setminus \{v\}}$ meets the discriminant divisor transversely away from $v$ and we have
\[
a\equiv a_v\mod \varpi_v^N.
\]
In this case, by the product formula (Proposition \ref{prop: product formula}), we have
	\[
	[\cM_{a}^\star/\cP_{a}]\underset{\text{Homeo}}{\simeq} [\cM_v^\star(a)/\cP_v(J_a)]\times \prod_{v'\in |\breve\Sigma -(U_a\cup\{v\})|} [\cM_{v'}^\star(a)/\cP_{v'}(J_a)] ,
	\]
where, for any $v'\in |\breve \Sigma - (U_a\cup\{v\})|$, we have $[\cM_{v'}^\star(a)/\cP_{v'}(J_a)] $ is equidimensional of dimension zero by Proposition \ref{prop: transverse is dimension zero}. Since $f^{\star,\ani}$ is flat (Corollary \ref{cor: ani is flat}), $[\cM_a^\star/\cP_a]$ is equidimensional (of dimension zero), and so $[\cM_v^\star(a)/\cP_v(J_a)]$ must also be equidimensional. The result then follows. 
	
	To prove the existence of such a point $a$, we argue similarly to \cite[Proposition 4.7.1]{ngo} and \cite[Lemma 5.2.6]{jingren}. By our assumptions on the degree of $D$, the evaluation map
	\begin{equation}
		\label{eqn: surj of local evaluation}
		L\to \fc(\cO_{v'}/\varpi_{v'}^2)
	\end{equation}
	is surjective. Let $\fD^\cent = \fD^{\mathrm{sm}}\cup\fD^{\mathrm{sing}}$ where $\fD^\mathrm{sm}$ is the smooth locus of $\fD^\cent$ and $\fD^{\mathrm{sing}}$ its complement (the intersection of root hyperplanes). Then, define
	\[
	Z_1 = \{ (c,a)\in \fD_D^{\mathrm{sm}}\times \cA \mid a\text{ meets }\fD_D\text{ at $c$ with multiplicity $\geq2$} \}
	\]
	\[
	Z_2 = \{ (c,a)\in \fD_D^{\mathrm{sing}}\times L \mid a\text{ meets at $\fD_D$ at $c$} \}
	\]
	For each of the $Z_j$, the surjectivity of \eqref{eqn: surj of local evaluation} implies that
	\[
	\dim(Z_j)\leq \dim(L) - 1
	\]
	Hence, the complement of the image of the projection $Z_1\cup Z_2\to L$ is a nonempty open set in $L$.
\end{proof}

Finally, we collect as a corollary a relationship with dimensions of endoscopic RASF's. Recall from Lemma \ref{lem: pullback of divisors along transfer map} that there exists a divisor $\fR_\fe\subset \fc_\fe$ so that $\nu_\fe^*\fD^\cent = \fD_\fe^\cent+2\fR_\fe$. Define $r_v(a_\fe) = \val(a_\fe^*\fR_\fe)$ for $a_\fe\in \fc^\heartsuit_\fe(\cO_v)$. We will denote by $\delta(a)$ the dimension of the RASF $\cM_v^\star(a)$, which is independent of $\star\in \{\siegel,\unitary\}$.

\begin{lem}
    \label{lem: delta vs endoscopic delta}
    For $a_\fe\in \fc_\fe^\heartsuit(\cO_v)$ and $a\in \fc^\heartsuit(\cO_v)$ such that $\nu_\fe(a_\fe) = a$, we have
    \[
    \delta_v(a) = \delta_{\fe,v}(a_\fe)+r_v(a_\fe)
    \]
\end{lem}
\begin{proof}
    Note that $c_v(a)$ and $c_{\fe,v}(a_\fe)$ depend only on the generic fiber of $J_a$, respectively $J_{a_\fe}$. By Proposition \ref{prop: centralizers under transfer}, $J_a$ and $J_{a_\fe}$ have isomorphic generic fibers, and so $c_v(a)=c_{\fe,v}(a_\fe)$. On the other hand, by Lemma \ref{lem: pullback of divisors along transfer map}, $d_v^\cent(a) = d_{\fe,v}^\cent(a_\fe)+2r_v(a_\fe)$. Hence, by the dimension formula, Proposition \ref{prop: dim of RASF}, we have
    \[
    \delta_v(a) = \frac{d_v^\cent(a)-c_v(a)}{2} = \frac{d_{\fe,v}^\cent(a_\fe)-c_{\fe,v}(a_\fe)}{2}+r_v(a_\fe) = \delta_{\fe,v}(a_\fe)+r_v(a_\fe) \qedhere
    \]
\end{proof}

Lemma \ref{lem: delta vs endoscopic delta} implies a global result as well. We fix global data $\vartheta\colon \Sigma'\to \Sigma$ and $D$ on $\Sigma$ as in \S\ref{sec: relative Hitchin spaces}. The divisor $\fR_\fe$ is homogeneous for the $\bG_m$ action on $\fc_\fe$, and so we can form the divisor $(\fR_\fe)_D$ in $\fc_{\fe,D}$. Put $r_v(D) = \deg a_\fe^*(\fR_\fe)_D$. When $X^\fe = X^\siegel_r\times X^\unitary_s$ for $r+s = n$, we get $r_v(D) = 2rs\deg(D)$. 

\begin{cor}
    \label{cor: delta vs endoscopic delta, global}
    Fix a divisor $D$ on $\Sigma$, and let $\fe = (\kappa,\rho)$ be an endoscopic datum. For any $a_\fe\in \cA_\fe^\ani(\ol k)$ and $a\in\cA^\ani(\ol k)$ such that $\nu_\fe(a_\fe) = a$, we have
    \[
    \delta(a) = \delta_\fe(a_\fe)+r_v(D)
    \]
\end{cor}
\begin{proof}
    By the product formula Proposition \ref{prop: product formula}, it suffices to note that
    \[
r_v(D) = \sum_{v\in |\breve\Sigma - U_{a_\fe}|} \deg(v)r_v(a_\fe).\qedhere
    \]
\end{proof}

\subsection{Density of $\cA^\dagger$ in $\delta$-strata} We now extract some global corollaries from these local results. Recall the $\delta$-stratification of $\cA$ in Section \ref{sec: delta stratification}, as well as the $\dagger$-locus $\cA^\dagger$ in Definition \ref{def: dagger locus}.
We prove the following density result.

\begin{prop}
    \label{prop: density of dagger locus in delta strata}
    Fix a positive integer $\delta$, and suppose that $\deg(D)\geq g+\frac{3}{2}\delta$. Then the locus $\cA^\dagger_\delta := \cA^\dagger\cap \cA_\delta$ is open and dense in every component of $\cA_\delta$.
\end{prop}

The proof of this proposition involves some results in deformation theory. We begin with the following claim.

\begin{lem}
\label{lem: local evaluation map for delta strata}
    Fix distinct points $x_0,x_1,\dots, x_r,y\in \breve\Sigma(\ol{k})$. If $\deg(D)> g+\frac{3}{2}r$, then the product of local evaluation maps
    \begin{equation}
    \label{eqn: local evaluation map is surjection}
\cA\to \fc(\cO_{y}/\varpi_{y})\oplus \fc(\cO_{x_0}/\varpi_{x_0})\oplus\prod_{i=1}^r \fc(\cO_{x_i}/\varpi_{x_i}^3)
    \end{equation}
    is surjective.
\end{lem}
\begin{proof}
    This morphism has cokernel given by
    \[
    H^1\!\left(\Sigma,\fc_D\!\!\left( -y-x_0 - 3\sum_{i=1}^r x_i \right)\right)
    \]
    Since $\fc_D = \bigoplus_i D^{\otimes 2i}$, this cohomology group vanishes by Riemann--Roch when
    \[
    2\deg(D) - 3r - 2> 2g-2
    \]
    which can be rearranged to read
    \[
    \deg(D)> g+\frac{3}{2}r.
    \]
    Starting with this assumption, we see that \eqref{eqn: local evaluation map is surjection} is a surjection. 
\end{proof}

We will use a local description of $\delta$ invariants.

\begin{definition}
    For $a\in \fc^\heartsuit(\cO_v)$, we denote by $\delta_v(a) = \dim\cM^\star_v(a)$ the dimension of the corresponding RASF at $a$.
\end{definition}

Note that Proposition \ref{prop: dim of RASF} gives a formula computing $\delta_v(a)$. Moreover, by the product formula (Proposition \ref{prop: product formula}), for any $a\in \cA^\ani(\ol k)$ we have an equality
\[
\delta(a) = \sum_{v\in |\breve\Sigma|} \delta_v(a),
\]
where only finitely many terms in the right hand side are nonzero.

Let $\cA_{(1^\delta)}\subset \cA_\delta$ denote the locus of points $a$ in $\cA_\delta$ such that $\delta_{x}(a)\leq 1$ for all $x\in \breve\Sigma(\ol{k})$. This subset was used in \cite{ngo} alongside other locally closed sets subsets $\cA_{\lambda}$ for partitions $\lambda\vdash \delta$ to give a stratification of $\cA_\delta$. In particular, Ng\^o shows, in the proof of \cite[Proposition 5.7.2]{ngo} that:

\begin{lem}
    \label{lem: density of 1,1,... locus in delta strata}
    The subset $\cA_{(1^\delta)}$ is open and dense in $\cA_\delta$.
\end{lem}

We will need the following refinement of $\cA_{(1^\delta)}$. Let $\cA_{(1^\delta)}^{\min}$ be the set of points $a$ in $\cA_{(1^\delta)}$ such that, for all $x$ such that $\delta_x(a) = 1$, we also have $d_x^\cent(a) = 2$. Note that by the dimension formula Proposition \ref{prop: dim of RASF}, we must have $d_x^\cent(a)\geq 2\delta_x(a)$ for all $x$, so this subset should be thought of as points of $\cA_\delta$ which meet the divisor $\fD^\cent_D$ ``as transversely as possible.''

\begin{lem}
\label{lem: d=2 and delta>1 detectable mod 3}
    Treat $\fc(\cO_v/\varpi_v^3)$ as an affine space. There exists a closed subvariety $\mathfrak{Z}\subset \fc(\cO_v/\varpi_v^3)$ and an open, dense subset $\mathfrak{U}\subset \mathfrak{Z}$ such that the condition $\delta_{v}(a)>0$ is equivalent to the image of $a$ under the local evaluation map 
    \[
    \cA\to \fc(\cO_v/\varpi_v^3)
    \]
    landing in $\mathfrak{Z}$ while the condition that \emph{both} $\delta_v(a) = 1$ and $d_v^\cent(a) = 2$ is equivalent to the image of $a$ landing in $\mathfrak{U}$.
\end{lem}
\begin{proof}
    Since our $\delta$-invariant coincides with the $\delta$-invariant for the local Jacobian of $\ol{\Sigma}_{a,v}$ (see Lemma \ref{lem: delta invariant is same as GLn}), the condition that $\delta_{v}(a)>0$ is equivalent to the condition that the germ of the associated spectral curve is singular. This can be described by vanishing of partial derivatives with respect to $x$ and $\varpi_v$ and is hence a closed condition.

    For the second claim, note that $d_v^\cent(a) = \mathrm{val}_{v}(a^*\fD^\cent_D)$ is an upper semicontinuous function, and so the locus with minimal valuation is open in the space where $\delta_v(a)>0$. In particular, by the dimension formula (Theorem \ref{prop: dim of RASF}), $d^\cent_v(a) = 2$ is the minimal valuation if we impose $\delta_v(a)>0$. As having valuation 2 is detectable mod $\varpi_v^3$, the lemma follows.
\end{proof}

Fix $\delta$ distinct points $\ul{x} = (x_1,\dots, x_\delta)$ on $\breve\Sigma$. For $\ul{\xi} = (\xi_1,\dots, \xi_\delta)$ with $\xi_i\in H^0(\Sigma,\cO_{x_i})$ satisfying both $\delta_{x_i}(a)=1$ and $d_{x_i}^\cent(a) = 2$, we denote by $\cA_{\ul{\xi}}$ the preimage of $\ul{\xi}$ under the local evaluation map
    \[
\cA\to \prod_{i=1}^\delta \fc(\cO_{x_i}/\varpi_{x_i}^3).
    \]
    Let $\cA_{\geq(1^\delta)}$ denote the subvariety of $a\in \cA$ for which there exist at least $\delta$ distinct points $x_1,\dots, x_\delta$ of $\breve\Sigma$ such that $\delta_{x_i}(a)\geq 1$. Then, by Lemma \ref{lem: d=2 and delta>1 detectable mod 3}, we have
    \begin{equation}
    \label{eqn: local decomposition of minimal delta cell}
        \cA_{\geq(1^\delta)} = \bigcup_{\ul{\xi}\in \ev^{-1}(\mathfrak{Z})} \cA_\xi,\quad \quad \cA_{(1^\delta)}^{\min} = \bigcup_{\ul{\xi}\in \ev^{-1}(\mathfrak{U})} \cA_\xi
    \end{equation}
    where $\mathfrak{Z}$ and $\mathfrak{U}$ are as in the statement of the lemma. In particular, we also deduce the following.

\begin{cor}
    \label{cor: density of valuation 2 locus in delta strata}
Suppose that $\deg(D)\geq g+\frac{3}{2}\delta$. The subset $\cA_{(1^\delta)}^{\min}$ contains an open and dense subset of $\cA_{(1^\delta)}$. In particular, it contains an open, dense subset of $\cA_\delta$.
\end{cor}
\begin{proof}
    By Lemma \ref{lem: local evaluation map for delta strata}, the hypothesis on degrees implies that the local evaluation map \eqref{eqn: local evaluation map is surjection} is surjective, and so $\ev^{-1}(\mathfrak{U})$ is open and dense in $\ev^{-1}(\mathfrak{Z})$. The result follows therefore from \eqref{eqn: local decomposition of minimal delta cell}.
\end{proof}

\begin{proof}[Proof of Proposition \ref{prop: density of dagger locus in delta strata}.]

    Fix any points $x_1,\dots, x_n$ on $\breve\Sigma$ which are pairwise distinct. For any fixed $x$, $y$ (with $x\neq y$) distinct from each of the $x_i$, by Lemma \ref{lem: local evaluation map for delta strata} and our assumption on $\deg(D)$, the map 
    \[
    \cA_{\ul{\xi}}\to \fc(\cO_x/\varpi_x)\oplus \fc(\cO_y/\varpi_y)
    \]
    remains surjective. Therefore, if we consider the variety $Z_1\subset \fD_D^\cent\oplus \cA_{\ul{\xi}}$ of pairs $(c,a)$ such that $a(\breve{\Sigma})$ meets $\fD_D^\cent$ at $c$ and $d_c^\cent(a) > 1$, we can show that $\dim(Z_1)\leq \dim(\cA_{\ul{\xi}}) - 1$, and similarly for $Z_2\subset \fD^\br_D\oplus \cA_{\ul{\xi}}$ consisting of pairs $(c,a)$ such that $a(\breve{\Sigma})$ meets $\fD_D^\br$ at $c$ and $d_c^\br(a) > 1$. Since $\cA_{\ul{\xi},\delta}^\dagger:=\cA_{\ul{\xi}}\cap \cA_\delta^\dagger$ contains the complement of the image of $Z_1\cup Z_2$ in $\cA_{\ul{\xi}}$, we deduce that $\cA_{\ul{\xi},\delta}^\dagger$ contains an open, dense subset of $\cA_{\ul{\xi}}$. By the decomposition \eqref{eqn: local decomposition of minimal delta cell}, it follows that $\cA_\delta^\dagger$ is dense in $\cA_{(1^\delta)}^{\min}$. The result now follows from Corollary \ref{cor: density of valuation 2 locus in delta strata}.
\end{proof}

\section{Point Counting}
\label{sec: point counting (whole)}

In this section, we connect the geometry of relative affine Springer fibers to the arithmetic of orbital integrals by counting points. In particular, this will permit us to access orbital integrals by studying cohomology of relative affine Springer fibers.

\subsection{Orbital Integrals}

We begin by recalling the definition of orbital integrals for each of the unitary and Siegel cases. Let $V^\star$ denote the representation of the group $H^\star$ over $k$, where $\star = \siegel$ or $\unitary$ as in \S\ref{sec: rational structures}. We may choose Haar measures $dh_v$ on $H^\star(F_v)$, normalized so that $\mathrm{vol}(H^\star(\cO_v),dh_v) = 1$, and $dt_v$ on $I_\gamma\subset H^\star(F_v)$. The orbital integral at $\gamma$ is defined as 
\[
\Orb^\star(\gamma,dt_v) = \int_{I_\gamma\backslash H^\star(F_v)} \mathbb{1}_{V^\star(\cO_v)}(h^{-1}\cdot \gamma)\, \frac{dh_v}{dt_v}
\]
where $\mathbb{1}_{V(\cO_v)}$ denotes the indicator function for the integral points $V^\star(\cO_v)$. 

Recall that an element $\gamma\in V^\star(F_v)$ is \emph{geometrically conjugate} to $\gamma_0$  when $\gamma = h\cdot \gamma_0$ for $h\in H(\breve F_v)$; in this setting, we have the Galois cocycle 
\[
\mathrm{inv}(\gamma,\gamma_0)\in H^1(F_v,I_{\gamma_0}) = H^1(F_v,J_a),
\]
where $a$ is the common image of $\gamma$ and $\gamma_0$ in $\fc(F_v)$, as follows:  choose $h\in H(\breve F_v)$ such that $h\cdot \gamma_0 = \gamma$. Then, for every $\varsigma\in \Gal(\ol{F}_v/F_v)$, the element $h^{-1}\cdot\varsigma(h)$ fixes $\gamma_0$ and hence must lie in the centralizer $I_{\gamma_0}$. The class $\mathrm{inv}(\gamma,\gamma_0)$ is the cocycle
\[
\varsigma\mapsto h^{-1}\cdot\varsigma(h);
\]
this is independent of the choice of $h$ since $h$ is well defined up to translation by the centralizer $I_{\gamma_0}$. When $\gamma_0$ is regular semi-simple (so that $I_{\gamma_0}$ is abelian), this cocycle is well defined up to multiplication by a 1-coboundary $\varsigma\mapsto x^{-1}\cdot\varsigma(x)$ for $x\in I_{\gamma_0}$.

For $\kappa\in H^1(F_v,I_{\gamma_0})^D$ (where $D$ here refers to the Pontryagin dual for finite abelian groups), we form the $\kappa$ orbital integral by  considering the weighted sum
\[
\Orb^{\star,\kappa}_a(dt_v) = \sum_{\gamma} \langle\kappa,\mathrm{inv}(\gamma,\gamma_0)\rangle \Orb^{\star}(\gamma,dt_v)
\]
We will frequently omit the dependence on the Haar measure $dt_v$ when it is understood from context. When $\kappa=\mathbf{1}$ is trivial, we denote $\SOrb_a^\star = \Orb_a^{\star,\mathbf{1}}$.

We first explain how a character on this cohomology group gives rise to an endoscopic datum  of $X^\Siegel$ in the sense of Definition \ref{def: endoscopic datum}. Recall that Tate--Nakayama duality gives $H^1(F_v,I_{\gamma_0})^D\simeq \pi_0(\widehat{I}_{\gamma_0}^{\Gamma})$. We may identify dual tori $\widehat{{I}}_{\gamma_0}\simeq \widehat{\canC}$, which may not respect the Galois actions on the two dual tori. Nevertheless, \cite[Lemma 7.14]{sl_endoscopy} implies that there exists a $1$-cocyle $w: \Gamma \to W_{X}$ valued in the little Weyl group of ${X}^\star$ such that this isomorphism becomes $\Gamma$-equivariant if we endow $\widehat{\canC}$ with the twisted action
\[
\sigma\ast t = w_\sigma(t).
\]
Thus we may identify these two tori, and if $\kappa\in H^1(F_v,I_{\gamma_0})^D\simeq \pi_0(\widehat{\canC}^{\Gamma,\ast})$, we may choose a lift $\dot{\kappa}\in \widehat{\canC}^{\Gamma,\ast}$. Composing with the symmetrization map $s: \widehat{\mathbf{X}}\to \widehat{\bG}$, we obtain an endoscopic datum via the classical construction (cf. \cite[\S  1.8]{ngo}).

For an endoscopic datum $\fe = (\kappa,\rho)$ for $X^\Siegel$, recall from Lemma \ref{lem: pullback of divisors along transfer map} that we have a reduced divisor $\fR_\fe$ such that
\begin{equation}
\label{eqn: identities for pulling back on transfer}
\nu_\fe^*\fD^\cent = \fD^\cent_\fe+2\fR_\fe, \quad \nu_\fe^*\fD^\br = \fD^\br_\fe.
\end{equation}
For $a\in \fc^\heartsuit(\cO_v)$, we denote 
\begin{equation}\label{eqn: dim names}
    r_\fe(a) = \deg(a^*\fR_\fe), \quad \quad d^\cent(a) = \deg(a^*\fD^\cent),\quad \quad d^\br(a) = \deg(a^*\fD^\br)
\end{equation}
We similarly denote $d_\fe^\cent(a_\fe)$ and $d_\fe^\br(a_\fe)$ for $a_\fe\in \fc_\fe^\heartsuit(\cO_v)$. Note that the identities \eqref{eqn: identities for pulling back on transfer} imply that, if $a_\fe$ maps to $a$ under $\nu_\fe$, then $d^\br(a) = d^\br_\fe(a_\fe)$ and $d^\cent(a) = d^\cent_\fe(a_\fe)+2r_\fe(a_\fe)$.

Note that if $\nu_\fe(a_\fe) = a$ for $a_\fe\in \fc_\fe(\cO_v)$ and $a\in \fc^\heartsuit$, then $J_a(F_v)\simeq J_{a_\fe}^\fe$ by Proposition \ref{prop: centralizers under transfer}. Hence, for a choice of Haar measure $dt_v$ on $J_a(F_v)$, we also obtain such a measure on $J_{a_\fe}^\fe$. The main theorem of this paper is the following.

\begin{thm}
    \label{thm: main}
    Let $\fe$ be an elliptic endoscopic datum, and suppose that $a\in \fc^\heartsuit(\cO_v)$ and $a_\fe\in \fc_\fe^
    \heartsuit(\cO_v)$ are matched in the sense that $\nu_\fe(a_\fe) = a$. Then, we have
    \[
    \Orb_a^{\Siegel,\kappa}(dt_v) = q^{r_\fe(a)}\SOrb_{a_\fe}^{\fe}(dt_v).
    \]
\end{thm}

We will drop the measure $dt_v$ from the notation for orbital integrals where it is understood. This result will be reinterpreted geometrically in terms of RASFs in Theorem \ref{thm: local main} after we've developed some point counting techniques for affine Springer fibers. The proof of Theorem \ref{thm: main} is completed in \S\ref{sec: proof of main}.

\subsection{Point Counting on RASFs}
\label{sec: point counting for RASFs}

We now interpret these orbital integral as weighted point counts of groupoids of $k$-points associated to relative affine Springer fibers. We first recall the classical setting in \cite[Chapter 8]{ngo}, which will apply to our relative Hitchin systems verbatim.

\subsubsection{General framework}
Let $P$ be a smooth commutative group acting on a scheme $M$ over $k$. We assume that $\pi_0(P)$ is a finite-type abelian group, the stabilizer $\mathrm{Stab}_P(m)$ is finite type for all points $m\in M$, and there exists a discrete subgroup $\Lambda\subset P$ such that $P/\Lambda$ and $M/\Lambda$ are finite type. This will apply to our cases of interest by Proposition \ref{prop: lattice action on RASF}.

For $P$ and $M$ satisfying the above properties, we put $\sX = [M/P]$ the stack over $k$ whose set of $k$-points is given by
\[
\sX(k) = \{(m,p)\in M(\ol{k})\times P(\ol{k})\mid p\cdot \sigma(m) = m\},
\]
where we recall that $\sigma$ denotes the Frobenius element of $k$. The $k$-points $\sX(k)$ naturally have the structure of a groupoid. For any $x\in \sX(k)$, we denote by $\Aut_{\sX(k)}(x)$ the automorphism group of $x$. Let $P_\sigma$ denote the set of $\sigma$-twisted conjugacy classes in $P(\ol{k})$, i.e. elements of $P_\sigma$ are sets
\[
\{ g p \sigma(g)^{-1}\mid g\in P(\ol{k})\}
\]
for $p\in P(\ol{k})$. The Galois cohomology $H^1(k,P)$ is the maximal torsion subgroup contained in $P_\sigma$.  We have a class map
\[
\cl\colon \sX(k)\to H^1(k,P)
\]
defined by taking $(m,p)$ to the $\sigma$-twisted conjugacy class of $p$.

\begin{definition}
	For any character $\kappa\colon H^1(k,P)\to \Q^\times$, we define
	\[
	\# \sX(k)_\kappa = \sum_{x\in |\sX(k)|} \frac{\kappa(x)}{\# \Aut_{\sX(k)}(x)}
	\]
	where the sum above ranges over isomorphism classes in $\sX(k)$. Here we abuse notation and set $\kappa(x) = \kappa(\cl(x))$. When $\kappa = 1$ is the trivial character, we abbreviate this sum simply by $\# \sX(k)$.
\end{definition}

We now specialize to our cases of interest. Let $a\in \fc^\heartsuit(\cO_v)$, and let $\sX = [\cM_v^\star(a)/\cP_v(J_a')]$ for $J_a'\to J_a$ a fixed group scheme with connected fibers over $\cO_v$ which induces an isomorphism on the generic fiber. We will deduce a formula relating the point counts $\#[\cM_v^\star(a)/\cP_v(J_a')](k)_\kappa$ with the $\kappa$ orbital integral $\Orb_a^{\star,\kappa}$. Note that by Theorem \ref{thm: regular centralizers for rational forms}, $J_a$ has connected fibers for all $a\in \fc^\heartsuit(\cO_v)$, and hence all results that follow will hold for $J_a'=J_a$. Recall the notation of $\fX_\gamma^\star$ from \eqref{eqn: RASF functor of points with lift} and  Proposition \ref{prop: RASF lattice description}. 

We begin by recalling a lemma from \cite{ngo}.

\begin{lem}(\cite[Lemma 8.2.4]{ngo})
	\label{lem: H1 from P to J} 
    Let $J'$ and $J$ be smooth, commutative group schemes over $\cO_v$ with $J'$ having connected fibers. Suppose that there is a map $J'\to J$ which is an isomorphism over $F_v$, and that the common fiber over $F_v$ is a torus. Then, there is an isomorphism $H^1(k,\cP_v(J'))\simeq H^1(F_v,J)$, where $\cP_v(J') = J'(F_v)/J'(\cO_v)$ is the affine Grassmannian of $J'$.
\end{lem}

Note in particular that $H^1(k,\cP_v(J_a))\simeq H^1(F_v,J_a)$ and $H^1(k,\cP_v(J_{\fe,a_\fe}))\simeq H^1(F_v,J_a)$

The class function now gives a map
\begin{equation}
\label{eqn: cl function}
\cl\colon [\cM_v^\star(a)/\cP_v(J_a')](k)\to H^1(F_v,J_a).
\end{equation}
This map is not surjective; we characterize its image in the following lemma.

\begin{lem}
	\label{lem:  image of class function lies in ker1}
	Let $a\in \fc^\heartsuit(\cO_v)$, and choose $\gamma\in V^\star(F_v)$ lifting  $a$. The canonical isomorphism $J_a\simeq I_{\gamma}\subset H^\star(F_v)$ from Proposition \ref{prop: regular centralizer package} induces a morphism
	\begin{equation}
		\label{eqn: ker1}
	H^1(F_v,J_a)\to H^1(F_v,H^\star)
	\end{equation}
	Let $\ker_1^{\star,\gamma}$ denote the kernel of \eqref{eqn: ker1}. Then, for any $x\in [\fX_\gamma^\star/\cP_v(J_a')](k)$, we have $\cl(x)\in \ker_1^{\star,\gamma}$. 
\end{lem}
\begin{proof}
	Let $x = (m,p)$. We write $m = gH^\star(\breve{\cO}_v)$ for $g\in H^\star(\breve{F}_v)$ and $p = jJ_a(\breve{\cO}_v)$ for $j\in J_a(\breve{F}_v)$. The condition $p\cdot \sigma(m) = m$ is equivalent to $g^{-1}\cdot j \cdot\sigma(g)\in H^\star(\breve{\cO}_v)$.  Since every element of $H^\star(\breve{\cO}_v)$ is $\sigma$-conjugate to the identity, we conclude that $j$ is $\sigma$-conjugate to the identity under $H^\star(\breve{F}_v)$, and $\cl(x) \in \ker_1^{\star,\gamma}$.
\end{proof}

\begin{prop}
	\label{prop: point counting for P(J')}
	For $\kappa\in H^1(F_v,J_a)^D$ a character, we have an equality
	\[
	\# [\cM_v^\star(a)/\cP_v(J_a')](k)_\kappa = \mathrm{vol}(J_a'(\cO_v),dt_v)\cdot \Orb_a^{\star,\kappa}
	\]
\end{prop}
\begin{proof}
We may decompose
	\[
	[\cM_v^\star(a)/\cP_v(J_a')](k) = \bigsqcup_{\xi\in \ker_1^{\star,\gamma}}[\fX_\gamma^\star/\cP_v(J_a')](k)_\xi
	\]
	where the disjoint union is over the image under the class function. That is, if $j_\xi\in \cP_v(J_a')$ denotes a representative for the $\sigma$-twisted conjugacy class $\xi\in H^1(F_v,J_a)\simeq H^1(k,\cP_v(J_a'))$, then we have
	\[
	[\fX_\gamma^\star/\cP_v(J_a')](k)_\xi = \{  h\in \Gr_H(\ol{k})\mid h^{-1}j_\xi \sigma(h) = 1\text{ and }\mathrm{ad}(h)^{-1}\gamma\in V^\star(\breve{\cO}_v) \}
	\]
	Since the image of $j_\xi$ lies in $\ker_1^{\star,\gamma}$, there exists some $h_\xi\in H^\star(\breve{F})$ such that $h_\xi^{-1}j_\xi\sigma(h_\xi) = 1$. For any $h\in [\fX_\gamma^\star/\cP_v(J_a')](k)_\xi$, the map sending $h$ to $h' = h_\xi^{-1}h\in \Gr_H(k)$ gives an isomorphism of groupoids
	\begin{equation}
		\label{eqn: xi part of quotient}
	[\fX_\gamma^\star/\cP_v(J_a')](k)_\xi = \{ h'\in \Gr_H(k)\mid \mathrm{ad}(h')^{-1}\gamma_\xi\in V^\star(\cO_v) \}
	\end{equation}
	where $\gamma_\xi = \mathrm{ad}(h_\xi)^{-1}\gamma$. Note the groupoid structure on the above is realized through the action of $I_{\gamma_\xi}$ on the right-hand side.
	
	It is evident that the isomorphism classes of the groupoid on the right-hand side of \eqref{eqn: xi part of quotient} are given by the double coset space
	\[
\{ h'\in I_{\gamma_\xi}\backslash H^\star(F_v)/H^\star(\cO_v)\mid \mathrm{ad}(h')^{-1}\gamma_\xi\in V^\star(\cO_v) \}	
	\]
	with automorphism groups
	\[
\Aut(h') = (I_{\gamma_\xi}\cap h'H^\star(\cO_v)(h')^{-1})/J_a'(\cO_v)	
	\]
	Thus, we compute that
	\begin{align*}
	\#[\fX_\gamma^\star/\cP_v(J_a')](k)_\xi &= \mathrm{vol}(J_a'(\cO_v),dt_v)\cdot \int_{I_{\gamma_\xi}\backslash \Gr_H} \mathbb{1}_{V^\star(\cO_v)}(\mathrm{ad}(h')^{-1}\gamma_\xi)\cdot dt_v \\
	& = \mathrm{vol}(J_a'(\cO_v),dt_v)\cdot \Orb({\gamma_\xi})
	\end{align*}
We conclude by summing over $\xi\in \ker_1^{\star,\gamma}$ according to the appropriate $\kappa$ weights.
\end{proof}

Since $J$ has connected fibers, we have the following special case.

\begin{cor}
    \label{cor: point counting for P(J)}
    Fix $a\in \fc^\heartsuit(\cO_v)$. Let $\kappa$ be a character of $H^1(k,\cP_v(J_a))\simeq H^1(F_v,J_a)$. Then, we have
    \[
    \#[\cM_v^\star(a)/\cP_v(J_a)](k)_\kappa = \vol(J_a(\cO_v),dt_v)\cdot \Orb_a^{\star,\kappa}
    \]
\end{cor}

\subsection{Point counting along the relative Hitchin fibration}

We now relate the local point-counting formula in Proposition \ref{prop: point counting for P(J')} to point counts along the anisotropic locus in the relative Hitchin fibration.

Let $a\in \cA^\ani(k)$, and let $J_a'\to J_a$ be a map of group schemes over $\Sigma$. As always, we demand in the sequel that $J_a'$ have connected geometric fibers and that the map be an isomorphism on the generic fiber. Define $\cP_a'$ to be the space of $J_a'$-torsors on $\Sigma$, equipped with the natural map $\cP_a'\to \cP_a$.

We form the stacky quotient $[\cM_a^\star/\cP_a']$. The class map for $x$ the groupoid $[\cM_a^\star/\cP_a'](k)$ lands in $\cl(x)\in H^1(k,\cP'_a) \simeq \pi_0(\cP'_a)_\sigma$, so we may consider the weighted point count $\#[\cM_a^\star/\cP_a'](k)_{\kappa'}$ as in \S\ref{sec: point counting for RASFs} for $\kappa'\in H^1(k,\pi_0(\cP_a'))^D$. We consider $\kappa'$ that arise as follows. Recall that, since $\pi_0(\cP_a)$ is finite, we have $H^1(k,\pi_0(\cP_a))\simeq \pi_0(\cP_a)_\sigma$ is naturally identified with the space of Frobenius-twisted conjugacy classes in $\pi_0(\cP_a)$, and similarly for $\cP_a'$. Given a $\sigma$-invariant character
\[
\kappa\colon \pi_0(\cP_a)\to \ol{\Q}_\ell^\times,
\]
we obtain a character $\kappa'$ on $H^1(k,\pi_0(\cP_a'))$ by pre-composing with the map $\pi_0(\cP_a')\to \pi_0(\cP_a)$. We denote the corresponding point count by $\#[\cM_a^\star/\cP_a'](k)_{\kappa}$.

Similarly, for any $v\in |\Sigma|$, we have the Beauville--Laszlo gluing map $\cP_v(J_a')\to \cP_a'$. Pre-composing $\kappa$ as above with the map $\pi_0(\cP_v(J_a'))_\sigma\to \pi_0(\cP_a')_\sigma$ gives a character $\kappa'_v$ of $H^1(k,\pi_0(\cP_v(J_a')))\simeq H^1(F_v,J_a')$. We again denote the corresponding weighted point count of the RASF at $v$ by $\#[\cM_v^\star(a)/\cP_v(J_a')](k)_\kappa$. We can relate the global and local point counts now as follows.

\begin{prop}
	\label{prop: local to global point count}
	Let $a\in \cA^\ani(k)$, and set $U_a = a^{-1}(\fc_D^\mathrm{rs})$ where $\fc^{\mathrm{rs}} = \fc - (\fD^\cent\cup\fD^\br)$. For any $\sigma$-invariant character $\kappa\colon \pi_0(\cP_a)\to \ol{\Q}_\ell^\times$, we have
	\begin{equation}
    \label{eqn: point count product formula}
	\#[\cM_a^\star/\cP_a'](k)_\kappa = \prod_{v\in |\Sigma - U_a|} \#[\cM_v^\star(a)/\cP_v(J_a')](k_v)_\kappa
	\end{equation}
\end{prop}
\begin{proof}
	Recall that by the product formula (Proposition \ref{prop: product formula}), we have an equivalence of groupoids
	\[
	[\cM_a^\star/\cP_a'](k)\simeq \prod_{v\in |\Sigma - U_a|}[\cM_v^\star(a)/\cP_v(J_a')](k_v)
	\]
	Consider the class functions $\cl$ defined in \eqref{eqn: cl function} on the groupoids $[\cM_a^\star/\cP_a'](k)$ and $[\cM_v^\star(a)/\cP_v(J_a')](k_v)$.  Then, by the definition of $\kappa'$ and $\kappa'_v$, the above equivalence preserves the class function; that is,
	\[
	\langle \cl(x),\kappa'\rangle = \prod_{v\in |\Sigma - U_a|} \langle \cl(x_v),\kappa'_v\rangle.
	\]
	for any $x\in [\cM_a^\star/\cP_a'](k)$ and $x_v\in [\cM_v^\star(a)/\cP_v(J_a')](k_v)$ corresponding to one another via the product formula. The proposition then follows.
\end{proof}

By computing the lefthand side of \eqref{eqn: point count product formula} using the Grothendieck--Lefschetz trace formula, we now deduce:
\begin{cor}
    \label{cor: Grothendieck-Lefschetz for Hitchin fibers}
    For any $\sigma$-invariant character $\kappa\colon \pi_0(\cP_a)\to \overline{\Q}_\ell^\times$, we have
    \[
    \sum_{i\geq0} (-1)^i\tr\left( \sigma, H^i(\cM_a^\star)_\kappa \right) = \#(\cP'_a)^\circ(k)\cdot \prod_{v\in |\Sigma - U_a|}\#[\cM_v^\star(a)/\cP_v(J_a')](k_v)_\kappa
    \]
\end{cor}
\begin{proof}
    Follows directly from Proposition \ref{prop: local to global point count} and \cite[Proposition 8.1.6]{ngo}.
\end{proof}

\subsection{The geometric relative Fundamental Lemma and cases}
\label{sec: cases and statements}

In this section, we state the precise form of our main result using the language of RASFs. We then establish this result for a few basic cases, which will serve as base cases for the global argument in \S\ref{sec: proof of main}.

By Proposition \ref{prop: point counting for P(J')}, Theorem \ref{thm: main} is equivalent to the (local) geometric incarnation below.

\begin{thm}
    \label{thm: local main}
    Let $\fe$ be an elliptic endoscopic datum, and suppose that $a\in \fc^\heartsuit(\cO_v)$ and $a_\fe\in \fc_\fe^
    \heartsuit(\cO_v)$ are such that $\nu_\fe(a_\fe) = a$. Let $J_a'\to J_a$ be a smooth commutative group scheme over $\cO_v$ with connected fibers which is isomorphic to $J_a$ on the generic fiber $\Sigma_v^\bullet = \Spec(F_v)$. Then, for all finite extensions $K/k_v$, we have
    \begin{equation}
    \label{eqn: local main}
    \#[\cM_v^\Siegel(a)/\cP_v(J_a')](K)_\kappa = q^{r_\fe(a)\deg(v)[K:k_v]}\cdot\#[\cM_v^\fe(a_\fe)/\cP_v(J_a')](K)
    \end{equation}
\end{thm}

\begin{remark}
    We note that since $J_a$ has connected fibers, Theorem \ref{thm: local main} remains true for $J_a'=J_a$. It is important to note, however, that the quotient on either side of \eqref{eqn: local main} is by the \emph{same} symmetry group. 
\end{remark}

The full proof of Theorem \ref{thm: local main} will be contained in \S\ref{sec: proof of main}. For now, we consider some particularly simple cases. As a first example, we consider the case of nontrivial obstruction class, where the theorem is a consequence of a functional equation; see also \cite[Lem 2.5.3]{YunJR} and \cite[Lem 3.6.4]{griffin2}.

\begin{prop}
	\label{prop: obstructed case}
	 Suppose that $a,a_\fe$ are as in Theorem \ref{thm: local main}, and $\kappa\in H^1(F_v,J_a)^D$ is a nontrivial character. If $\Ob_v(a_\fe)\neq 0$ (i.e., $\mathrm{val}(\det(a))$ is odd), then $\Orb_a^{\Siegel,\kappa} =0$. In particular, Theorem \ref{thm: local main} holds in this case with both sides being equal to zero.
\end{prop}
\begin{proof}
Note that since $\Ob_v(a)\neq 0$, we have $\cM_v^\fe(a)(k) = \emptyset$, and hence the stable point count is $\#[\cM_v^\fe(a)/\cP_v(J_a)](k) = 0$. To show the $\kappa$-twisted point count on $\cM_v^\siegel(a)$ vanishes, we construct a duality $\bbD$ on $\cM_v^\siegel(a)(\ol k)$, which will produce a functional equation for the point count at $a$. Recall from Proposition \ref{prop: RASF lattice description} that $\cM_v^\Siegel(a)(\ol k)$ is the moduli of tuples $(\cE',h,\iota,\phi';\xi)$ where $(\cE',h,\iota,\phi')$ is a Higgs object consisting of:
\begin{enumerate}
	\item A free rank $2n$ $\cO'\otimes_k\ol k$-module $\cE'$;
	\item A nondegenerate Hermitian form $h\colon \cE'\to \sigma^*\cE'^\vee$;
	\item An $\cO'\otimes_k\ol k$-linear involution $\iota\colon \cE'\to \cE'$ such that $h\circ \iota =-h$;
	\item An $\cO'\otimes_k\ol k$-linear endomorphism $\phi'\colon \cE'\to \cE'$ such that $\iota\circ \phi' = -\phi'\circ \iota$, $h(\phi'v,w)+h(v,\phi'w)=0$, and $\phi'$ has characteristic polynomial $a$;
\end{enumerate}
and $\xi$ is a trivialization on the punctured infinitesimal disk $\breve\Sigma_v^\bullet = \Spec \breve F_v'$, i.e.
\[
\xi\colon (\cE',h,\iota,\phi')|_{\breve\Sigma_v^\bullet}\xrightarrow{\sim} (R_a(\breve F_v'), h^\siegel, \iota, [x]).
\]
There is a natural duality on Higgs objects given by 
\[
\bbD(\cE',h,\iota,\phi') = (\cE'^\vee,(h^{-1})^\vee,\iota^\vee,\phi'^\vee)
\]
 Note that over any extension of $F_v$ where  $F_v'$ splits, there exists an isomorphism
\[
\theta\colon \bbD(R_a(\breve F_v'), h^\siegel, \iota, [x])\to (R_a(\breve F_v'), h^\siegel, \iota, [x]).
\]
In fact, the space of isomorphisms is a torsor $\cT_a$ under the action of $J_a$, and so determines a class $[\cT_a]\in H^1(F_v,J_a)$. For any choice of $\theta$, we have an explicit cocyle representing $[\cT_a]$ given by
\[
d_\theta = \theta\cdot \sigma(\theta)^{-1},
\]
where $\sigma(\theta)$ is the Frobenius-twisted isomorphism $\sigma\circ \theta\circ \sigma^{-1}$.

Fix a choice of such a $\theta$, and now define a duality $\mathbb{D} = \mathbb{D}_\theta$ on $\cM_v(a)(\ol k)$ by
\[
\bbD(\cE',h,\iota,\phi';\xi) = (\bbD(\cE',h,\iota,\phi');\theta\circ (\xi^{-1})^\vee).
\]
This is compatible with inversion on $\cP_v(J_a)$ in the sense that for every $m\in \cM_v(a)(\ol k)$ and $p\in \cP_v(J_a)(\ol k)$, we have
\[
\bbD(p\cdot m) = p^{-1}\cdot \bbD(m).
\]
The duality $\bbD$ satisfies the compatibility with Frobenius given by the cocyle $d_\theta$; that is, for any $m\in \cM_v(a)(\ol k)$,
\[
\bbD(\sigma(m)) = d_\theta\cdot \sigma(\bbD(m))
\]
Therefore, $\bbD$ induces an involution (which we also call $\bbD$) on the groupoid $[\cM_v^\siegel(a)/\cP_v(J_a)](k)$ by $\bbD(m,p) = (\bbD(m),d_\theta\cdot  p^{-1})$. For any $x = (m,p)$, we compute
\[
\cl(\bbD(x)) = \cl(\bbD(m),d_\theta\cdot  p^{-1}) = d_\theta\cdot  p^{-1}\in \pi_0(\cP_v(J_a))_\sigma = H^1(F,J_a)
\]
and hence $\cl(\bbD(x)) = [\cT_a]\cdot\cl(x)^{-1}$ in $H^1(F_v,J_a)$.

Suppose now that $\kappa\in H^1(F_v,J_a)^D$ and $\nu_\fe(a_\fe) = a$. For any $x\in [\cM_v^\siegel(a)/\cP_v(J_a)](k)$, by our formula for the class function of $\bbD(x)$, we have 
\[
\langle \kappa,\cl(\bbD(x))\rangle = \kappa([\cT_a])\cdot \kappa(\cl(x))^{-1}.
\]
Since $H^1(F_v,J_a)$ is 2-torsion, $\kappa(\cl(x))^{-1}=\kappa(\cl(x))$. On the other hand, we prove the following lemma to compute $\kappa([\cT_a])$.

\begin{lem}
    With notation as above, write $X^\fe =X^\siegel_r\times X^\unitary_s$ and decompose $a_\fe = (a_\fe^\siegel,a_\fe^\unitary)$. Then 
    \[
    \kappa([\cT_a]) = (-1)^{\val (a_\fe^\unitary)_{2s}} = \begin{cases}
        1 & \text{if }\Ob(a_\fe) = 0 \\
        -1 & \text{if }\Ob(a_\fe) \neq 0
    \end{cases}
    \]
\end{lem}
\begin{proof}
    By the description given in Theorem \ref{thm: regular centralizers for rational forms}, we compute
    \[
    H^1(F_v,J_a) = R_a(F_v)^{\iota}/\mathrm{Nm}(R_a(F_v')^{\iota})
    \]
    where $\mathrm{Nm}$ denotes the norm map
    \[
    \mathrm{Nm}\colon R_a(F_v')^{\iota}\to R_a(F_v)^{\iota}.
    \]
    Recall that $R_a(F_v')^{\iota} = F_v[x^2]/(Q_a(x^2))$ for the polynomial $Q_a$ satisfying $Q_a(x^2) = P_a(x)$. By Proposition \ref{prop: centralizers under transfer}, we see that the endoscopic character $\kappa$ determines an ordered factorization $Q_a = Q^{\siegel}\cdot Q^\unitary$ and hence a decomposition 
\[
R_a(F_v)^{\iota}/\mathrm{Nm}(R_a(F_v')^{\iota})\simeq R^\siegel\oplus R^\unitary.
\]
    Note that the coefficients of $Q^\siegel$ (resp., $Q^\unitary$) are precisely those given by the tuple $a_\fe^\siegel$ (resp., $a_\fe^\unitary$). Pairing a class in $(r^\siegel,r^\unitary)\in R^\siegel\oplus R^\unitary\simeq H^1(F_v,J_a)$ with $\kappa$ is given by (compare \cite[Lemma A.2]{Lesliedescent})
\[
\kappa(r^\siegel,r^\unitary) = (-1)^{\val(\mathrm{Nm}_{R^\unitary/F_v}(r^\unitary))}.
\]
    
    We now compute the class of the torsor $\cT_a$ in this group explicitly. If we express the multiplication-by-$x$ map $[x]$ and its dual $[x]^\vee$ with respect to the decomposition $R_a(F_v') = R_a(F_v')^{\iota}\oplus x\cdot R_a(F_v')^{\iota}$, we get
    \[
[x] = \begin{pmatrix}
    0 & x^2 \\ 1 & 0
\end{pmatrix},\quad [x]^\vee = \begin{pmatrix}
    0 & 1 \\ x^2 & 0
\end{pmatrix}.
    \]
    If an isomorphism $\theta\colon R_a(F_v')^{\iota}\to R_a(F_v')^{\iota}$ preserves $\iota$ and intertwines $[x]$ and $[x]^\vee$, then $\theta$ must have the form
    \[
    \theta = \begin{pmatrix}
    x^2t & 0 \\ 0 & t
\end{pmatrix}
    \]
    for some $t\in (R_a(F_v')^{\iota})^\times$. To preserve $h$, we obtain the further condition on $t$ that $t\cdot \sigma(t)^{-1} = x^{-2}$. Hence, we conclude that
    \[
    [\cT_a] = x^{-2} = x^2\in R_a(F_v)^{\iota}/\mathrm{Nm}(R_a(F_v')^{\iota})
    \]
    Since $Q^\unitary$ is monic of degree $s$, we have $\mathrm{Nm}_{R^\unitary/F_v}(x^2) = (-1)^sQ^\unitary(0)$ is the constant term of $Q^\unitary$ up to sign, which is exactly $(a_\fe^\unitary)_{2s}$. We therefore compute
    \[
\kappa([\cT_a]) = (-1)^{\val (a_\fe^\unitary)_{2s}}.
    \]
    The final line of the lemma follows from the computation of $\Ob(a_\fe)$ given in Lemma \ref{lem: obstruction for unitary}.
\end{proof}

 Hence, we conclude from Corollary \ref{cor: point counting for P(J)} the functional equation
	\[
\Orb_a^{\Siegel,\kappa} = (-1)^{\Ob(a_\fe)}\Orb_a^{\Siegel,\kappa}
	\]
	from which the Proposition \ref{prop: obstructed case} follows.
\end{proof}

We conclude this section by proving Theorem \ref{thm: local main} for some special cases of split elements for which the dimension of the corresponding relative affine Springer fiber is at most one. We note that the geometry of all but the first of these cases appear in \cite[Lem 8.5.7]{ngo}.

\begin{prop}
    \label{prop: zero dimensional cases}
    Let $\fe$ be an elliptic endoscopic datum. Theorem \ref{thm: local main} holds for all $a_\fe\in \fc_\fe^\heartsuit(\cO_v)$ with image $a\in \fc^\heartsuit(\cO_v)$ satisfying either (c.f. \eqref{eqn: dim names})
    \[
    d^\br(a) = d^\cent(a) = 0
    \]
    or
    \[
    d^\br(a) = 0, \quad d^\cent(a) = d^\cent_\fe(a_\fe)= 1.
    \]
    Moreover, in this case, both sides of \eqref{eqn: local main} are nonzero rational numbers.
\end{prop}
\begin{proof}
    By the dimension formula (Theorem \ref{prop: dim of RASF}), we must have $\delta_v(a) = \delta_{\fe,v}(a) = 0$, so that by the proof of Proposition \ref{prop: dim formula split case}, the affine Springer fiber $\cM_v^\Siegel(a)$, respectively $\cM_v^{\fe}(a)$, is a torsor for the action of the lattice $\cP_v(J_a) = \cP_v(J_{\fe,a_\fe}) = \Gamma_a$. Therefore, on the endoscopic side, we have
    \[
    1 = \#[\cM_v^\fe(a_\fe)/\cP_v(J_{\fe,a_\fe})](k) = \vol(J_{\fe,a_\fe}(\cO_v))\cdot \SOrb_{a_\fe}
    \]
    by Corollary \ref{cor: point counting for P(J)} as well as
    \[
    \#[\cM_v^\fe(a_\fe)/\cP_v(J_a')](k) = \vol(J_{a}'(\cO_v))\cdot \SOrb_{a_\fe}
    \]
    by Proposition \ref{prop: point counting for P(J')}. Combining these statements, we get that
    \[
    \#[\cM_v^\fe(a_\fe)/\cP_v(J_a')](k) = \frac{\vol(J_{a}'(\cO_v))}{\vol(J_{\fe,a_\fe}(\cO_v))}.
    \]
    By an identical argument, we also get
    \[
    \#[\cM_v^\Siegel(a)/\cP_v(J_a')](k) = \frac{\vol(J_{a}'(\cO_v))}{\vol(J_{a}(\cO_v))}.
    \]
    To conclude, we note in these cases that the morphism $J_a\to J_{\fe,a_\fe}$ is an isomorphism so that the two denominators above agree.
\end{proof}

\begin{prop}
\label{prop: 1 dimensional case}
    Let $\fe$ be an elliptic endoscopic datum. Theorem \ref{thm: local main} holds for $a\in \fc^\heartsuit(\cO_v)$ such that $d^\br(a) = 0$, $d^\cent(a) = 2$ and $d^\cent_\fe(a) = 0$. Moreover, in this case, both sides of \eqref{eqn: local main} are nonzero rational numbers.
\end{prop}
\begin{proof}
    Since $d^\cent_\fe(a) = 0$, by the dimension formula (Theorem \ref{prop: dim of RASF}) we have $c_{\fe,v}(a_\fe) = 0$. Since $c_v(a)$ depends only on the generic fiber and $J_a\to J_{\fe,a_\fe}$ is an isomorphism on the generic fiber (see Proposition \ref{prop: centralizers under transfer}), it also follows that $c_v(a) = 0$. Therefore, we have $\delta_{\fe,v}(a_\fe) = 0$ while $\delta_v(a) = 1$, and $a$ is split.

    On the endoscopic side, an identical argument to the proof of Proposition \ref{prop: zero dimensional cases} implies that
    \begin{equation}
    \label{eqn: endoscopic dim zero point count}
    \#[\cM_v^\fe(a)/\cP_v(J_a')](k) = \frac{\vol(J_{a}'(\cO_v))}{\vol(J_{\fe,a_\fe}(\cO_v))}.
    \end{equation}

    We now consider the Siegel case. Consider the polynomial $Q_a(x^2)=P_a(x)$ (\emph{cf.} Proposition \ref{prop: lattice action on RASF}). Since $a$ is split, it follows that $Q_a(y)$ splits as a product of linear factors $Q_a(z)=\prod_{i=1}^n(z-u_i)$ over $\breve F_v$. Furthermore, by Gauss's Lemma, $u_i\in \breve\cO_v$ for all $i$, and because $d^{\mathrm{br}}(a)=\val(a_{2n})=0$ and $a_{2n} = \prod_i u_i$, the $u_i$ are units in $\breve\cO_v$. Since $d^{\mathrm{cent}}(a)=\operatorname{val}\operatorname{Disc}(Q_a)=2$, there exists exactly one pair of roots, $u_i,u_j$, such that $\val(u_i-u_j)=1$, and all other differences $u_{i'}-u_{j'}$ are units. By the assumption that $d^\cent_{\fe}(a_\fe)=0$, the $\kappa$ factorization $Q_a=Q^\siegel Q^\unitary$ places $u_i$ and $u_j$ in different factors.

We reuse notation $R = R_a(\breve\cO_v)$ and $R^\iota$ from Section \ref{sec: symmetries of RASFs}. We also use $\wt{R^\iota}$ to denote the normalization of $R^\iota$. By Theorem \ref{thm: regular centralizers for rational forms} and Proposition \ref{prop: centralizers under transfer}, we have identifications
\[
J_a(\breve\cO_v) \simeq \breve\cO_v[y]/Q_a(y)=R^\iota, 
\]
and of $J_{\fe,a_\fe}(\breve\cO_v)$ with the normalization 
\[
J_a(\breve\cO_v)\simeq \breve\cO_v[y]/Q^\siegel(y)\times \breve\cO_v[y]/Q^\unitary(y) = \wt{R^\iota}.
\]
Define a torus $A_a:=J_a(\breve\cO_v)/J_{\fe,a_\fe}(\breve\cO_v) =\wt{R^\iota}/R^\iota$, with the natural $k$ structure defined by the Galois actions on $J_a(\breve\cO_v)$ and $J_{\fe,a_\fe}(\breve\cO_v)$. By out descriptions of $J_a(\breve\cO_v)$ and $J_{\fe,a_\fe}(\breve\cO_v)$ above, the torus $A_a$ is of rank 1.

\begin{lemma}
    There is a short exact sequence 
    \begin{equation}
    \label{eqn: exact sequence of rank 1 torus}
0\to A_a(\ol k)\to \cP_v(J_a)(\ol k)\to \Gamma_a\to 0,
\end{equation}
    where $\Gamma_a$ is the lattice of Section \ref{sec: symmetries of RASFs}. Moreover, all morphisms are equivariant with respect to the Galois action of $\mathrm{Gal}(\ol k/k_v)$.
\end{lemma}
\begin{proof}
    Explicitly, the lattice $\Gamma_a$ is given by
    \[
    \Gamma_a = (R^\iota\otimes_{\breve\cO_v}\breve F_v)/\wt{R^\iota}
    \]
    Hence, the sequence \eqref{eqn: exact sequence of rank 1 torus} is explicitly
    \[
0\to \wt{R^\iota}/R^\iota\to (R^\iota\otimes_{\breve\cO_v}\breve F_v)/R^\iota\to (R^\iota\otimes_{\breve\cO_v}\breve F_v)/\wt{R^\iota}\to 0,
    \]
    whose exactness is evident.
\end{proof}

Note in particular, that $\Gamma_a$ is of rank $n$. We may compute
\[
\#A_a(k_v) = \frac{\mathrm{vol}\; J_{\fe,a_\fe}(\cO_v)}{\mathrm{vol}\; J_a(\cO_v)}.
\]

The relative affine Springer fiber $\cM_v^\siegel(a)$ was implicitly computed in Proposition \ref{prop: dim formula split case}; we make this computation more explicit here. Over $\ol k$, we have a lift of $a$ of the form
\[
\gamma = \begin{pmatrix}
    0 & D \\ I_n & 0
\end{pmatrix}\in V^\fe(\ol k),\quad D = \mathrm{diag}(u_1,\dots,u_n).
\]
By Lemma \ref{lem: kostant section and canonical lift agree}, we compute that $\cM_v^\siegel(a)\otimes_k\ol k\simeq \fX_{\gamma}$ is represented by a pair $(g_1,g_2)$ with $g_1^{-1}Dg_2$ and $g_2^{-1}g_1$ having integral entries. The determinants multiply to the unit $\det D$, so both $g_1^{-1}Dg_2$ and $g_2^{-1}g_1$ are invertible over $\breve{\mathcal O}_v$. In particular, $g_1 = g_2$ in the affine Grassmannian for $\GL_n$, and we can identify $\cM_v(a)\otimes_k\ol k$ with the $\GL_n$ affine Springer fiber of the diagonal matrix $D$. With the condition on $D$, this is precisely the case considered in \cite[\S8.3]{ngo}, with our torus $A_a$ coinciding with Ng\^o's torus $A_{\pm \alpha}$. Namely, let 
\[
M = \bigcup_{j\in \Z}C_j,\quad C_j\simeq \bP^1\text{ for all }j
\]
be an infinite chain of projective lines over $\ol k$ with a Galois action swapping components $\sigma(C_i) = C_{1-i}$ and a unique fixed point. There is an action of $P = A_a\times \Z$ on $M$ with the torus $A_a$ acting on each irreducible component $C_j$ by fixing the two intersection points with consecutive components $C_{j-1}$ and $C_{j+1}$, and $\Z$ acting by translating along the chain. There is a natural inclusion $P\subset \cP_v(J_a)$, and the elementary computations in \emph{loc. cit.} show that
\[
\cM_v(a)\simeq M\times^P\cP_v(a),
\]
with the natural action of $\cP_v(a)$ through the second factor. Therefore, we have 
\[
\#[\cM_v(a)/\cP_v(J_a)](k)_\kappa = \#[M/P](k)_{\kappa'}
\]
for the induced character $\kappa'\colon H^1(k,P)\hookrightarrow H^1(k,\cP_v(J_a))\to \ol\Q_\ell^\times$. By \cite[Example 8.1.15]{ngo}, we deduce that
\[
 \#[\cM^\siegel_v(a)/\cP_v(J_a)](k)_\kappa
       =\frac{q_v}{\#A_a(k_v)} = q_v\cdot \frac{\mathrm{vol}\; J_a'(\cO_v)}{\mathrm{vol}\; J_{\fe,a_\fe}(\cO_v)}.
\]
Therefore, $\#[\cM^\siegel_v(a)/\cP_v(J_a')](k)_\kappa = q\cdot\#[\cM^\fe_v(a)/\cP_v(J_a')](k)$, and Theorem \ref{thm: local main} follows in this case since $r_\fe(a)=1$.
\end{proof}

\section{The support theorem and computation of cohomology}
\label{sec: support theorem (whole)}

In this section, we recall Ng\^o's support theorem. In \S\ref{sec: support thm for relative case}, we will apply this support theorem by analyzing top degree cohomology over $\cA^\dagger$, which will demonstrate the importance of the regular quotient in relative endoscopy.

\subsection{The support theorem}

We begin by introducing some notions needed to state the support theorem.

\begin{definition}
\label{def: weak abelian fibration}
    Let $S$ be a scheme. A \emph{weak abelian fibration} over $S$ is the data of a proper morphism $f\colon M\to S$, with $M$ a Deligne--Mumford stack, together with a smooth commutative group scheme $g\colon P\to S$ acting on $M$ over $S$ such that
    \begin{enumerate}
        \item the morphisms $f$ and $g$ are of the same relative dimension $d$;
        \item for every $m\in M(k)$ mapping to $s\in S(k)$, the stabilizer of $m$ in $P_s$ is affine;
        \item the Tate module $T_{\ol{\Q}_\ell}(P^0)$ of the fiberwise neutral component $P^0$ admits a polarization. More precisely, \'etale-locally on $S$ there exists a pairing 
        \begin{equation}
        \label{eqn: polarization pairing}
        T_{\ol{\Q}_\ell}(P^0)\times T_{\ol{\Q}_\ell}(P^0)\to \ol{\Q}_\ell(1)
        \end{equation}
        such that, for any $s\in S(k)$, if
        \[
        0\to R_s\to P_s^0\to A_s\to 0
        \]
        is the Chevalley decomposition of $P_s^0$--with $R_s$ the maximal connected affine subgroup of $P_s^0$ and $A_s$ an abelian variety--then when restricted to $s$, the bilinear map \eqref{eqn: polarization pairing} factors through $T_{\ol \Q_\ell}(A_s)$ and induces a perfect pairing on $T_{\ol{\Q}_\ell}(A_s)$.
    \end{enumerate}
\end{definition}

Given a weak abelian fibration, we have a $\delta$-stratification on $S$ determined by the function $\delta(s) = \dim R_s$, where $R_s\subset P_s$ is the maximal connected affine subgroup of $P_s$ as above.

\begin{definition}
    \label{def: delta regularity}
    A weak abelian fibration $(f,g)$ on $S$ is \emph{$\delta$-regular} if for all $\delta\geq 0$, we have $\codim_S(S_\delta)\geq \delta$.
\end{definition}

We now can state Ng\^o's support theorem. 
\begin{thm}
\label{thm: ngos support thm}
(\cite[Theorem 7.2.1]{ngo})
    Let $S$ be a geometrically irreducible scheme over $k$, and let $f\colon M\to S$ and $g\colon P\to S$ form a $\delta$-regular weak abelian fibration of relative dimension $d$. Let $\cF$ be a pure and self-dual constructible sheaf on $M$\footnote{For example, we will take $\cF = \bar{\Q}_{\ell,M}$ when $M$ is smooth.}.

    Let $K$ be a simple constituent of the perverse cohomology sheaf $^p\H^n(f_*\cF)$ with support $Z$. Then $Z$ satisfies the following two properties.
    \begin{enumerate}
        \item $Z$ is $\delta$-critical, i.e. $\codim_S(Z) = \delta_Z$ where $\delta_Z$ is the minimum value of the $\delta$-invariant on $Z$.
        \item There exists a non-empty open set $U\subset S_{\ol{k}}$ with $U\cap Z\neq \emptyset$ and a local system $L$ on $U\cap Z$ such that $i_*L$ is a direct factor of $\bR^{2d}f_*\cF|_U$, the top degree ordinary cohomology sheaf restricted to $U$.
    \end{enumerate}
\end{thm}

\subsection{Application to relative Hitchin systems}
\label{sec: support thm for relative case}

We now apply Theorem \ref{thm: ngos support thm} to our cases of interest. We begin by showing that $(f^{\star,\ani},g^\ani)$ forms a weak abelian fibration.

\begin{prop}
\label{prop: h ani is a weak abelian fibration}
    The action of the group scheme $g^\ani\colon \cP^\ani\to \cA^\ani$ on $f^{\star,\ani}\colon \cM^{\star,\ani}\to \cA^\ani$ gives $(f^{\star,\ani},g^\ani)$ the structure of a weak abelian fibration.
\end{prop}
\begin{proof}
Proposition \ref{prop: smoothness and properness of relative hitchin fibration} shows that $f^{\star,\ani}$ is a proper and flat morphism  whose (equidimensional) fibers are of the same relative dimension as $g^\ani$ (see Corollary \ref{cor: ani is flat}). Moreover, the centralizers of the action of $\cP_a$ on $\cM_a$ are identified with $H^0(\Sigma,J_a)$, which is affine since $J_a$ is commutative. Finally, we note that $\cP$ has polarizable Tate module via an identical argument to \cite[\S4.12]{ngo}.
\end{proof}

As the property of being a weak abelian fibration is stable under base change, the above also implies that $(\cameral{f}^{\star,\ani},\cameral{g}^\ani)$ is a weak abelian fibration, where the underline denotes the base change to the cover $\cameral{\cA}^\ani$ defined in \S\ref{sec: component groups of P}.

The property of $\delta$-regularity is more delicate. While we expect $(\cameral{f}^{\star,\ani},\cameral{g}^\ani)$ is $\delta$-regular, do not prove this statement. Rather we will skirt the issue of $\delta$-regularity over the entire anisotropic locus by applying Corollary \ref{cor: delta regularity in the limit} and arguing for $\deg(D)$ large enough. For now, we restrict to the following locus.

\begin{definition}
\label{def: good locus}
 Fix an effective divisor $D$ on $\Sigma$. For any fixed positive integer $\delta\geq 0$, let $\cameral{\cA}_{<\delta}$ denote the union 
    \[
    \cameral{\cA}_{<\delta} = \bigsqcup_{\delta'< \delta} \cameral{\cA}_{\delta'}.
    \]
Let $\delta^\bad(D)$ denote the minimal number $\delta$ such that the inequality
    \[
    \codim_{\cameral{\cA}^\ani}(\cameral{\cA}_\delta)\geq \delta
    \]
    fails; note that $\delta^{\bad}(D)>0$. The locus $\cameral{\cA}^{\good}$ is defined as $\cameral{\cA}_{< \delta^\bad(D)}$.
\end{definition}
By definition, the restriction of $(\cameral{f}^{\star,\ani},\cameral{g}^\ani)$ to $\cameral{\cA}^\good$ is $\delta$-regular.

Recall from Definition \ref{def: dagger locus} that the dagger locus $\cA^\dagger\subset\cA^\ani$ denotes the subvariety of $\cA^{ani}$ whose $S$-points $a\colon \Sigma\times S\to \fc_D$ satisfy that the image $a(\Sigma)$ intersects the branching divisor $\fD_D^\br$ transversely. We now prove a density statement for $\cM^{\star,\reg}$ over $\cA^\dagger$. 

\begin{lem}
	\label{lem: density of regular locus}
	For any $a\in \cA^\dagger(\ol{k})$, the regular locus $\cM^{\star,\reg}_a$ is dense in $\cM_a^\star$. That is, the inclusion $\cM^{\star,\reg}\times_\cA\cA^\dagger\subset \cM^{\star,\dagger}$ is fiberwise dense.
\end{lem}
\begin{proof}
	Without loss of generality, we may assume that $k = \ol k$. By the product formula in Proposition \ref{prop: product formula}, it suffices to check the above for RASFs $\cM_v^\star(a)$ where $a\in \fc^\heartsuit(\cO_v)$ is such that either $a$ intersects transversely with $\fD^\br$ or $a$ does not intersect $\fD^\br$ at all. In the latter case, $\cM_v^\star(a)/\Gamma_a$ is connected, and so $\cM_v^\star(a)$ contains a dense $\cP_v(J_a)$-orbit which is regular. In the former case, $\cM_v^\star(a)/\Gamma_a$ has exactly two connected components. There are two non-conjugate choices of Kostant-Rallis section over $\ol k$, so each component of $\cM_v^\star(a)/\Gamma_a$ contains the regular $\cP_v(J_a)$-orbit of a choice of Kostant-Rallis section. Each such orbit is dense in its connected component in $\cM_v^\star(a)/\Gamma_a$, and so its preimage is dense in $\cM_v^\star(a)$.
\end{proof}

Recall from \eqref{eqn: kappa decomp} that we have a decomposition
\[
\bR \cameral{f}_*^{\ani}\ol\Q_\ell = \bigoplus_{\kappa\in \wh{\canC}} (\bR \cameral{f}_*^\ani\ol\Q_\ell)_\kappa,
\]
with $(\bR \cameral{f}_*^{\ani}\ol\Q_\ell)_\kappa$ being the locus on which $\mathbb{X}^*(\wh\canC)$ acts through $\cameral\cP$ by evaluation at $\kappa\in \wh\canC$. (See Proposition \ref{prop: C surjects on components of P}.) We denote the case $\kappa = 1$ by $(\bR \cameral{f}_*^\ani\ol\Q_\ell)_\st$ and call it the stable constituent.

\begin{prop}
	\label{prop: stable cohomology}
	In both the Siegel and unitary cases, the top degree stable cohomology over the anisotropic locus is a local system when restricted to $\cameral\cA^\dagger\subset\cameral\cA^\ani$. That is,
    \[
    (\bR^{\top}\cameral{f}_*^{\star,\ani}\ol{\Q}_\ell)_\st|_{\cameral\cA^\dagger} = (\bR^{\top}\cameral{f}_*^{\star,\dagger}\ol{\Q}_\ell)_\st
    \]
    is a local system on $\cameral\cA^\dagger$.
\end{prop}
\begin{proof}
	By Lemma \ref{lem: density of regular locus}, we have $\bR^{\top}\cameral{f}_*^{\star,\dagger}\ol{\Q}_\ell = \bR^{\top}(\cameral{f}^{\star,\dagger}|_{\cameral\cM^\reg})_!\ol{\Q}_{\ell}$, where in the latter we push forward the constant sheaf on $\cameral\cM^{\star,\reg}$. Proposition \ref{prop: factorization thru A reg} (cf. \cite[Thm 1.1]{HM}) gives a factorization of $f^\star|_{\cM^{\star,\reg}}$ into
	\[
	\cameral\cM^{\star,\reg}\xrightarrow{\cameral f^{\star,\reg}} \cameral\cA^{\star,\reg} \xrightarrow{\cameral\varphi^\star} \cameral\cA,
	\]
	where $\cameral f^{\star,\reg}$ is a torsor for the smooth, commutative group scheme $\cameral\cP\times_{\cameral\cA}\cameral\cA^{\star,\reg}$ over $\cameral\cA^{\star,\reg}$, and $\cameral\varphi^\star$ is finite and \'etale. On top degree cohomology, this implies 
	\begin{equation}
		\label{eqn: dagger to regular on top degree cohomology}
\bR^{\top}(\cameral{f}^{\star,\dagger}|_{\cameral\cM^\reg})_!\ol{\Q}_{\ell} \simeq \cameral\varphi^\star_*\left(\bR^{\top}\cameral{f}_!^{\star,\reg}\ol{\Q}_\ell|_{\cameral\cA^\dagger\times_{\cameral\cA}\cameral\cA^{\star,\reg}}\right).
	\end{equation}
	The group scheme $(\varphi^{\star})^*\cameral\cP$ acts on $\cameral\cM^{\star,\reg}$ over the regular base $\cA^{\star,\reg}$ and hence on the sheaf $\bR^{\top}\cameral{f}_!^{\star,\reg}\ol{\Q}_\ell$. This gives an action of $\cameral\cP$ on $\cameral\varphi^\star_*(\bR^{\top}\cameral{f}_!^{\star,\reg}\ol{\Q}_\ell|_{\cameral\cA^\dagger})$ via the unit $\cameral\cP\to \cameral\varphi_*\cameral\varphi^*\cP$. With this action, the isomorphism \eqref{eqn: dagger to regular on top degree cohomology} intertwines the action of $\cameral\cP$. In particular, we have an equality of stable components
	\[
(\bR^{\top}\cameral{f}_*^{\star,\dagger}\ol{\Q}_{\ell})_\st \simeq \cameral\varphi^\star_*((\bR^{\top}\cameral{f}_!^{\star,\reg}\ol{\Q}_\ell|_{\cameral\cA^\dagger})_\st)
	\]
	As $\cameral\cM^{\star,\reg}$ is a torsor for the action of $\cameral\cP\times_{\cameral\cA}\cameral\cA^{\star,\reg}$, the complex $(\bR^{\top}\cameral{f}_!^{\star,\reg}\ol{\Q}_\ell|_{\cameral \cA^\dagger})_\st$ is a local system on $\cameral \cA^{\star,\reg}\times_{\cameral \cA}\cameral \cA^\dagger$. Finally, since $\cameral \varphi^\star$ is finite \'etale (see Proposition \ref{prop: A reg over dagger locus}), the right-hand side is also a local system on $\cameral \cA^\dagger$, and the proposition follows.
\end{proof}

\begin{cor}
	\label{cor: supports over stable locus}
	Fix an integer $\delta$ and choose a divisor $D$ such that $ \delta^\bad(D)>\delta$ and $\deg(D)> g+\frac{3}{2}\delta$. Then, any simple constituent of the stable cohomology ${}^p\bH^n(\cameral{f}_*^{\star,\ani}\ol{\Q}_\ell)_\st$ whose support intersects $\cameral \cA_{<\delta}$ has full support.
\end{cor}
\begin{proof}
	Let $Z$ be the support of a simple constituent of ${}^p\H^n(\cameral{f}_*^{\star,\ani}\ol{\Q}_\ell)_\st$ which intersects $\cameral \cA_{<\delta}$. By Theorem \ref{thm: ngos support thm}(1) combined with $\delta$-regularity on $\cameral \cA_{<\delta}\subset \cA^\good$, $Z$ must be the closure of a component of $\cA_{\delta_Z}$. The claim is that $\delta_Z=0$.
    
    To see this, consider the open set $\cameral\cA^\dagger\subset \cameral\cA$. By Proposition \ref{prop: density of dagger locus in delta strata} and the assumption on $\deg(D)$, $\cameral \cA^\dagger\cap \cameral \cA_\delta$ is dense in each component of $\cameral \cA_\delta$. Moreover, Proposition \ref{prop: stable cohomology} shows that the top ordinary cohomology $(\bR^{\top}\cameral{f}_*^{\star,\ani}\ol{\Q}_\ell)_\st|_{\cameral \cA^\dagger}$ is a local system. Now by condition (2) of Theorem \ref{thm: ngos support thm}, there exists some $U\subset \cameral{\cA}^\dagger$ such that $U\cap Z\neq\emptyset$ and the local system $(\bR^{\top}\cameral{f}_*^{\star,\ani}\ol{\Q}_\ell)_\st|_{\cameral \cA^\dagger}$ admits a nontrivial direct factor supported in $U\cap \cA_\delta$. This is clearly a contradiction $Z$ if $\delta\geq 1$, forcing $Z$ to be the closure of $\cameral\cA_0$, which is the whole of $\cameral\cA^\ani$.
\end{proof}

We now turn to $\kappa$-isotypic components in the case $\star = \siegel$. Recall from Proposition \ref{prop: endoscopic description of components of P} that for a 2-torsion element $\kappa\in \wh{\canC}$, we have a closed set $\cameral\cA_\kappa = \cameral\nu_\fe(\cameral\cA_\fe)$, where $\fe$ is the (unique) elliptic elliptic endoscopic datum containing $\kappa$, which is the support of the sheaf $(\bR \cameral g_!^\ani\ol\Q_\ell)_\kappa$. We begin with a lemma about the density of the dagger locus in $\cameral\cA_\kappa$.

\begin{lem}
    \label{lem: dagger locus is dense in A kappa}
    Assume that $\deg(D)>2g$. The subset $\cameral\cA^\dagger_\kappa:=\cameral\cA_\kappa\cap \cameral\cA^\dagger$ is open and dense in every component of $\cameral\cA_\kappa$.
\end{lem}
\begin{proof}
    By Lemma \ref{lem: pullback of divisors along transfer map}, we have $\cameral\nu_\fe(\cameral\cA_\fe)\cap \cameral\cA^\dagger = \cameral\nu_\fe(\cameral\cA_\fe^\dagger)$. By Lemma \ref{lem: density of dagger locus}, $\cameral\cA_\fe^\dagger$ is open and dense in $\cameral\cA_\fe$, so since $\cameral\nu_\fe$ is a closed embedding, we conclude that $\cameral\cA^\dagger_\kappa$ is dense in each component of $\cameral\cA_\kappa$.
\end{proof}

In this case, the crucial calculation needed to apply the support theorem is the following.

\begin{prop}
    \label{prop: kappa piece of top degree cohomology}
    The sheaf $(\bR \cameral{f}^{\siegel,\ani}_*\ol{\Q}_\ell)_\kappa$ has support contained in $\cameral \cA_\kappa$. Moreover, the top degree cohomology $(\bR^\top \cameral{f}^{\Siegel,\ani}_*\ol{\Q}_\ell)_\kappa$ is a local system when restricted to the locally closed set $\cameral \cA_\kappa^\dagger$.
\end{prop}
\begin{proof}
    By Proposition \ref{prop: C surjects on components of P}, the action of $\mathbb{X}_*(\canC)\times \cameral\cA^\ani$ on $\bR \cameral{f}^{\siegel,\ani}_*\ol{\Q}_\ell$ is through $\pi_0(\cameral\cP/\cameral\cA^\ani)$. We have a $\pi_0(\cameral\cP/\cameral\cA^\ani)$ equivariant isomorphism
    \[
    \bR^{\top}\cameral g_!^\ani\ol \Q_\ell\simeq \ol\Q_\ell[\pi_0(\cameral\cP/\cameral\cA^\ani)](-d),
    \]
    where $(-d)$ denotes Tate twist. Therefore, by Proposition \ref{prop: endoscopic description of components of P}, $\ol\Q_\ell[\pi_0(\cP/\cA^\ani)]_\kappa$ is supported on $\cA_\kappa$, and hence, the support of $(\bR\cameral{f}^{\Siegel,\ani}_*\ol{\Q}_\ell)_\kappa$ is contained in $\cameral \cA_\kappa$. Over $\cameral \cA_\kappa^\dagger$, an identical argument to that of Proposition \ref{prop: stable cohomology} works:  Namely, we use the density of the regular locus in Hitchin fibers over $\cameral \cA^\dagger$ (Lemma \ref{lem: density of regular locus}) to reduce to the study of the top degree cohomology over the regular locus. Over this space, we use the factorization
    \[
\cameral \cM^{\siegel,\reg}\xrightarrow{\cameral{f}^{\siegel,\reg}}\cameral \cA^{\siegel,\reg}\xrightarrow{\cameral \varphi^\siegel}\cameral \cA
    \]
from Proposition \ref{prop: factorization thru A reg}, noting that the $\kappa$-isotypic piece of top degree cohomology is then a pushforward $\cameral \varphi^\siegel_*(\cL_\kappa)$, where by Proposition \ref{prop: endoscopic description of components of P}, $\cL_\kappa$ is a local system on $(\cameral \varphi^{\siegel})^{-1}(\cameral \cA_\kappa^{\dagger})$ determined by the isomorphism class of the $\cameral \varphi^{\siegel,*}\left(\cameral \cP\right)$-torsor $\cameral \cM^{\siegel,\reg}$.
\end{proof}

\begin{cor}
	\label{cor: supports over kappa locus}
	Fix an integer $\delta$, and choose the divisor $D$ on $\Sigma$ so that $\deg(D)>g+\frac{3}{2}\delta$ and $\delta^\bad(D)>\delta$. Let $Z$ be the support of an irreducible summand of $^p\bH^n(\cameral{f}^{\siegel,\ani}_*\ol{\Q}_\ell)_\kappa$. If $Z\cap \cameral\nu_{\fe}(\cameral \cA_{\fe,<\delta})\neq \emptyset$, then $Z = \cameral\nu_{\fe,*}(\cameral \cA_{\fe}^\ani)$ is the image of the transfer map for the unique endoscopic datum $\fe = (\kappa, \rho)$ extending $\kappa$.
\end{cor}
\begin{proof}
	First, we relate $\delta$-regularity in $\cA_\fe$ to $\delta$-regularity in $\cA$. Suppose that $Z\subset \cA_\kappa$ has preimage $Z_\fe=\nu_\fe^{-1}(Z)$ which satisfies the $\delta$-regularity condition
    \[
\mathrm{codim}_{\cA_\fe}(Z_\fe)\geq \delta_{\fe,Z_\fe}.
    \]
    Recall the invariant $r_v(D) = \sum_{v\in |\Sigma|}\deg(v)r_v(a_\fe)$ from Corollary \ref{cor: delta vs endoscopic delta, global}. Then, as $\nu_\fe$ is finite, we compute
    \[
\mathrm{codim}_\cA(Z) = \mathrm{codim}_{\cA_\fe}(Z_\fe)+r_v(D)\geq \delta_{\fe,Z_\fe}+r_v(D) = \delta_Z
    \]
    where the last equality follows from Corollary \ref{cor: delta vs endoscopic delta, global}. Hence, $\nu_\fe(\cA_\fe^\good)\subset \cA^\good$, and we can apply Theorem \ref{thm: ngos support thm}.
    
    The result now follows from Proposition \ref{prop: kappa piece of top degree cohomology} by an identical argument to that of Corollary \ref{cor: supports over stable locus}, the key ingredient being Proposition \ref{prop: density of dagger locus in delta strata}, which shows in particular that, given the assumptions on $\deg(D)$, $\cameral \cA_{\kappa}\cap \cameral\cA_{\delta}^\dagger\subset \cameral \cA_{\kappa}\cap \cameral\cA_{\delta}$ is open and dense.
\end{proof}

\section{Proof of the Main Theorems}
\label{sec: proof of main}

In this section, we prove the main results of the paper, beginning with the geometric stabilization over the locus $\cA^{\good}$ (Theorem \ref{thm: global main}), and then using this result to obtain the local geometric statement Theorem \ref{thm: local main}.

\subsection{Global Stabilization}

We prove the main global result:

\begin{thm}
    \label{thm: global main}
    Fix a divisor $D$ on $\Sigma$, and put $\delta = \min\{\delta^\bad(D),\frac{2}{3}(\deg(D) - g)\}$; we choose $D$ so that $\delta>0$. For any 2-torsion element $\kappa\in \wh{\canC}[2]$, we have an isomorphism of complexes in $D^b(\cameral\cA_{\fe,<\delta})$ over $\ol{k}$
    \[
    \nu_\fe^*(\bR \cameral{f}^{\Siegel,\ani}_*\ol{\Q}_\ell)_{\kappa}[2r_\fe(D)](r_\fe(D)) =  (\bR\cameral{f}^{\fe,\ani}_*\ol{\Q}_\ell)_{\st}
    \]
     (restricted to the locus $\cameral \cA_{\fe,<\delta}$) which induces an isomorphism of semi-simplifications over $k$.
\end{thm}
\begin{proof}
    Since $\cameral{f}^\siegel$ and $\cameral{f}^\fe$ are proper over $\cameral \cA^\ani$ and the spaces $\cameral \cM^{\Siegel,\ani}$, $\cameral \cM^{\fe,\ani}$ are smooth by Proposition \ref{prop: smoothness and properness of relative hitchin fibration}, both sides of the proposed identity are pure and hence split as a direct sum of their perverse cohomology sheaves. The theorem reduces to the identity
    \[ \bigoplus_n\nu_\fe^*\big({}^p\bH^n(\bR\cameral{f}^{\Siegel,\ani}_*\ol{\Q}_\ell)_{\kappa}[2r_\fe(D)](r_\fe(D))\big)[-n] =  \bigoplus_n {}^p\bH^{n}(\bR\cameral{f}^{\fe,\ani}_*\ol{\Q}_\ell)_{\st}[-n]
    \]
    Moreover, by the decomposition theorem of Beilinson, Bernstein, Deligne, and Gabber \cite{BBD}, the perverse cohomology groups on either side are semisimple. By Corollary \ref{cor: supports over kappa locus} and our assumption on $\deg(D)$, it is enough to show there is an isomorphism after restricting both sides to an open set $\cU\subset \cA_\fe^\dagger$, where both sides are local systems. Namely, let $\cU$ denote the open subset of $\cameral \cA_\fe^\dagger$ whose $\ol k$-points are given by $a\in \cameral \cA_\fe^\dagger(\ol k)$ for which $a(\Sigma)$ intersects $\fD_\fe^\cent$ transversely, and $a^{-1}(\fD^\cent_\fe)$ and $a^{-1}(\fD_\fe^\br)$ are disjoint divisors in $\Sigma$. Both sides are local systems on $\cameral \cA_\fe^\dagger$ by Propositions \ref{prop: stable cohomology} and \ref{prop: kappa piece of top degree cohomology}.
    
    Denote
    \[
    L_\kappa^n = \left.\nu_\fe^*\big({}^p\bH^n(\bR\cameral{f}^\Siegel_*\ol{\Q}_\ell)_{\kappa}[2r_\fe(D)](r_\fe(D))\big)\right|_{\cU}\quad \text{ and }\quad L_{\fe,\st}^n = \left. {}^p\bH^{n-2r_\fe(D)}(\bR\cameral{f}^\fe_*\ol{\Q}_\ell)_{\st}\right|_{\cU}
    \]
    which are local systems by Propositions \ref{prop: stable cohomology} and \ref{cor: supports over kappa locus}. By Chebotarev density theorem and purity of the pushforwards $\bR f_*^{\star,\ani}\ol\Q_\ell$, it suffices to show that for all finite extensions $K$ over $k$ with Frobenius $\sigma_K$ and for all $a\in \cU(K)$,
    \begin{equation}
    \label{eqn: GL equation}
   q^{-r_\fe(D)\cdot [K:k]}\cdot \sum_n (-1)^n\tr(\sigma_K,(L_\kappa^n)_a) = \sum_n (-1)^n\tr(\sigma_K,(L_{\fe,\st}^n)_a)
    \end{equation}
    Without loss of generality, we assume that $K = k$. Choose a smooth group scheme $J_a'$ with connected fibers equipped with a map $J_a'\to J_a$ which is an isomorphism on generic fibers, and let $\cameral\cP'_a$ be the commutative group stack classifying $J_a'$-torsors on $\cameral\cA\times \Sigma$. 
    Recall that $U_a\subset \Sigma$ denotes the preimage of $\fc_D^\rs$ under $a$. Then, using Corollary \ref{cor: Grothendieck-Lefschetz for Hitchin fibers}, we compute that the left hand side of \eqref{eqn: GL equation} is
    \[
    \#(\cameral \cP_{a}')^\circ(k)\prod_{v\in |\Sigma - U_a|}\#[\cM_v^\siegel(a_v)/\cP_v(J_{a_v}')](k_v)_\kappa,
    \]
    while the right hand side is
    \[
    q^{r_\fe(D)}\cdot\#(\cameral \cP_a')^\circ(k)\prod_{v\in |\Sigma - U_a|}\#[ \cM_v^\fe(a_{\fe,v})/ \cP_v(J_{a}')](k_v)_\st.
    \]
    Note that the extra power of $q$ in the right hand side is from the Tate twist. Since $r_\fe(D) = \sum_v \deg(v)r_{\fe,v}(a_\fe)$ can be computed locally, the identity will therefore follow from the cases of Theorem \ref{thm: local main} given in Propositions \ref{prop: obstructed case}, \ref{prop: zero dimensional cases}, and \ref{prop: 1 dimensional case}; namely, those $a_\fe\in \fc^\heartsuit_\fe(\cO_v)$ and $a\in \fc^\heartsuit(\cO_v)$ which intersect the divisor $\fD^\br$ transversely but do not intersect $\fD^\cent$ ($d^\br_v(a)\leq 1$ for all $v$ and $d^\cent_v(a) =d^\cent_v(a_\fe) = 0$) and those $a_\fe,a$ which do not intersect the divisor $\fD^\br$ and for which $d^\cent_v(a_\fe)\leq 1$. Theorem \ref{thm: global main} therefore follows.
\end{proof}

\subsection{Proof of Theorem \ref{thm: local main}}

In this section, we prove Theorem \ref{thm: local main}, which immediately implies Theorem \ref{thm: main} via Corollary \ref{cor: point counting for P(J)}.

Fix a place $v\in |\Sigma|$ with residue field $k_v = k$ such that there is a unique place $v'\in |\Sigma'|$ over $v$. Let $\fe = (\kappa,\rho)$ be an endoscopic datum over $F_v = k(\!(\varpi_v)\!)$ and $a_{\fe,v}\in \fc_\fe^\heartsuit(\cO_v)$ with image under the transfer map $\nu_\fe$ given by $a_v$. We wish to prove Theorem \ref{thm: local main} for the pair $a_{\fe,v},a_v$. By Proposition \ref{prop: obstructed case}, we can assume that $\Ob_v(a_{\fe,v}) = 0$. Fix $N\gg 0$ such that Proposition \ref{prop: local constancy} holds for both $a_v$ and $a_{\fe,v}$, and fix a line bundle $D$ on $\Sigma$ with 
\[
\deg(D)\geq 2g+nN+1
\]

\begin{prop}
	\label{prop: existence of interpolation point}
	There exists an integer $N'$ such that, for all finite extensions $K$ of $k$ of degree greater than $N'$, there exists a point $a_\fe\in \cA_\fe(K)$ such that 
	\begin{enumerate}
		\item $a_\fe\equiv a_{\fe,v}\mod \varpi_v^N$;
		\item the support of the divisors $a_\fe^{-1}(\fD_D^\cent)$ and $a_\fe^{-1}(\fD_D^{\br})$ intersects non-trivially at most at the place $v$ in $\Sigma$;
		\item $\Ob_u(a_\fe)$ is trivial for all places $u\in |\Sigma|$;
		\item $a_\fe\in \cA_\fe^\ani(K)$;
		\item $\delta_\fe(a_\fe) = \delta_{\fe,v}(a_{\fe,v})$.
	\end{enumerate}
\end{prop}
\begin{proof}

	We recall how the obstruction is computed on the endoscopic base. We denote ranks of symmetric varieties by subscripts. Write $X^\fe = X^\siegel_r\times X^\unitary_s$ with $r+s = n$, so that $\cA_\fe = \cA_r\times \cA_s$, where the first factor is the Hitchin base of the Siegel factor and the second that of the unitary factor. By Lemma~\ref{lem: obstruction is always trivial for siegel}, for every non-split place $u$ we have
	\[
	\Ob_u(a_\fe) = \val_u\!\big((a_\fe^\unitary)_{2s}\big)\bmod 2.
	\]
	where $a_\fe = (a_\fe^\siegel,a_\fe^\unitary)$ and $(a_\fe^\unitary)_{2s}$ is the determinant coordinate of $a_\fe^\unitary$. In particular, condition~(3) holds for \emph{every} place $u$ whenever $a^\unitary_{2s}$ is in the image of the norm map
	\begin{equation}
		\label{eqn: norm map on det}
\mathrm{Nm}\colon H^0(\Sigma',\vartheta^*D^{\otimes s})	\to H^0(\Sigma,D^{\otimes 2s}),\quad  t\mapsto t\cdot \sigma(t),
	\end{equation}
	which motivates the following definition.
	
	\newcommand{\sq}{\mathrm{sq}}
	\newcommand{\Sq}{\mathrm{Sq}}
	\newcommand{\Nm}{\mathrm{Nm}}

	\begin{definition}
		\label{defn: square locus}
		Let $\Nm \subset H^0(\Sigma, D^{\otimes 2s})$ denote the image of the norm map \eqref{eqn: norm map on det}. We define the \emph{norm locus}
		\[
		\cA_s^{\Nm} := \big\{\, (a^\unitary_2,\dots,a^\unitary_{2s}) \mid a^\unitary_{2s}\in \Nm \,\big\}\subset \cA_s
		\]
		and $Q := \cA_r \times\cA_{s}^{\Nm} \subset \cA_\fe$.
	\end{definition}

	\begin{lem}
		\label{lem: complement of ani locus}
		Let $Q^\ani = Q\cap \cA_\fe^\ani$. The codimension of $Q\setminus Q^\ani$ in $Q$ is greater than or equal to $\deg(D)+1-g$.
	\end{lem}
	\begin{proof}
		Note that if $X^\fe = X^\siegel$, then we are done by Lemma \ref{lem: codim of anisotropic}. Therefore, as $X^\fe$ splits as a product and $Q$ only imposes conditions on the unitary factor, it suffices to prove the lemma for $X^\fe = X^\unitary$. We will drop subscripts $\fe$ and work with $X^\unitary$ throughout. Recall from \S\ref{sec: unitary Higgs} that we denote the Hitchin base for the unitary group $\U_{2n}$ by $\cB = \cB_{2n}$. We have a natural inclusion $\cA\to \cB$ which identifies $\cA$ with the fixed points of the involution $\theta(a_i) = (-1)^ia_i$ on $\cB$. By definition of the anisotropic locus, $\cA^\ani = \cA\cap \cB^\ani$.
		
		Fix a $\theta$- and $\sigma$-stable torus $T$ of $\U_{2n}$ which is maximally $\theta$-split; that is, for which the dimension of the subtorus on which $\theta$ acts by inversion is maximal. Note that there is a unique $H^{\unitary}(F)$-conjugacy class of such tori \cite{HelminckWang}. For every Levi subgroup $M$ of $\U_{2n}$ containing $T$, we associate a subvariety $\cB_M\subset \cB$ consisting of those points in $\cB = H^0(\Sigma,(\mathrm{u}_{2n}\git\U_{2n})_{\cL(D)})$ which factor through $\fm\git M\to \mathrm{u}_{2n}\git\U_{2n}$. We note that the subvarieties $\cB_M$ depend only on the conjugacy class of $M$ in $\U_{2n}$. Since $\U_{2n}$ is a quasisplit group scheme over $\Sigma$, \cite[Proposition 6.3.5]{ngo} implies that
		\[
		\cB\setminus \cB^\ani = \bigcup_{T\subset M\subset\U_{2n}} \cB_M
		\]
		where the union ranges over all proper Levi subgroups $M$ containing $T$.
		
		Put $\cA_M = \cA\cap \cB_M$. It now suffices to show $\mathrm{codim}_{Q}(Q\cap \cA_M)>\deg(D)+1-g$. All Levi subgroups $M$ are of the form
		\[
		M = \U_{2m_0}\oplus \prod_{i=1}^\ell \mathrm{Res}_{\Sigma'/\Sigma}(\GL_{m_i})
		\]
		where $\sum_{i=0}^\ell m_i = n$. Note in particular, $\cA_M$ is the image of a product of Hitchin bases for $\U_{2m_0}/(\U_{m_0}\times \U_{m_0})$ and $\mathrm{Res}_{\Sigma'/\Sigma}(\GL_{m_i})/\U_{m_i}$. In particular, using Riemann-Roch, one can compute
		\begin{align*}
			\dim(\cA_M) &= m_0(m_0+1)\deg(D)+m_0(1-g)+\sum_{i=1}^\ell\left[ \frac{m_i(m_i+1)}{2}\deg(D)+m_i(1-g) \right] \\
			& = \left[ m_0(m_0+1)+\sum_{i=1}^\ell\frac{m_i(m_i+1)}{2} \right]\deg(D)+n(1-g).
		\end{align*}
		It is easy to check that, for fixed $\ell$, the above quantity is maximized when $m_0 = n-\ell$ and $m_i=1$ for all $i\geq 1$, in which case we obtain
		\begin{equation}
			\label{eqn: dim of AM}
			\dim(\cA_M) \leq  \left[ (n-\ell)^2+n\right]\deg(D)+n(1-g).
		\end{equation}
		We also compute by Riemann-Roch that 
		\begin{equation}
			\label{eqn: dim of Q}
			\dim(Q) = \sum_{i=1}^{n-1} (2id+1-g)+2nd+2(1-g)
			= n(n+1)\deg(D)+(n+1)(1-g).
		\end{equation}
		Hence, we have
		\begin{align*}
			\mathrm{codim}_Q(Q\cap \cA_M) &= \dim(Q) - \dim(\cA_M\cap Q) \geq \dim(Q) - \dim(\cA_M) \\
			&  = [2\ell n - \ell^2]\deg(D)+1-g \geq n^2\deg(D)+1-g \\
            & \geq \deg(D)+1-g
		\end{align*}
		where the second to last inequality follows since $2\ell n - \ell^2$ is maximized for $1\leq \ell\leq n$ when $\ell = n$. \qedhere
	\end{proof}
	
	Now, again by Riemann--Roch and our assumption on $\deg(D)$, we have a surjective evaluation map
	\begin{equation}
		\label{eqn: evaluation at v}
		\cA_\fe= H^0(\Sigma,\fc_{\fe,D})\to H^0(\Sigma_v,\fc_\fe(\cO_v/\varpi_v^N))	
	\end{equation}
	Let $Z$ be the preimage of $a_{\fe,v}$, where we view $a_{\fe,v}$ as an element in $H^0(\Sigma_v,\fc(\cO_v/\varpi_v^N))$. By surjectivity of the linear map \eqref{eqn: evaluation at v}, the codimension of $Z$ in $\cA$ is $nN$. As a consequence, we have
	\[
	\codim_Q(Q\cap Z)\leq nN.
	\]
	We now impose condition (5) by using a transversality argument. Namely, let $(Q\cap Z)^{\min}\subset Q\cap Z$ denote the subvariety of points $a$ satisfying $d^\cent_{\fe,u}(a) \leq 1$ for all $u\neq v$. By an identical argument to Corollary \ref{cor: density of valuation 2 locus in delta strata}, to prove $(Q\cap Z)^{\min}$ contains an open, dense subset of $Q\cap Z$, it suffices to prove the following.
	
	\begin{lem}
		Fix a place $u\neq v$, and let $Q_u^{(j)}\subset \fc(\cO_u/\varpi_u^j)$ denote the subset of tuples $(a_2,\dots, a_{2n})\in \fc(\cO_u/\varpi_u^j)$ such that $a_{2n}\in \cO_u/\varpi_u^j$ is the norm of an element in $\cO_{u'}/\varpi_{u}^j$. With the assumptions on $\deg(D)$ above, the product of local evaluation maps at $v$ and $u$ gives a surjection
		\[
		Q\to Q_v^{(N)}\times Q_u^{(2)}\subset \fc(\cO_v/\varpi_v^N)\times \fc(\cO_u/\varpi_u).
		\]
	\end{lem}
	\begin{proof}
		The map $Q\to Q_v^{(N)}\times Q_u^{(2)}$ is identified with the map on $H^0$ induced by map
		\[
		0\to \left(\vartheta_*\vartheta^*D^{\otimes n}\oplus \bigoplus_{i=1}^{n-1} D^{\otimes 2i}\right)\otimes \cO(-Nv-2u) \to D^{\otimes n}\oplus \bigoplus_{i=1}^{n-1} D^{\otimes 2i}\to \left.D^{\otimes n}\oplus \bigoplus_{i=1}^{n-1} D^{\otimes 2i}\right|_{Nv+2u}\to 0
		\]
		so the cokernel is
		\[
		H^1\!\!\left( \Sigma,\left(D^{\otimes n}\oplus \bigoplus_{i=1}^{n-1} D^{\otimes 2i}\right)\otimes \cO(-Nv-2u) \right),
		\]
		which vanishes for $\deg(D)\geq 2g+nN+1$.
	\end{proof}
	
	We now have
	\[
	\codim_Q(Q\cap Z)^{\min}\leq nN.
	\]
	In particular, by our assumptions on the degree of $D$ and Lemma \ref{lem: complement of ani locus}, we have
	\[
	\codim_{Q}(Q\cap Z)<\codim_{Q}(Q\setminus Q^\ani)
	\]
	so that the locally closed set $Q^\ani\cap(Q\cap Z)^{\min}$ is nonempty. Since this is a nonempty locally closed subvariety, by Lang-Weil there is an integer $N'$ such that it has $K$ points for any finite extension $K$ of $k$ with $[K:k]\geq N'$. In particular, if we choose any $a_\fe\in (Q^\ani\cap(Q\cap Z)^{\min})(K)$, then by Proposition \ref{prop: local constancy}, its local $\delta$ invariant at $v$ agrees with $\delta_{\fe,v}(a_{\fe,v})$ and by Proposition \ref{prop: dim of RASF} transversality implies $\delta_{\fe,u}(a_{\fe,u})=0$. Hence, $a_\fe$ satisfies condition (5). The other conditions are clear from our choice, so $a_\fe$ is a $K$ point satisfying the assumptions of the proposition.
\end{proof}

Let $a_\fe$ be as in Proposition \ref{prop: existence of interpolation point}, and $a = \nu_\fe(a_\fe)$. Let $\cameral a, \cameral a_\fe$ be lifts in $\cameral \cA^\ani(K)$, respectively $\cameral \cA^\ani_\fe(K)$. By condition (5), we may also assume that $\deg(D)> g+\frac{3}{2}\delta_\fe(a_\fe)$ and that $\deg(D)$ is large enough so that 
\[
\delta_\fe(a_\fe) = \delta_{\fe,v}(a_{\fe,v})<\delta^\bad(D)
\]
Hence, we assume that $\cameral\cA_{\fe,< \delta(a)}\subset \cA_\fe^\good$. Therefore, Theorem \ref{thm: global main} holds over $\cameral\cA_{\fe,< \delta_\fe(a_\fe)}$. Then, by the point counting formula for the product formula (Proposition \ref{prop: local to global point count}) combined with Theorem \ref{thm: global main}, we have
\begin{equation}
	\label{eqn: point count product}
	\#(\cP'_a)^\circ(K)\cdot \prod_{v\in |\Sigma_K-U_a|}\#[\cM_{v}^\siegel(a_v)/\cP(J_{a_v}')](K_v)_\kappa = q^{r_\fe(D)[K:k]} \cdot\#(\cP'_a)^\circ(K)\cdot \prod_{v\in |\Sigma_K-U_a|}\#[\cM_{v}^\fe(a_{\fe,v})/\cP(J_{a_v}')](K_v).
\end{equation}

Note here that the extra power of $q$ appears in the Tate twist in Theorem \ref{thm: global main}. Moreover, by the cases in Propositions \ref{prop: zero dimensional cases} and \ref{prop: 1 dimensional case}, for each of the places $u\neq v$ of $\Sigma$, we have identities
\[
\#[\cM_{u}^\siegel(a_u)/\cP_u(J_{a_u}')](K_u)_\kappa = q^{r_\fe(a_{\fe,u})\cdot [K:k]}\#[\cM_{u}^\fe(a_{\fe,u})/\cP_u(J_{a_{\fe,u}}')](K_u)\neq 0.
\]
By canceling these terms from \eqref{eqn: point count product}, we get the identity
\[
\#[\cM_{v}^\siegel(a_v)/\cP_v(J_{a_v}')](K)_\kappa = q^{r_\fe(a_{\fe,v})\cdot [K:k]}\#[\cM_{v}^\fe(a_{\fe,v})/\cP_v(J_{a_v}')](K).
\]
Since this holds for every finite extension $K$ of $k$ of degree greater than $N'$, we can conclude the same for $k$, which proves Theorem \ref{thm: local main}. \qed\\

It is now simple to deduce Theorem \ref{thm: global main} over $\cA_\fe^\ani$ (without the restriction to $\cA^\good$) by reversing the flow of information from local back to global.

\begin{thm}
    \label{thm: global main over ani}
    Fix any divisor $D$ on $\Sigma$ with $\deg(D)\geq 2g$. For any 2-torsion element $\kappa\in \wh{\canC}[2]$, we have an isomorphism of complexes in $D^b(\cameral\cA_{\fe}^\ani)$ over $\ol{k}$
    \[
    \nu_\fe^*(\bR \cameral{f}^{\Siegel,\ani}_*\ol{\Q}_\ell)_{\kappa}[2r_\fe(D)](r_\fe(D)) =  (\bR\cameral{f}^{\fe,\ani}_*\ol{\Q}_\ell)_{\st}
    \]
     which induces an isomorphism of semisimplifications over $k$.
\end{thm}
\begin{proof}
    By an argument identical to Theorem \ref{thm: global main}, using Chebotarev density theorem and Corollary \ref{cor: Grothendieck-Lefschetz for Hitchin fibers}, it suffices to prove Theorem \ref{thm: local main}. Having proved this, the global isomorphism follows on all of $\cameral\cA^\ani$.
\end{proof}

\begin{appendix}

\section{Transfer factors and Descent}\label{Appendix A}
In this appendix, we introduce the transfer factors needed for the comparison of $\kappa$-orbital integrals on $\mathrm{X}^\Siegel$ and show that they descend to the infinitesimal level of $V^\Siegel$. We then sketch the transfer from characteristic zero to positive characteristic via motivic integration, where we finally make contact with the formulation of main document.

More precisely, we take as the starting point that Theorem \ref{thm: main_intro} may be transferred to characteristic zero local fields for sufficiently high residue characteristic; see Section \ref{Sec: on transfer to characteristic zero}.

Let $\mathcal{O}_F$ (resp. $\mathcal{O}_{F'}$) denote the ring of integers in $F$ (resp. $F'$). Let $V_n$ be a \emph{split} Hermitian space of dimension $n$, and  set $\mathrm{U}_{n}:=\mathrm{U}(V_n)$. This affords the self-dual lattice $\Lambda_n=\mathcal{O}_{F'}^n\subset V_n$. Note that we may fix forms so that $\Lambda_{2n} = \Lambda_n\oplus \Lambda_n$. On the other hand, the self-dual lattice $\Lambda_{2n}$ admits a polarization
\[
\Lambda_{2n} =\mathcal{L}\oplus \mathcal{L}^\vee, \qquad \mathrm{rank}(\mathcal{L}) = n.
\]
With respect to these integral structures, we obtain hyperspecial subgroups
\[
\rH_n^{\Siegel}(\mathcal{O}_F):=\GL(\mathcal{L})\subset \rH_n^{\Siegel}(F),\quad \rH^{\unitary}_n(\mathcal{O}_F):=\U(\Lambda_n)\times\U(\Lambda_n)\subset \rH^{\unitary}_n(F),
\]and 
\[
\mathrm{G}_n(\mathcal{O}_F):=\U(\Lambda_n\oplus \Lambda_n)\subset \mathrm{G}_n(F).
\]
Set $\mathrm{X}^\star:=\mathrm{G}_n/\rH^\star_n$ and let $\mathbf{1}_{\mathrm{X}^\star(\mathcal{O}_F)}$ to be the associated characteristic function. 

Recall from Section \ref{sec: endoscopic spaces} and Proposition \ref{prop: endoscopic description of components of P}, that an unramified elliptic relative endoscopic datum $(\kappa,\rho)$ for $\mathrm{X}^{\Siegel}$ is determined uniquely up to equivalence by $\kappa\in \hat{\mathbf{C}}[2]$. This determines the variety
    \[
    X^\fe = H^\fe\backslash G^\fe = \big(\Res_{F'/F}\GL_r\backslash \U_{2r}\big)\times \big((\U_s\times \U_s)\backslash \U_{2s}\big).
    \]
    {over $F$, }the \emph{endoscopic symmetric space} associated to $\fe$ as in Definition \ref{def: end space}. Note that we also have the smooth $\mathcal{O}_F$-scheme 
    \[
    \mathrm{X}^\fe = \rH_r^{\Siegel}\backslash \mathrm{G}_{2r}\times \rH_s^{\unitary}\backslash \mathrm{G}_{2s},
    \]
  {with special fiber $X^\fe$ }and the corresponding characteristic function $\mathbf{1}_{\mathrm{X}^\fe(\mathcal{O}_F)}$ 
    
For a regular semi-simple element $x\in X^{\Siegel}(F)$, the endoscopic datum determines a character $\kappa: H^1(F,\rH^{\Siegel}_x)\to \mathbb{C}^\times$ of the Galois cohomology of the stabilizer $\rH^{\Siegel}_x\subset \rH^\Siegel$.  The set $\mathcal{O}_{\st}(x)$ of rational orbits inside the stable orbit of $x$ is naturally a torsor for this group, and we consider the  $\kappa$-orbital integral
\[
\Orb^{\Siegel,\kappa}_x(\mathbf{1}_{\mathrm{X}^\star(\mathcal{O}_F)})=\sum_{[x']\in \mathcal{O}_{\st}(x)}\langle\kappa,\mathrm{inv}(x,x')\rangle\Orb^{\Siegel}_x(\mathbf{1}_{\mathrm{X}^\star_n(\mathcal{O}_F)}),
\]
where $\mathrm{inv}(x,x')$ is the cohomological invariant associated to the rational orbit of $x'$. 

In Sections \ref{Section: transfer factor Lie}  and \ref{Section: transfer factor variety}, we recall the notion of the matching of regular semi-simple elements
\[
x\in X^{\Siegel,\rs}(F)\:\:\text{and}\:\:x_{\fe}\in{X}^{\fe,\rs}(F),
\]
and introduce transfer factors 
\[
\Delta_{rel}:{X}^{\fe,\rs}(F)\times X^{\Siegel,\rs}(F)\to \mathbb{C}
\]
to facilitate the matching $\kappa$-orbital integrals on $X^{\Siegel}(F)$ with stable orbital integrals on ${X}^{\fe}(F)$. We then prove following result.
\begin{thm}\label{Thm: fundamental lemma intro}
Let $F$ be a $p$-adic field and assume that the residue characteristic of $F$ is large enough that the statement of Corollary \ref{Cor: transfer} holds. Let $\kappa\in \widehat{\mathbf{C}}[2]$ and let $X^\fe$ denote the associated endoscopic variety. For any regular semi-simple $x\in X_n^{\Siegel}(F)$ and matching element $x_{\fe}\in X^{\fe}(F)$,we have
\begin{equation}\label{first identity}
    \Delta_{rel}(x_\fe,x)\Orb^{\Siegel,\kappa}_x(\mathbf{1}_{\mathrm{X}^\Siegel(\mathcal{O}_F)})=\SO^\fe_{x_{\fe}}(\mathbf{1}_{\mathrm{X}^\fe(\mathcal{O}_F)}).
\end{equation}
\end{thm}
This is the fundamental lemma on the ``group'' or variety level, which is what is actually required for the stabilization of the trace formula considered in \cite{LXZrtf}. We prove it by imitating the descent argument of \cite{Lesliedescent} to show that it follows from Theorem \ref{thm: main_intro} and Corollary \ref{Cor: transfer}. More precisely, we will introduce the transfer factors for the Siegel variety $X^\Siegel$ and its infinitesimal model $V^\Siegel$, then sketch the analog of the descent argument of \cite{Lesliedescent} and \cite[Appendix B]{LXZfund} in this setting. As the Siegel variety is a rational form of the unitary and linear cases considered in these previous works, the underlying tools and steps are very similar. The main obstacle we overcome is the definition of the transfer factors in this setting, which is more involved in this case than in the unitary setting of \cite{Lesliedescent,spencer_annals}. We therefore emphasize the descent of the transfer factors, leaving the step by step reduction on the compliment of the very regular loci (see Section \ref{Sec: descent} below) to the interested reader.

For the rest of the appendix, we assume that $F'/F$ is a quadratic extension of non-archimedean local fields, and let $V^{\Siegel}\simeq \Herm_n\times \Herm_n$ denote the infinitesimal symmetric variety. We recall from Section \ref{sec: invariant theory} the invariant quotient
\[
\chi=\chi_n: V^\siegel\longrightarrow \fc\simeq \A^n.
\]
For an endoscopic datum $\fe$ with $X^\fe = X^S_a\times X^\U_b$ $(a+b=n)$ the associated  endoscopic variety, we set
\[
 V^\fe = V^\Siegel_a\times V^\unitary_b, 
 \]and use the notation $\chi_\fe = \chi_{a} \times \chi_b$ for the invariant quotient {map}.
 
\subsection{The product map and Kostant sections}

 We first recall the notion of Jacquet--Langlands transfer between two non-isomorphic Hermitian spaces $W$ and $W'$. If we identify the underlying vector spaces (but not necessarily the Hermitian structures)
\begin{equation*}
  W\cong W'\cong {F'}^n,  
\end{equation*}
we have embeddings 
\[
\Herm(W),\:\Herm(W')\hookrightarrow\fgl_n(F').
\]
Then $\delta\in \Herm(W)$ and $\delta'\in \Herm(W')$ are said to be \textbf{Jacquet--Langlands transfers} if they are $\GL_n(F')$-conjugate in $\fgl_n(F')$.  This is well defined since the above embeddings are determined up to $\GL_n(F')$-conjugacy. Note that if $\delta$ and $\delta'$ are Jacquet--Langlands transfers, then
\[
\delta'=\mathrm{Ad}(g)(\delta)
\]
for some $g\in \GL(W)$ and we obtain a well-defined
\[
\mathrm{inv}(\delta,\delta')=[\sigma\in \Gal(\overline{F}/F)\mapsto g^{-1}\sigma(g)]\in H^1(F,T_\delta),
\]
where $T_\delta\subset \U(W')$ is the usual centralizer.  For a given $\delta\in \Herm(W)$, let
\[
\mathcal{O}_{JL}(\delta) = \mathcal{O}_{W}(\delta)\sqcup \mathcal{O}_{W'}(\delta)\subset \fgl_n(F')
\] be the union of all the Jacquet--Langlands transfers of $\delta$ in $\Herm(W)\cup \Herm(W')$, so that $\mathcal{O}_{W}(\delta)$ is the stable $\mathrm{U}(W)$-orbit (consisting of $\frac{1}{2}\#H^1(F,T_\delta)$ rational orbits). 

\subsubsection{Retraction maps}

Now consider the ``product map''
\begin{align*}
    r:\Herm_{n}\times\Herm_n&\longrightarrow \Res_{F'/F}(\fgl_{n})\\
                (x_1,x_2)&\longmapsto x_1x_2.
\end{align*}
This intertwines the $\rH$-action with the usual conjugacy action of $\rH$ on $\Res_{F'/F}(\fgl_n)\supset \Herm_n$. We fix a $F$-rational Kostant slice 
\begin{equation}\label{eqn: fixed slice}
    \epsilon: \fc\to \Herm_n
\end{equation} for the $\mathrm{U}_n$-action on $\Herm_n$.

\begin{lem}
     For each regular semi-simple $x=(x_1,x_2)\in  \Herm_{n}(F)\times\Herm_n(F)$,  there exists a $g\in \rH(F) = \GL_n(F')$ such that $y=g\cdot x = (y_1,y_2)$ satisfies that 
     \[
     r(y) = y_1y_2\in \Herm_n(F).
     \]
\end{lem}
\begin{proof}
Setting $a= \chi(x)$, consider $x'=(I,\epsilon(a))\in V^\siegel(F)$. Then clearly $r(x') \in \Herm_n(F)$ and $\chi(x) = \chi(x')$. This forces $gr(x)g^{-1} = \epsilon(a)$ for some $g\in\GL_n(F)$, so that $y = g \cdot x$ works.
\end{proof}
In fact, the argument shows more. Given $x\in V^{\Siegel}(F)$ such that $r(x)\in \Herm_n(F)$, then 
\[
r\left(\rH(F)\cdot x\right)\supset \mathcal{O}_{JL}(r(x)),
\] so that the $\rH(F)$-orbit of $x$ in $V^{\Siegel}(F)$ surjects onto the full $H^1(F,\rH_x)$-torsor of rational orbits $\Herm(W),\:\Herm(W')\subset\fgl_n(F')$ under the appropriate unitary groups. In the following, we will denote a representative of this class via $r^{\Siegel}(x)$. This will only be applied to constructions depending only on the stable orbit of $r^{\Siegel}(x)$.

We also recall the retraction map in the unitary case from \cite[Section 2.1]{spencer_annals}. For this we recall that the general rational form in the unitary setting is given by taking two $n$-dimensional Hermitian spaces $W_1$ and $W_2$ with Hermitian forms $\langle\cdot,\cdot\rangle_{i}$ for $i=1,2$, and setting $W= W_1\oplus W_2$ to be our $2n$-dimensional Hermitian space. Then
\[
H^\unitary = \U(W_1)\times \U(W_2)\subset \U(W)=G,
\] and $V^\unitary = \Hom(W_2,W_1)$ with the natural pre and post composition action of $H^\unitary$. The contraction map is the map
\begin{align*}
    r^\unitary:V^\unitary= \mathrm{Res}_{F'/F}\fgl_n&\longrightarrow \Herm(W_1)\\
                                    x &\longmapsto -x x^\ast,
\end{align*}
where $x^\ast\in \Hom(W_1,W_2)$ satisfies that for all $v\in W_2$ and $w\in W_2$ 
\[
\langle xv,w\rangle_1 =\langle v, x^\ast w\rangle_2. 
\]
Then $r^\unitary$ intertwines the $H^\unitary$-action on $V^\unitary$ with the adjoint $\U(W_1)$-action on $\Herm(W_1)$.
\subsubsection{Lifting the Kostant section}
Set $\tilde{\epsilon}(a) =(I,\epsilon(a))\in V^{\Siegel,\reg}(F)$. Below, the symbol $\sim_{\st}$ means stably conjugate under the $\rH$-action.
\begin{lem}\label{Lem: section on Lie}
    For each regular semi-simple $x=(x_1,x_2)\in  \Herm_{n}(F)\times\Herm_n(F)$, if $a=\chi(x)$, then $x\sim_{\st} \tilde{\epsilon}(a)$.  Moreover, there is a unique $\rH(F)$-orbit in $\mathcal{O}_{\st}(x)$ meeting $\tilde{\epsilon}(\fc(F))$.
\end{lem}
\begin{proof}
It is clear that $\chi(I,\epsilon(a))=\chi(x)$, so the first claim follows. The second is clear from the property that $\tilde{\epsilon}: \fc\to V^{\Siegel, \reg}$ is a section of $\chi:V^{\Siegel, \reg}\to \fc$. 
\end{proof}

For any regular semi-simple $a\in \fc(F)$, set $\mathcal{O}(a):=\chi^{-1}(a)(F)\subset V^\Siegel(F)$ and $\rH_a=\rH_{\tilde{\epsilon}(a)}$. For any $x\in \mathcal{O}(a)$, we set
\begin{equation}
   \mathrm{inv}_{\epsilon}(x):=\mathrm{inv}(\widetilde{\epsilon}(a),x)\in H^1(F,\rH_x). 
\end{equation}
The next lemma follows from the fact that two stably $\mathrm{U}_n$-conjugate Kostant sections of the quotient map $\pi_n:\Herm_n\to \fc$ are $\GL_n(F')$-conjugate.
\begin{lem}\label{Lem: change of section}
 If we replace our choice of $\epsilon$ by another $F$-rational Kostant section $\epsilon'$, then there exists $g\in \rH(F) = \GL_n(F')$ such that $\epsilon'=\mathrm{Ad}(g)\epsilon$.  
\end{lem}

Given such a change, rationality implies that we have
\[
\epsilon'(a) = \mathrm{Ad}(g)\epsilon(a) = \mathrm{Ad}(g^{-\ast})\epsilon(a)
\]
 The corresponding cocycle  $z_\sigma=(\sigma\in \Gal(F'/F)\mapsto  g^{-1}g^{-\ast})$ satisfies
\[
\mathrm{Ad}(z_\sigma) \epsilon(a) =\epsilon(a)
\] for all $a\in \fc(F)$, so gives a class 
$\mathrm{inv}({\epsilon'},{\epsilon})\in H^1(F,Z(\mathrm{U}_n)).$ We denote
\[
\mathrm{inv}_a({\epsilon'},{\epsilon})\in  H^1(F,\rH_a)
\]
the image of this cocycle under the map $Z(\mathrm{U}_n)\hookrightarrow \rH_a$. This satisfies that
\[
\mathrm{inv}_a({\epsilon'},{\epsilon})=\mathrm{inv}(\tilde{\epsilon}'(a),\tilde{\epsilon}(a)).
\]
By standard properties of the cohomological invariant, we have 
\begin{equation}\label{eqn: invariant identity}
    \mathrm{inv}(\tilde{\epsilon}'(a),\tilde{\epsilon}(a)) = \mathrm{inv}_{\epsilon}(x)\mathrm{inv}_{\epsilon'}(x)^{-1}\in H^1(F,\rH_x),
\end{equation}
for any $x\in \mathcal{O}(a)$, where we abuse notation so that the left-hand side denotes the transport of this class from $\rH_{\tilde{\epsilon}'(a)}$ to $\rH_{x}$. Similarly, if $x'\in \mathcal{O}(a)$ is another representative then
\begin{equation}\label{eqn: invariant identity 2}
    \mathrm{inv}(x',x) = \mathrm{inv}_{\epsilon}(x')\mathrm{inv}_{\epsilon}(x)^{-1}\in H^1(F,\rH_x),
\end{equation}

\subsection{Transfer factors on the Lie algebra}\label{Section: transfer factor Lie}

Recall now the notion of elliptic endoscopic datum from Definition \ref{def: endoscopic datum}. For a fixed endoscopic datum $\fe$, let $X^\fe = X^S_a\times X^\U_b$ $(a+b=n)$ be the associated  endoscopic variety equipped with an invariant map 
\begin{align*}
\chi_{\fe} = \chi_a\times \chi_b: V^\Siegel_a\times V^\unitary_b&\longrightarrow\A^a\times \A^b,\\
(x_a,x_b)&\longmapsto (\chi_a(x_a),\chi_b(x_b)).
\end{align*}
Viewing  $\chi_a(x_a) $ as the coefficients of a monic polynomials $P_{x_a}(t)$, recall that $x\in V^{\Siegel,\rs}(F)$ and $(x_a,x_b)\in V^{\fe,\rs}$ are matching orbits if we have the identity
\[
P_x(t) = P_{x_a}(t)P_{x_b}(t).
\]

For $x_{\fe} =(x_a,x_b)\in V^{\fe,\rs}(F)$, we set $r^\fe(x_\fe) = (r^\Siegel(x_a),r^\unitary(x_b))$ for the unique stable $G_{\fe}$-orbit in $\Herm_a(F)\times \Herm_b(F)$ defined in the previous section.

\begin{definition} Given an elliptic endoscopic datum $\fe$ for $X^\Siegel$, we define the transfer factor $\Delta_{\fe}$ as follows: if $x\in V^{\Siegel,\rs}(F)$ and $x_\fe\in V^{\fe,\rs}(F)$ do not match, we set
\begin{equation}\label{eqn: TF vanishing}
    \Delta_{\fe}(x_\fe,x)=0.
\end{equation}
 Now suppose that $x\in V^{\siegel,\rs}(F)$ and $x_\fe\in V^{\fe,\rs}(F)$ match as above. Then we set
\begin{equation}\label{eqn: TF formula}
   \Delta_{\fe}(x_\fe,x) = \kappa(\mathrm{inv}_\epsilon(x))\Delta_{LS,\fe}(r^\fe(x_\fe),\epsilon(\chi(x))), 
\end{equation}
where $\Delta_{LS,\fe}$ denotes the transfer factor of Langlands--Shelstad and Kottwitz
\[
\Delta_{LS,\fe}:[\Herm(V_a)\oplus \Herm(V_b)]^{\rs}\times  \Herm(W)^{\rs}\to \mathbb{C}
\]
associated to the induced endoscopic datum for $G$; see Definition \ref{def: endoscopic datum}.
\end{definition}

The notation $\Delta_\fe$ is justified by the following lemma.
\begin{lem}\label{Lem: indep of section}
    The transfer factor $\Delta$ is independent of the choice of $\epsilon$. Moreover, if $x'\sim_{\st}x$ then
    \[
         \Delta_{\fe}(x_\fe,x')=\kappa(\mathrm{inv}(x',x)) \Delta_{\fe}(x_\fe,x).
    \]
\end{lem}
\begin{proof}
The two important properties of $\Delta_{LS,\fe}$ are
\begin{enumerate}
    \item $\Delta_{LS,\fe}((\delta_a,\delta_b),\delta) = 0$ if $\delta$ does not match $(\delta_a,\delta_b),$
    \item\label{second} if $\delta$ is stably conjugate to $\delta'$, then 
    \[
    \Delta_{LS,\fe}((\delta_a,\delta_b),\delta) = \Delta_{LS,\fe}((\delta_a,\delta_b),\delta')\kappa(\mathrm{inv}(\delta,\delta')).
    \]
\end{enumerate}
The first is compatible with our definition \eqref{eqn: TF vanishing}. For the second, we note that if we change $\epsilon$ to another $F$-rational Kostant section $\epsilon'$, then
\[
\kappa(\mathrm{inv}(\tilde{\epsilon}'(a),\tilde{\epsilon}(a)))\Delta_{LS,\fe}(r^\fe(x_\fe),\epsilon'(\chi(x))) = \Delta_{LS,\fe}(r^\fe(x_\fe),\epsilon(\chi(x))), 
\]
Applying \eqref{eqn: invariant identity}, we see that $\Delta$ is independent of the choice of rational Kostant slice. The final claim is an immediate consequence of \eqref{eqn: invariant identity 2}.
\end{proof}
\begin{remark}
    Note that the definition of $\Delta_{LS,\fe}$ \emph{does} depend on a choice of pinning for the quasi-split unitary group $\mathrm{U}_n$, which determines a choice of rational Kostant section \cite{kottwitztransfer}. This corresponds to a particularly simple formula if $\epsilon$ is chosen compatibly with the choice of pinning of $\mathrm{U}_n$ in defining $\Delta_{LS}$.
\end{remark}
\begin{ex}
    Suppose that $\kappa = \varepsilon$ is the geometric cocycle from \cite{LXZfund}. More precisely, if we fix representatives $(x_{1,\alpha},x_{2,\alpha})\in \mathcal{O}_{\st}(x_1,x_2)$ where $\alpha\in H^1(F,\rH_x)$, we obtain a non-trivial character $\varepsilon\in H^1(F,\rH_{x})^D$ given by
  \begin{equation}\label{eqn: special kappa lie}
       \varepsilon(\alpha) = \eta_{F'/F}(\det(x_{1,\alpha}))\eta_{F'/F}(\det(x_{1}))^{-1},
  \end{equation}where $\eta_{F'/F}$ is the quadratic character associated to the extension $F'/F$

Here $\varepsilon|_{\ker^1(F;\rH_x,\mathrm{U}_n)}\equiv1$ so that the endoscopic space is just $V^{\unitary}(F)$ again, so transfer reduces to stable conjugacy, and $\Delta_{LS}=1$. Our definition now reduces to
\[
\Delta^\varepsilon(x,x) = \varepsilon(\mathrm{inv}(x,\tilde{\epsilon}(\chi(x))) = \eta_{F'/F}(\det(x_1))^{-1}, 
\]
which matches the transfer factor in \cite[Sec. 5.2.1]{LXZfund} since $\eta_{F'/F}$ is quadratic.

\end{ex}

\subsubsection{Compatibility of transfer factors} 
For an endoscopic datum $\fe = (\kappa,\rho)$ for $X^\Siegel$, recall  the reduced divisor $\fR_\fe$ such that
\begin{equation*}
\nu_\fe^*\fD^\cent = \fD^\cent_\fe+2\fR_\fe, \quad \nu_\fe^*\fD^\br = \fD^\br_\fe.
\end{equation*}
For $a\in \fc^\heartsuit(\cO_v)$, we defined in \eqref{eqn: dim names} the integer $r_\fe(a) = \deg(a^*\fR_\fe)$. Noting that if $Q_a(t)$ is the degree $n$ monic polynomial associated to $a\in \mathfrak{c}^\heartsuit(\mathcal{O}_F)$, the divisor of $\fD^\cent$ is given by its discriminant
\[
\fD^\cent = \left((-1)^{\frac{n^2-n}{2}}\prod_{i<j}(\delta_i-\delta_j)^2\right) = \left(\mathrm{Disc}(x^n+a_2x^{n-1}+a_4x^{n-2}+\cdots +a_{2n}) \right).
\]
In particular, when $\nu_\fe(a_1,a_2) = a$ so that $Q_{a}(t) = Q_{a_1}(t)Q_{a_2}(t)$, then $\fR_\fe$ is cut out by the resultant of $Q_{a_1}(t)$ and $Q_{a_2}(t)$.  

As stated in Theorem \ref{Thm: fundamental lemma intro}, the transfer factor associated to a point $a\in \mathfrak{c}^\heartsuit(\mathcal{O}_F)$ is given by $q^{r_\fe(a)}$. On the other hand, inspecting the transfer factor $\Delta_{LS,\fe}$ as in \cite[Appendix A]{Lesliedescent} shows that if $\delta_\fe\in Herm(V_a)\times\Herm(V_b)$ represents $(a_1,a_2)\in \mathfrak{c}_\fe^\heartsuit(\mathcal{O}_F)$ and matches $\delta\in \Herm(W)$ lying over $a$, then $\Delta_{LS,\fe}(\delta_\fe, \delta)= \pm q^{-r_\fe(a)}$. Finally, we note that if since the $\kappa$-orbital integral appearing in Theorem \ref{Thm: fundamental lemma intro} is normalized by a Kostant section $\epsilon: \fc\to \Herm(W)$, we see from \cite[Theorem 5.1]{kottwitztransfer} that $\Delta_{LS,\fe}(\delta_\fe, \delta)=  q^{-r_\fe(a)}$, noting that Kottwitz removes the discriminant factor $\Delta_{IV}$ from the transfer factor. This shows that the transfer factor constructed in the previous section reduces to the power of $q$ appearing in Theorem \ref{Thm: fundamental lemma intro}.

\subsection{Transfer factors on the variety}\label{Section: transfer factor variety} To model the transfer factor as explicitly as necessary, we will choose a particular Hermitian form on $\rG=\U_{2n}$. Letting $\Herm_n^\circ\subset\Herm_n$ denote the subspace of non-degenerate Hermitian forms, we take
\[
\tau= \left(\begin{array}{cc}
    \tau' &  \\
     & \tau'
\end{array}\right)\in \Herm_{2n}^\circ(F),
\]
where $\tau'\in \Herm_{n}^\circ(F)$. This gives 
\[
 \left(\begin{array}{cc}
    a& b \\
    c &d
\end{array}\right)^\sharp= \left(\begin{array}{cc}
    d^{\sharp} & b^{\sharp} \\
    c^{\sharp} & a^{\sharp}
\end{array}\right),
\]
where $a^\sharp = \tau'{}^Ta^\sigma(\tau')^{-1}$ for $a\in\fgl_n(F')$, where $a^\sigma$ denotes applying the non-trivial Galois element $\sigma\in \Gal(F'/F)$ entry-wise. Under these assumptions, we find that
\[
\rH = \left\{ \left(\begin{array}{cc}
    c &  \\
     & (c^\sharp)^{-1}
\end{array}\right): c\in \Res_{F'/F}(\GL_n) \right\}\simeq \Res_{F'/F}(\GL_n),\qquad \rG=\mathrm{U}(V_{\tau}).
\]

Then  any element $x\in {X}^\Siegel(F)$ may be written
\begin{equation}\label{eqn: element of inert-inert}
x=\left(\begin{array}{cc}A&B\\C&A^\sharp\end{array}\right),
\end{equation}
where $A\in \fgl_n(F')$ and $B,C\in \Herm(V)$. As recalled in \cite[Section 3]{LXZfund}, the blocks satisfy the polynomial relations
\begin{equation}\label{eqn: inertinert eqns}
A^2=I-BC,\: (A^\sharp)^2=I-CB,\: AB=BA^\sharp, CA= A^\sharp C.
\end{equation}

We also have the explicit formula for the categorical quotient.   
\begin{lem}\label{Lem: cat quotient ii}
 The morphism  $\chi_{n}: {X^\Siegel}\to \A^n$ which takes as coordinates the coefficients of the polynomial
\begin{align*}
 	x=\left(\begin{array}{cc}A&B\\C&A^\sharp\end{array}\right)&\longmapsto \mathrm{char}_A(t)
\end{align*}
where $\mathrm{char}_a(t)$ denotes the reduced characteristic polynomial, gives a categorical quotient for $(\rH,{X^\Siegel})$. 
\end{lem}

Recalling \cite[Lemma 4.6]{LXZfund}, the invariant map $\chi_n:{X^\Siegel}(F)\to [{X^\Siegel}\sslash H^\Siegel](F)$ is surjective on regular semi-simple orbits. Indeed, consider the $\rH$-equivariant map $R: {X^\Siegel}\to \End(L)=\Res_{F'/F}(\fgl_{n})$ given by $R(X) = A$. Similarly to the product map on $V^\Siegel$, $R$ intertwines the natural $\rH$-actions. As before, there exists $g\in H^\Siegel(F)$ such that 
\[
R(g\cdot x)\in \Herm_n(F),
\]
and the stable $\U_n$-orbit of this element is uniquely determined; we let $R^{\Siegel}(x)$ for a representative of this stable orbit. 

If we consider the case of the unitary variety $X^\unitary$, then a similar map 
\[
R^\unitary: X^\unitary\longrightarrow \Herm_n
\]
was defined and studied in \cite[Section 5]{Lesliedescent}. For an endoscopic datum $\fe$ such that $X^{\fe}= X_a^\Siegel\times X_{b}^\unitary$, we set $R^\fe(x_\fe) = (R^\siegel(x_a),R^\unitary(x_b))$ to represent the corresponding stable $G_\fe$-orbit in $\Herm_a\times \Herm_b$.

 If we fix a Kostant section $\epsilon: \A^n\to \Herm_n$ as before, this implies that $\epsilon(\chi_n(x))\in \Herm_{n}^{\circ}(F)$ lies in the same $\GL_n(F)$-orbit as $\chi_n(x)$. In particular, setting $a=\chi_n(x)$
    \begin{equation}\label{eqn: relative steinberg}
    \overline{\epsilon}(a) =\left(\begin{array}{cc}\epsilon(a)&I_n\\ I_n-\epsilon(a)^2&\epsilon(a)\end{array}\right)\in {X^\Siegel}(F),
    \end{equation}
lies in the same stable $\rH$-orbit as $s\in {X}^{\Siegel,\rs}(F)$. This gives an analogue of the Kostant section on the level of the variety $X^\Siegel$. We record this as the following analogue of Lemma \ref{Lem: section on Lie}.
\begin{lem}
    For each regular $x\in{X}^{\Siegel, \rs}(F)$, if $a=\chi_n(x)$, then $x\sim_{\st} \overline{\epsilon}(a)$. Moreover, there is a unique $\rH(F)$-orbit in $\mathcal{O}_{\st}(x)$ meeting  the image of $\overline{\epsilon}$.
\end{lem}

For any regular semi-simple $a\in [{X^\Siegel}\sslash H^\Siegel](F)$, set $\mathcal{O}(a):=\chi_n^{-1}(a)(F)$ and $\rH_a=\rH_{\overline{\epsilon}(a)}$. For any $x\in \mathcal{O}(a)$, we set
\begin{equation}
   \mathrm{inv}_{\epsilon}(x):=\mathrm{inv}(\overline{\epsilon}(a),x)\in H^1(F,\rH_x). 
\end{equation}

We now mimic the construction of the transfer factor from the linearized setting using $\overline{\epsilon}$. By Lemma \ref{Lem: change of section}, if we replace $\epsilon$ by another Kostant section $\epsilon'$, then there exists $g\in \rH(F) = \GL_n(F')$ such that $\epsilon'=\mathrm{Ad}(g)\epsilon$.  Given such a change, we have
\[
\epsilon'(a) = \mathrm{Ad}(g)\epsilon(a) = \mathrm{Ad}(g^{-\ast})\epsilon(a)
\]
 The corresponding cocycle  $z_\sigma=(\sigma\in \Gal(F'/F)\mapsto  g^{-1}g^{-\ast})$  now gives a class 
$\mathrm{inv}({\epsilon'},{\epsilon})\in H^1(F,Z(\mathrm{U}_n)).$ We denote
\[
\mathrm{inv}_a({\epsilon'},{\epsilon})\in  H^1(F,\rH_a)
\]
the image of this cocycle under the map $Z(\mathrm{U}_n)\hookrightarrow \rH_a$. This satisfies that
\[
\mathrm{inv}_a({\epsilon'},{\epsilon})=\mathrm{inv}(\tilde{\epsilon}'(a),\tilde{\epsilon}(a)).
\]
As in the Lie algebra setting,
\begin{equation}\label{eqn: invariant identity var}
    \mathrm{inv}(\tilde{\epsilon}'(a),\tilde{\epsilon}(a)) = \mathrm{inv}_{\epsilon}(x)\mathrm{inv}_{\epsilon'}(x)^{-1}\in H^1(F,\rH_x),
\end{equation}
for any $x\in \mathcal{O}(a)$, where we abuse notation so that the left-hand side denotes the transport of this class from $\rH_{\tilde{\epsilon}'(a)}$ to $\rH_{x}$. Similarly, if $x'\in \mathcal{O}(a)$ is another representative then
\begin{equation}\label{eqn: invariant identity var 2}
    \mathrm{inv}(x',x) = \mathrm{inv}_{\epsilon}(x')\mathrm{inv}_{\epsilon}(x)^{-1}\in H^1(F,\rH_x).
\end{equation}

\subsubsection{Transfer factors for ${X}^\siegel$} We copy over all the notions from the previous section on endoscopy. For a fixed endoscopic datum $\fe$, the endoscopic variety is equipped with an invariant map 
\begin{align*}
\chi^\fe: \mathrm{X}^\fe&\longrightarrow\A^a\times \A^b,\\
x_\fe=(x_a,x_b)&\longmapsto (\chi_a(x_a),\chi_b(x_b)).
\end{align*}

As before, we let $P_{x}(t)$ denote the monic polynomial whose coefficients are given by $\chi(x)$.  Then the natural map of invariant quotients \cite[Section 8]{sl_endoscopy}
\[
\nu_\fe :[X^\fe\sslash H^\fe]\longrightarrow[X^\siegel\sslash H^\siegel], 
\]
corresponds on invariant polynomials to taking products.
\begin{definition} Given an elliptic endoscopic datum $\fe$, we define the transfer factor $\Delta_{\mathrm{X},\fe}$ as follows: if $x\in X^{\Siegel,\rs}(F)$ and $x_\fe=(x_a,x_b)\in X^{\fe,\rs}(F)$ do not match, we set
\begin{equation}
    \Delta_{\mathrm{X},\fe}(x_\fe,x)=0.
\end{equation}
 Now suppose that $x\in {X}^{\Siegel, \rs}(F)$ and $x_\fe=(x_a,x_b)\in X^{\fe,\rs}(F)$ match as above. Then we set
\begin{equation}
   \Delta_{\mathrm{X},\fe}(x_\fe,x) = \kappa(\mathrm{inv}_\epsilon(x))\Delta_{LS,\fe}((R_a(x_a),R_b(x_b)),\epsilon(\chi(x))), 
\end{equation}
where $\Delta_{LS,\fe}$ denotes the transfer factor of Langlands--Shelstad and Kottwitz associated to the induced endoscopic datum for $G$; see Definition \ref{def: endoscopic datum}.
See \cite{kottwitztransfer}.
\end{definition}

The argument proving Lemma \ref{Lem: indep of section} gives the following.
\begin{lem}
    The transfer factor $\Delta_{\mathrm{X},\fe}$ is independent of the choice of $\epsilon$ and satisfies
    \[
      \Delta_{\mathrm{X},\fe}(x_\fe,x')=\kappa(\mathrm{inv}(x',x)) \Delta_{\mathrm{X},\fe}(x_\fe,x),
    \]
    whenever $x'$ lies in the same stable $H$-orbit as $x\in X^{\Siegel,\rs}(F)$.
\end{lem}

\begin{ex}
    Suppose that $\kappa = \varepsilon$ is the geometric character from \cite{LXZfund}, so that $X^\fe = X^\U_n$. More precisely, if we fix representatives $x_\alpha\in \mathcal{O}_{\st}(x)$ where $\alpha\in H^1(F,\rH_x)$, we obtain a non-trivial character $\varepsilon\in H^1(F,\rH_{x})^D$ given by
  \begin{equation}\label{eqn: special kappa group}
       \varepsilon(\alpha) = \eta(\det(B_{\alpha}))\eta(\det(B))^{-1},
  \end{equation}
  where 
  \[
  x_\alpha = \left(\begin{array}{cc}
      A_\alpha & B_\alpha \\
     C_{\alpha}  &A_{\alpha}^\sharp 
  \end{array}\right).
  \]

Here $\varepsilon|_{\ker^1(F,\rH_x,\mathrm{U}_n)}\equiv1$ so that the endoscopic space is just $\mathrm{X}_n(F)$ again, so transfer reduces to stable conjugacy, and $\Delta_{LS}=1$. Our definition now reduces to
\[
\Delta_{\mathrm{X},\varepsilon}(x,x) = \varepsilon(\mathrm{inv}(x,\tilde{\epsilon}(\chi(x))) = \eta_{F'/F}(\det(A)^{-1}), 
\]
which matches the transfer factor in \cite[Sec. 5.2.1]{LXZfund}, given by $ \eta_{F'/F}(\det(A))$. Note that this works since we used the section 
\[
\overline{\epsilon}(a) =\left(\begin{array}{cc}\epsilon(a)&I_n\\ I_n-\epsilon(a)^2&\epsilon(a)\end{array}\right)\in \mathrm{X}(F).
\]

\end{ex}

\subsection{Descent}
We need to apply the tools from \cite{Lesliedescent} to relate the orbital integrals on $X^\siegel$ with those on $V^\siegel$.

\begin{prop}\label{Prop: very reg FL}
For an elliptic relative endoscopic datum $\fe$, the fundamental lemma between $V^\Siegel$ and $V^\fe$ implies  the fundamental lemma holds for ${X}^\siegel$ and $X^\fe$. 
\end{prop}

The main content is to consider the relationship between the transfer factors through the Cayley transform. This will allow for a simple proof of the proposition on a certain open subset of $\mathrm{X}^\Siegel(\mathcal{O}_F)$. The remaining orbits need to be dealt with via descent as in \emph{loc. cit.}.

\subsubsection{The Cayley transform}\label{Section: cayley}
Let $W=L\oplus L^\vee$ be our $2n$ dimensional Hermitian space with a polarization. For any $\xi\in F'$, we define
\[
D_\xi=\{X\in \End(W): \det(\xi I-X)=0\}.
\]
In this notation, we have
\[
V^\siegel= \Herm(L^\vee,L)\times \Herm(L,L^\vee),
\]
and $R:V^\siegel(F)\to \fgl(L)$ via $(X,Y) \mapsto X\circ Y$.
\begin{lem}\label{Lem: Cayley map}
 The Cayley transform
\begin{align*}
\mathrm{Cay}_{\pm1}:\End(W)-D_1(F)&\longrightarrow \GL(W)\\
			X&\longmapsto \mp(1+X)(1-X)^{-1}
\end{align*}
induces a $\GL(L)$-equivariant isomorphism between $V^\siegel-D_1(F)$ and $\mathrm{X}(F)-D_{\pm1}(F)$, where we identify 
\[
X^\Siegel \subset \U(W)\subset \GL(W)
\]via the symmetrization map. The images of $V^\Siegel-D_1(F)$ under $\mathrm{Cay}_{\pm}$ form a finite cover by open subsets of $\mathrm{X}(F)-(D_1\cap D_{-1})(F)$.In particular, the images contain the regular semi-simple locus of $\mathrm{X}(F)$.
\end{lem}

The next few lemmas show how the Cayley transform is compatible with the categorical quotients, our notions of matching, and transfer factors. These are all easy variations on the results of \cite[Section 7.4]{Lesliedescent}, with analogous proofs.

\begin{lem}\label{Lem: Cayley contraction}
For $\nu=\pm1$, there is a commutative diagram
\[
\begin{tikzcd}
V^\siegel(F)-D_1(F)\ar[r,"\mathrm{Cay}_{\nu}"]\ar[d,"r"]& {{X}^\Siegel}(F)-D_{\nu}(F)\ar[d,"R"]\\
\fgl(L)-D_1(F)\ar[r,"\mathrm{Cay}_\nu"]&\fgl(L)-D_{\nu}(F),
\end{tikzcd}
\]
where by abuse of notation we let $\mathrm{Cay}_\nu$ denote the Cayley transform on both $\mathfrak{u}(W)$ and $\fgl(L)$.
\end{lem}
\begin{proof}
Let $\delta\in \mathfrak{u}(W)_1-D_1(F)$ such that
\[
\delta = \left(\begin{array}{cc}
     & X \\
    Y & 
\end{array}\right),\text{ so that  } r(\delta)=XY.
\]
We need to calculate
\[
\mathrm{Cay}_\nu(\delta) = -\nu\left(\begin{array}{cc}
    I & X \\
   Y & I
\end{array}\right)\left(\begin{array}{cc}
    I & -X \\
    -Y & I
\end{array}\right)^{-1}.
\]
A simple matrix computation shows that
\[
\left(\begin{array}{cc}
    I & -X \\
    -Y & I
\end{array}\right)^{-1}= \left(\begin{array}{cc}
    (I-XY)^{-1} & X (I-Y X)^{-1} \\
    Y (I-XY)^{-1}&  (I-Y X)^{-1}
\end{array}\right),
\]
In particular, we have
\[
\mathrm{Cay}_{\nu}(\delta) =  -\nu\left(\begin{array}{cc}(I+XY)(I-X Y)^{-1}&2X(I-YX)^{-1}\\2Y(I-XY)^{-1}&(I+Y X)(I-Y X)^{-1}\end{array}\right). 
\]
 The commutativity of the diagram now follows from the definitions of $r$ and $R$.
\end{proof}

We now consider the effect of the Cayley transform on the invariant polynomial maps.

\begin{lem}\label{Lem: Cayley cahracteristic}
Let $\delta\in \End(W)-D_1(F)$ and set $x=\mathrm{Cay}_\nu(\delta)$. Then
\begin{equation*}
\mathrm{char}_x(t) = (t-\nu)^{\dim(W)}\mathrm{char}_\delta(1)^{-1}\left(\mathrm{char}_\delta\left(\frac{t+\nu}{t-\nu}\right)\right).
\end{equation*}
\end{lem}

Let $\mathfrak{l}_{\nu}: \A^n \dashrightarrow{} \A^n$ denote the birational morphism defined by the above identity (identifying $a\in \A^n$ with the coefficients of a monic degree polynomial). That is, if $x=\mathrm{Cay}_{\nu}(\delta)$, then
\[
\mathrm{char}_n(\mathrm{Cay}_{\nu(\delta)}) = \mathfrak{l}_\nu(\chi(\delta)).
\]

Before we apply this to comparing transfer factors, we check that the Cayley transform preserves our notions of matching orbits.

\begin{lem}\label{Lem: cayley matching}
Suppose that $\delta\in V^{\Siegel,\rs}(F)-D_{\pm1}(F)$, and let $x=\mathrm{Cay}_{\nu}(\delta)\in {X}^{\Siegel, \rs}(F)$. Fix an elliptic datum $\fe$. 
Then a regular semi-simple element
\[\delta_\fe=(\delta_a,\delta_b)\in V^\fe(F)= V^\Siegel_a(F)\times V^\unitary_a(F)
\]matches $\delta$ if and only if 
\[
(x_a,x_b)= (\mathrm{Cay}_{\nu}(\delta_a),\mathrm{Cay}_{\nu}(\delta_b))\in {X}_{a}^\siegel(F)\times {X}^\unitary_{b}(F)
\]matches $x$. 
\end{lem}

\begin{lem}\label{Lem: Cayley disc}\cite[Lemma 7.6]{Lesliedescent}
With notation as in the previous lemma with $x=\mathrm{Cay}_\nu(\delta)$, consider the relative discriminants
\[
\tilde{D}_{a,b}(\delta) = \prod_{t_a,t_b}(t_a-t_b)\quad\text{ and }\quad {D}_{a,b}(x) = \prod_{z_a,z_b}(z_a-z_b),
\]
where $t_a$ (resp. $z_a$) runs over the roots of the invariant polynomial of $\delta$ (resp. $x$) arising from $\delta_a$ (resp. $x_a$) and likewise with $t_b$ (resp. $x_b$), then
\[
\tilde{D}_{a,b}(\delta) =(-2\nu)^{a\cdot b}\prod_{z_a,z_b}\frac{1}{(z_a-\nu)(z_b-\nu)}D_{a,b}(x).
\]

\end{lem}

\subsubsection{Kostant sections and descent}\label{Section: very regular} 

Let $\nu=\pm1$. 
    Fix a Kostant section $\epsilon: \A^n\to \Herm(L)$. Fix $\delta\in V^{\siegel,\rs}(F)$ and set $\mathrm{Cay}_\nu(\delta)=x\in {X}^{\Siegel, \rs}(F)$. Then we compute that 
    \[
    \mathrm{Cay}_{\nu}(I,\epsilon(a)) = -\nu\left(\begin{array}{cc}(I+\epsilon(a))(I-\epsilon(a))^{-1}&2\epsilon(a)(I-\epsilon(a))^{-1}\\2(I-\epsilon(a))^{-1}&(I+\epsilon(a))(I-\epsilon(a))^{-1}\end{array}\right). 
    \]
    Thus $R( \mathrm{Cay}_{\nu}(I,\epsilon(a)))=-\nu(I+\epsilon(a))(I-\epsilon(a))^{-1} = \mathrm{Cay}_\nu(\epsilon(a))$ is obtained from the Cayley transform applied to the Kostant section.

 If the pairs $(x,(x_a,x_b))$ and $(\delta,(\delta_a,\delta_b))$ are as in Lemmas \ref{Lem: cayley matching} and \ref{Lem: Cayley disc}, we set
\[
C_{a,b,\nu}(x,\delta):=(-2\nu)^{a\cdot b}\prod_{z_a,z_b}\frac{1}{(z_a-\nu)(z_b-\nu)}.
\]
Following the steps from \cite[Section 7]{Lesliedescent}, the results of the previous section imply the formula
\begin{equation*}
   {\Delta}_{\kappa}((\delta_a,\delta_b),\delta) = \eta_{F'/F}(C_{a,b,\nu}(x,\delta))|C_{a,b,\nu}(x,\delta)|_F\Delta_{\mathrm{X},\fe}((x_a,x_b),x).
\end{equation*}
Since $F'/F$ is unramified and the residue characteristic is odd, these transfer factors agree whenever $C_{a,b,\nu}(x,\delta)$ is a unit.

\subsubsection{A sketch of the argument}\label{Sec: descent} Fix an elliptic endoscopic datum $\fe$ for $X^\Siegel$. The goal is to show that the identity in Corollary \ref{Cor: transfer} implies the identity in Theorem \ref{Thm: fundamental lemma intro}. To this end, it suffices to consider $x\in \mathrm{X}^{\Siegel,\rs}(\mathcal{O}_F)$.  Following \cite[Section 7]{Lesliedescent}, we define certain open subsets $\mathrm{X}_n^{\heartsuit,\nu}(F)$ of $\mathrm{X}_n(F)$ on which $C_{a,b,\nu}(x,\delta)$ is a unit. More precisely for $\nu=\pm1$, we define the $\nu$-very regular locus by
\[
\mathrm{X}^{\Siegel,\heartsuit,\nu}(\mathcal{O}_F)=\{x\in\mathrm{X}^{\Siegel, \rs}(\mathcal{O}_F): \overline{x}\in \mathrm{X}^\siegel(k)-D_{\nu}(k)\},
\]
where $\overline{x}$ denotes the reduction modulo the maximal ideal of $\mathcal{O}_F$. Then any regular semi-simple $x\in \mathrm{X}^{\Siegel,\heartsuit,\nu}(\mathcal{O}_F)$ is in the image of the Cayley map $\mathrm{Cay}_\nu$, so we write $x= \mathrm{Cay}_\nu(\delta)$ for an element $\delta\in V^{\Siegel}-D_1(F)$. This element may be chosen to be integral, and moreover if $x_\fe\in \mathrm{X}^{\fe}(F)$ matches $x$, then we may similarly pullback along the Cayley map on the endoscopic side, say $x_\fe = \mathrm{Cay}_{\nu,\fe}(\delta_\fe)$. An easy counting argument (see the proof of \cite[Proposition 7.9]{Lesliedescent}) shows that the orbital integrals appearing in Theorem \ref{Thm: fundamental lemma intro} for $x$ and $x_\fe$ agree with the orbital integrals appearing in Theorem \ref{Thm: fundamental lemma intro} and Corollary \ref{Cor: transfer} for $\delta$ and $\delta_\fe$. The work of the prior section shows that the transfer factors also match, so the result is proved when $x\in \mathrm{X}^{\Siegel,\heartsuit,+}(\mathcal{O}_F)\cup \mathrm{X}^{\Siegel,\heartsuit,-}(\mathcal{O}_F)$.

It remains to consider those $x\in \mathrm{X}^{\Siegel}(\mathcal{O}_F)$ such that $\overline{x}\in D_{+}(k)\cap D_{-}(k)$. Following the approach developed in \cite{Lesliedescent}, we now apply topological descent to these remaining elements. The details are entirely analogous to the reduction carried out in \cite[Section 8]{Lesliedescent} in the unitary Friedberg--Jacquet setting via descent at certain absolutely semi-simple elements of $\mathrm{X}^{\Siegel}(\mathcal{O}_F)$ and using properties of the topological Jordan decomposition established in \cite[Section 4]{Lesliedescent}. The arguments in \emph{loc. cit.} go through in the Siegel case as well. Indeed, the special case of $X^\fe= X^\unitary_n$ is handled in detail in \cite[Appendix B]{LXZfund}. The general case follows with minor modifications, so we omit the details.

\end{appendix}

\bibliographystyle{amsalpha}
\bibliography{bib.bib}

@article {jingren,
    AUTHOR = {Chi, Jingren},
     TITLE = {Geometry of {K}ottwitz-{V}iehmann varieties},
   JOURNAL = {J. Inst. Math. Jussieu},
  FJOURNAL = {Journal of the Institute of Mathematics of Jussieu. JIMJ.
              Journal de l'Institut de Math\'{e}matiques de Jussieu},
    VOLUME = {21},
      YEAR = {2022},
    NUMBER = {1},
     PAGES = {1--65},
      ISSN = {1474-7480,1475-3030},
   MRCLASS = {22E35 (14D06 14D23 22E67)},
  MRNUMBER = {4366333},
MRREVIEWER = {Alan\ Koch},
       DOI = {10.1017/S1474748019000604},
       URL = {https://doi.org/10.1017/S1474748019000604},
}

@incollection {CluckersHalesLoeser,
    AUTHOR = {Cluckers, Raf and Hales, Thomas and Loeser, Fran\c cois},
     TITLE = {Transfer principle for the fundamental lemma},
 BOOKTITLE = {On the stabilization of the trace formula},
    SERIES = {Stab. Trace Formula Shimura Var. Arith. Appl.},
    VOLUME = {1},
     PAGES = {309--347},
 PUBLISHER = {Int. Press, Somerville, MA},
      YEAR = {2011},
      ISBN = {978-1-57146-227-5},
   MRCLASS = {22E35 (14E18)},
  MRNUMBER = {2856374},
}

@article {YunJR,
    AUTHOR = {Yun, Zhiwei},
     TITLE = {The fundamental lemma of {J}acquet and {R}allis},
      NOTE = {With an appendix by Julia Gordon},
   JOURNAL = {Duke Math. J.},
  FJOURNAL = {Duke Mathematical Journal},
    VOLUME = {156},
      YEAR = {2011},
    NUMBER = {2},
     PAGES = {167--227},
      ISSN = {0012-7094},
   MRCLASS = {22E35 (17B45)},
  MRNUMBER = {2769216},
MRREVIEWER = {Anne-Marie H. Aubert},
       DOI = {10.1215/00127094-2010-210},
       URL = {https://doi.org/10.1215/00127094-2010-210},
}

@article{Lesliedescent,
    AUTHOR = {Leslie, Spencer},
     TITLE = {On the stabilization of relative trace formulae: descent and
              the fundamental lemma},
   JOURNAL = {Adv. Math.},
  FJOURNAL = {Advances in Mathematics},
    VOLUME = {394},
      YEAR = {2022},
     PAGES = {Paper No. 108026, 68},
      ISSN = {0001-8708,1090-2082},
   MRCLASS = {11F70 (11F55 11F72 11F85)},
  MRNUMBER = {4355728},
MRREVIEWER = {Chen\ Wan},
       DOI = {10.1016/j.aim.2021.108026},
       URL = {https://doi.org/10.1016/j.aim.2021.108026},
}

@article{kottwitztransfer,
  title={Transfer factors for {L}ie algebras},
  author={Kottwitz, Robert E.},
  journal={Representation Theory of the American Mathematical Society},
  volume={3},
  number={6},
  pages={127--138},
  year={1999}
}

@article {ngo,
    AUTHOR = {Ngo, Bao Chau},
     TITLE = {Le lemme fondamental pour les alg\`ebres de {L}ie},
   JOURNAL = {Publ. Math. Inst. Hautes \'{E}tudes Sci.},
  FJOURNAL = {Publications Math\'{e}matiques. Institut de Hautes \'{E}tudes
              Scientifiques},
    NUMBER = {111},
      YEAR = {2010},
     PAGES = {1--169},
      ISSN = {0073-8301,1618-1913},
   MRCLASS = {22E35 (11S37 14D23 14G35 22E50)},
  MRNUMBER = {2653248},
MRREVIEWER = {R.\ P.\ Langlands},
       DOI = {10.1007/s10240-010-0026-7},
       URL = {https://doi.org/10.1007/s10240-010-0026-7},
}

@article {gppn,
    AUTHOR = {Garc\'{\i}a-Prada, Oscar and Pe\'{o}n-Nieto, Ana},
     TITLE = {Abelianization of {H}iggs bundles for quasi-split real groups},
   JOURNAL = {Transform. Groups},
  FJOURNAL = {Transformation Groups},
    VOLUME = {28},
      YEAR = {2023},
    NUMBER = {1},
     PAGES = {285--325},
      ISSN = {1083-4362,1531-586X},
   MRCLASS = {14H60 (14D23)},
  MRNUMBER = {4552177},
       DOI = {10.1007/s00031-021-09658-9},
       URL = {https://doi.org/10.1007/s00031-021-09658-9},
}

@article {ngo_endoscopie,
    AUTHOR = {Ngo, Bao Chau},
     TITLE = {Fibration de {H}itchin et endoscopie},
   JOURNAL = {Invent. Math.},
  FJOURNAL = {Inventiones Mathematicae},
    VOLUME = {164},
      YEAR = {2006},
    NUMBER = {2},
     PAGES = {399--453},
      ISSN = {0020-9910,1432-1297},
   MRCLASS = {14D20 (11F72 11R39 22E55)},
  MRNUMBER = {2218781},
MRREVIEWER = {Ulrich\ G\"{o}rtz},
       DOI = {10.1007/s00222-005-0483-7},
       URL = {https://doi.org/10.1007/s00222-005-0483-7},
}

@article{LXZfund,
  title={Unitary {F}riedberg-{J}acquet periods and their twists: {F}undamental lemmas},
  author={Leslie, Spencer and Xiao, Jingwei and Zhang, Wei},
  journal={arXiv preprint arXiv:2503.09500},
  year={2025}
}

@article{LXZrtf,
  title={Unitary {F}riedberg-{J}acquet periods and their twists: {R}elative trace formulas},
  author={Leslie, Spencer and Xiao, Jingwei and Zhang, Wei},
  journal={arXiv preprint arXiv:2503.09664},
  year={2025}
}

@article{hitchin,
    AUTHOR = {Hitchin, Nigel},
     TITLE = {Stable bundles and integrable systems},
   JOURNAL = {Duke Math. J.},
  FJOURNAL = {Duke Mathematical Journal},
    VOLUME = {54},
      YEAR = {1987},
    NUMBER = {1},
     PAGES = {91--114},
      ISSN = {0012-7094,1547-7398},
   MRCLASS = {58F07 (14F05 32G13 32L05 32L10)},
  MRNUMBER = {885778},
MRREVIEWER = {G.\ M.\ Khenkin},
       DOI = {10.1215/S0012-7094-87-05408-1},
       URL = {https://doi.org/10.1215/S0012-7094-87-05408-1},
}

@article{ngo-laumon,
  author    = {Gérard Laumon and Bao Chau Ngo},
  title     = {Le lemme fondamental pour les groupes unitaires},
  journal   = {Annals of Mathematics},
  year      = {2008},
  volume    = {168},
  number    = {2},
  pages     = {477--573},
  doi       = {10.4007/annals.2008.168.477}
}

@article{schaposnik,
    AUTHOR = {Schaposnik, Laura P.},
     TITLE = {Spectral data for {$U(m,m)$}-{H}iggs bundles},
   JOURNAL = {Int. Math. Res. Not. IMRN},
  FJOURNAL = {International Mathematics Research Notices. IMRN},
      YEAR = {2015},
    NUMBER = {11},
     PAGES = {3486--3498},
      ISSN = {1073-7928,1687-0247},
   MRCLASS = {14J60 (14D20)},
  MRNUMBER = {3373057},
MRREVIEWER = {Sanjay\ Kumar\ Singh},
}

@article{KR,
    AUTHOR = {Kostant, B. and Rallis, S.},
     TITLE = {Orbits and representations associated with symmetric spaces},
   JOURNAL = {Amer. J. Math.},
  FJOURNAL = {American Journal of Mathematics},
    VOLUME = {93},
      YEAR = {1971},
     PAGES = {753--809},
      ISSN = {0002-9327,1080-6377},
   MRCLASS = {22E45 (53C35)},
  MRNUMBER = {311837},
MRREVIEWER = {W.\ T.\ van Est},
       DOI = {10.2307/2373470},
       URL = {https://doi.org/10.2307/2373470},
}

@article{levy,
    AUTHOR = {Levy, Paul},
     TITLE = {Involutions of reductive {L}ie algebras in positive
              characteristic},
   JOURNAL = {Adv. Math.},
  FJOURNAL = {Advances in Mathematics},
    VOLUME = {210},
      YEAR = {2007},
    NUMBER = {2},
     PAGES = {505--559},
      ISSN = {0001-8708,1090-2082},
   MRCLASS = {17B50 (14L40)},
  MRNUMBER = {2303231},
MRREVIEWER = {George\ J.\ McNinch},
       DOI = {10.1016/j.aim.2006.07.002},
       URL = {https://doi.org/10.1016/j.aim.2006.07.002},
}

@article{sekiguchi,
    AUTHOR = {Sekiguchi, Jiro},
     TITLE = {The nilpotent subvariety of the vector space associated to a
              symmetric pair},
   JOURNAL = {Publ. Res. Inst. Math. Sci.},
  FJOURNAL = {Kyoto University. Research Institute for Mathematical
              Sciences. Publications},
    VOLUME = {20},
      YEAR = {1984},
    NUMBER = {1},
     PAGES = {155--212},
      ISSN = {0034-5318,1663-4926},
   MRCLASS = {14L30 (17B20 32C40)},
  MRNUMBER = {736100},
MRREVIEWER = {Wim\ H.\ Hesselink},
       DOI = {10.2977/prims/1195181836},
       URL = {https://doi.org/10.2977/prims/1195181836},
}

@article{leslie,
    AUTHOR = {Leslie, Spencer},
     TITLE = {An analogue of the {G}rothendieck-{S}pringer resolution for
              symmetric spaces},
   JOURNAL = {Algebra Number Theory},
  FJOURNAL = {Algebra \& Number Theory},
    VOLUME = {15},
      YEAR = {2021},
    NUMBER = {1},
     PAGES = {69--107},
      ISSN = {1937-0652,1944-7833},
   MRCLASS = {20G05 (17B08 32S45)},
  MRNUMBER = {4226983},
MRREVIEWER = {Paul\ D.\ Levy},
       DOI = {10.2140/ant.2021.15.69},
       URL = {https://doi.org/10.2140/ant.2021.15.69},
}

@article{knop1996automorphisms,
    AUTHOR = {Knop, Friedrich},
     TITLE = {Automorphisms, root systems, and compactifications of
              homogeneous varieties},
   JOURNAL = {J. Amer. Math. Soc.},
  FJOURNAL = {Journal of the American Mathematical Society},
    VOLUME = {9},
      YEAR = {1996},
    NUMBER = {1},
     PAGES = {153--174},
      ISSN = {0894-0347,1088-6834},
   MRCLASS = {14L30 (14M17)},
  MRNUMBER = {1311823},
MRREVIEWER = {Michel\ Brion},
       DOI = {10.1090/S0894-0347-96-00179-8},
       URL = {https://doi.org/10.1090/S0894-0347-96-00179-8},
}

@article{schaposnik2018geometric,
    AUTHOR = {Schaposnik, Laura P.},
     TITLE = {A geometric approach to orthogonal {H}iggs bundles},
   JOURNAL = {Eur. J. Math.},
  FJOURNAL = {European Journal of Mathematics},
    VOLUME = {4},
      YEAR = {2018},
    NUMBER = {4},
     PAGES = {1390--1411},
      ISSN = {2199-675X,2199-6768},
   MRCLASS = {14D20 (14D21 14H70 14P25 20C33 53C07)},
  MRNUMBER = {3866701},
MRREVIEWER = {\'{A}lvaro\ Ant\'{o}n Sancho},
       DOI = {10.1007/s40879-017-0206-9},
       URL = {https://doi.org/10.1007/s40879-017-0206-9},
}

@phdthesis{schaposnik2013spectral,
    AUTHOR = {Schaposnik, Laura P.},
     TITLE = {Spectral data for {G}-{H}iggs bundles},
      NOTE = {Thesis (D.Phil.)--University of Oxford (United Kingdom)},
 PUBLISHER = {ProQuest LLC, Ann Arbor, MI},
      YEAR = {2013},
     PAGES = {1},
   MRCLASS = {99-05},
  MRNUMBER = {3389247},
       URL =
              {http://gateway.proquest.com/openurl?url_ver=Z39.88-2004&rft_val_fmt=info:ofi/fmt:kev:mtx:dissertation&res_dat=xri:pqm&rft_dat=xri:pqdiss:U633158},
}

@misc{polak,
      title={Dimensions of Affine {S}pringer Fibers for Some $\mathrm{GL}_n$ Symmetric Spaces}, 
      author={Jason K. C. Polák},
      year={2015},
      eprint={1506.04703},
      archivePrefix={arXiv},
      primaryClass={math.AG},
      note={\url{https://arxiv.org/abs/1506.04703}}, 
}

@incollection{schaposnik2014introduction,
    AUTHOR = {Schaposnik, Laura P.},
     TITLE = {An introduction to spectral data for {H}iggs bundles},
 BOOKTITLE = {The geometry, topology and physics of moduli spaces of {H}iggs
              bundles},
    SERIES = {Lect. Notes Ser. Inst. Math. Sci. Natl. Univ. Singap.},
    VOLUME = {36},
     PAGES = {65--101},
 PUBLISHER = {World Sci. Publ., Hackensack, NJ},
      YEAR = {2018},
      ISBN = {978-981-3229-08-2},
   MRCLASS = {14H60 (57R22)},
  MRNUMBER = {3837869},
MRREVIEWER = {Francesco\ Bottacin},
}

@article {beauville-laszlo,
    AUTHOR = {Beauville, Arnaud and Laszlo, Yves},
     TITLE = {Un lemme de descente},
   JOURNAL = {C. R. Acad. Sci. Paris S\'er. I Math.},
  FJOURNAL = {Comptes Rendus de l'Acad\'emie des Sciences. S\'erie I.
              Math\'ematique},
    VOLUME = {320},
      YEAR = {1995},
    NUMBER = {3},
     PAGES = {335--340},
      ISSN = {0764-4442},
   MRCLASS = {14J60},
  MRNUMBER = {1320381},
}

@article{mcninch,
  title={Nilpotent orbits over ground fields of good characteristic},
  author={McNinch, George J},
  journal={Mathematische Annalen},
  volume={329},
  pages={49--85},
  year={2004},
  publisher={Springer}
}

@misc{sl_endoscopy,
      title={Symmetric varieties for endoscopic groups}, 
      author={Leslie, Spencer},
      year={2024},
      eprint={2401.09156},
      archivePrefix={arXiv},
      primaryClass={math.NT},
      note={arXiv preprint, arXiv: 2401.09156},
      url={https://arxiv.org/abs/2401.09156}, 
}

@article{kl,
  title={Fixed point varieties on affine flag manifolds},
  author={Kazhdan, David and Lusztig, George},
  journal={Israel Journal of Mathematics},
  volume={62},
  number={2},
  pages={129--168},
  year={1988},
  publisher={Springer}
}

@misc{HML,
      title={Relative {D}olbeault Geometric {L}anglands via the Regular Quotient}, 
      author={Thomas Hameister and Zhilin Luo and Benedict Morrissey},
      year={2024},
      eprint={2409.15691},
      archivePrefix={arXiv},
      primaryClass={math.AG},
      note={accepted at Selecta Mathematica},
      url={https://arxiv.org/abs/2409.15691}, 
}

@article{HM,
title = {The {H}itchin fibration for symmetric pairs},
journal = {Advances in Mathematics},
volume = {482},
pages = {110560},
year = {2025},
issn = {0001-8708},
doi = {https://doi.org/10.1016/j.aim.2025.110560},
author = {Thomas Hameister and Benedict Morrissey}
}

@book{griffin,
  author    = {Wang, X. Griffin},
  title     = {Multiplicative {H}itchin Fibrations and the Fundamental Lemma},
  series    = {Annals of Mathematics Studies},
  publisher = {Princeton University Press},
  address   = {Princeton, NJ},
  year      = {2026}
}

@article{spencer_annals,
  author    = {Leslie, Spencer},
  title     = {The endoscopic fundamental lemma for unitary {F}riedberg--{J}acquet periods},
  journal   = {Annals of Mathematics},
  volume    = {201},
  number    = {3},
  pages     = {551--645},
  year      = {2025},
  doi       = {10.4007/annals.2025.201.3.1}
}

@article{bzsv,
  title={Relative {L}anglands duality},
  author={Ben-Zvi, David and Sakellaridis, Yiannis and Venkatesh, Akshay},
  journal={arXiv preprint, arXiv:2409.04677},
  year={2023}
}

@article {artin,
    AUTHOR = {Artin, M.},
     TITLE = {Algebraic approximation of structures over complete local
              rings},
   JOURNAL = {Inst. Hautes \'Etudes Sci. Publ. Math.},
  FJOURNAL = {Institut des Hautes \'Etudes Scientifiques. Publications
              Math\'ematiques},
    NUMBER = {36},
      YEAR = {1969},
     PAGES = {23--58},
      ISSN = {0073-8301,1618-1913},
   MRCLASS = {14.05},
  MRNUMBER = {268188},
MRREVIEWER = {H.\ Kurke},
       URL = {http://www.numdam.org/item?id=PMIHES_1969__36__23_0},
}

@misc{griffin2,
  author       = {Wang, X. Griffin and Zhang, Zhiyu},
  title        = {Relative Fundamental Lemmas for Spherical {H}ecke Algebras and Multiplicative {H}itchin Fibrations: the {J}acquet--{R}allis Case},
  year         = {2024},
  eprint       = {2408.15155},
  archivePrefix= {arXiv},
  primaryClass = {math.NT},
  note         = {accepted by Compositio Mathematica}
}

@article{bez,
  author  = {Bezrukavnikov, Roman},
  title   = {The dimension of the fixed point set on affine flag manifolds},
  journal = {Mathematische Research Letters},
  volume  = {3},
  number  = {2},
  year    = {1996},
  pages   = {185--189}
}

@article{Romagny1,
  author  = {Romagny, Matthieu},
  title   = {Group actions on stacks and applications},
  journal = {Michigan Mathematical Journal},
  year    = {2005},
  volume  = {53},
  number  = {1},
  pages   = {209--236},
  doi     = {10.1307/mmj/1114021093}
}

@misc{Romagny2,
  author        = {Romagny, Matthieu},
  title         = {Algebraicity and smoothness of fixed point stacks},
  year          = {2022},
  eprint        = {2205.11114},
  archivePrefix = {arXiv},
  primaryClass  = {math.AG},
  doi           = {10.48550/arXiv.2205.11114}
}

@article {sv,
    AUTHOR = {Sakellaridis, Yiannis and Venkatesh, Akshay},
     TITLE = {Periods and harmonic analysis on spherical varieties},
   JOURNAL = {Ast\'erisque},
  FJOURNAL = {Ast\'erisque},
    NUMBER = {396},
      YEAR = {2017},
     PAGES = {viii+360},
      ISSN = {0303-1179,2492-5926},
      ISBN = {978-2-85629-871-8},
   MRCLASS = {22E50 (11F67)},
  MRNUMBER = {3764130},
MRREVIEWER = {Arnab\ Mitra},
}

@article {FJ,
    AUTHOR = {Friedberg, Solomon and Jacquet, Herv\'e},
     TITLE = {Linear periods},
   JOURNAL = {J. Reine Angew. Math.},
  FJOURNAL = {Journal f\"ur die Reine und Angewandte Mathematik. [Crelle's
              Journal]},
    VOLUME = {443},
      YEAR = {1993},
     PAGES = {91--139},
      ISSN = {0075-4102,1435-5345},
   MRCLASS = {11F70 (22E55)},
  MRNUMBER = {1241129},
MRREVIEWER = {Stephen\ Gelbart},
       DOI = {10.1515/crll.1993.443.91},
       URL = {https://doi.org/10.1515/crll.1993.443.91},
}

@article {GWZ,
    AUTHOR = {Groechenig, Michael and Wyss, Dimitri and Ziegler, Paul},
     TITLE = {Geometric stabilisation via {$p$}-adic integration},
   JOURNAL = {J. Amer. Math. Soc.},
  FJOURNAL = {Journal of the American Mathematical Society},
    VOLUME = {33},
      YEAR = {2020},
    NUMBER = {3},
     PAGES = {807--873},
      ISSN = {0894-0347,1088-6834},
   MRCLASS = {11S37 (11S80 14D24 14H60 20G40)},
  MRNUMBER = {4127904},
MRREVIEWER = {Aprameyo\ Pal},
       DOI = {10.1090/jams/948},
       URL = {https://doi.org/10.1090/jams/948},
}

@article {BBD,
    AUTHOR = {Beilinson, Alexander and Bernstein, Joseph and Deligne, Pierre
              and Gabber, Ofer},
     TITLE = {Faisceaux pervers},
   JOURNAL = {Ast\'erisque},
  FJOURNAL = {Ast\'erisque},
      YEAR = {2018},
    NUMBER = {100},
     PAGES = {vi+180},
      ISSN = {0303-1179,2492-5926},
      ISBN = {978-2-85629-878-7},
   MRCLASS = {32S30 (14F08 18G80 19F27 32S60 55N33)},
  MRNUMBER = {4870047},
}

@incollection {CHL,
    AUTHOR = {Cluckers, Raf and Hales, Thomas and Loeser, Fran\c cois},
     TITLE = {Transfer principle for the fundamental lemma},
 BOOKTITLE = {On the stabilization of the trace formula},
    SERIES = {Stab. Trace Formula Shimura Var. Arith. Appl.},
    VOLUME = {1},
     PAGES = {309--347},
 PUBLISHER = {Int. Press, Somerville, MA},
      YEAR = {2011},
      ISBN = {978-1-57146-227-5},
   MRCLASS = {22E35 (14E18)},
  MRNUMBER = {2856374},
}

@article{HelminckWang,
  author  = {Helminck, Aloysius G. and Wang, S. P.},
  title   = {On rationality properties of involutions
             of reductive groups},
  journal = {Advances in Mathematics},
  volume  = {99},
  number  = {1},
  pages   = {26--96},
  year    = {1993}
}

\end{document}